\documentclass{article}

 \PassOptionsToPackage{numbers,sort&compress}{natbib}

\usepackage[main, final]{neurips_2026}

\usepackage[T1]{fontenc}    
\usepackage{url}            
\usepackage{booktabs}       
\usepackage{amsfonts}       
\usepackage{nicefrac}       
\usepackage{microtype}      
\usepackage{xcolor}    
\usepackage{cancel}
\usepackage{mathrsfs}
\newcommand{\CMscr}{}
\let\CMscr=\mathscr
\usepackage{yfonts}
\usepackage[mathscr]{euscript}

\usepackage{xparse}

\usepackage{amsmath, amsthm, amssymb}

\usepackage{nccmath}
\usepackage{bbold}
\usepackage{booktabs}
\usepackage{empheq}
\usepackage{pifont}
\usepackage{enumitem}
\usepackage{graphicx} 
\usepackage{wrapfig}
\usepackage{empheq}
\usepackage{titlesec}
\usepackage{floatrow}
\usepackage{float}

\usepackage{thmtools}
\usepackage{thm-restate}

\usepackage{nicefrac}
\usepackage{booktabs}
\usepackage{multirow}
\usepackage{rotating}
\usepackage{bm}
\usepackage{tabularx}

\usepackage{tikz,pgfplots}
\usetikzlibrary{shapes}
\usetikzlibrary{decorations.markings}
\usetikzlibrary{plotmarks}
\usetikzlibrary{patterns,patterns.meta}
\usetikzlibrary{hobby} 
\usetikzlibrary{arrows.meta} 
\tikzset{>={Latex[length=4,width=4]}} 
\usetikzlibrary{calc,intersections,decorations.markings}
\usetikzlibrary{decorations.pathreplacing,%
                    calligraphy
                    }
\usepackage{siunitx}
\usepackage{ctable}

\usepackage{color, colortbl}
\usepackage{arydshln}
\definecolor{mydarkblue}{rgb}{0,0.08,0.45}
\usepackage[colorlinks=true,
    linkcolor=mydarkblue,
    citecolor=mydarkblue,
    filecolor=mydarkblue,
    urlcolor=mydarkblue,
    pdfview=FitH]{hyperref}
    
\newcommand{\cut}[1]{}
\usepackage{empheq}

\definecolor{myteal}{RGB}{27,158,119}
\definecolor{myorange}{RGB}{217,95,2}
\definecolor{myred}{RGB}{231,41,138}
\definecolor{mypurple}{RGB}{152,78,163}
\definecolor{myblue}{RGB}{55,126,184}
\definecolor{mygreen}{RGB}{0,100,0}
\definecolor{myroyalblue}{HTML}{4169E1}
\definecolor{mygrey}{RGB}{100,100,100}

\definecolor{PIIcolor}{HTML}{4e79a7}
\definecolor{PIcolor}{HTML}{7D476e}
\definecolor{PIVcolor}{HTML}{9a9c07}
\definecolor{PIIIcolor}{HTML}{337B29}
\definecolor{PVcolor}{HTML}{B5443B}

\definecolor{FPPcolor}{HTML}{7D476e}
\definecolor{FACcolor}{HTML}{4e79a7}
\definecolor{F0color}{HTML}{cc6805}
\definecolor{KPPcolor}{HTML}{337B29}

\definecolor{region1}{HTML}{CAC7FF}
\definecolor{region2}{rgb}{0.683,0.9335,0.9726}
\definecolor{region2b}{HTML}{66dde4}
\definecolor{region3a}{rgb}{0.407,0.764,0.686}
\definecolor{region3b}{rgb}{0.729411, 1.0, 0.7882}
\definecolor{region3}{rgb}{0.729411, 1.0, 0.7882}
\definecolor{region4a}{rgb}{1.0,0.8745,0.729411}
\definecolor{region4b}{rgb}{1.0,0.7019,0.729411}
\definecolor{region1c}{HTML}{E4D3FC}
\definecolor{region1b}{HTML}{FEE0F9}

\usepackage{cleveref}

\newtheorem{definition}{Definition}[section]

\newtheorem{proposition}{Proposition}[section]
\newtheorem{lemma}{Lemma}[section]
\newtheorem{theorem}{Theorem}[section]
\newtheorem{remark}{Remark}[section]
\newtheorem{corollary}{Corollary}[section]

\usepackage[framemethod=tikz]{mdframed}
\usepackage{subcaption}

\definecolor{myGreen}{HTML}{90EE90}

\global\mdfdefinestyle{exampledefault_1}{%
outerlinewidth=0pt,innerlinewidth=0pt,
roundcorner=5pt,backgroundcolor=white,
leftmargin=0.2cm,rightmargin=0.2cm
}

\usepackage[framemethod=tikz]{mdframed}

\global\mdfdefinestyle{exampledefault}{%
outerlinewidth=0pt,innerlinewidth=0pt,
roundcorner=5pt,backgroundcolor=myGreen,
leftmargin=0.2cm,rightmargin=0.2cm
}

\definecolor{myblue}{HTML}{D2E4FC}
\definecolor{Gray}{gray}{0.92}

\let\Im\relax
\DeclareMathOperator{\Im}{Im}

\usetikzlibrary{decorations.markings,positioning}

\newtheorem{assumption}{Assumption}
\crefname{assumption}{Assumption}{Assumptions}
\Crefname{assumption}{Assumption}{Assumptions}

\crefname{limitstatement}{Limit Statement}{Limit Statements}
\Crefname{limitstatement}{Limit Statement}{Limit Statements}

\makeatletter
\newcommand{\vast}{\bBigg@{4}}
\newcommand{\Vast}{\bBigg@{5}}

\usepackage[utf8]{inputenc} 
\usepackage[T1]{fontenc}    
\usepackage{hyperref}       
\usepackage{url}            
\usepackage{booktabs}       
\usepackage{amsfonts}       
\usepackage{nicefrac}       
\usepackage{microtype}      
\usepackage{xcolor}         

\title{Phases of \Muon: When \Muon  Eclipses \SignSGD }

\author{%
  Elliot Paquette\thanks{Email: \texttt{elliot.paquette@mcgill.ca} and Website: \url{https://elliotpaquette.github.io}} \\
  McGill University\\
  \And
  Noah Marshall\\
  McGill University\\
  \And
  Lucas Benigni\\
  Université de Montréal\\
  \AND 
  Guangyuan Wang\\
  McGill University\\
  \And
  Atish Agarwala\\
  Google DeepMind\\
  \And
  Courtney Paquette\\
  McGill University\\
}

\newcommand{\EE}{\operatorname{\mathbb{E}}}
\newcommand{\R}{\mathbb{R}}
\newcommand{\T}{\top}
\newcommand{\F}{\mathrm{F}}
\DeclareMathOperator{\Tr}{Tr}
\DeclareMathOperator{\Id}{Id}

\DeclareMathOperator{\rank}{rank}
\DeclareMathOperator{\sign}{sign}
\newcommand{\defeq}{\stackrel{\mathrm{def}}{=}}
\newcommand{\norm}[1]{\|#1\|}

\newcommand{\gin}{\gamma_\mathrm{in}}
\newcommand{\gout}{\gamma_\mathrm{out}}

\newcommand{\Nin}{{N_\mathrm{in}}}
\newcommand{\Nout}{{N_\mathrm{out}}}

\newcommand{\lr}{\eta}

\renewcommand{\vec}[1]{#1}

\newcommand{\abs}[1]{|#1|}

\DeclareMathOperator{\diag}{diag}

\newcommand{\tensor}{\otimes}

\newcommand{\Xin}{x_\mathrm{in}}
\newcommand{\Xout}{x_\mathrm{out}}
\newcommand{\Sigmain}{{\Sigma_\mathrm{in}}}
\newcommand{\Sigmaout}{{\Sigma_\mathrm{out}}}

\newcommand{\Idin}{I_\mathrm{in}}
\newcommand{\Idout}{I_\mathrm{out}}
\newcommand{\Wstar}{W^\star}
\newcommand{\Wt}{W_t}
\newcommand{\Dt}{\Delta_t}
\newcommand{\U}{U}
\newcommand{\uve}{u}
\newcommand{\Gt}{G_{t+1}}
\newcommand{\Gtilde}{\widetilde{G}_{t+1}}
\newcommand{\At}{\vec{A}_t}     

\newcommand{\cR}{\CMscr{R}}  
\newcommand{\cD}{\CMscr{D}}  

\newcommand{\fQ}{\mathscr{Q}}      
\newcommand{\fRisk}{\mathscr{R}}    
\newcommand{\fri}{\fQ} 

\newcommand{\dif}{\mathrm{d}}

\newcommand{\bmom}{\beta}                        
\newcommand{\ridge}{\varrho}                     
\newcommand{\ridgephys}{\varrho^{\circ}}         
\newcommand{\Qridge}[1]{Q_{#1}}                  
\newcommand{\Mt}{M_t}                            
\newcommand{\emaw}[1]{w_{#1}}                    
\newcommand{\sw}[1]{s_{#1}}                      

\newcommand{\texp}{\varsigma}                    
\newcommand{\resscale}{\kappa}                   

\newcommand{\Ahat}{\widehat{A}_t}                

\newcommand{\Jfun}{J}                            
\newcommand{\Fw}{F_{\bmom}}                      
\newcommand{\Lw}{L_{\bmom}}                      

\newcommand{\Rbar}{\overline{R}}                 
\newcommand{\Rtbar}{\overline{\widetilde R}}     
\newcommand{\Gtildebar}{\overline{\widetilde G}} 
\newcommand{\volraw}{\mathfrak{v}^{(0)}}         

\newcommand{\macrostep}{h}                       
\newcommand{\macrotime}{\tau}                    
\newcommand{\clock}{\phi}                        
\newcommand{\BL}{\mathrm{BL}}                    

\usepackage{xspace}
\usepackage{textcomp}

\newcommand{\Adam}{{\scshape Adam}\xspace}

\newcommand{\Muon}{{\scshape Muon}\xspace}
\newcommand{\SPEC}{{\scshape SPEC}\xspace}

\newcommand{\SignSGD}{{\scshape SignSGD}\xspace}
\newcommand{\SignSVD}{{\scshape SignSVD}\xspace}

\AddToHook{env/thebibliography/begin}{\raggedright}
\makeatletter
\renewcommand*\l@subsection{\@dottedtocline{2}{1.5em}{3.2em}}
\makeatother

\makeatletter
\long\def\muon@cref@first#1#2\@nil{#1}
\long\def\muon@cref@second#1#2#3\@nil{#2}
\def\cref@getref#1#2{%
  \expandafter\let\expandafter#2\csname r@#1@cref\endcsname%
  \expandafter\expandafter\expandafter\def%
    \expandafter\expandafter\expandafter#2%
    \expandafter\expandafter\expandafter{%
      \expandafter\muon@cref@first#2\@nil}}
\def\cpageref@getref#1#2{%
  \expandafter\let\expandafter#2\csname r@#1@cref\endcsname%
  \expandafter\expandafter\expandafter\def%
    \expandafter\expandafter\expandafter#2%
    \expandafter\expandafter\expandafter{%
      \expandafter\muon@cref@second#2\@nil}}
\makeatother
\AddToHook{cmd/appendix/after}{%
  \crefalias{section}{appendix}%
  \crefalias{subsection}{subappendix}%
  \crefalias{subsubsection}{subsubappendix}}
\AddToHook{env/definition/begin}{\crefalias{section}{definition}}
\AddToHook{env/parametrization/begin}{\crefalias{section}{parametrization}}
\AddToHook{env/example/begin}{\crefalias{section}{example}}
\AddToHook{env/proposition/begin}{\crefalias{section}{proposition}}
\AddToHook{env/lemma/begin}{\crefalias{section}{lemma}}
\AddToHook{env/theorem/begin}{\crefalias{section}{theorem}}
\AddToHook{env/remark/begin}{\crefalias{section}{remark}}
\AddToHook{env/corollary/begin}{\crefalias{section}{corollary}}
\AddToHook{env/claim/begin}{\crefalias{section}{claim}}
\AddToHook{env/conjecture/begin}{\crefalias{section}{conjecture}}
\AddToHook{env/assumption/begin}{\crefalias{section}{assumption}}
\AddToHook{env/limitstatement/begin}{\crefalias{section}{limitstatement}}

\newcommand{\skpointer}[2]{\par\smallskip{\small\itshape\raggedright\noindent
  Full statement and proof: \cref{S-#1} of the supplement. #2.\par}\smallskip}

\newenvironment{sketch}{\par\noindent\textit{Proof sketch.}\ }{\hfill$\diamond$\par\medskip}

\begin{document}

\maketitle

\begin{abstract}
\Muon and related spectral optimizers have recently demonstrated strong empirical performance as scalable stochastic methods, often outperforming \Adam. Yet their behavior remains poorly understood. We analyze stochastic \Muon, which orthogonalizes an exponential moving average of minibatch gradients, on a high-dimensional matrix-valued least squares problem, and we compare it with stochastic \SignSGD, a simple baseline comparable to \Adam. We derive explicit deterministic dynamics that provide a tractable framework for studying learning behavior. We show that \Muon has non-linear sensitivity to the effective batch $B_{\rm eff}=B(1+\bmom)/(1-\bmom)$. With exact whitening and small $B_{\rm eff}$ \Muon fails to resolve smaller eigenmodes, slowing down convergence. Momentum and finite Newton--Shulz iterations can act as a ridge which lessen this effect. This behavior is in contrast with \SignSGD which has no transition, leading to different optimal learning rates, but with matching convergence rate up to a constant factor in the isotropic setting. For anisotropic data, large batch setting \Muon performs a square-root preconditioning with respect to the data covariance spectrum while \SignSGD provides no preconditioning for generic covariances. We analyze time to target loss in a power law covariance model and derive a phase diagram in the data exponent $\alpha$ and target exponent $\texp$. \Muon is favored for steep covariance spectra with rich targets ($\alpha>2$, $\texp<\alpha+1$), \SignSGD is favored for flat spectra ($\alpha<1$), and in other regions the winner depends on how the target loss scales with dimension. Our results shed light on why replacing \Adam with \Muon is not trivial even from a pure optimization perspective.
\end{abstract}

\section{Introduction}

Practical training in machine learning is dominated by stochastic first-order methods with diagonal preconditioning, in particular, robust, scalable variants of \Adam.
A long-standing research goal is to develop efficient and robust \emph{non-diagonal} preconditioners
which exploit the matrix structure of neural network parameters \citep{martens2015optimizing,gupta2018shampoo,shi2023distributed,vyas2024soap}.
Recently the \Muon optimizer \citep{jordan2024muon} has gained traction as a simple yet scalable
optimizer in this class which can outperform \Adam in many settings \cite{wen2025fantastic,shah2025practical,jordan2024muon,liu2025muon}.
\Muon efficiently approximates spectral whitening of matrix-valued gradients
via Newton--Schulz iterations.
However, deploying \Muon can be difficult; the relationship between the training dynamics and
batch size, dimension, and data structure remain poorly understood. This leads to a natural
question:
\emph{under what conditions does \Muon outperform \Adam}?


To answer this question, we perform an analysis of the learning dynamics of the
family of \emph{stochastic spectral optimizers}, particularly \Muon with
momentum. We focus on a motivated high-dimensional
setting where we can derive deterministic equations for the stochastic loss trajectories. This allows for explicit computation
of the scaling of optimal hyperparameters with batch size and dimension. This also reveals a rich phase diagram that quantifies the performance of non-diagonal preconditioners and characterizes the joint conditions on the data and optimization setup under which these methods outperform Adam-like methods. Taken together, our results suggest new directions
for future empirical and theoretical study towards improving spectral optimizers.

\textbf{Related Work.} The original motivation for \Muon \cite{jordan2024muon} stems from steepest descent with respect to the spectral norm \cite{bernstein2024old}. The method is closely related to stochastic spectral descent \cite{carlson2015stochastic} and orthogonalized gradient methods \cite{tuddenham2022orthogonalising}, and can also be interpreted as a form of basis-aware \SignSGD \cite{bernstein2018signsgd}. Its strong empirical performance \cite{wen2025fantastic, shah2025practical, jordan2024muon,liu2025muon} has led to a number of recent variants \cite{ma2024swan, liu2025cosmos, riabinin2025gluon, pethick2025training, amsel2025polar, deepseekai2026deepseekv4}. On the theoretical side, several works analyze \Muon and related methods using worst-case complexity arguments, establishing $\mathcal{O}(1/\sqrt{T})$ convergence rates comparable to SGD \cite{shen2025convergence,shulgin2025beyond,chen2025muon,kim2026convergence}. These analyses connect \Muon to frameworks such as inexact linear minimization oracles \cite{shulgin2025beyond} and broader algorithmic families like Lion-$\mathcal{K}$ \cite{chen2025muon}, while also characterizing the role of approximate polarization \cite{kim2026convergence}. Complementary regret guarantees are provided in \cite{jiang2026adaptive}.

A growing body of work studies \Muon and the exact whitening variant \SignSVD in simplified settings. Gradient-flow analyses yield explicit loss dynamics for \SignSVD and enable comparisons with gradient descent \cite{vasudeva2025muon, peyre2026muon}, while other works highlight nontrivial behavior even on quadratic objectives \cite{gonon2026insights}. High-dimensional analyses remain limited: \cite{wang2025high} derives risk dynamics for \SignSVD in isotropic least-squares settings but does not explore their implications, while \cite{du2026newton} studies the dynamics of \Muon and \SignSVD in full-batch isotropic least-squares settings. Beyond least squares, implicit bias has been studied in linear classification \cite{fan2025implicit, tsilivis2024flavors}, and recent work connects \Muon to attention via associative memory models, including results on critical batch size scaling \cite{kim2026sharp, li2026muon, wang2025muon}. Several papers also attempt to explain the empirical success of spectral methods through single-step analyses \cite{davis2025spectral, su2025isotropic}, though these do not reliably predict full training behaviour \cite{gonon2026insights}. 
We instead provide explicit stochastic training dynamics over the full course of training.
This extends the deterministic-equivalent / homogenization line of analysis developed for SGD and \SignSGD on scalar-valued least squares \cite{paquette2022homogenization,xiao2025exact,paquette20244+} to a matrix-valued setting in which spectral preconditioning is meaningful.
See Sec.~\ref{sec:additional_related_work} for an extended discussion of background works.

\begin{wrapfigure}[16]{r}{0.39\textwidth}
  \centering
  \vspace{-0.65cm}
\includegraphics[width=.95\linewidth]{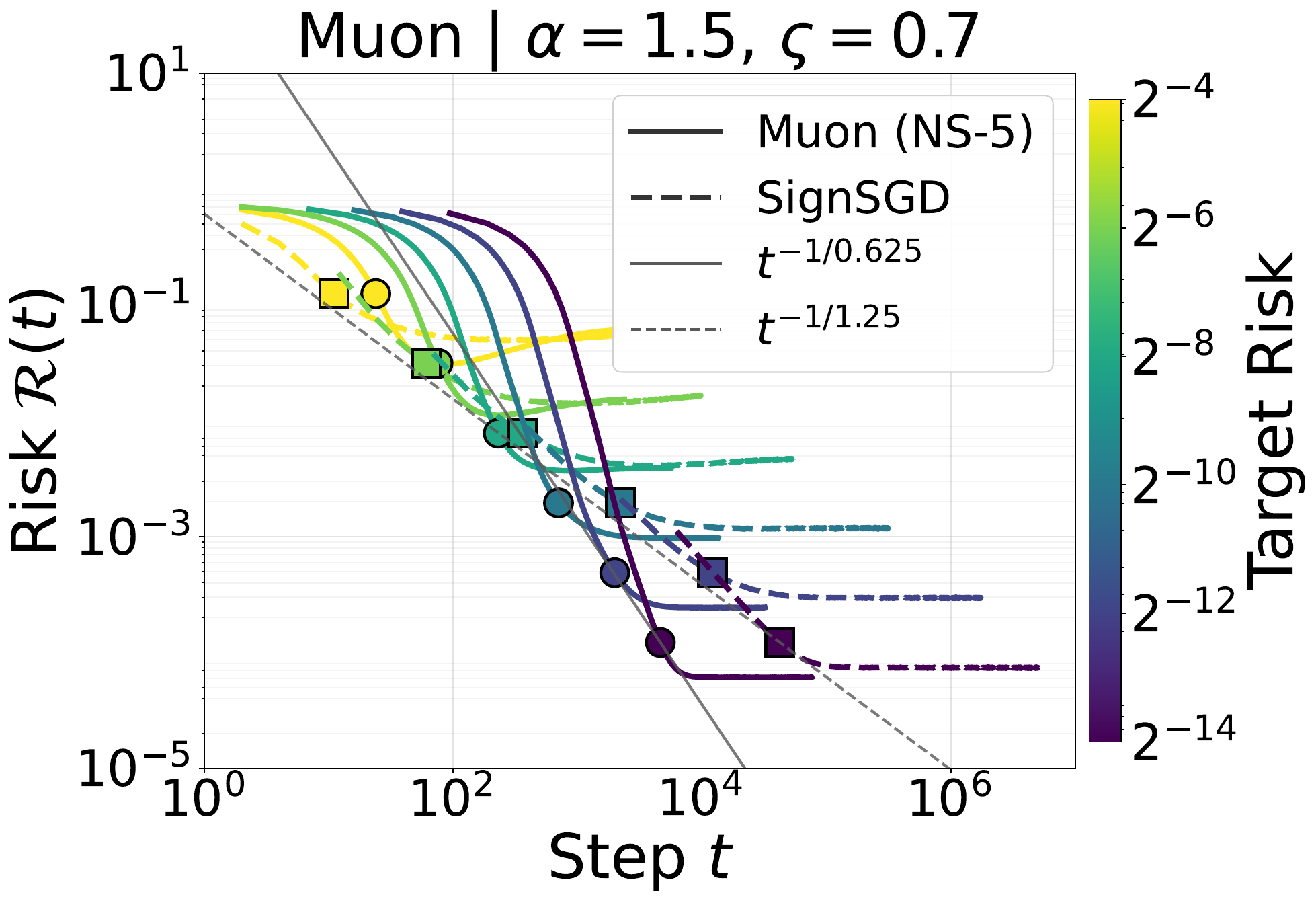}
  \vspace{-0.0em}
  \caption{\small \textbf{Random-matrix theory predicts \Muon and \SignSGD loss curves.} Power-law data ($\alpha=1.5$, $\texp=0.7$, $N=1024$, $B=2048$) in the mixed phase~C of Fig.~\ref{fig:tte-phase-diagram}, where \Muon wins at this dimension once $\epsilon$ is small. Solid: mean of 16 runs of \Muon with 5 Newton--Schulz steps; gray: predicted exponents (Sec.~\ref{sec:timetoepsilon}).}
  \vspace{-0.3cm}
  \label{fig:lead-phase-a}
\end{wrapfigure}

\textbf{Main contributions.}
We introduce a matrix-valued quadratic optimization problem that captures key aspects of neural-network training, in particular the outer-product structure of fully connected layer gradients (Sec.~\ref{sec:model_set_up}). This setting enables the study of a general family of \emph{stochastic spectral optimizers}. The algorithm we analyze is stochastic \Muon with exponential-moving-average momentum; its case without momentum is stochastic \SignSVD, the exact spectral whitening of the gradient, which goes back to stochastic spectral descent \cite{carlson2015stochastic}. The framework also covers diagonal methods such as stochastic \SignSGD \cite{bernstein2018signsgd}, a simple baseline comparable to \Adam (Sec.~\ref{sec:spectral_algs}). 

We derive deterministic learning dynamics for these algorithms via a random matrix theory analysis. Specifically, we show that the stochastic risk is tracked by a deterministic function $\fRisk(t)$, characterized through drift and volatility kernels; this is a theorem on a finite horizon in the proportional limit (Sec.~\ref{sec:deterministic_dynamics}, Thm.~\ref{thm:deterministic-dynamics}). This framework enables us to predict risk trajectories without running the stochastic algorithm and to explicitly relate performance to batch size $B$, matrix dimension $(N \times N)$, and optimizer hyperparameters (e.g., non-constant learning rate $\eta$), thereby providing a precise characterization of scaling behavior (see Fig.~\ref{fig:iso-risk-constlr}). \textit{\textbf{The dynamics cover a broad class of spectral optimizers as well as the fully anisotropic setting.}} We further leverage this framework to uncover a rich set of phenomena:
\begin{itemize}[leftmargin=*]
\item \textbf{Critical batch effect for \Muon (Sec.~\ref{sec:isotropic} \& Sec.~\ref{sec:half-aniso}).} A minibatch gradient resolves only part of the parameter eigenspace, and \Muon acts as a \emph{square-root preconditioner} only on the resolved eigenspaces. Momentum enlarges the effective batch $B_{\rm eff}=B(1+\bmom)/(1-\bmom)$, and as $B_{\rm eff}$ increases, more eigenspaces are resolved, leading to a progressive improvement in convergence. A \textit{critical transition} occurs at $B_{\rm eff}\asymp N$, where all eigenspaces are resolved and \Muon preconditions the full spectrum.


\item \textbf{With isotropic data, \Muon has a larger learning rate and the same convergence rate up to a constant (Sec.~\ref{sec:isotropic}, Figs.~\ref{fig:iso-lr}, \ref{fig:isotropic-speedup}).} The optimal learning rate of \Muon exceeds \SignSGD's by a factor of order $\sqrt{N}$ for $B_{\rm eff}\gg N$, recovering and refining the empirical $\sqrt{N}$ \textsc{Moonshot} rule of \cite{liu2025muon}, and of order $\sqrt B$ for $B_{\rm eff}\ll N$ ($N/\sqrt B$ at $\bmom=0$). 
For $B_{\rm eff}\gg N$, momentum is a rescaling of the learning rate and for $B_{\rm eff}\ll N$, it sets the bias of \Muon to $1/(2\sqrt\pi)$. 
However at the optimal learning rate there is no speedup from \Muon on isotropic data.
\item \textbf{Newton--Shulz is a useful regularizer.}
We quantitatively connect finite Newton--Schulz steps to a ridge-regularized whitening, which can be advantageous at small batch: with momentum and $B_{\rm eff}<N$, a few Newton--Schulz steps outperform exact whitening (Fig.~\ref{fig:isotropic-speedup}, App.~\ref{sec:app-nsridge}).
\item \textbf{Anisotropic power-law data reveals a phase diagram (Sec.~\ref{sec:half-aniso} \& Sec.~\ref{sec:timetoepsilon}).} A power-law covariance with data exponent $\alpha$ and target exponent $\texp$ gives a solvable setting in which we can tell which problems favor \Muon and which favor \SignSGD: comparing the time each method needs to reach a target loss divides the $(\alpha,\texp)$-plane into \Muon-favored and \SignSGD-favored phases (Figs.~\ref{fig:lead-phase-a}, \ref{fig:tte-phase-diagram}). Spectral preconditioning is not uniformly advantageous, even at large batch sizes.


\end{itemize}


\begin{figure}[t]
\centering
\begin{minipage}[c]{0.62\textwidth}
\renewcommand{\arraystretch}{1.55}
\setlength{\tabcolsep}{4pt}
\footnotesize
\begin{tabular}{@{}>{\raggedright\arraybackslash}p{0.59\linewidth} >{\raggedright\arraybackslash}p{0.37\linewidth}@{}}
\toprule
\textbf{Fixed-point equation} & \textbf{Random-matrix meaning} \\
\midrule
\rowcolor{blue!8}
$\displaystyle \sigma = \frac{1}{B}\sum_i \frac{\mu_i}{\tilde{s}_1\,\mu_i - z}$
& output-side resolvent trace $\tfrac{1}{B}\operatorname{Tr}\!\big(R(z)\Sigmaout\big)$ \\
$\displaystyle \tilde\sigma = \frac{1}{B}\sum_j \frac{\lambda_j}{s_1\,\lambda_j - z}$
& input-side resolvent trace $\tfrac{1}{B}\operatorname{Tr}\!\big(\tilde R(z)\Sigmain\big)$ \\
\rowcolor{blue!8}
$\tilde{s}_1\,\sigma = s_1\,\tilde\sigma = q\,F_\bmom(q)$, $\ q=-z\,\sigma\,\tilde\sigma$
& Gaussian closure over the moving-average weights \\
\hdashline
\multicolumn{2}{@{}c@{}}{$F_\bmom(q)=\sum_{j\ge0}w_j^2J(w_j^2q),\quad L_\bmom(q)=\sum_{j\ge0}w_jJ(w_j^2q)$} \\
\multicolumn{2}{@{}c@{}}{$J(x)=\EE_{V\sim\mathcal N(0,1)}\!\big[V^2/(1+xV^2)\big],\quad w_j=(1-\bmom)\bmom^j$} \\
\bottomrule
\end{tabular}
\end{minipage}\hfill
\begin{minipage}[c]{0.36\textwidth}
\centering
\includegraphics[width=\linewidth]{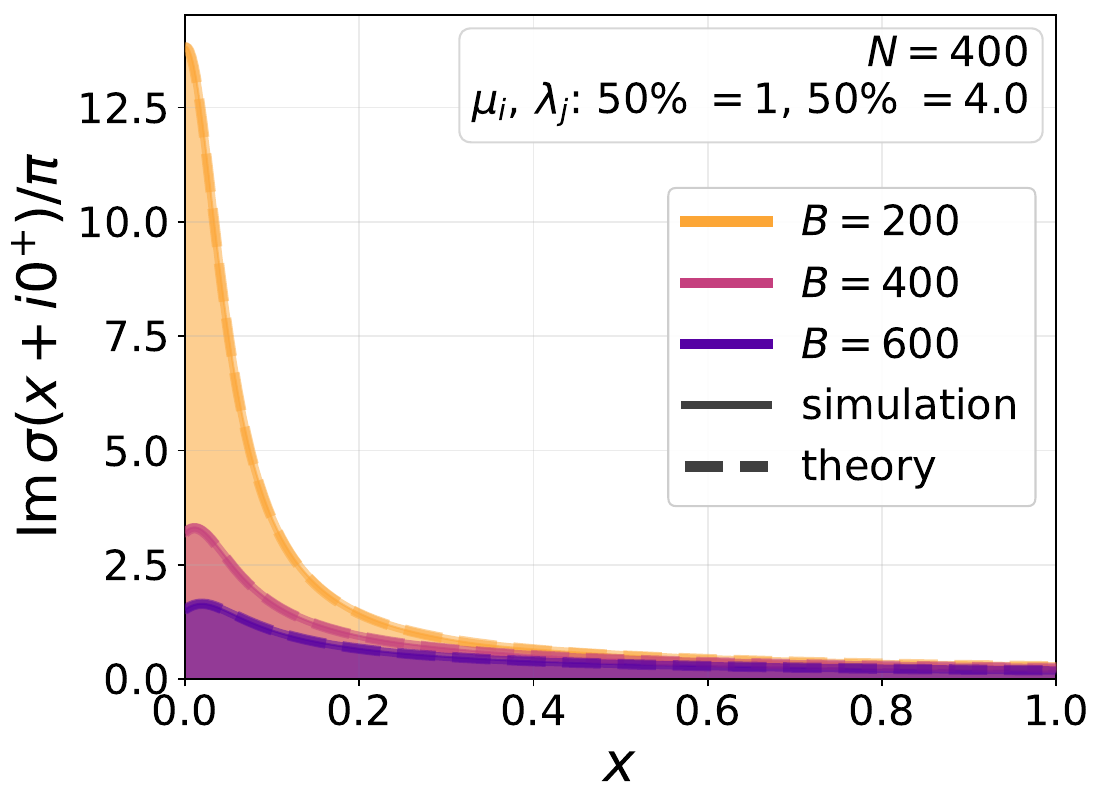}
\end{minipage}
\caption{Self-consistent fixed-point system for the deterministic equivalents of the resolvents $R(z) = (M_tM_t^\T/(2\fRisk) - z\Id_{\Nout})^{-1}$ and $\tilde R(z) = (M_t^\T M_t/(2\fRisk) - z\Id_{\Nin})^{-1}$ of the moving-average buffer, for $z<0$ (\cref{def:app-de-branch,thm:app-de-branch}); at $\bmom=0$, $F_0=L_0=J$. \emph{Left:} the governing equations. \emph{Right:} simulation validation at $\bmom=0$ of $\Im\,\sigma(x+i0^+)/\pi$ against the fixed-point solution for a doubly anisotropic spectrum. 50 random data matrices drawn at each batch size $B$, $N=400$, imaginary part of $z$ is $10^{-3}$.}
\label{fig:fp-system}
\end{figure}

\subsection{Model setup} \label{sec:model_set_up}
We study matrix-valued linear regression with $\Nout \times \Nin$ parameters $W$, a fixed teacher $\Wstar$ of the same shape, single-sample loss $\tfrac12\langle\Xout\tensor\Xin,\,\Wstar - W\rangle^2$, population risk
\begin{equation}\label{eqn:risk}
\cR(W) \;\defeq\; \tfrac12\,\EE\!\left[\langle\Xout\tensor\Xin,\,W - \Wstar\rangle^2\right],
\end{equation}
and no label noise; we write $\cR(t) \defeq \cR(\Wt)$. This is a minimal toy model that
retains the gradient structure spectral preconditioners are designed to exploit.
The single sample gradient
$(\Xout\tensor\Xin)\,\langle\Xout\tensor\Xin,\,W-\Wstar\rangle$ has exactly the rank-one outer-product structure of a single linear layer's gradient under backprop, with $\Xout$ playing the role of the upstream error and $\Xin$ the activation (see Sec.~\ref{sec:formal-assumptions} for longer discussion). The data $\Xin,\Xout$ are independent Gaussian with covariances $\Sigmain,\Sigmaout$; as we will see, different
choices in the data covariance will dramatically change the dynamics.
The proportional-regime assumptions (and aspect ratios) controlling $\Nin,\Nout,B$ and $\|\Sigmain\|_{\mathrm{op}},\|\Sigmaout\|_{\mathrm{op}}$ are formally defined in Assumption~\ref{assum:data-target} and Assumption~\ref{assum:proportional-regime}.

\subsection{Stochastic spectral algorithms.} \label{sec:spectral_algs}
The minibatch gradient $\Gt$ of \eqref{eqn:risk} is almost surely a rank-$\min\{\Nin, \Nout, B\}$ matrix where $B$ is the size of the minibatch. \emph{Stochastic spectral optimizers} process the singular values of a moving average of minibatch gradients through a function $\varphi:[0,\infty)\to\R$:
\begin{mdframed}[style=exampledefault_1]
\textbf{Algorithm {\scshape Muon}$_\varphi$.} Fix a momentum $\bmom\in[0,1)$, a learning rate $\lr>0$ and $M_{-1}=0$. With $\Delta_t \defeq \Wt - \Wstar$, for $t\ge0$,
\begin{gather}
\Gt \;\defeq\; \frac{1}{B}\sum_{i=1}^{B}(\Xout^{i}\tensor\Xin^{i})\,\langle\Xout^{i}\tensor\Xin^{i},\,\Delta_t\rangle, \label{eqn:minibatch-gradient}\\
M_t \;=\; \bmom M_{t-1}+(1-\bmom)\,\Gt,\qquad W_{t+1} \;=\; \Wt-\lr\,\varphi(M_tM_t^\T)\,M_t. \label{eqn:phi-update}
\end{gather}
\textbf{\Muon} is $\varphi(z)=z^{-1/2}$, the polar factor $\mathrm{polar}(M)=(MM^\T)^{-1/2}M=UV^\T$ of an SVD $M=U\Sigma V^\T$:
\begin{equation}
W_{t+1} \;=\; \Wt-\lr\,\mathrm{polar}(M_t). \label{eqn:sinsvd-update}
\end{equation}
\textbf{\SignSVD} is \Muon at $\bmom=0$. The \textbf{regularized polar map} is $\varphi(z)=(z+2\cR(t)\ridge)^{-1/2}$, with a normalized ridge $\ridge>0$.
\end{mdframed}
\SignSVD, the exact spectral whitening of the gradient, goes back to stochastic spectral descent \cite{carlson2015stochastic}. In practice the polar factor is computed by $K$ Newton--Schulz iterations on the Frobenius-normalized buffer \cite{jordan2024muon,amsel2025polar}. The regularized polar map is the mathematical simplification for which the dynamical theorem is proved (Thm.~\ref{thm:deterministic-dynamics}; Def.~\ref{def:app-setup-algorithm}), and a qualitative model of finite Newton--Schulz: $K$ iterations map a small singular value $\sigma$ to about $a^K\sigma$, with $a\approx3.44$ for the standard coefficients, an effective ridge of order $a^{-2K}$ (about $4\cdot10^{-6}$ at $K=5$); Thm.~\ref{thm:app-iso-fixed-ridge} and \ref{thm:app-iso-crossover} describe what a small ridge does. Sec.~\ref{sec:momentum-ablation} compares the momentum theory with simulations, and we validate predictions on \Muon (NS-5) numerically (e.g.\ Fig.~\ref{fig:lead-phase-a}). Note that recent work has introduced \Muon with \emph{hybrid Newton-Schulz} which results in the near exact \SignSVD whitening of singular values \cite{deepseekai2026deepseekv4}.\footnote{Early \Muon implementations aggressively normalize the gradient, prior to using a Newton-Schulz scheme that prioritizes rapid increase of singular values from $0$, at the cost of slower convergence to $1$. }

\textbf{\SignSGD.} We compare with stochastic \SignSGD \citep{bernstein2018signsgd}, a simple baseline comparable to \Adam, which acts \emph{entrywise} (not spectrally) on $\Gt$:
\begin{equation}\label{eq:main_signSGD}
W_{t+1} = \Wt - \eta\,\widehat{G}_{t+1}, \qquad (\widehat{G}_{t+1})_{ij} \defeq \sign\!\big((\Gt)_{ij}\big).
\end{equation}
\SignSGD coincides with \Adam at $\beta_1=\beta_2=0$, so comparing it against \SignSVD captures, to leading order, the structural difference between \Adam and \Muon. We write \emph{s-SGD} for \SignSGD, \emph{s-SVD} for \SignSVD, and \emph{\SPEC} for the general spectral family throughout.

\section{Learning curves and deterministic dynamics}
\label{sec:deterministic_dynamics}
We derive a deterministic function $\fRisk(t)$ that tracks the stochastic risk $\cR(\Wt)$ in the proportional limit (Thm.~\ref{thm:app-dyn-main}).
Working in a projected basis, we extend prior 
deterministic-equivalent / homogenization analyses for SGD and \SignSGD \cite{paquette2022homogenization,xiao2025exact} to \SPEC. \Muon and \SignSGD both admit closed deterministic risk recursions of the same form. They differ only through their drift $\mathfrak{d}_{ij}$ and volatility $\mathfrak{v}_{ij}$ kernels.

Let $(\mu_i,u_i)$ and $(\lambda_j,v_j)$ be the eigenpairs of $\Sigmaout$ and $\Sigmain$. Define the \emph{projected risk} $\CMscr{Q}_{ij}(t) \defeq (u_i^\T \Dt v_j)^2$, so that $\cR(\Wt) = \tfrac{1}{2}\sum_{ij}\mu_i\lambda_j\,\CMscr{Q}_{ij}(t)$. In the proportional limit, the projected risks $\CMscr{Q}_{ij}$ follow, in aggregate, a deterministic trajectory $\fQ_{ij}$ (Thm.~\ref{thm:deterministic-dynamics} below) governed by the universal recursion
\begin{equation}\label{eqn:fQ-recursion}
\fQ_{ij}(t{+}1) \;=\; \fQ_{ij}(t) \;-\; \frac{2\eta\,\textcolor{PIcolor}{\mathfrak{d}_{ij}(t)}}{\sqrt{\fRisk(t)}}\,\fQ_{ij}(t) \;+\; \eta^2\,\textcolor{PIIIcolor}{\mathfrak{v}_{ij}(t)},
\end{equation}
with $\fRisk(t) \defeq \tfrac12\sum_{ij}\mu_i\lambda_j\,\fQ_{ij}(t)$ the theory prediction for $\CMscr{R}$. \begin{theorem}[Deterministic dynamics, informal]\label{thm:deterministic-dynamics}
Let $\Nin,\Nout,B\to\infty$ proportionally, and run the algorithm of Sec.~\ref{sec:spectral_algs} with a fixed momentum $\bmom\in[0,1)$, a fixed normalized ridge $\ridge>0$ and step $\eta=\macrostep/B$. Then $\sup_{t\le BT_\star}\abs{\cR(\Wt)-\fRisk(t)}\to0$ in probability, where $T_\star>0$ depends only on $\macrostep$, $\ridge$, $\bmom$, the aspect ratios, the covariance bound and the initial risk.
\end{theorem}
\begin{wrapfigure}[16]{r}{0.4\textwidth}
  \vspace{-0.1em}
  \centering
  \includegraphics[width=\linewidth]{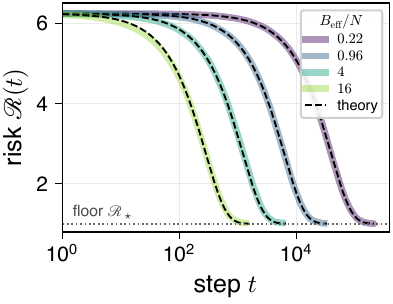}
  \vspace{-0.9em}
  \caption{\small \textbf{Theory tracks \Muon.} Isotropic data, $N=256$, $\bmom=0.9$, exact polar factor, constant $\eta$, floor $\fRisk_\star=1$. Thick: simulations (App.~\ref{sec:app-numerics}); dashed: the recursion \eqref{eqn:comparison-general-form}. Relative error below $2\%$.}
  \label{fig:iso-risk-constlr}
  \vspace{-1em}
\end{wrapfigure}
See Thm.~\ref{thm:app-dyn-main} for the precise form (hypotheses on the initial state, the spectral-measure topology and the error terms). The rest of the paper analyzes this deterministic equivalent at zero ridge ($\ridge\downarrow0$) as the proxy for \Muon. The system is closed in the high-dimensional limit. The \textcolor{PIcolor}{drift kernel} $\mathfrak{d}_{ij}$ and the \textcolor{PIIIcolor}{volatility kernel} $\mathfrak{v}_{ij}$ depend only on $(\mu_i,\lambda_j)$, the aspect ratios, $\bmom$ and the algorithm. They need no further information from $\Dt$. The kernels are computed in Sec.~\ref{sec:isotropic} and Sec.~\ref{sec:half-aniso}. For complete details, see Sec.~\ref{sec:app-setup} to \ref{sec:app-dyn} and Sec.~\ref{sec:signsgd}.

\textbf{Computing the kernels: a self-consistent fixed-point system.} For the spectral family, $\mathfrak{d}_{ij}$ and $\mathfrak{v}_{ij}$ are spectral integrals against the deterministic equivalent of the resolvent of the Gram matrix of the moving-average buffer. After rescaling spectral variables by the risk ($z \to 2\fRisk \cdot z$), the deterministic equivalent solves the self-consistent fixed-point system in Fig.~\ref{fig:fp-system}. The system needs only the eigenvalues $\{\mu_i\},\{\lambda_j\}$, the aspect ratios and the momentum. It belongs to the family of deterministic-equivalent / homogenization fixed-point systems used in high-dimensional learning theory and random-matrix theory \cite{paquette2022homogenization,couillet2022random,pennington2017nonlinear,adlam2020understanding,hu2022universality,wei2022more,simon2023more}.

The drift kernel is
\begin{equation}\label{eqn:dij}
\mathfrak{d}_{ij} \;=\; \frac{\sqrt2}{\pi}\int_0^\infty \frac{z\,L_\bmom\big(q(-z)\big)\,\mu_i\,\lambda_j}{\big(z+\tilde s_1(-z)\,\mu_i\big)\big(z+s_1(-z)\,\lambda_j\big)}\,\dif u,\qquad z=u^2,
\end{equation}
where $\tilde s_1,s_1,q$ solve the system of Fig.~\ref{fig:fp-system} at $-z$. At a positive ridge, $z=\ridge+u^2$ (\cref{cor:app-drift-kernel}). The volatility kernel $\mathfrak v_{ij}$ is the noise the update injects into mode $(i,j)$ per step. It includes the correlation between successive updates that momentum creates. App.~\ref{sec:app-vol-brief} gives it in closed form. Neither kernel depends on $\fRisk$, because the spectral variables are normalized by the risk. For simplicity, the main text takes $\Nin = \Nout = N$. The appendices treat the general case.

\section{\Muon and \SignSGD in the isotropic setting}
\label{sec:isotropic}

The simplest comparison of \Muon and \SignSGD uses isotropic Gaussian data, $\Xin, \Xout \sim \mathcal{N}(0, \Id_N)$. With identity covariances there is no spectral structure for \Muon to exploit. The comparison comes down to how each method handles minibatch noise relative to the dimension $N$. For \Muon, the batch size and the momentum set that noise jointly, through the \textbf{\textit{effective batch}} $B_{\rm eff}\defeq B(1+\bmom)/(1-\bmom)$. The behaviour of \Muon changes qualitatively at $B_{\rm eff} \asymp N$.

\textbf{Deterministic risk recursion.} With identity covariances, the dynamics of the projected risk $\fQ_{ij}(t) = (u_i^\T \Dt v_j)^2$ are \emph{eigenmode-independent}. The kernels are constant, $\mathfrak{d}_{ij} = \mathfrak{d}$ and $\mathfrak{v}_{ij} = \mathfrak{v}/N^2$. Summing over $(i,j)$ gives a deterministic \emph{scalar risk recursion} for $\fRisk = \frac{1}{2}\sum_{ij} \fQ_{ij}$, of the same form for both algorithms (Muon: Prop.~\ref{prop:app-iso-scalar} \& Cor.~\ref{cor:app-iso-dynamics}; \SignSGD: Prop.~\ref{prop:signsgd-recursion-B}):
\begin{equation}\label{eqn:comparison-general-form}
  \fRisk(t{+}1) = \fRisk(t)
    - \textcolor{PIcolor}{\underbrace{2\eta \, \mathfrak{d} \,  \sqrt{\fRisk(t)}}_{\text{drift}}}
    + \textcolor{PIIIcolor}{\underbrace{\tfrac{\eta^2}{2}\,\mathfrak{v}}_{\text{volatility}}}.
\end{equation}
For \Muon we write the kernels as a \textcolor{PIcolor}{drift} $\mathfrak d = \mathfrak d_\bmom$ and a \textcolor{PIIIcolor}{volatility} $\mathfrak v = \mathfrak v_\bmom$. They come from the fixed-point system of Fig.~\ref{fig:fp-system} at $\mu_i=\lambda_j=1$ and depend only on $B_{\rm eff}/N$ and $\bmom$ (Fig.~\ref{fig:isotropic-main}; Eq.~\eqref{eqn:app-iso-C}, Def.~\ref{def:app-iso-vol-zero} and Prop.~\ref{prop:app-iso-zero}, where they are written $C_\bmom$ and $\mathfrak v_{\mathrm{iso},\bmom}$). For \SignSGD the kernels are explicit, $\mathfrak d=\mathcal N_B/\sqrt\pi\approx\sqrt{B/\pi}$ and $\mathfrak v=N^2$ (Prop.~\ref{prop:signsgd-recursion-B}, \ref{prop:app-sgd-batch-coefficient}). We use $\mathcal N_B\approx\sqrt B$ throughout.
The remaining analyses depend on the kernels only through two ratios:
The drift to volatility ratio
$\mathfrak d/\mathfrak v$ sets the learning rate, and the signal-to-noise ratio $\mathfrak d^2/\mathfrak v$ sets the speed of training.

\begin{figure}[t]
  \centering
  {\small\textbf{Drift and volatility of \Muon}}\par\vspace{2pt}
  \begin{subfigure}[b]{0.32\textwidth}
    \includegraphics[width=\textwidth]{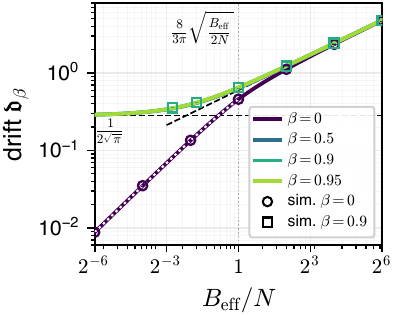}
    \caption{Drift}
    \label{fig:isotropic-main-drift}
  \end{subfigure}
  \hfill
  \begin{subfigure}[b]{0.32\textwidth}
    \includegraphics[width=\textwidth]{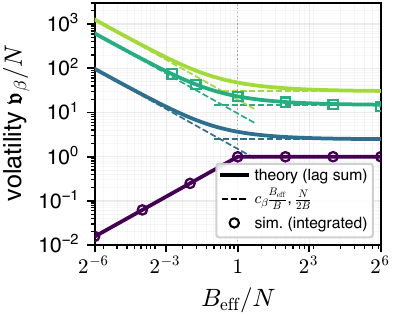}
    \caption{Volatility}
    \label{fig:isotropic-main-vol}
  \end{subfigure}
  \hfill
  \begin{subfigure}[b]{0.335\textwidth}
    \centering
    \fcolorbox{gray!60}{blue!4}{\begin{minipage}{0.93\linewidth}
    \footnotesize
    {\centering\textbf{Approximate forms}\par}
    \vspace{-0.5em}
    \begin{equation}\label{eqn:C-integral}
    \begin{aligned}
      \mathfrak d_\bmom&\approx\Big[\tfrac1{4\pi}+\tfrac{32\,B_{\rm eff}}{9\pi^2N}\Big]^{1/2},\\
      \mathfrak v_\bmom&\approx\tfrac{B_{\rm eff}}{B}\Big(c_\bmom N+\tfrac{N^2}{2B_{\rm eff}}\Big),\\
      \mathfrak d_0&\approx\Big[\tfrac{\pi N}{B}\Big(\tfrac{9\pi}{32}+\tfrac NB\Big)\Big]^{-1/2},\\
      \mathfrak v_0&=\min\{B,N\},
    \end{aligned}
    \end{equation}
    \vspace{-1.2em}

    with $c_\bmom$ decreasing from $1$ to $0.778$ as $\bmom\uparrow1$ (Lem.~\ref{lem:app-iso-V}).
    \end{minipage}}
    \caption{}
    \label{fig:isotropic-main-approx}
  \end{subfigure}
  \caption{\small \textbf{Isotropic kernels of \Muon against the effective batch} $B_{\rm eff}=B(1+\bmom)/(1-\bmom)$, square case.
  Colored curves represent the theory (Eq.~\eqref{eqn:app-iso-C}, Def.~\ref{def:app-iso-vol-zero}), black dashed lines are the limits $B_{\rm eff}\gg N$ and $B_{\rm eff}\ll N$,
  which differ for $\bmom =0$ and $\bmom>0$ (c). Markers are simulations at $N=512$ with the exact polar factor. Drift and volatility change character at $B_{\rm eff}\asymp N$, vanishing for $\bmom = 0$ for $B\ll N$, since the gradient is only rank $B$ even after \SignSVD. With momentum the buffer has full rank and the drift tends to $1/(2\sqrt\pi)$ (Thm.~\ref{thm:app-iso-zero-ridge}). See Sec.~\ref{sec:app-numerics}. \vspace{-0.5cm}}
  \label{fig:isotropic-main}
  \vspace{-0.3cm}
\end{figure}

\subsection{Learning rates: floors and schedules}
\label{sec:iso-lr}

\textbf{Constant learning rate: floors.} Both updates have a fixed size, so at a constant learning rate the risk does not go to zero, but rather balances at the risk floor (Fig.~\ref{fig:iso-risk-constlr}):
\begin{equation}\label{eqn:fixed-point-iso}
  \fRisk_\star^{\text{Muon}}=\bigg(\frac{\eta\,\mathfrak v_\bmom}{4\,\mathfrak d_\bmom}\bigg)^{2}
  \quad\text{and}\quad
  \fRisk_\star^{\text{s-SGD}}=\frac{\pi\,\eta^2N^4}{16\,B}.
\end{equation}

\begin{wrapfigure}[14]{r}{0.4\textwidth}
  \vspace{-0.5cm}
  \centering
  \includegraphics[width=\linewidth]{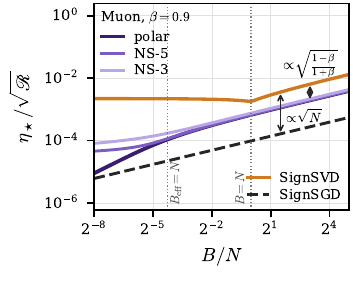}
  \vspace{-1.9em}
  \caption{\small \textbf{Optimal learning rates} $\eta_\star/\sqrt{\fRisk}=2\mathfrak d/\mathfrak v$, $N=512$; NS-$K$: calibrated ridges (App.~\ref{sec:app-nsridge}). Arrows: the large-batch gaps in \eqref{eqn:lr-ratio}.}
  \label{fig:iso-lr}
  \vspace{-1em}
\end{wrapfigure}

At the same $\eta$, the \Muon floor is far lower. The factor is of order $N$ for $B_{\rm eff}\gg N$, and of order $B$ for $B_{\rm eff}\ll N$ with momentum ($N^2/B$ at $\bmom=0$). To reach the same floor, the \Muon learning rate must be larger by the square root of this factor.

\textbf{Scheduled learning rate.} The learning rate that decreases the risk the most at time $t$ is $\eta_\star(t)=2\mathfrak d\sqrt{\fRisk(t)}/\mathfrak v$.  This turns out to be a schedule that gives the lowest risk at all times (something which need not even exist for anisotropic data). With this learning rate, the recursion \eqref{eqn:comparison-general-form} gives $\fRisk(t{+}1)=(1-2\mathfrak d^2/\mathfrak v)\,\fRisk(t)$.  
Thus we can compare the optimal learning rates of \Muon and \SignSGD at a given risk by looking at the ratio: 
\begin{equation}\label{eqn:lr-ratio}
\frac{\eta_\star^{\text{Muon}}}{\eta_\star^{\text{s-SGD}}}
\;=\; \frac{\sqrt{\pi}\,N^2\,\mathfrak d_\bmom}{\sqrt B\,\mathfrak v_\bmom}
\;\approx\;
\begin{cases}\dfrac{8\sqrt{N}}{3c_\bmom\sqrt{2\pi}}\sqrt{\dfrac{1-\bmom}{1+\bmom}}, & B_{\rm eff}\gg N,\\[10pt] \sqrt{B}, & B_{\rm eff}\ll N,\end{cases}
\end{equation}
At $\bmom=0$ and $B\le N$ the ratio is $N/\sqrt B$, since $\mathfrak v_0=\min\{B,N\}$ (Prop.~\ref{rem:app-iso-volatility-zero}). With momentum, the \Muon learning rate bends at $B_{\rm eff}\asymp N$, from $\eta_\star\propto B$ below to $\eta_\star\propto\sqrt B$ above (Fig.~\ref{fig:iso-lr}). Finite Newton--Schulz flattens the part below $B_{\rm eff}=N$. With five steps the slope drops from about $1$ to about $0.3$, close to the picture at $\bmom=0$ (Fig.~\ref{fig:iso-lr}; App.~\ref{sec:app-nsridge-lr}). The \Muon learning rate is always the larger, as in practice \cite{liu2025muon, deepseekai2026deepseekv4}. Momentum narrows the gap at large batch by a factor $\sqrt{(1-\bmom)/(1+\bmom)}$. The batch limits hold up to a factor $1+o(1)$, as $N/B\to0$ for $B_{\rm eff}\gg N$ (Thm.~\ref{thm:app-iso-small-momentum}, \ref{thm:app-iso-small}, \ref{thm:app-iso-small-volatility}, Lem.~\ref{lem:app-iso-V}) and as $N/B\to\infty$ for $B_{\rm eff}\ll N$ (Thm.~\ref{thm:app-iso-zero-ridge}).

\begin{mdframed}[style=exampledefault]
\textbf{Key Takeaway:}
The \Muon learning rate should exceed that of \SignSGD by a factor of order $\sqrt{N}$ for $B_{\rm eff} \gg N$ and $\sqrt B$ for $B_{\rm eff} \ll N$, whether the learning rate is constant (matched floors) or scheduled. The $\sqrt N$ factor is the \textsc{Moonshot} scaling rule \cite{liu2025muon} (see also \cite{deepseekai2026deepseekv4}).
\end{mdframed}

\subsection{Rate comparison: the \Muon speedup}
\label{sec:signsvd-vs-signsgd}

Under the scheduled learning rates both algorithms contract by $\approx e^{-2\mathfrak d^2/\mathfrak v}$ per step. We call the ratio of the two exponents the \emph{\Muon speedup}:
\begin{equation}\label{eqn:rate-signsgd}
\mathcal S_{\rm Muon}\defeq\frac{2\mathfrak d_\bmom^2/\mathfrak v_\bmom}{2B/(\pi N^2)}
\;\approx\;
\begin{cases}
\tfrac{32}{9\pi c_\bmom}, & B_{\rm eff}\gg N,\\[1pt]
\tfrac12, & B_{\rm eff}\ll N,\ \bmom>0,\\[1pt]
1, & B_{\rm eff}\ll N,\ \bmom=0.
\end{cases}
\end{equation}
It is the number of \SignSGD steps that one \Muon step is worth. It is also the ratio of the times to a common floor at constant learning rates (Thm.~\ref{thm:app-phase-iso}, \ref{thm:app-phase-matched}).

\begin{wrapfigure}[17]{r}{0.42\textwidth}
  \vspace{-0.75cm}
  \centering
  \includegraphics[width=\linewidth]{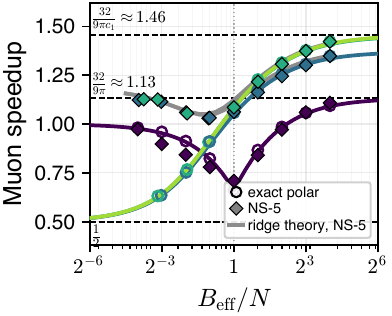}
  \vspace{-0.9em}
  \caption{\small \textbf{\Muon speedup} \eqref{eqn:rate-signsgd}; above $1$, \Muon is faster. Colors as in Fig.~\ref{fig:isotropic-main}. Curves: zero ridge; markers: simulations at $N=512$, exact polar (open) and NS-5 (filled); grey: ridge theory at the calibrated NS-5 ridge (App.~\ref{sec:app-nsridge}).}
  \label{fig:isotropic-speedup}
  \vspace{-1em}
\end{wrapfigure}

\textbf{Whitening does not help on isotropic data.} The speedup stays between $\tfrac12$ and $1.46$ for every $\bmom$, $B$ and $N$ (Fig.~\ref{fig:isotropic-speedup}). Once each learning rate is tuned, the two algorithms are interchangeable up to an absolute constant. Whitening the minibatch is in fact mildly harmful. Figure \ref{fig:isotropic-speedup} shows where.
\emph{Large batch.} For $B_{\rm eff}\gg N$ the minibatch singular values already approach 1, the population limit, and so the speedup approaches $\pi/2$ as $B_{\rm eff}/N$ continues to grow (Rmk.~\ref{rem:app-iso-momentum}). One Newton--Schulz step, the softest whitening and essentially normalized SGD, attains $\pi/2$ at every batch size (App.~\ref{sec:app-nsridge}).
\emph{Small batch without momentum.} For $B_{\rm eff}\ll N$ the speedup tends to $1$, as at small batch, the whitening of the individual per-example gradients can essentially be done per example.
\emph{Near $B_{\rm eff}= N$.} Without momentum the speedup dips to about $0.7$ at $B_{\rm eff}=N$. There the minibatch gradient is square and its singular values reach down to zero, as in double descent. Whitening amplifies these weakest directions the most.
\emph{Small batch with momentum.} For $B_{\rm eff}\ll N$ the buffer has full rank but many weak directions. Exact whitening amplifies all of them and halves the rate.

\textbf{Newton--Schulz is ridge-regularized whitening.} Practical \Muon does not exactly whiten. $K$ Newton--Schulz steps map small singular values $\sigma$ to about $a^K\sigma$. This acts like the regularized polar map with ridge $\ridge\approx0.31\,a^{-2K}\norm{M_t}_\F^2$ (App.~\ref{sec:app-nsridge}). At this calibrated ridge the theory predicts the NS-5 speedups without refitting (Fig.~\ref{fig:isotropic-speedup}, grey; Thm.~\ref{thm:app-iso-crossover}, $a=3.4445$). On isotropic data, whitening less is better. NS-5 beats exact whitening for $B_{\rm eff}<N$, and the speedup decreases with $K$ at every $\bmom$ and $B_{\rm eff}$ (Fig.~\ref{fig:app-nsridge-dynamic}).

\begin{mdframed}[style=exampledefault]
\textbf{Key Takeaway:}
On isotropic data, tuned \Muon and \SignSGD converge at the same rate up to a factor of at most $2$. There is no structure to whiten, and the less \Muon whitens small sinbular values (fewer Newton--Schulz steps, larger effective ridge), the faster it converges.
\end{mdframed}

\section{Anisotropic output covariance ($\Sigmain = \Id_N$, $\Sigmaout = \text{general}$)}
\label{sec:half-aniso}

The previous analysis suggests, not surprisingly, that any real advantage of \Muon
stems from anisotropy. To explore this further we consider the
\emph{half-anisotropic} setting: $\Sigmain = \Id_N$ and $\Sigmaout$ is a general covariance matrix. This setting captures the effect of non-uniform signal strengths across output modes but has vastly simpler random matrix theory. In particular we focus on the setting:
\begin{equation}\label{eqn:half-aniso-unitary}
  \Sigmaout \;=\; \U\,\diag(\mu_1,\ldots,\mu_N)\,\U^\T\,,
  \qquad \U \sim \mathrm{Haar}\bigl(N\bigr)\,,
\end{equation}
with $\U$ Haar distributed and independent of the data. We note that the distribution
of $\U$ matters for \SignSGD but \emph{not} for \Muon. Proofs are in the appendix (Cor.\ref{rem:app-drift-halfaniso}, \ref{cor:app-dyn-half}; Sec.~\ref{sec:app-vol-special}, \ref{sec:app-pl}, \ref{sec:SignSGD_anisotropic_data}). We begin with the deterministic expression for the risk. Recall the projected risk $\fQ_{ij}$ from Sec.~\ref{sec:deterministic_dynamics}. In this setting we can partially sum the $\fQ_{ij}$ to produce:
\[
\fRisk(t) = \tfrac{1}{2}\sum_i^N \mu_i\,\fri_i, \quad \text{where the \textit{projected row risks} satisfy} \quad \fri_i(t) \defeq \sum_{j=1}^{N} \fQ_{ij}(t) = \|\Dt^\T \uve_i\|^2.
\]
The $\fri_i$ play the role of the $\fQ_{ij}$ in Sec.~\ref{sec:deterministic_dynamics} and satisfy a deterministic recursion, which the algorithm follows in the sense of \cref{cor:app-dyn-half},
\begin{equation}\label{eqn:ri-recursion}
\fri_i(t{+}1) = \fri_i(t) - \frac{2\eta\,\textcolor{PIcolor}{\mathfrak{d}_i}}{\sqrt{\fRisk(t)}}\fri_i(t) + \eta^2\,\textcolor{PIIIcolor}{\mathfrak{v}_i}.\,
\end{equation}
Here the drift kernel $\mathfrak{d}_i$ is the common value of the $\mathfrak{d}_{ij}$, $j=1,\dots,N$, which do not depend on $j$ when $\Sigmain=\Id_N$ (\cref{rem:app-drift-halfaniso}), and the volatility kernel $\mathfrak{v}_i = \sum_{j=1}^N \mathfrak{v}_{ij}$ is their sum (\cref{prop:app-vol-half-recovery}); both are deterministic functions of the spectrum $\{\mu_i\}_{i=1}^N$. \textit{Both} \Muon and \SignSGD admit risk formulations and projected row risks of the form \eqref{eqn:ri-recursion}; once more the \textit{only differences} are the specific drift and volatility kernels. \SignSGD admits the same row recursion on a Haar eigenbasis.\footnote{We derive the \SignSGD row kernels in a more restricted setting: one step, a Haar eigenbasis, comparable dimensions and a covariance bounded below (Prop.~\ref{prop:signsgd-recursion-anisotropic}, Thm.~\ref{thm:app-sgd-modes}). There the per-row volatility carries an extra nonnegative term. Power-law spectra are not bounded below, so for them we use the simpler kernels of Table~\ref{tab:power-law-kernels}, which omit this term (Rmk.~\ref{rem:app-phase-sgd}).}

\subsection{Large batch: square-root preconditioning}\label{sec:half-aniso-large-batch}

For sufficiently large effective batch $B_{\rm eff}\gg N$, \Muon gives simple drift and volatility coefficients:

\textbf{Constant volatility.} The volatility is the main simplification of the half-anisotropic setting. Without momentum, the update $UV^\T$ is orthogonal, because $B\ge N$. So $\|(UV^\T)^\T u_i\|=1$, and $\mathfrak v_i=1$ for every mode and every spectrum (Rmk.~\ref{rem:app-vol-nomomentum}). With momentum, successive buffers are correlated, and the volatility grows. We use $\mathfrak v_i=(1+\beta)/(1-\beta)$, the value proved on the resolved modes at small batch (Sec.~\ref{sec:power_law_main}). In the isotropic case the exact value lies between $1$ and $B_{\rm eff}/B$ (Lem.~\ref{lem:app-iso-V}). A constant factor in the volatility only rescales the time to reach a target (Rmk.~\ref{rem:app-phase-momentum}).

\textbf{Square-root drift.} The drift is a square-root law for every spectrum, up to a constant. The constant depends only on how the spectrum is packed around each mode.

\begin{definition}[Packed mode]\label{def:packed}
Let $K\ge1$ and $\delta\in(0,1]$. Mode $i$ is \emph{$(K,\delta)$-packed} if at least $\delta N$ eigenvalues satisfy $\mu_k\le K\mu_i$.
\end{definition}

\noindent Modes high in the spectrum are packed with $\delta$ close to $1$. The count of eigenvalues below $K\mu_i$ grows with $\mu_i$, so if the smallest eigenvalue is packed, so is every mode. Packing holds when the spectrum does not decay too fast. Every mode of a power law $\mu_k=k^{-\alpha}$ is $(2,1-2^{-1/\alpha})$-packed. A geometric spectrum $\mu_k=q^k$ with $q<1$ decays so fast that its smallest eigenvalues are not packed.

\begin{theorem}[Large-batch drift; Cor.~\ref{cor:app-os-packed}]\label{thm:large-batch}
Let $\bmom\in[0,1)$ and $\epsilon_{\rm os}\defeq3(1+\bmom)^2N/B_{\rm eff}<1$. The zero-ridge drift kernel satisfies
\[
(1-\epsilon_{\rm os})\,c_{K,\delta}\sqrt{\tfrac{B_{\rm eff}}{2N}}\,\sqrt{\mu_i}\;\le\;\mathfrak d_i\;\le\;\tfrac{1}{1-\epsilon_{\rm os}}\sqrt{\tfrac{B_{\rm eff}}{2N}}\,\sqrt{\mu_i},\qquad c_{K,\delta}=0.94\sqrt{\delta/(1+K)} .
\]
The upper bound holds for every spectrum and every mode $i$. The lower bound holds if mode $i$ is $(K,\delta)$-packed.
\end{theorem}

\noindent In expectation and with small learning rate, the \Muon update on packed modes is therefore, up to constants (Cor.~\ref{cor:app-os-update}),
\[
W_{t+1}-W_t\approx-\eta\,\sqrt{\tfrac{B_{\rm eff}}{2N}}\; \cdot \sqrt{\mathcal K}\, \cdot \tfrac{W_t-W^*}{\sqrt{\fRisk(t)}},
\]
where $\mathcal{K}$ is the covariance operator of the problem. We note that gradient descent would instead perform the update $W_{t+1} - W_t = -\eta\mathcal{K}(W_t-W^*)$, so \Muon operates like a \emph{square-root preconditioner}. For modes high in the spectrum, the constant tends to $1$ (Prop.~\ref{prop:app-os-head}). When every mode is packed, this preconditioning has an effect comparable to optimal momentum methods, which \emph{accelerate} strongly convex optimization problems by a factor of the square root of the condition number \cite{Polyak1962Some,nesterov2004introductory}. Not every mode benefits, however. A mode that is not packed, such as an isolated small eigenvalue, can be as slow as under gradient descent, up to a logarithm (Prop.~\ref{prop:app-os-tail}).

If $\mathcal{K}$ is diagonal, \SignSGD performs the same update as large-batch \SignSVD, after rescaling the learning rate by $\sqrt{N}$ \citep[Eqn.\,(29)]{xiao2025exact}. In this case \SignSVD is similar to \SignSGD, and \Muon may be similar to \Adam. But \Muon achieves this preconditioning in \emph{any} basis. In the Haar basis above, \SignSGD loses all of it.

\begin{mdframed}[style=exampledefault]
\textbf{Key Takeaway:} For large effective batch, \Muon is a square-root preconditioner on every packed eigendirection, in any basis.
\end{mdframed}
\vspace{-0.25cm}

\subsection{Power-law specialization $\mu_i = i^{-\alpha}$} \label{sec:power_law_main}

We now take the \textit{power law} $\mu_i = i^{-\alpha}$ with $\alpha > 0$. It captures key phenomenology of power-law scaling on real data \cite{paquette20244+, bordelon2024dynamical, lin2024scaling, qiu2025scaling}.


\begin{assumption}[Power-law covariance.] \label{assump:power_law} Suppose $\mu_i=i^{-\alpha}$ with $\alpha > 0$, and $\fQ_i(0) = i^{-\texp}$ such that $\alpha+\texp>1$. The parameters $(\alpha, \texp)$ are related to what is known in the literature as source and capacity conditions \cite{caponnetto2007optimal,carratino2018learning, dieuleveut2016nonparametric, pillaud2018statistical} and the half-space $\alpha + \texp > 1$ is precisely where we have summability at initialization. 
\end{assumption}
\vspace{-0.3cm}

\noindent \textbf{Small batch.} When $B_{\rm eff}\ll N$, the batch cannot resolve every direction of the spectrum. \Muon whitens only the top modes, up to a \emph{resolution index} $i_\#$. Modes $i \ll i_\#$ keep the square-root drift, $\mathfrak d_i\propto\sqrt{\mu_i}$. Modes $i \gg i_\#$ are not whitened. At $\bmom=0$ their drift is $\mathfrak d_i\propto\mu_i$. This is the drift of \textsc{SGD}, and of \SignSGD in the Haar basis. As $N/B\to\infty$, the resolution index is (Prop.~\ref{prop:app-pl-scales}, Cor.~\ref{cor:app-pl-endpoints}(iv))
\begin{equation}\label{eqn:i-sharp}
\small
i_\# \asymp
\begin{cases}
B_{\rm eff}, & \alpha > 1, \\[1pt]
B_{\rm eff}/\ln^3(\tfrac NB), & \alpha = 1, \\[1pt]
\big(B_{\rm eff}/\ln^2(\tfrac NB)\big)^{1/\alpha}/N^{1/\alpha-1}, & \alpha < 1.
\end{cases}
\end{equation}
\noindent \textbf{Large batch.} The theory of Sec.~\ref{sec:half-aniso-large-batch} applies directly. Every mode of the power law is $(2,1-2^{-1/\alpha})$-packed (Rmk.~\ref{rem:app-os-packed}). So every mode keeps the square-root law up to a constant that depends only on $\alpha$. The drift is also sharper than Theorem~\ref{thm:large-batch}. For the modes $i\ll N$, we have $\mathfrak d_i\sim\sqrt{\mu_iB_{\rm eff}/(2N)}$ as $B_{\rm eff}/N\to\infty$ (Cor.~\ref{cor:app-os-powerlaw}). The rest of this section treats small batch, $B_{\rm eff}\ll N$.

\noindent \textbf{Drift and volatility kernels.} Table~\ref{tab:power-law-kernels} gives the kernels on either side of $i_\#$. The resolved modes $i\ll i_\#$ keep the large-batch kernels of Sec.~\ref{sec:half-aniso-large-batch}.

\begin{center}
\begin{minipage}[c]{0.62\textwidth}
\centering
\footnotesize
\setlength{\tabcolsep}{4pt}
\renewcommand{\arraystretch}{1.15}
\begin{tabular}{lcc}
\toprule
& \textcolor{PIcolor}{\textbf{Drift}} $\mathfrak{d}_i$
& \textcolor{PIIIcolor}{\textbf{Volatility}} $\mathfrak{v}_i$ \\
\midrule
\multicolumn{3}{l}{\Muon, $B_{\rm eff}\ll N$} \\
\quad $i\ll i_\#$
& $\sqrt{\mu_iB_{\rm eff}/(2N)}$
& $B_{\rm eff}/B$ \\[3pt]
\quad $i\gg i_\#$, $\bmom>0$
& $N^{-\alpha/2}\,\Psi_\alpha(i/N)/\sqrt2$
& $\tfrac NB\,\Gamma_\alpha(i/N)$ \\[3pt]
\quad $i\gg i_\#$, $\bmom=0$
& $A_\alpha\,\mu_i\sqrt{\resscale B/(2N)}$
& $\resscale \mu_i/(1 + \resscale \mu_i)$ \\[3pt]
\midrule
\SignSGD
& $\mu_i \sqrt{B/(\pi \bar\mu)}$
& $N \big(1 + \tfrac{2}{\pi}\big(\tfrac{\mu_i}{\bar\mu}-1\big)\big)$ \\
\bottomrule
\end{tabular}
\end{minipage}\hfill
\begin{minipage}[c]{0.36\textwidth}
\centering
\small
\setlength{\tabcolsep}{5pt}
\renewcommand{\arraystretch}{1.45}
\begin{tabular}{lcc}
\toprule
$x\downarrow0$ & $\Psi_\alpha(x)$ & $\Gamma_\alpha(x)$ \\
\midrule
$\alpha<1$ & $\log\frac1x$ & $x^\alpha\log^2\frac1x$ \\
$\alpha=1$ & $\log^{3/2}\frac1x$ & $x\log^3\frac1x$ \\
$\alpha>1$ & $x^{-(\alpha-1)/2}$ & $x$ \\
\bottomrule
\end{tabular}\\[4pt]
$\bar\mu=\frac1N\sum_k\mu_k$; $A_\alpha$ is a constant.\\
$\resscale$ solves $\sum_{i}\frac{\resscale\mu_i}{1+\resscale\mu_i}=B$.
\end{minipage}
\captionof{table}{\small \textbf{Zero-ridge kernels for power-law covariance}, under Assumption~\ref{assump:power_law} and \eqref{eqn:half-aniso-unitary}, at fixed $\bmom\in[0,1)$. The \Muon rows are proved as $N/B\to\infty$ (Cor.~\ref{cor:app-pl-endpoints}; Thm.~\ref{thm:app-pl-drift}, \ref{thm:app-pl-volatility}; Rmk.~\ref{rem:app-pl-nomomentum}). The profiles $\Psi_\alpha,\Gamma_\alpha$ of $x=i/N$ are positive, continuous and independent of $\bmom$ (Def.~\ref{def:app-pl-profiles}). Right: their behavior as $x\downarrow0$, up to constants (Prop.~\ref{prop:app-pl-profiles}), and the remaining symbols (Prop.~\ref{prop:app-pl-nomomentum-endpoints}).}
\label{tab:power-law-kernels}
\end{center}

\begin{mdframed}[style=exampledefault]
\textbf{Key Takeaway:} The effective batch $B_{\rm eff}$ is a spectral resolution scale for \Muon. It sets how many top eigenmodes are whitened. This number depends sharply on the tail exponent $\alpha$.
\end{mdframed}
\vspace{-0.5cm}

\section{Time to $\epsilon$-approximate solution: performance of \Muon vs. \SignSGD}
\label{sec:timetoepsilon}
%

We compare the two algorithms at matched limit loss on power-law data with $B\ge N$. Given a target $\epsilon$, each algorithm runs at the constant learning rate
\begin{equation}\label{eqn:eta-star-tte}
\eta^\star=\frac{2\sqrt{\epsilon}}{\mathcal S},\qquad\mathcal S=\sum\nolimits_{i=1}^N\frac{\mu_i\,\mathfrak v_i}{2\,\mathfrak d_i},
\end{equation}
which makes its risk floor $\fRisk(\infty)=\epsilon$, and $t_{4\epsilon}=\inf\{t\ge0:\fRisk(t)\le4\epsilon\}$ is the first time its risk is within a factor $4$ of the floor. A constant learning rate avoids comparing two optimized schedules, which in the anisotropic case differ between the methods. On the other hand, this constant learning rate ensures a fair comparison since the risk floors are the same for both algorithms. 

\begin{theorem}[Speed ratio, informal, see Thm.~\ref{thm:app-phase-map}, Cor.~\ref{cor:app-phase-unified-ratio}, Thm.~\ref{thm:app-phase-frozen} for precise form.]\label{thm:tte-half-aniso}
Under Assumption~\ref{assump:power_law}, take $B\ge N$, the large-batch \Muon kernels of Sec.~\ref{sec:half-aniso-large-batch}, the \SignSGD kernels of Table~\ref{tab:power-law-kernels}, $\eta=\eta^\star$, and $\alpha\notin\{\frac12,1\}$. For large $N$ and small $\epsilon$, the ratio of times-to-$4\epsilon$ $t_{4\epsilon}$ is
\begin{equation}\label{eqn:tte-ratio}
\frac{t_{4\epsilon}^{\text{s-SGD}}}{t_{4\epsilon}^{\text{Muon}}}\asymp D_N\max\{\epsilon,\epsilon_N\}^{-\kappa},\qquad
D_N\defeq\frac{\pi\sum_{i\le N}i^{-\alpha}}{\sum_{i\le N}i^{-\alpha/2}},\qquad
\epsilon_N\defeq N^{-(\alpha+\texp-1)}.
\end{equation}
Here $\kappa\defeq e_{\text{s-SGD}}-e_{\text{Muon}}\ge0$, where
\begin{equation}\label{eqn:tte-exponents}
e_{\text{Muon}}=\max\Big\{\frac12,\ \frac{\alpha}{2(\alpha+\texp-1)}\Big\},\qquad
e_{\text{s-SGD}}=\max\Big\{\frac12,\ \frac{\alpha}{\alpha+\texp-1}\Big\}.
\end{equation}
\end{theorem} 
\newcommand{\phaseswatch}[1]{\tikz[baseline=0.55em]{#1\draw[line width=0.4pt] (0,0) rectangle (3.0em,1.6em);}}
\newcommand{\swA}{\phaseswatch{\fill[PIcolor!70] (0,0) rectangle (3.0em,1.6em);}}
\newcommand{\swB}{\phaseswatch{\fill[PIIIcolor!70] (0,0) rectangle (3.0em,1.6em);}}
\newcommand{\swC}{\phaseswatch{\shade[left color=PIIIcolor!80,right color=white] (0,0) rectangle (1.5em,1.6em);\shade[left color=white,right color=PIcolor!80] (1.5em,0) rectangle (3.0em,1.6em);}}
\newcommand{\swD}{\phaseswatch{\fill[PIIIcolor!70] (0,0) rectangle (3.0em,1.6em);\foreach \x in {0.45,1.15,...,2.7} \foreach \y in {0.45,1.15} \fill[PIIcolor] (\x em,\y em) circle[radius=1.8pt];}}
\newcommand{\swE}{\phaseswatch{\fill[PIIcolor!70] (0,0) rectangle (3.0em,1.6em);}}
\newcommand{\swref}[2]{\begin{tabular}[t]{@{}c@{}}#1\\[-1pt]\makebox[3.0em]{\scriptsize #2}\end{tabular}}

\begin{figure}[t]
\centering
\begin{minipage}[c]{0.37\textwidth}
\centering
\begin{tikzpicture}[x=1.35cm,y=0.95cm,font=\small]
  \fill[PIIIcolor!55] (0,1) -- (1,1) -- (2,3) -- (2,4.5) -- (0,4.5) -- cycle;        
  \begin{scope}
    \clip (1,0) -- (2,0) -- (2,3) -- (1,1) -- cycle;
    \foreach \k in {0,...,199} {
      \pgfmathsetmacro{\ao}{1+\k/200}
      \pgfmathsetmacro{\at}{min(2,1+(\k+1.4)/200)}
      \pgfmathsetmacro{\am}{1+(\k+0.5)/200}
      \pgfmathsetmacro{\rr}{(2-\am)/\am}
      \pgfmathsetmacro{\pg}{max(0,110*\rr-55)}
      \pgfmathsetmacro{\pp}{max(0,55-110*\rr)}
      \pgfmathparse{\rr>0.5 ? 1 : 0}
      \ifnum\pgfmathresult=1 \colorlet{cC}{PIIIcolor!\pg!white}\else \colorlet{cC}{PIcolor!\pp!white}\fi
      \fill[cC] (\ao,0) rectangle (\at,1.001);
      \fill[cC] (2,3) -- (\ao,1) -- (\at,1) -- cycle;
    }
  \end{scope}
  \fill[PIcolor!55]   (2,0) -- (3,0) -- (3,4) -- (2,3) -- cycle;                     
  \fill[PIIcolor!55]  (2,3) -- (3,4) -- (3,4.5) -- (2,4.5) -- cycle;                 
  \fill[PIIIcolor!55]  (0,1) -- (1,0) -- (1,1) -- cycle;                              
  \begin{scope}
    \clip (0,1) -- (1,0) -- (1,1) -- cycle;
    \foreach \x in {0.06,0.18,...,1.0} \foreach \y in {0.06,0.2,...,1.0}
      \fill[PIIcolor] (\x,\y) circle[radius=1.5pt];
  \end{scope}                                                                       
  \draw (0,0) rectangle (3,4.5);
  \foreach \x in {0,1,2,3} \draw (\x,0) -- (\x,-0.08) node[below] {$\x$};
  \foreach \y in {0,1,2,3,4} \draw (0,\y) -- (-0.05,\y) node[left] {$\y$};
  \node[below=12pt,font=\normalsize] at (1.5,0) {$\alpha$};
  \node[left=14pt,font=\large] at (0,2.25) {$\texp$};
  \node[font=\bfseries\large] at (0.8,3.0)  {B};
  \node[font=\bfseries\large] at (1.55,0.6) {C};
  \node[font=\bfseries\large] at (0.62,0.62){D};
  \node[font=\bfseries\large] at (2.55,1.6) {A};
  \node[font=\bfseries\large] at (2.5,4.05) {E};
  \draw[thick] (0,1) -- (1,1);                       
  \draw[thick] (1,0) -- (1,1) -- (2,3) -- (3,4);     
  \draw[thick] (2,0) -- (2,4.5);                     
  \draw[very thick,red] (0,1) -- (1,0);              
\end{tikzpicture}
\end{minipage}\hfill
\begin{minipage}[c]{0.61\textwidth}
\centering
\footnotesize
\renewcommand{\arraystretch}{1.3}
\setlength{\tabcolsep}{4pt}
\begin{tabular}{@{}>{\raggedright\arraybackslash}p{0.31\linewidth} >{\raggedright\arraybackslash}p{0.51\linewidth} c@{}}
\toprule
\textbf{Phase} & \textbf{Ratio $t_{4\epsilon}^{\text{s-SGD}}/t_{4\epsilon}^{\text{Muon}}$ at $\epsilon=N^{-\omega}$} & \\
\midrule
$A$: $\alpha>2$, $\texp<\alpha{+}1$ & \Muon\ favored (Prop.~\ref{prop:app-phase-winners}). & \swA \\
\arrayrulecolor{gray!50}\midrule\arrayrulecolor{black}
$B$: $\alpha<1<\texp$; $1{<}\alpha{<}2$, $\texp{>}2\alpha{-}1$ & 
\SignSGD\ favored (Thm.~\ref{thm:app-phase-omega}). & \swB \\
\arrayrulecolor{gray!50}\midrule\arrayrulecolor{black}
$C$: $1{<}\alpha{<}2$, $\texp{<}2\alpha{-}1$ & if $\omega<\omega_c$, \SignSGD favored;\newline if $\omega>\omega_c$, \Muon favored (Cor.~\ref{cor:app-phase-unified-ratio}). & \swC \\
\arrayrulecolor{gray!50}\midrule\arrayrulecolor{black}
$D$: $\alpha<1$, $\texp<1$ & if $\omega<\alpha{+}\texp{-}1$, \SignSGD favored;\newline if $\omega>\alpha{+}\texp{-}1$, algs. perform equally. & \swref{\swD}{(Thm.~\ref{thm:app-phase-frozen})} \\
\arrayrulecolor{gray!50}\midrule\arrayrulecolor{black}
$E$: $\alpha>2$, $\texp>\alpha{+}1$ & Ratio tends to a constant; both algs. perform equally (Cor.~\ref{cor:app-phase-ratio}). & \swE \\
\bottomrule
\end{tabular}
\end{minipage}
\caption{\small \textbf{Phase diagram for \Muon\ vs.\ \SignSGD\ in the $(\alpha,\texp)$ plane} when $B\ge N$, matched limit loss, and target $\epsilon=N^{-\omega}$. The analysis requires $\alpha+\texp>1$, above the red line. In C, the shading shows the fraction of accuracies ($\epsilon$) at which \SignSGD\ still wins, from green ($1$) through white ($\frac12$) to purple ($0$), \Muon favored. Dots in D indicate the \SignSGD favored and tie performance. 
Boundary points are given in Cor.~\ref{cor:app-phase-omega-winners}.
\vspace{-0.5cm}
}
\label{fig:tte-phase-diagram}
\end{figure}

\textbf{Phases: which algorithm wins?} To determine which algorithm wins, we analyze the time-to-$4\epsilon$ ratio \eqref{eqn:tte-ratio}, which is a competition between an $N$ factor and an $\epsilon$ factor. We set the target accuracy to $\epsilon=N^{-\omega}$, so that the final ratio is a power of $N$. The sign of the resulting exponent of $N$ then gives the phases of Fig.~\ref{fig:tte-phase-diagram} (Cor.~\ref{cor:app-phase-omega-winners}).
For ``hard'' targets with poor conditioning and a lot of loss in small eigenmodes, the exponent is positive and \Muon is favored
(Phase A). In the rest of the phase diagram, the exponent is either negative (favoring \SignSGD), or zero (tie).
The only subtlety is Phase C ($1<\alpha<2$, $\texp<2\alpha-1$), where the winner depends on the target accuracy $\omega$. Define
\[
{\small \omega_c=\frac{(2-\alpha)(\alpha+\texp-1)}{\alpha+1-\texp}\ \ (\texp\ge1),\qquad
\omega_c=\frac{(2-\alpha)(\alpha+\texp-1)}{\alpha}\ \ (\texp<1).}
\]
For a given $N$, the ratio \eqref{eqn:tte-ratio} grows as $\epsilon$ decreases (Cor.~\ref{cor:app-phase-unified-ratio}), so \SignSGD wins at low accuracy, $\epsilon\gtrsim N^{-\omega_c}$, and \Muon takes over at high accuracy, $\epsilon\lesssim N^{-\omega_c}$. Fig.~\ref{fig:tte-four-panel} illustrates the loss behavior across the different phases.

\begin{mdframed}[style=exampledefault]
\textbf{Key Takeaway:} At matched limit loss, the winner depends on both the data, through $(\alpha,\texp)$, and the target accuracy, so the advantage of \Muon is highly data dependent. \Muon wins where whitening pays off, on steep spectra and on hard targets at high accuracy. \SignSGD tends to win on easy targets or at low accuracy, especially for flatter spectra, where whitening does little (Fig.~\ref{fig:tte-phase-diagram} and Fig.~\ref{fig:tte-four-panel}).
\end{mdframed}



\vspace{-0.25cm}

\begin{figure}[h!]
\centering
\includegraphics[width=0.245\textwidth]{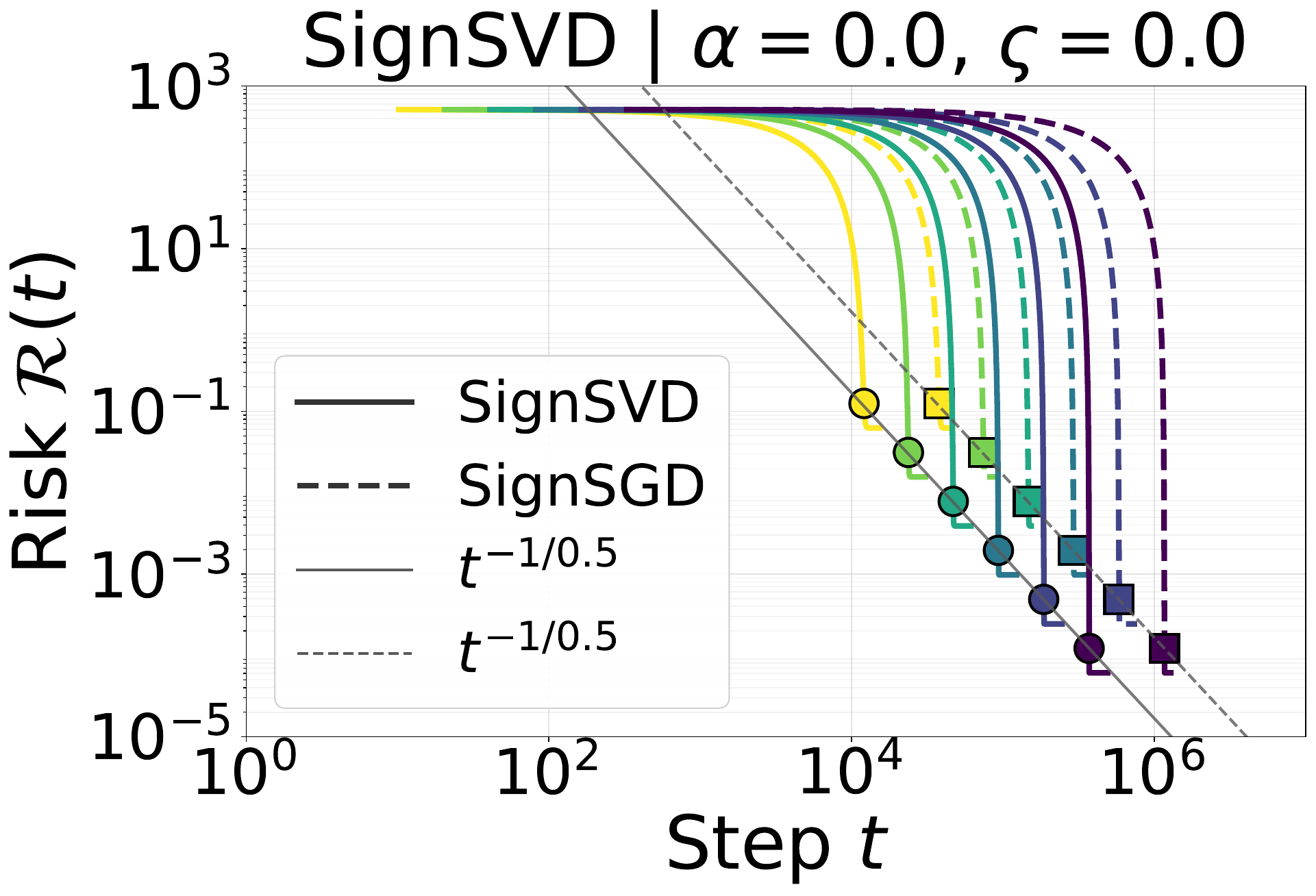}\hfill
\includegraphics[width=0.245\textwidth]{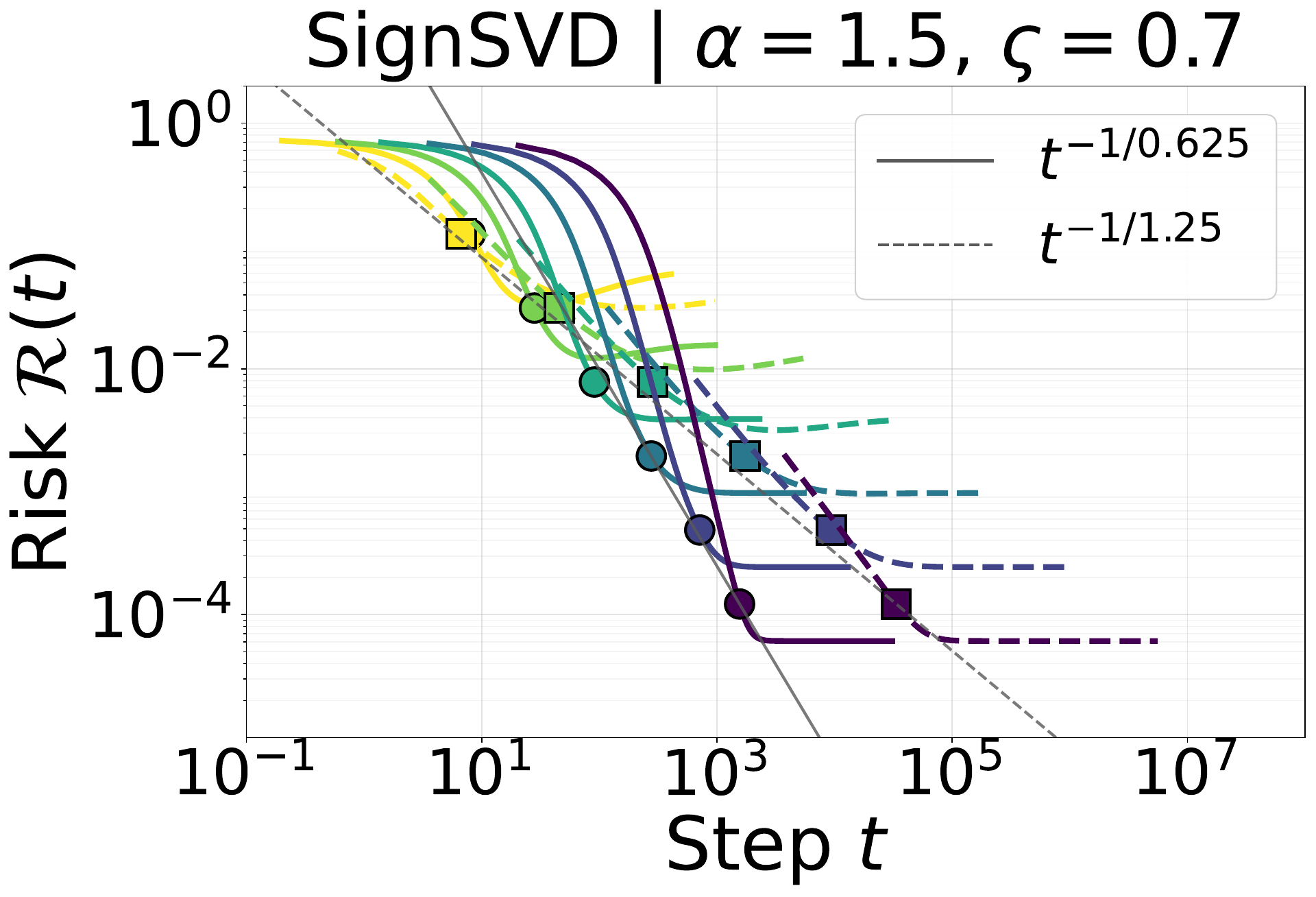}\hfill
\includegraphics[width=0.245\textwidth]{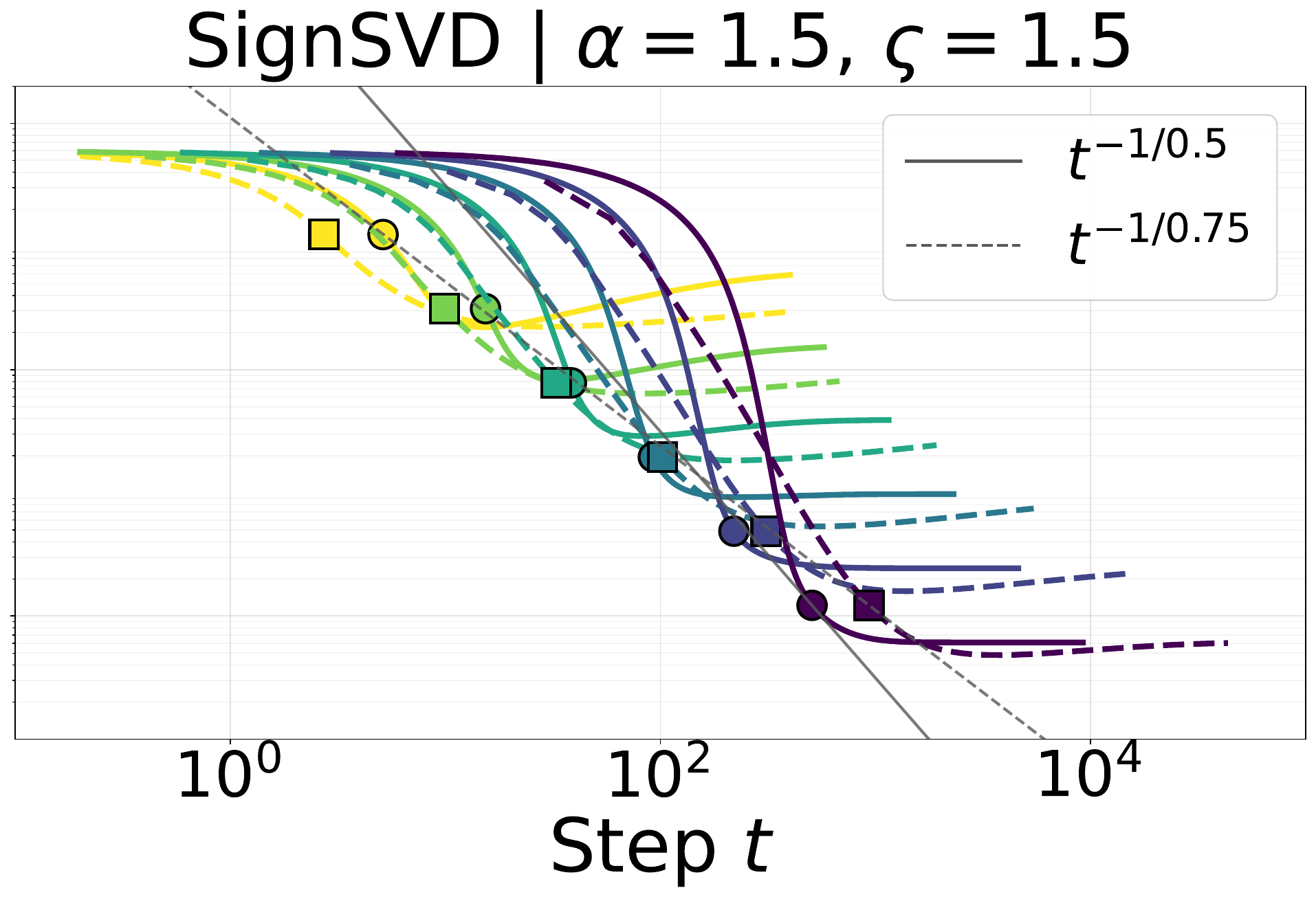}\hfill
\includegraphics[width=0.245\textwidth]{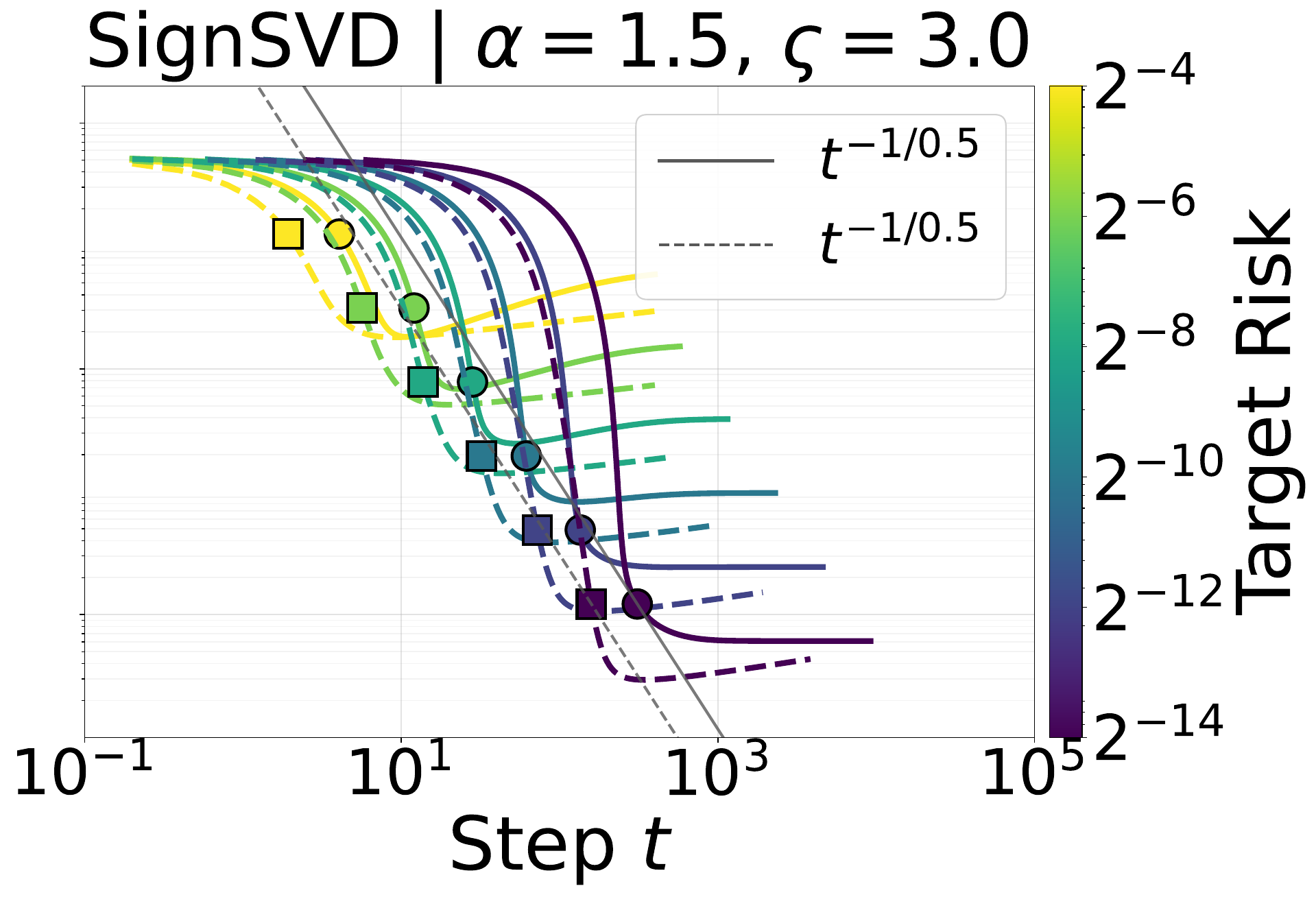}
\vspace{-0.25cm}
\caption{\small \textbf{Deterministic risk trajectories $\fRisk(t)$ and $t_{4\epsilon}$ at $N=1024$, $B=2048$}: isotropic data and the power-law points $(\alpha,\texp)=(1.5,0.7)$, $(1.5,1.5)$, $(1.5,3.0)$, for $\epsilon\in[2^{-14},2^{-4}]$. Circles: $t_{4\epsilon}^{\text{\Muon}}$ (the same for every $\bmom$ in the power-law panels); squares: $t_{4\epsilon}^{\text{\SignSGD}}$; grey lines: the predicted powers of $\epsilon$. At $\texp=0.7$ and $\texp=1.5$ (Phase~C) \Muon pulls ahead polynomially in $1/\epsilon$; at $\texp=3.0$ (Phase~B) both scale as $\epsilon^{-1/2}$ and \SignSGD\ wins by the factor $D_N\approx0.42$.}
\label{fig:tte-four-panel}
\end{figure}
\vspace{-0.5cm}
\paragraph{Conclusions.} Our analysis of the high-dimensional dynamics of spectral algorithms in the high-dimensional
matrix-valued linear regression suggests that even this relatively simple model can capture
key relationships between diagonal and non-diagonal optimization methods. Future work includes studying power-law random features models and deriving explicit compute-optimal neural scaling laws.
In particular, our work does not capture the effects of feature learning, or non-linear dynamical effects like the edge-of-stability \citep{cohen2021gradient, cohen2022adaptive}. 
The results in the anisotropic case suggest that there is an important interplay between the
batch size, data distribution, and the ability to usefully apply non-diagonal preconditioning.
This may be helpful to explain the mixed success of deploying \Muon in different workloads, and may
point to new directions in algorithmic design space.

\clearpage
\newpage

\begin{ack}
This research was enabled in part by compute resources, software and technical help provided by McGill math cluster (\url{https://www.mcgill.ca/mathstat/}) as well as by support provided by Calcul Québec (\url{https://www.calculquebec.ca/}) and the Digital Research Alliance of Canada (alliancecan.ca). No Google computing resources or data were used in this submission. 

N. Marshall is supported by a CGS-D from the Natural Science and Engineering Research Council (NSERC). L. Benigni is supported by NSERC RGPIN 2023-03882 \& DGECR 2023-00076 and FRQNT \'Etablissement de la Rel\`eve Professorale 364387. C. Paquette is a Canadian Institute for Advanced Research (CIFAR) AI
chair, Quebec AI Institute (Mila) and a Sloan Research Fellow in
Computer Science (2024). C. Paquette is supported by a Discovery Grant from the
 Natural Science and Engineering Research Council (NSERC) of Canada,
 NSERC CREATE grant Interdisciplinary Math and Artificial Intelligence
 Program (INTER-MATH-AI), Google x Mila research grant, Fonds de recherche du Quebec Nature et technologies (FRQNT) NOVA Grant (DOI: \href{https://doi.org/10.69777/363444}{https://doi.org/10.69777/363444}), and CIFAR AI Catalyst Grant. Research by E. Paquette was supported by a Discovery Grant from the Natural Science and Engineering Research Council (NSERC) and a gift from Google Canada.  E. Paquette also acknowledges support from AFOSR grant FA9550261B232. Additional revenues related to this work: C. Paquette has 20\% part-time employment at Google DeepMind. 
\end{ack}

\bibliographystyle{plain}
\bibliography{cr-references}

\clearpage
\begin{center}
{\Large Supplementary Material}\\
\end{center}

\paragraph{Broader Impact Statement.} The work presented in this paper is foundational research and it is not tied to any particular application. The primary set-up is on a simple well-studied matrix-valued linear regression problem on synthetic data and solved using commonly deployed algorithms, variants of \Adam (\SignSGD) and \Muon (\SignSVD). We present deterministic dynamics for the training loss for a class spectral algorithms. The results are theoretical and we do not anticipate any direct ethical and societal issues. The results show in a highly synthetic setting the scaling behaviour of \Muon/\SignSVD relative to \SignSGD as batch size and parameter size grow. This does have some implications to practitioners in designing better compute-optimal scalings for these commonly used algorithms in practice on large language models. While a common positive foreseeable impact of algorithms with better scaling properties is that it would allow more efficient model training and hence lower energy consumption of AI, it is important to acknowledge that, many times across history, cost-lowering technological improvements have led to increase in consumption due to Jevons paradox. Finally, since our main results are theorems regarding scaling laws, we anticipate that one potential negative impact of such a line of research is that fundamental advance in scaling laws will advantage large entities that have more computational resource, thus enhancing centralization of technological power. 

\paragraph{LLM Declaration.} Large language models were used extensively in this work. The
project started from several proofs, some written by hand and some generated
semi-autonomously. These were then greatly expanded through an LLM-led exploration, which
added the analysis with momentum and a complete proof of the dynamical comparison between the
algorithm and its deterministic dynamics. The resulting appendix was audited for style and
correctness. Its statements and proofs were then formalized in Lean~4 with Mathlib and checked
by machine, using standard results proved in the accompanying foundations library,
and the formalization status of every statement is listed in \cref{sec:supp-lean} of the
supplement. Independently, the theorems were verified numerically. In the main text, LLMs were
used for some initial drafting, for generating figures, and for connecting the text to the
proofs; all text was revised by hand. LLMs were also used to write code for the numerical
validation, for example to solve the deterministic equivalent.

\paragraph{Proofs and code.} The accompanying website provides interactive access to the
statements, proof sketches and complete written proofs:
\url{https://mlexpos.github.io/phases-of-muon/}.
The Lean formalization is available at \url{https://github.com/mlexpos/phases-of-muon},
with its accompanying foundations libraries at \url{https://github.com/mlexpos/alea-foundations}.


\appendix

\tableofcontents

\clearpage

\section{Additional Related Work} \label{sec:additional_related_work}

\paragraph{\Muon \& \SignSVD.} The original motivation for \Muon \cite{jordan2024muon} stems from steepest descent with respect to the spectral norm \cite{bernstein2024old}. The method shares key features with related approaches, including stochastic spectral descent \cite{carlson2015stochastic} and orthogonalized gradient methods \cite{tuddenham2022orthogonalising}, and can also be interpreted as a form of \SignSGD \cite{bernstein2018signsgd}. Its empirical success \cite{wen2025fantastic, shah2025practical, jordan2024muon,liu2025muon} has motivated a number of recent variants \cite{ma2024swan, liu2025cosmos, riabinin2025gluon, pethick2025training}. Given the growing interest in \Muon, the discussion below is necessarily non-exhaustive and highlights only a subset of recent theoretical and empirical developments.

On the theoretical side, several works have analyzed the performance of \Muon. For instance, \cite{shen2025convergence,shulgin2025beyond,chen2025muon,kim2026convergence} study \Muon or closely related variants such as \SignSVD using classical worst-case complexity arguments for non-convex optimization, showing that these methods achieve $\mathcal{O}(1/\sqrt{T})$ convergence rates comparable to SGD. More specifically, \cite{shulgin2025beyond} analyzes \Muon via an inexact linear minimization oracle, while \cite{chen2025muon} connects \Muon to a broader class of algorithms known as Lion-$\mathcal{K}$. In \cite{kim2026convergence}, the authors further characterize the dependence on the number of Newton--Schultz steps and compare \Muon with exact polarization. In addition, \cite{jiang2026adaptive} establishes regret bounds for \Muon.

Several works have studied \SignSVD \cite{vasudeva2025muon} and \Muon \cite{gonon2026insights} in simplified linear settings. In \cite{vasudeva2025muon}, explicit loss dynamics for \SignSVD are derived in the gradient flow regime (vanishing learning rate and negligible stochasticity), enabling analysis of generalization for both \SignSVD and standard gradient descent; see also \cite{peyre2026muon} which studies gradient flow \SignSVD but on more complicated loss functions. Meanwhile, \cite{gonon2026insights} shows that even for simple quadratic objectives, \Muon exhibits nontrivial behavior. However, these works do not operate in the high-dimensional regime considered here. In contrast, \cite{wang2025high} studies \SignSVD on a high-dimensional isotropic Gaussian least squares problem and derives explicit risk dynamics, but does not analyze the implications of the dynamics. Building on this line of work, this work moves beyond the isotropic setting and analyze the precise scaling behavior of both \Muon and \SignSVD as well as the explicit dynamics.

Beyond least squares, \cite{fan2025implicit, tsilivis2024flavors} study the implicit bias of spectral methods in linear multiclass logistic regression with separable data. Additionally, \citep{kim2026sharp, li2026muon} analyze \Muon in attention through associative memory models; in particular, \cite{kim2026sharp} shows that the critical batch size scales as $\sqrt{\Nin}$ under power-law covariance. To our knowledge, this is the first work to provide explicit training dynamics of spectral methods over the full course of optimization with anisotropic data.

Finally, several works aim to explain the effectiveness of \Muon and related spectral optimizers in deep learning, primarily through the lens of a \textit{single update}. For instance, \cite{davis2025spectral} identifies a layerwise condition, known as low stable rank, under which a spectral update yields a larger loss decrease than a Euclidean gradient step in the deterministic \Muon setting, and shows that intermediate-layer activations in deep networks often satisfy this property. Similarly, \cite{su2025isotropic} analyzes a single deterministic \Muon update under an isotropic assumption on the data and the loss, reducing the loss to a univariate function and showing that \Muon is directionally optimal, though not strictly optimal. However, \cite{gonon2026insights} demonstrates on simple quadratic problems that single-step analyses do not reliably predict overall performance. In contrast, our work provides explicit dynamics for \textit{stochastic} spectral optimizers over the full course of training.

\paragraph{Random matrix theory and machine learning.} This paper uses random matrix theory to analyze a matrix least squares problem. In particular, our analysis requires analysis of spectral properties of sample covariance matrices (e.g., power-law population covariance), and the use of deterministic equivalents for the resolvents of random matrices. Both these sets of tools have long predated modern machine learning literature, yet have found growing use in the field. See \cite{couillet2022random} for a modern introduction.

\paragraph{High-dimensional optimization literature for SGD and its variants.}
Most of the literature on high-dimensional optimization focuses on streaming SGD under certain settings such as least squares beyond isotropic settings \cite{paquette2022homogenization}. A few works have expanded beyond least squares, see e.g., \cite{collinswoodfin2023hitting, mei2019generalization,ben2022high,montanari2026phase}, but again with focus on streaming SGD with many working on power-law covariances \cite{bordelon2025how,paquette20244+,lin2024scaling,ferbach2025dimension,bordelon2024dynamical,arous2025learning}. There have been some extensions beyond SGD. Notably, \cite{paquette2022homogenization, paquette2021sgd,zou2022risk} gave exact dynamics for multi-pass SGD on a high-dimensional empirical least squares problem with non-Gaussian covariances. Extensions for deterministic dynamics of momentum algorithms in small batch \cite{li2023risk,varre2022accelerated,paquette2021dynamics,yarotsky2025sgd} and large batch setting \cite{lee2022trajectory} as well as power-law covariances in \cite{ferbach2025dimension} where Nesterov acceleration methods with time-dependent momentum parameters were explored on least squares problems and single/multi-index \cite{gheissari2025universality}. Deterministic dynamics for high-dimensional optimization beyond the SGD/SGD with momentum framework have been considered, see \cite{jagannath2025high,collins2024high} for stochastic adaptive learning rates (e.g., AdaGrad-Norm), \SignSGD \cite{kim2026scaling,xiao2025exact}, and clipped SGD \cite{marshall2024clip}. High-dimensional limits for spectral optimizers have not yet been explored in detail, with the only related result being the workshop paper \cite{wang2025high}.

\paragraph{\SignSGD.} \SignSGD has been proposed as a simplified proxy for \Adam, of which 
it is a special case (setting $\beta_1 = \beta_2 = \varepsilon = 0$) 
\cite{balles_dissecting_adam, balles_pedregosa2020_geom_signSGD}. 
In \cite{xiao2025exact}, the authors study \SignSGD on a 
high-dimensional linear regression problem, deriving explicit loss 
dynamics that precisely quantify the \emph{diagonal preconditioning} 
effect of \SignSGD and characterize how it reshapes the distribution 
of gradient noise. 

\clearpage

\section{Numerical simulations for Figures~\ref{fig:iso-risk-constlr}, \ref{fig:isotropic-main} and \ref{fig:isotropic-speedup}}\label{sec:app-numerics}

This appendix describes the simulations behind Figs.~\ref{fig:iso-risk-constlr}, \ref{fig:isotropic-main} and \ref{fig:isotropic-speedup} and the checks of the
approximations \eqref{eqn:C-integral}. Nothing in it is used in a proof.

\paragraph{Model and algorithms.}
All runs use the square isotropic model with $\Nin=\Nout=N=512$ and $\Sigmain=\Sigmaout=\Id_N$.
At an error $\Delta=W-W^\star$, a batch of size $B$ gives the gradient
\[
G=\frac1B\sum_{b=1}^B x_{\rm out}^b\otimes x_{\rm in}^b\,\big\langle x_{\rm out}^b\otimes x_{\rm in}^b,\Delta\big\rangle,
\qquad \EE\,G=\Delta,\qquad \cR=\tfrac12\norm{\Delta}_\F^2 .
\]
\Muon keeps the buffer $M_t=\bmom M_{t-1}+(1-\bmom)G_{t+1}$ and steps $W_{t+1}=W_t-\eta S_t$.
With the exact polar factor, $S_t=UV^\T$ is computed from a singular value decomposition of $M_t$
over the singular values above $10^{-13}$ times the largest. With five Newton--Schulz steps (NS-5),
$S_t$ is the result of five iterations of
\[
X\leftarrow aX+\big(bXX^\T+c(XX^\T)^2\big)X
\]
from $X_0=M_t/\norm{M_t}_\F$, with $(a,b,c)=(3.4445,-4.7750,2.0315)$.
\SignSGD steps $W_{t+1}=W_t-\eta\,\mathrm{sign}(G_{t+1})$, without momentum. All arithmetic is in double precision.

\paragraph{Panels (a) and (b): frozen error.}
The error is frozen at $\Delta=\Id_N/\sqrt N$; by the invariance of the model under
$\Delta\mapsto O_1\Delta O_2$ with orthogonal $O_1,O_2$, this is equivalent in law to any $\Delta$
with the same singular values. The buffer is run to stationarity for
$\lceil 37/\log(1/\bmom)+N/B\rceil$ steps (none at $\bmom=0$), after which $T$ steps are recorded,
with $T=100$ at $\bmom=0$ and $T=\max(600,3N/B+300)$ at $\bmom=0.9$, two seeds each.
The drift estimate is the mean of $\langle S_t,\Delta\rangle/(\sqrt2\norm{\Delta}_\F)$.
The volatility estimate is the mean of $\norm{S_t}_\F^2$ plus twice the centred lag sum
$\sum_{h\ge1}\langle S_t-\bar c\hat\Delta,\,S_{t-h}-\bar c\hat\Delta\rangle$, where $\hat\Delta=\Delta/\norm{\Delta}_\F$ and
$\bar c$ is the mean of $\langle S_t,\hat\Delta\rangle$; it is averaged over $t\ge T/2$, so that every lag up to
$T/2$ enters with full weight.
Both estimates agree with the quadrature to within about $1\%$ for $B_{\rm eff}/N\le4$. For larger
$B_{\rm eff}/N$ the measured volatility lies below the $N\to\infty$ constant by a finite-$N$ correction of
relative size about $0.7\,B_{\rm eff}/N^2$: at $\bmom=0.9$ it is $2\%$ at $B_{\rm eff}/N=16$ and $9\%$ at
$B_{\rm eff}/N=64$ for $N=512$, and it roughly doubles each time $N$ is halved ($18\%$ and $32\%$ at
$B_{\rm eff}/N=64$ for $N=256$ and $128$). Quadrupling $T$ does not change it. The drift is $1.4\%$ low at
$B_{\rm eff}/N=64$.

\paragraph{Panel (c): rates.}
Each run starts from a Gaussian $\Delta_0$ with $\cR(0)=6.25$ (or $4$ at the smallest batches)
and uses the constant step $\eta=4\mathfrak d/\mathfrak v$ built from the theory constants, which puts
the floor at $\cR_\star=1$ for full whitening. The recursion \eqref{eqn:comparison-general-form},
$\cR(t{+}1)=\cR(t)-a\sqrt{\cR(t)}+b$ with $a=2\eta\mathfrak d$ and $b=\eta^2\mathfrak v/2$, is fitted to
the seed-averaged risk by least squares in $\log\cR$. Near the floor $\sqrt{\cR_\star}=b/a$ the
recursion contracts by the factor $1-a^2/(2b)=1-4\mathfrak d^2/\mathfrak v$ per step, which does not depend
on $\eta$, and we report
\[
1-\text{rate}=\frac{a^2}{4b}=\frac{2\mathfrak d^2}{\mathfrak v},
\]
the factor of \eqref{eqn:rate-signsgd}. So the measured rate does not depend on the step, even when NS-5
reaches a floor different from the one it was tuned for. Runs last between $7/\kappa$ and $9/\kappa$ steps,
$\kappa=8\mathfrak d^2/\mathfrak v$.

The \SignSGD drift is $\mathcal N_B/\sqrt\pi$ with $\mathcal N_B$ as in Sec.~\ref{sec:isotropic}, and the
main text writes $\sqrt B$ for $\mathcal N_B$. In panel (c) each measured ratio is multiplied by
$\mathcal N_B^2/B$, which puts \SignSGD on the $\sqrt B$ scale of the curves. We use
$\mathcal N_1=2/\pi$, a Monte Carlo table with $2\cdot10^6$ samples for $2\le B\le64$, and
\[
\mathcal N_B=\sqrt B\Big(1-\frac1B+\frac5{B^2}+O(B^{-3})\Big)
\]
for $B>64$. The expansion follows from the cumulants $1$, $8$ and $200$ of $(xy)^2$; at $B=64$ it agrees with the table to $4\cdot10^{-5}$.

\paragraph{Grid.}
Panel (c) uses $B_{\rm eff}/N=2^k$, $k=-4,\dots,4$, at $\bmom\in\{0,0.5,0.9\}$, with \SignSGD, NS-5 and the
exact polar factor at each point and two seeds, between $10^2$ and $6\cdot10^5$ steps. Full whitening
inverts the oldest direction of the buffer, whose weight is $\bmom^{N/B}$. Where this is below
$10^{-11}$ the exact polar factor is not computable in double precision, and that marker is
omitted: $\bmom=0.5$ at $B_{\rm eff}/N=1/16$ ($0.5^{46.5}\approx10^{-14}$) and $\bmom=0.9$ at
$B_{\rm eff}/N=1/16$ ($0.9^{256}\approx2\cdot10^{-12}$). The same rule removes the frozen point at
$\bmom=0.9$, $B=2$ from panels (a) and (b).

\paragraph{Newton--Schulz.}
The quintic coefficients above do not converge to the polar factor: with many steps the singular
values of $S_t$ oscillate in a band near $[0.68,1.13]$, so $\norm{S_t}_\F^2\approx0.86N$. This is close
to a scalar multiple of the polar factor; it changes the floor, but it cancels in
$2\mathfrak d^2/\mathfrak v$. In frozen runs at $\bmom=0.9$, $B=13$ ($B_{\rm eff}/N\approx0.48$), the rate
with $K$ steps equals the exact rate within the seed spread for every $K\ge7$. With $K=5$ the drift is
$16\%$ lower and the volatility $38\%$ lower, so the rate is $14\%$ higher: five steps do not reach the
oldest buffer directions, whose singular values are below $a^{-5}\norm{M_t}_\F\approx2\cdot10^{-3}\norm{M_t}_\F$,
and act as a small positive ridge. The effect grows as $B_{\rm eff}/N$ decreases; at $\bmom=0.5$ the
NS-5 rate exceeds the full-whitening rate by factors $1.15$, $1.42$, $1.76$ and $1.96$ at
$B_{\rm eff}/N=1/2$, $1/4$, $1/8$ and $1/16$, and at $\bmom=0.9$ by $1.37$, $1.79$ and $1.92$ at
$B=7$, $3$ and $2$ ($B_{\rm eff}/N\approx0.26$, $0.11$ and $0.07$). For $B_{\rm eff}\ge 2N$ NS-5 and full whitening agree within $3\%$.
Without momentum the update has rank $\min\{B,N\}$ and there is no stale direction to suppress, and the sign
of the effect reverses: for $B<N$ the NS-5 rate is $5$--$8\%$ below the exact one ($6\%$, $8\%$ and $5\%$ at
$B/N=1/8$, $1/4$ and $1/2$), and at $B=N$ it is $7\%$ above.

\paragraph{The approximations \eqref{eqn:C-integral}.}
The forms for $\bmom>0$ interpolate between the two batch limits: additively in $C_\bmom^2$, because the
thin-batch limit $1/(2\sqrt\pi)$ is a positive constant, and additively in
$\mathfrak v_{\mathrm{iso},\bmom}B/B_{\rm eff}$. Inserted in the ratio of the factors $1-\text{rate}$, they give
\[
\frac{1-\text{rate}_{\text{Muon}}}{1-\text{rate}_{\text{s-SGD}}}\approx
\frac{32/(9\pi)+N/(4B_{\rm eff})}{c_\bmom+N/(2B_{\rm eff})},
\]
which tends to $32/(9\pi c_\bmom)$ for $B_{\rm eff}\gg N$ and to $\tfrac12$ for $B_{\rm eff}\ll N$.
Figure~\ref{fig:app-numerics-approx} compares them with the quadrature of the exact constants, and the
table gives the largest relative error over $1/64\le B_{\rm eff}/N\le64$; all three approximations
overestimate, most near $B_{\rm eff}/N\approx0.3$.
\begin{figure}[h]
\centering
\includegraphics[width=\textwidth]{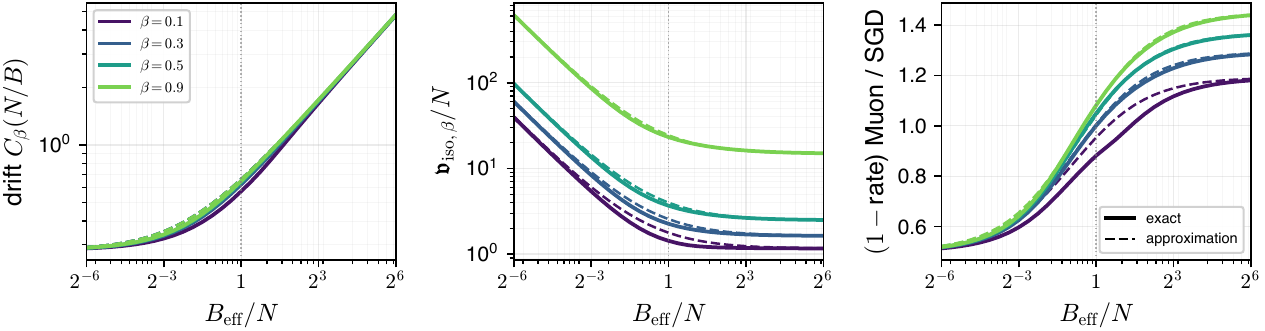}
\caption{The algebraic forms \eqref{eqn:C-integral} (dashed) against the exact zero-ridge constants (solid),
for $\bmom\in\{0.1,0.3,0.5,0.9\}$. \textbf{Left:} drift constant. \textbf{Centre:} integrated volatility.
\textbf{Right:} ratio of the factors $1-\text{rate}$.}
\label{fig:app-numerics-approx}
\end{figure}
\begin{center}
\begin{tabular}{lccccc}
\toprule
$\bmom$ & $0.1$ & $0.3$ & $0.5$ & $0.9$ & $0.95$\\
\midrule
$C_\bmom$ & $16\%$ & $8.0\%$ & $5.7\%$ & $4.6\%$ & $4.6\%$\\
$\mathfrak v_{\mathrm{iso},\bmom}$ & $26\%$ & $16\%$ & $12\%$ & $8.8\%$ & $8.7\%$\\
ratio of $1-\text{rate}$ & $8.7\%$ & $2.8\%$ & $2.1\%$ & $2.1\%$ & $2.1\%$\\
\bottomrule
\end{tabular}
\end{center}
The accuracy degrades as $\bmom\downarrow0$ because the limit $\bmom\downarrow0$ at zero ridge is
discontinuous for $N>B$ (\cref{prop:app-iso-beta-zero}).

\clearpage

\section{Finite-ridge effects}\label{sec:app-nsridge}

The theory of the paper is proved for the regularized \SignSVD map $\Qridge\ridge$ of
\cref{def:app-setup-algorithm}, and its sharpest statements are made in the limit of vanishing
ridge, which is exact whitening. Practical \Muon does not whiten exactly: it applies a few
Newton--Schulz steps, which amplify small singular values only up to a point. This appendix
collects numerical evidence on what such imperfect whitening does, and on when a positive ridge
describes it. Nothing in this appendix is used in a proof.

\subsection{Newton--Schulz as a ridge}\label{sec:app-nsridge-ns}

Fig.~\ref{fig:isotropic-speedup} shows that \Muon with five Newton--Schulz steps (NS-5) and momentum behaves
like a small positive ridge. This subsection measures, for every number of steps $K\le7$, which
ridge reproduces NS-$K$ in the isotropic model, how accurately, and where the description fails.
The calibration is made on a frozen error (Secs.~\ref{sec:app-nsridge-method} and \ref{sec:app-nsridge-results})
and then tested, without refitting, on the dynamics (Sec.~\ref{sec:app-nsridge-dynamic}).

\subsubsection{Frozen-error calibration: method}\label{sec:app-nsridge-method}

\paragraph{Model and updates.}
All runs use the square isotropic model of Appendix~\ref{sec:app-numerics}, with $N=512$ unless
stated otherwise, momentum buffer $M_t=\bmom M_{t-1}+(1-\bmom)G_t$ and $\bmom\in\{0,0.05,0.1,0.25,0.5,0.9\}$.
Every update we compare is a spectral function of the buffer. If $M_t=\sum_k\sigma_k\,p_k\otimes r_k$ is a
singular value decomposition, then the update is $\sum_k\phi(\sigma_k)\,p_k\otimes r_k$ with
\[
\phi_{\rm polar}(\sigma)=1,\qquad
\phi_\varrho(\sigma)=\frac{\sigma}{\sqrt{\sigma^2+\varrho}},\qquad
\phi_{{\rm NS}\text{-}K}(\sigma)=p^{\circ K}\Big(\frac{\sigma}{\norm{M_t}_\F}\Big),
\]
where $p(x)=ax+bx^3+cx^5$ with $(a,b,c)=(3.4445,-4.7750,2.0315)$ and $p^{\circ K}$ is its $K$-fold
composition. The last formula is exact: $K$ Newton--Schulz iterations from $X_0=M_t/\norm{M_t}_\F$
act on each singular value separately.

\paragraph{Units of the ridge.}
Newton--Schulz acts on $M_t/\norm{M_t}_\F$, so we parametrize the ridge in the same units,
\[
\varrho=\varrho_X\,\norm{M_t}_\F^2 .
\]
In the frozen state of Appendix~\ref{sec:app-numerics}, which has $2\cR=1$, the normalized and the
physical ridge coincide, and the measured $\norm{M_t}_\F^2$ agrees with $\sw2N^2/B=N^2/B_{\rm eff}$ to
within $5\%$ in every run (median ratio $1.013$). A fixed $\varrho_X$ therefore corresponds to the
normalized ridge $\varrho_X N^2/B_{\rm eff}$, which is $\varrho_X N$ times the mean eigenvalue of
$M_tM_t^\T$. The heuristic of \cref{sec:spectral_algs} matches the small-$\sigma$ slope $a^K$ of
$\phi_{{\rm NS}\text{-}K}$ with the slope $\varrho_X^{-1/2}$ of $\phi_\varrho$, which gives
$\varrho_X=a^{-2K}$.

\paragraph{The frozen filter bank.}
The error is frozen at $\Delta=\Id_N/\sqrt N$, and the buffer is burnt in and recorded exactly as in
Appendix~\ref{sec:app-numerics}: $\lceil37/\log(1/\bmom)+N/B\rceil$ burn-in steps (none at $\bmom=0$),
then $T=\max(600,3N/B+300)$ recorded steps ($T=400$ at $\bmom=0$). Because the error is frozen, the
sequence $M_t$ does not depend on the update rule. One double-precision singular value decomposition
per step therefore gives, on the same random numbers, the frozen statistics of the exact polar
factor, of NS-$K$ for $K=1,\dots,7$, and of exact ridges at the $16$ values
$\varrho_X\in\{10^{-10},10^{-9.33},\dots,1\}$. For each update we record the drift $\mathfrak d$, the
integrated volatility $\mathfrak v$ and the same-time volatility $\mathfrak v_{\rm same}=\EE\norm{S_t}_\F^2$ with
the estimators of Appendix~\ref{sec:app-numerics}.

\paragraph{Two scale-free statistics.}
NS-$K$ does not output singular values equal to one, so its updates differ from any ridge by an
overall amplitude. Two statistics do not change when the update is multiplied by a constant:
\[
I_1=\frac{\mathfrak d^2}{\mathfrak v},\qquad I_2=\frac{\mathfrak v_{\rm same}}{\mathfrak v}.
\]
The first sets the rate at the optimal step, $1-\text{rate}=2I_1$, and $\pi\gamma N I_1$ is the ratio
plotted in Fig.~\ref{fig:isotropic-speedup} (the \Muon speedup). The second is the share of the volatility that comes
from a single step.

\paragraph{Calibration.}
For each configuration $(N,B,\bmom)$ and each $K$, the effective ridge $\varrho_X(K)$ is the value at
which the exact-ridge bank has the $I_1$ of NS-$K$, found by linear interpolation in
$\log\varrho_X$. This calibration uses no theory and no finite-$N$ correction, because the bank and
NS-$K$ are evaluated on the same buffers. As a second, independent calibration we solve the same
equation with the positive-ridge theory of Appendix~\ref{sec:app-iso}, that is $\mathfrak d_{\rm iso}(\varrho)$
and $\mathfrak v_{\rm iso}(\varrho)$ at $\varrho=\varrho_X\sw2N^2/B$. When the $I_1$ of NS-$K$ lies outside the
range spanned by the ridges, we report that no ridge matches. The fitted ridge then predicts $I_2$,
which was not used in the fit.

\paragraph{Grid.}
At $N=512$: $\bmom\in\{0,0.05,0.1,0.25,0.5,0.9\}$ and $B_{\rm eff}/N=2^k$, $k=-4,\dots,4$, two seeds,
$108$ runs. For the dependence on $N$: $N\in\{256,1024\}$ at $\bmom\in\{0.1,0.9\}$ and
$B_{\rm eff}/N\in\{1/8,1/2,2\}$, two seeds, $24$ runs.

\subsubsection{Frozen-error calibration: results}\label{sec:app-nsridge-results}

\begin{figure}[h]
\centering
\includegraphics[width=\textwidth]{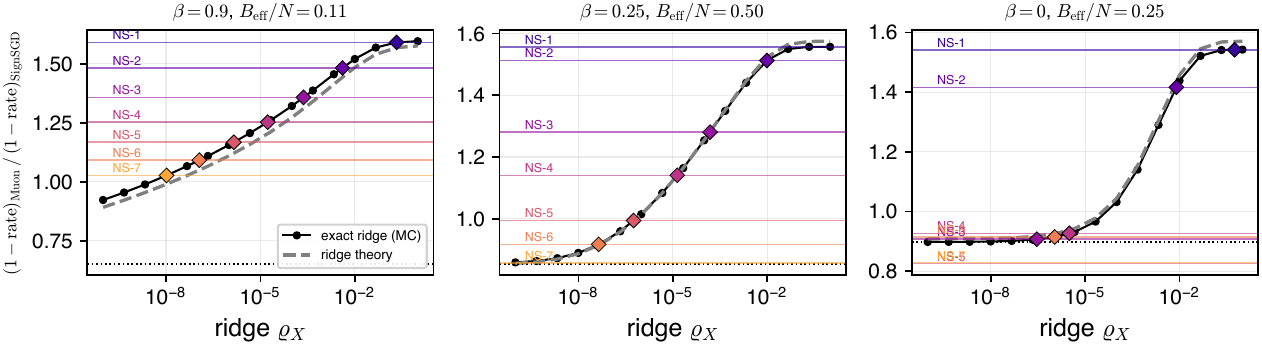}
\caption{Rate statistic $\pi\gamma NI_1=(1-\text{rate})_{\rm Muon}/(1-\text{rate})_{\rm SignSGD}$ (values above $1$: \Muon converges faster) of the exact ridge as a function of $\varrho_X$ (black: frozen bank;
dashed: ridge theory), the levels of NS-$K$ (horizontal lines) and the calibrated ridges (diamonds),
for three configurations. Dotted: exact polar factor. Right: at $\bmom=0$ the levels of NS-5 and NS-7
lie below every ridge, and no ridge matches.}
\label{fig:app-nsridge-bank}
\end{figure}

\paragraph{Calibration curves.}
Fig.~\ref{fig:app-nsridge-bank} shows the mechanism. The rate statistic of the ridge increases
monotonically from the exact polar value at $\varrho_X\to0$ to the value of a linear update at large
$\varrho_X$, and each NS-$K$ level crosses it once. The theory curve follows the bank to within a
few percent; the gap at $\bmom=0.9$, $B=3$ is the finite-$N$ drift bias of the polar factor at
large $N/B$ noted in Appendix~\ref{sec:app-numerics}. At $\bmom=0$ the NS-$K$ levels with $K\ge5$
lie below the exact polar value, which no ridge reaches.

\begin{figure}[h]
\centering
\includegraphics[width=0.48\textwidth]{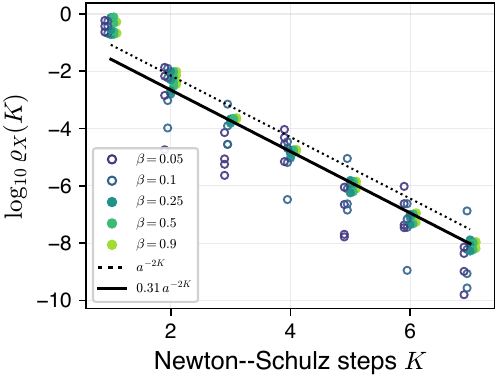}\hfill
\includegraphics[width=0.48\textwidth]{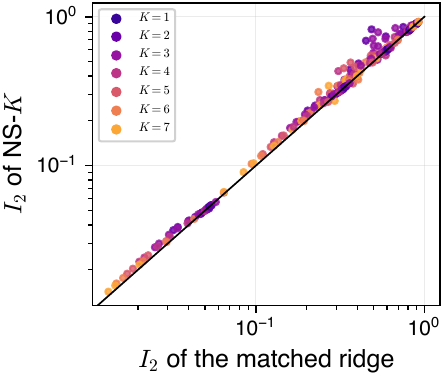}
\caption{Left: calibrated $\log_{10}\varrho_X(K)$ for every configuration with $B_{\rm eff}<N$ and $\bmom>0$
(filled: $\bmom\ge0.25$; open: $\bmom\in\{0.05,0.1\}$), with the heuristic $a^{-2K}$ (dotted) and the fit
$0.31\,a^{-2K}$ (solid). Right: the lag share $I_2$ of NS-$K$ against that of the matched ridge, all
$\bmom>0$; the line is equality.}
\label{fig:app-nsridge-law}
\end{figure}

\begin{table}[h]
\centering
\small
\begin{tabular}{cccccc}
\toprule
$K$ & bank: median [range] & theory: median & $\log_{10}a^{-2K}$ & $\log_{10}0.31a^{-2K}$ & $I_2$: NS/ridge$-1$ \\
\midrule
1 & $-0.58$ $[-0.72,-0.10]$ & $-1.37$ & $-1.07$ & $-1.58$ & $+0.006$\\
2 & $-2.34$ $[-2.80,-2.00]$ & $-2.36$ & $-2.15$ & $-2.65$ & $+0.11$\\
3 & $-3.67$ $[-3.82,-3.63]$ & $-3.75$ & $-3.22$ & $-3.73$ & $+0.13$\\
4 & $-4.78$ $[-5.03,-4.73]$ & $-4.84$ & $-4.30$ & $-4.80$ & $+0.10$\\
5 & $-5.92$ $[-6.25,-5.81]$ & $-6.03$ & $-5.37$ & $-5.87$ & $+0.09$\\
6 & $-7.01$ $[-7.35,-6.94]$ & $-7.14$ & $-6.45$ & $-6.95$ & $+0.07$\\
7 & $-8.01$ $[-8.28,-7.89]$ & $-8.19$ & $-7.52$ & $-8.02$ & $+0.07$\\
\bottomrule
\end{tabular}
\caption{Calibrated $\log_{10}\varrho_X(K)$ at $N=512$, $\bmom\in\{0.25,0.5,0.9\}$, $B_{\rm eff}/N\in\{1/16,\dots,1/2\}$
(twelve configurations per $K$, nine at $K=7$ where the polar end of the bank is not reached).}
\label{tab:app-nsridge}
\end{table}

\paragraph{The effective ridge.}
For $\bmom\ge0.25$ and $B_{\rm eff}<N$ the calibrated ridges collapse onto one law
(Fig.~\ref{fig:app-nsridge-law}, Table~\ref{tab:app-nsridge}),
\[
\varrho_X(K)\approx0.31\,a^{-2K},\qquad a=3.4445,
\]
with a scatter of $0.13$ in $\log_{10}\varrho_X$ over the $57$ calibrations with $K\ge3$. A fit with a
free slope gives $-1.11$ decades per step, against $-2\log_{10}a=-1.07$. The heuristic $a^{-2K}$ has the
right dependence on $K$ and is a factor of about three too large. The theory calibration agrees
with the bank calibration to within $0.2$ in $\log_{10}\varrho_X$ for $K\ge2$, and lies slightly below it. At $K=1$ the rate statistic of the
ridge is nearly flat at large $\varrho_X$ (one Newton--Schulz step on $M_t/\norm{M_t}_\F$ is almost linear), so
$\varrho_X(1)$ is only weakly determined.

\begin{figure}[h]
\centering
\includegraphics[width=\textwidth]{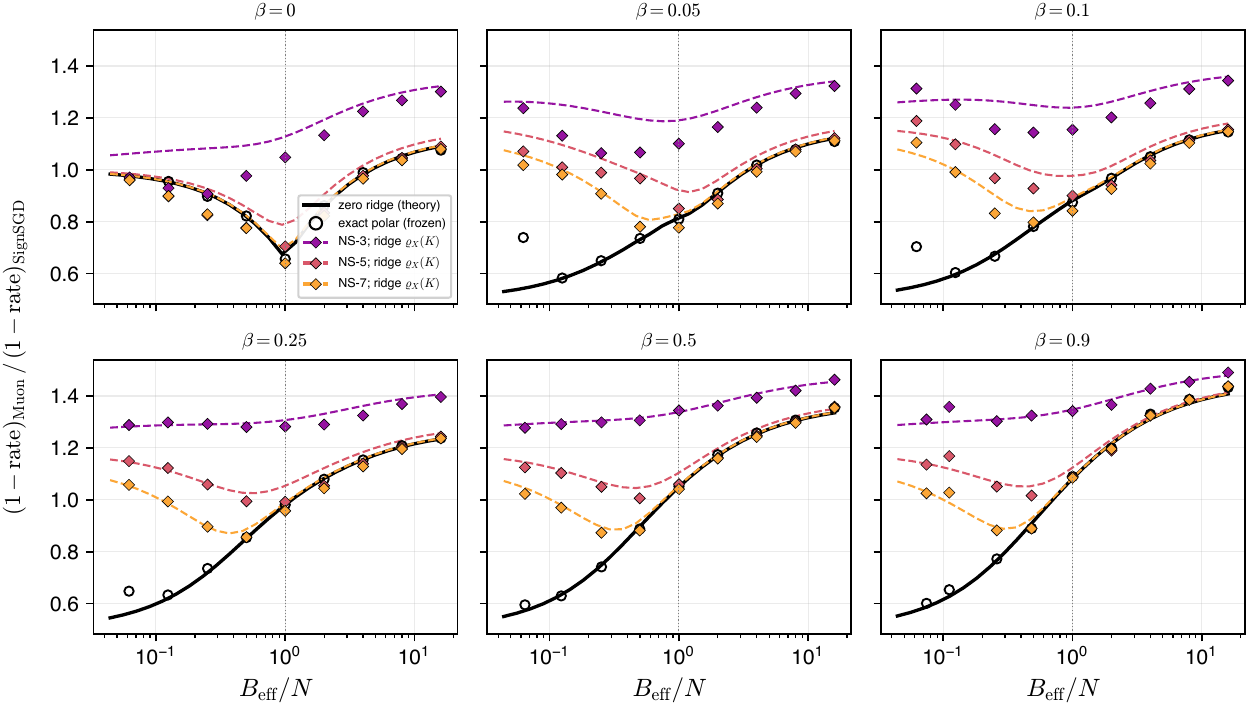}
\caption{Frozen rate statistic $\pi\gamma NI_1=(1-\text{rate})_{\rm Muon}/(1-\text{rate})_{\rm SignSGD}$ against $B_{\rm eff}/N$ for each momentum; values above $1$ mean that \Muon converges faster than \SignSGD. Black: zero-ridge theory
and frozen exact polar factor. Diamonds: frozen NS-3, NS-5, NS-7. Dashed: ridge theory at the single
calibrated ridge $\varrho_X(K)=0.31\,a^{-2K}$, the same for every panel.}
\label{fig:app-nsridge-frozen}
\end{figure}

\paragraph{Where the ridge describes Newton--Schulz.}
Fig.~\ref{fig:app-nsridge-frozen} compares every frozen NS-$K$ measurement with the ridge theory at the
calibrated $\varrho_X(K)$, with no further fitting.
\begin{enumerate}[label=(\roman*)]
\item For $\bmom\ge0.25$ the calibrated ridge reproduces NS-3, NS-5 and NS-7 across the whole range
$1/16\le B_{\rm eff}/N\le16$, including the minimum near $B_{\rm eff}/N\approx1/2$ and the gain over the exact
polar factor at small $B_{\rm eff}/N$.
\item For $B_{\rm eff}\ge N$ the effect of the ridge is small: for $K\ge4$ the rate statistic of NS-$K$ is
within $2\%$ of that of the exact polar factor, and the effective ridge is not determined there.
\item For $\bmom\in\{0.05,0.1\}$ a ridge still matches each NS-$K$, and Newton--Schulz still improves on
the exact polar factor at small batch, but the matched ridge scatters over one to two decades and
is not monotone in $K$ at $\bmom=0.05$. With small momentum the weights $\bmom^j$ of old batches
separate the singular values of the buffer into widely spaced groups, and any ridge between two
groups has the same effect.
\item At $\bmom=0$ no ridge describes NS-$K$. For $K\ge5$ and $B<N$, NS-$K$ is slower than the exact polar
factor (by $6\%$ at $K=5$), while every ridge is faster. The buffer is then a single minibatch
gradient of rank $\min(B,N)$, with no small singular values for a ridge to suppress, and the
remaining effect of Newton--Schulz is the spread of its outputs over $[0.68,1.2]$, which adds
volatility.
\end{enumerate}
At the smallest batches with $\bmom>0$ the frozen exact polar factor lies above the zero-ridge curve.
There the oldest directions of the buffer have weights below the rank threshold $10^{-13}$ of
Appendix~\ref{sec:app-numerics}, so the computed polar factor is itself a polar factor with a tiny ridge.

\paragraph{One Newton--Schulz step is SGD.}
In every panel the rate statistic decreases with $K$, and NS-1 is the fastest. One Newton--Schulz step
is nearly linear on $M_t/\norm{M_t}_\F$, whose bulk singular values are of order $N^{-1/2}$, so NS-1 is
essentially momentum SGD with a normalized step. In the isotropic model there is nothing to
precondition, and SGD is the most efficient of these updates: for nearly Gaussian gradient entries a
linear update beats the sign update by the factor $\pi/2\approx1.57$, and NS-1 sits close to this value for
every $\bmom$ and $B_{\rm eff}/N$. In this model, taking the fewest steps therefore amounts to running SGD.

\paragraph{Consistency.}
The ridge fitted on $I_1$ predicts the lag share $I_2$ of NS-$K$ to within about $10\%$
(Fig.~\ref{fig:app-nsridge-law}, right): NS-$K$ has $7$--$13\%$ more same-time volatility than the matched
ridge for $2\le K\le7$, and the same at $K=1$. The excess is the volatility added by the spread of
the Newton--Schulz outputs, which a ridge does not have.

\begin{figure}[h]
\centering
\includegraphics[width=\textwidth]{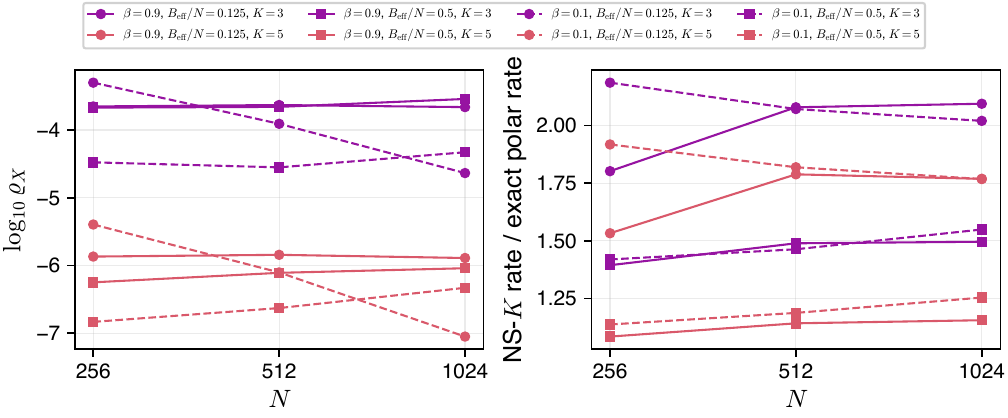}
\caption{Dependence on $N$. Left: calibrated $\log_{10}\varrho_X$. Right: rate statistic of NS-$K$ divided by that
of the exact polar factor. Solid: $\bmom=0.9$; dashed: $\bmom=0.1$.}
\label{fig:app-nsridge-N}
\end{figure}

\paragraph{Dependence on the dimension.}
At $\bmom=0.9$ the calibrated $\varrho_X$ does not depend on $N$:
\[
\log_{10}\varrho_X(5)=-5.87,-5.84,-5.89,\qquad
\log_{10}\varrho_X(3)=-3.65,-3.63,-3.66
\]
at $N=256,512,1024$ and $B_{\rm eff}/N=1/8$
(Fig.~\ref{fig:app-nsridge-N}). The Frobenius-normalized ridge is therefore the right unit, and the
normalized ridge $\varrho_X N^2/B_{\rm eff}$ grows with $N$ at a fixed $B_{\rm eff}/N$. Accordingly the gain of
NS-5 over the exact polar factor grows with $N$, from $1.09$ to $1.14$ to $1.16$ at $B_{\rm eff}/N=1/2$.
At $\bmom=0.1$ the calibrated ridge drifts with $N$, as the scatter in (iii) suggests.

\subsubsection{Dynamical validation}\label{sec:app-nsridge-dynamic}

\paragraph{Protocol.}
The dynamical runs follow the protocol of Fig.~\ref{fig:isotropic-speedup}, described in
Appendix~\ref{sec:app-numerics}. Each run starts from a Gaussian $\Delta_0$ with $\cR(0)=6.25$ ($4$ at the
smallest batches), warms the buffer at $W_0$ for the burn-in of the frozen runs, and uses the constant
step $\eta=4\mathfrak d/\mathfrak v$ built from the frozen $(\mathfrak d,\mathfrak v)$ of the same NS-$K$, so that its floor
is near $\cR_\star=1$. Runs last $7/\kappa$ to $9/\kappa$ steps, $\kappa=8\mathfrak d^2/\mathfrak v$, between $2\cdot10^2$ and
$3.5\cdot10^5$ steps. Newton--Schulz is computed by matrix products in single precision, as in practice,
not by the spectral formula. The grid is $N=512$, $\bmom\in\{0,0.1,0.25,0.5,0.9\}$, $B_{\rm eff}/N=2^k$ with
$k=-4,\dots,3$, $K=1,\dots,7$ and two seeds, $560$ runs; the runs with different $K$ share their initial
errors. The rate is read from the recursion fit of Appendix~\ref{sec:app-numerics}, and we report the
rate statistic $(1-\text{rate})_{\rm Muon}/(1-\text{rate})_{\rm SignSGD}$ on the $\sqrt B$ scale of the main text.

\begin{figure}[h]
\centering
\includegraphics[width=\textwidth]{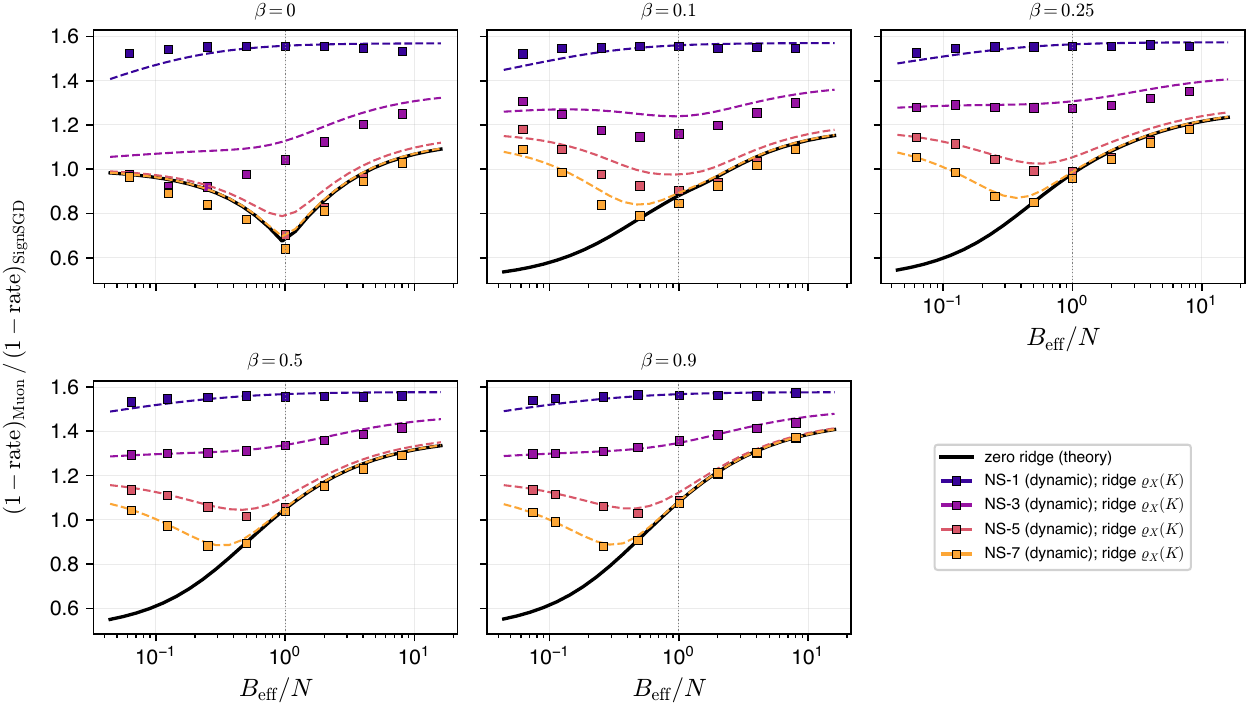}
\caption{Dynamical rate statistic $(1-\text{rate})_{\rm Muon}/(1-\text{rate})_{\rm SignSGD}$ of NS-$K$ (squares; values above $1$:
\Muon converges faster) against the ridge theory at the frozen calibration $\varrho_X(K)=0.31\,a^{-2K}$ (dashed),
with no refitting. Solid: zero ridge.}
\label{fig:app-nsridge-dynamic}
\end{figure}

\paragraph{Results.}
The dynamics reproduce the frozen calibration, and the calibrated ridge predicts the dynamical rates
where the frozen calibration said it would (Fig.~\ref{fig:app-nsridge-dynamic}).
\begin{enumerate}[label=(\roman*)]
\item The dynamical rate statistic agrees with the frozen $\pi\gamma NI_1$ of the same NS-$K$ to $0.8\%$ on
average and $4.5\%$ at most, over all $280$ pairs of configuration and $K$; the two seeds differ by $0.3\%$ in the median.
Single-precision Newton--Schulz therefore behaves as its double-precision spectral formula.
\item For $\bmom\ge0.25$ the ridge theory at $\varrho_X(K)$ predicts the dynamical rate of NS-$K$ to $0.6\%$ on average
($2.7\%$ at most) for $B_{\rm eff}/N\le1/4$, and to $2.1\%$ on average at $B_{\rm eff}/N=1/2$. For $B_{\rm eff}\ge N$ it is
$2.6\%$ too fast on average ($8.6\%$ at most): there NS-$K$ is close to the exact polar factor, and the
spread of its outputs, which a ridge does not have, is the larger effect.
\item For $\bmom=0.1$ the error is $4.8\%$ on average and up to $11\%$, and for $\bmom=0$ it is $6.5\%$ on average
and up to $15\%$, with the ridge predicting too fast a rate, as the frozen results anticipated.
\end{enumerate}

\subsubsection{Learning rates}\label{sec:app-nsridge-lr}

Sec.~\ref{sec:iso-lr} gives the optimal learning rate $\eta_\star/\sqrt{\fRisk}=2\mathfrak d/\mathfrak v$ of exact
whitening. Fig.~\ref{fig:app-nsridge-lr} shows the same quantity for NS-$K$, measured on the frozen bank
($N=512$), against the ridge theory at $\varrho_X(K)$.

\begin{figure}[h]
\centering
\includegraphics[width=\textwidth]{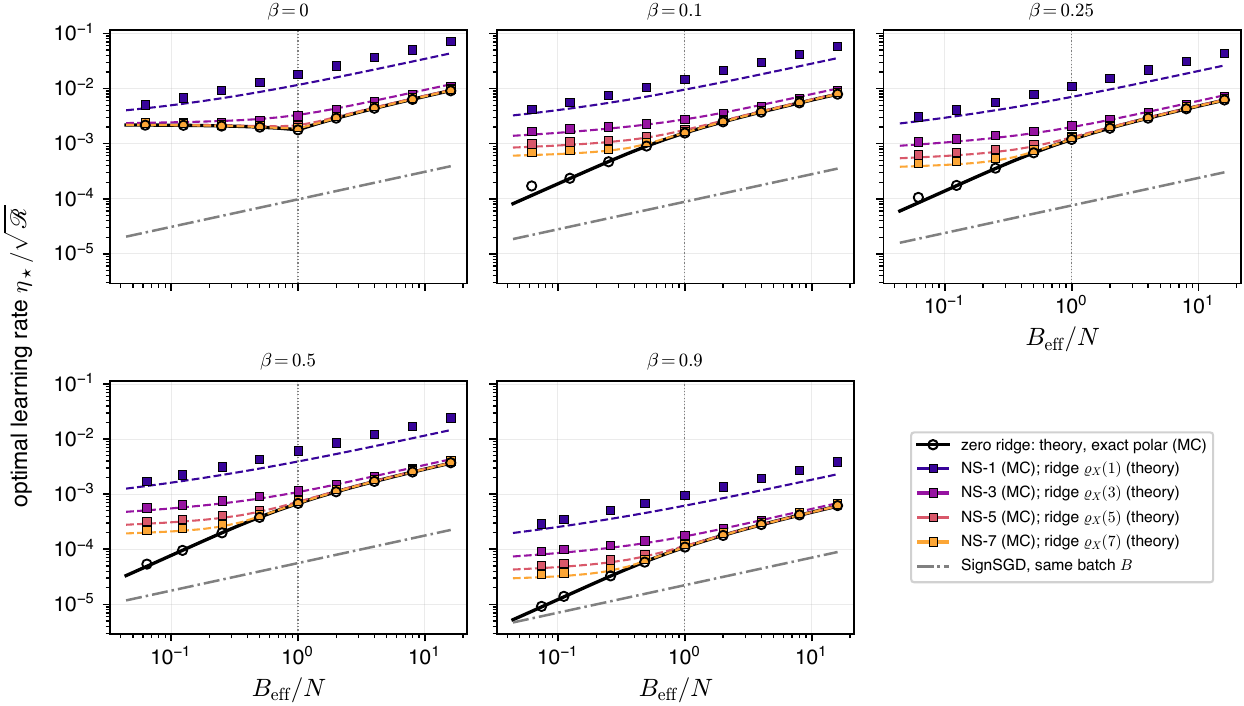}
\caption{Optimal learning rate $\eta_\star/\sqrt{\fRisk}=2\mathfrak d/\mathfrak v$ against $B_{\rm eff}/N$, $N=512$.
Black: zero ridge (theory) and exact polar factor (open circles, frozen bank). Dashed: ridge theory at the calibrated
$\varrho_X(K)=0.31\,a^{-2K}$; squares: NS-$K$ (frozen bank). Grey: \SignSGD at the same batch $B$.}
\label{fig:app-nsridge-lr}
\end{figure}

\paragraph{Finite Newton--Schulz flattens the small-batch learning rate.}
With momentum and exact whitening, the learning rate grows like $B$ for $B_{\rm eff}<N$ (slope $0.95$ in
$\log B_{\rm eff}$ between $B_{\rm eff}/N=1/16$ and $1$), because the volatility of the whitened weak directions
of the buffer grows like $N^2/B$. A ridge caps the gain on those directions: the slope falls to $0.25$--$0.34$
for NS-5 at every $\bmom\in[0.1,0.9]$, and at $B_{\rm eff}/N=1/16$ the NS-5 learning rate is $6$--$8$ times that
of exact whitening. The picture is then close to that at $\bmom=0$, where the rank-$B$ gradient has no weak
directions, NS-5 and exact whitening nearly coincide, and the learning rate is nearly flat in $B$ for $B<N$.
For $B_{\rm eff}\ge N$ the ridge matters little, and every curve rises with slope between $\tfrac12$ and $0.6$,
as in Sec.~\ref{sec:iso-lr}. NS-$K$ stays above \SignSGD at the same batch in every configuration.

\paragraph{Amplitude.}
The calibration matches the scale-free statistic $I_1=\mathfrak d^2/\mathfrak v$, while the learning rate
$2\mathfrak d/\mathfrak v$ also carries the overall amplitude of the update. For $K\ge3$ the NS-$K$ learning rate is
between $0.92$ and $1.16$ times the ridge theory at $\varrho_X(K)$ over all $54$ frozen configurations at $N=512$.
NS-1 is nearly linear, $S_t\approx a\,M_t/\norm{M_t}_\F$, a smaller update than the ridge map at $\varrho_X(1)$, and
its learning rate is $1.15$--$1.64$ times higher. The exact-polar points at the smallest $B_{\rm eff}/N$ for
$\bmom\in\{0.1,0.25\}$ lie above the zero-ridge curve, for the rank-threshold reason of
Sec.~\ref{sec:app-nsridge-results}.

\paragraph{Scope.}
These measurements are in the isotropic model with the standard Muon coefficients.
The calibration concerns the rate at the optimal step and the lag share; it does not claim that
NS-$K$ and the ridge have the same drift or volatility separately, since they differ by an overall
amplitude.

\clearpage
%

\section{Momentum in \Muon: theory against an ablation}
\label{sec:momentum-ablation}

This appendix compares the momentum theory of the appendix with simulations of \Muon at
$\bmom\in\{0,0.9,0.95,0.99\}$ on the half-anisotropic power-law problem at $B=2N$. The
theory makes two predictions in this oversampled regime. The noise floor changes with $\bmom$
through an explicit noise constant $\mathcal S(\bmom)$, which we compute and which matches
the measurements to within $1$--$6\%$ with momentum (Sec.~\ref{ssec:mom-floor}). The
$\epsilon$-exponent of the time to reach a target does not change with $\bmom$; the time
constants with momentum are outside the theory at these settings, because the buffer is
dominated by the signal early in the descent (Sec.~\ref{ssec:mom-tte}).

\paragraph{Algorithms.}
Both updates are run on the half-anisotropic problem of Sec.~\ref{sec:half-aniso}, with
$\Sigmain=\Id$ and $\Sigmaout=U\diag(\mu)U^\T$ for a Haar-random orthogonal $U$ and
$\mu_i=i^{-\alpha}$. Without momentum the update is $\Delta_{t+1}=\Delta_t-\eta\,\mathrm{orthog}(G_t)$,
where $\mathrm{orthog}$ is five Newton--Schulz steps with the coefficients of
\cite{jordan2024muon}. With momentum a heavy-ball buffer $H_{t+1}=\bmom H_t+G_t$ is
orthogonalized instead, as in \texttt{optax.contrib.scale\_by\_muon} with
\texttt{nesterov=False} \cite{deepmind2020jax}. The heavy-ball buffer is the normalized
average $M_t$ of \cref{def:app-setup-algorithm} divided by $1-\bmom$
(\cref{rem:app-setup-heavyball}); the polar factor, and Newton--Schulz applied to a
Frobenius-normalized input, do not see this factor, so the simulated algorithm is the
algorithm of the appendix at zero ridge, up to the Newton--Schulz approximation.

\paragraph{Setup.}
We use the matched-floor protocol of Sec.~\ref{sec:timetoepsilon}: for each target $\epsilon$
the step is $\eta^\star=2\sqrt\epsilon/\mathcal S$, which puts the risk floor at $\epsilon$. In the
simulations $\mathcal S=\mathcal S_{\rm eff}$ is calibrated per $(\bmom,\texp)$ from a short probe
at $\epsilon=0.25$, as $2\sqrt{\cR_\infty}/\eta$ with $\cR_\infty$ the measured plateau. The
problems are $\alpha=1.5$ with $\texp\in\{0.7,1.5,3.0\}$, one in each $\epsilon$-exponent regime of
Thm.~\ref{thm:tte-half-aniso}; $N=1024$, $B=2048$, $16$ trials, and
$\epsilon\in\{2^{-4},2^{-6},2^{-8},2^{-10}\}$. \SignSGD at $\bmom=0$ is run for reference.
Fig.~\ref{fig:muon-vs-signsvd-tte} first compares \Muon at $\bmom=0$ with the deterministic
model of Sec.~\ref{sec:timetoepsilon} over a wider range of $\epsilon$.

\begin{figure}[t]
  \centering
  \begin{subfigure}[t]{0.49\textwidth}
    \includegraphics[width=\linewidth]{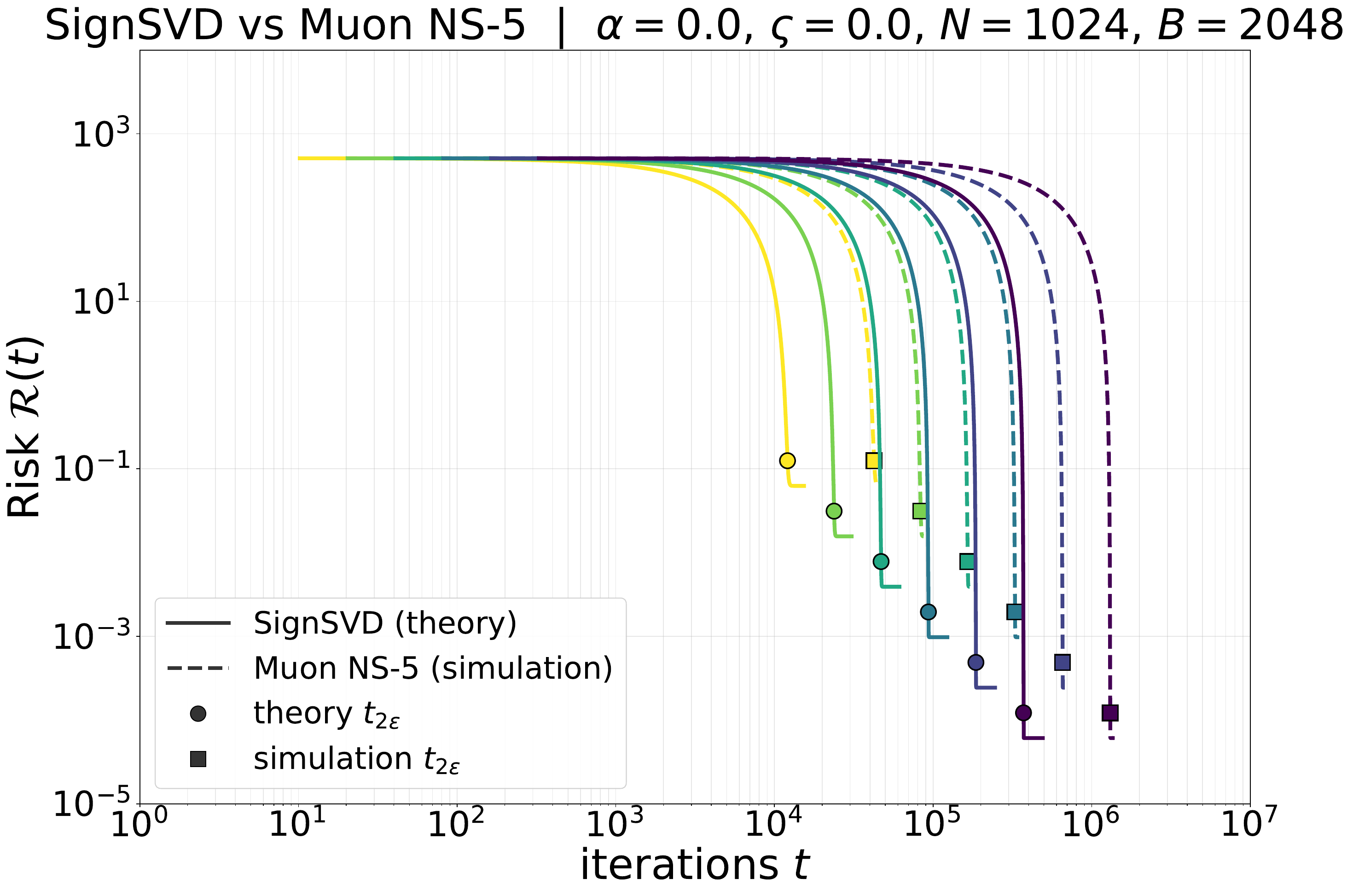}
    \caption{Isotropic ($\alpha=0$, $\texp=0$).}
  \end{subfigure}\hfill
  \begin{subfigure}[t]{0.49\textwidth}
    \includegraphics[width=\linewidth]{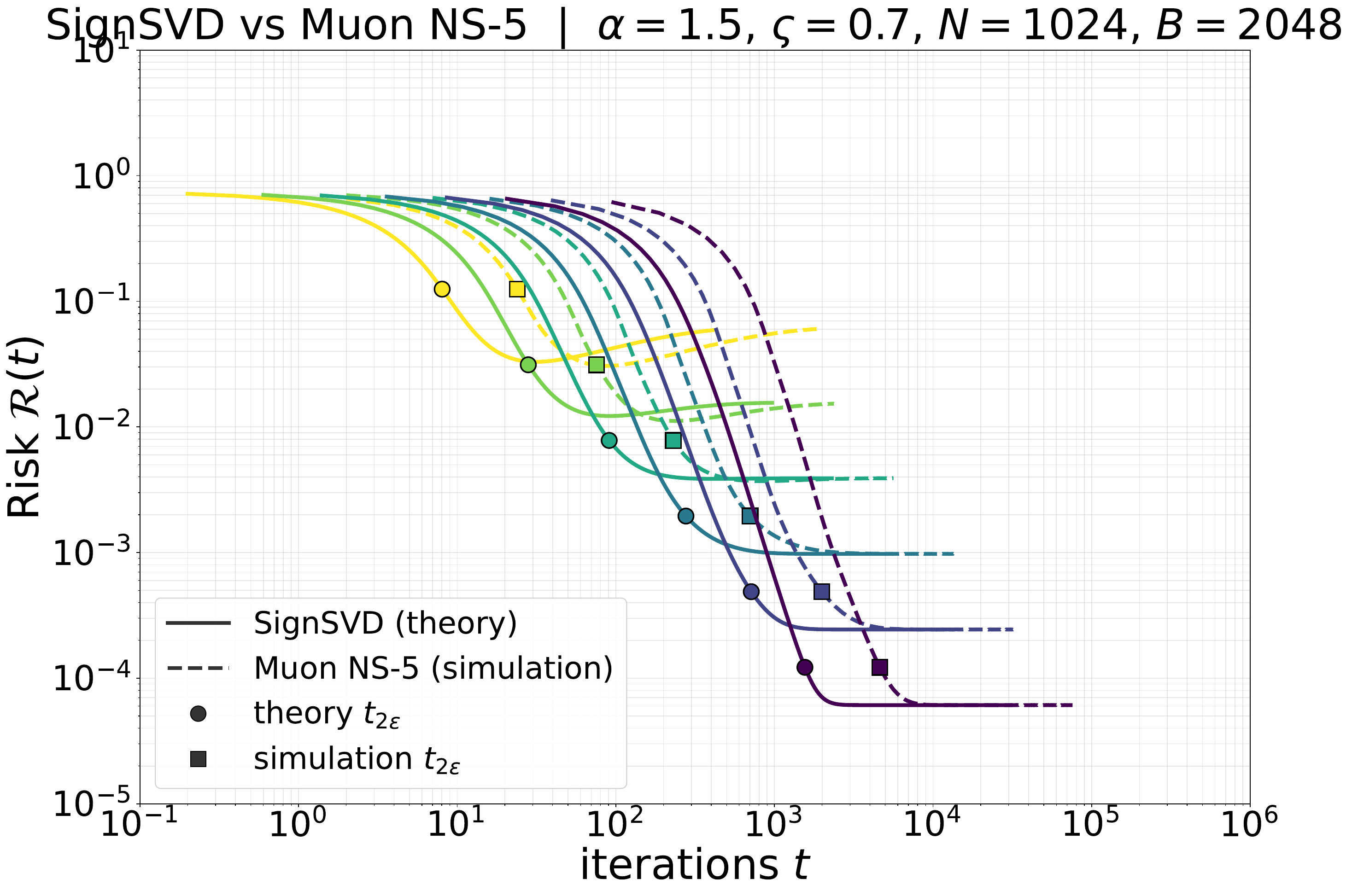}
    \caption{Half-anisotropic, $\alpha=1.5$, $\texp=0.7$.}
  \end{subfigure}

  \vspace{0.4em}

  \begin{subfigure}[t]{0.49\textwidth}
    \includegraphics[width=\linewidth]{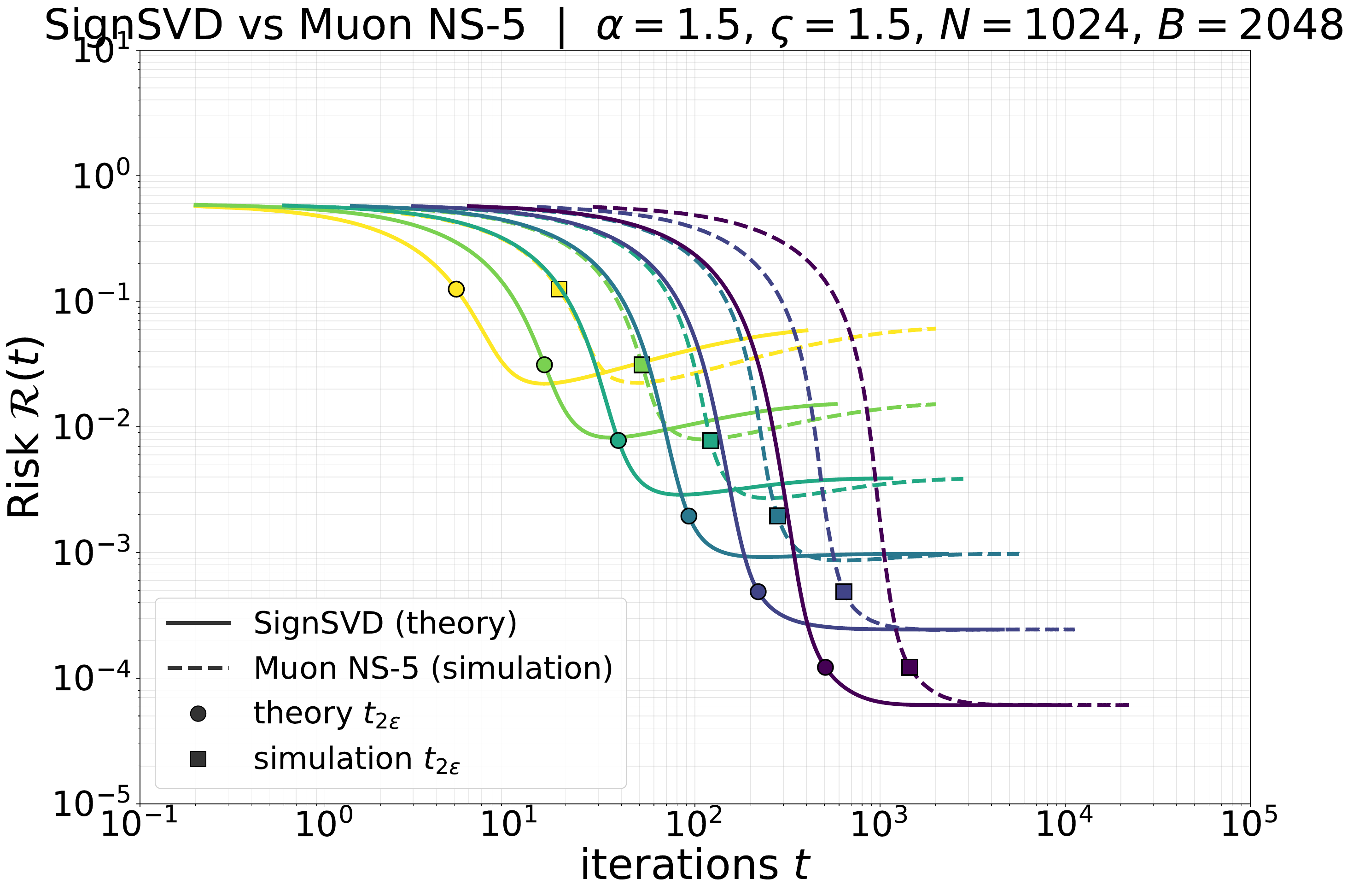}
    \caption{Half-anisotropic, $\alpha=1.5$, $\texp=1.5$.}
  \end{subfigure}\hfill
  \begin{subfigure}[t]{0.49\textwidth}
    \includegraphics[width=\linewidth]{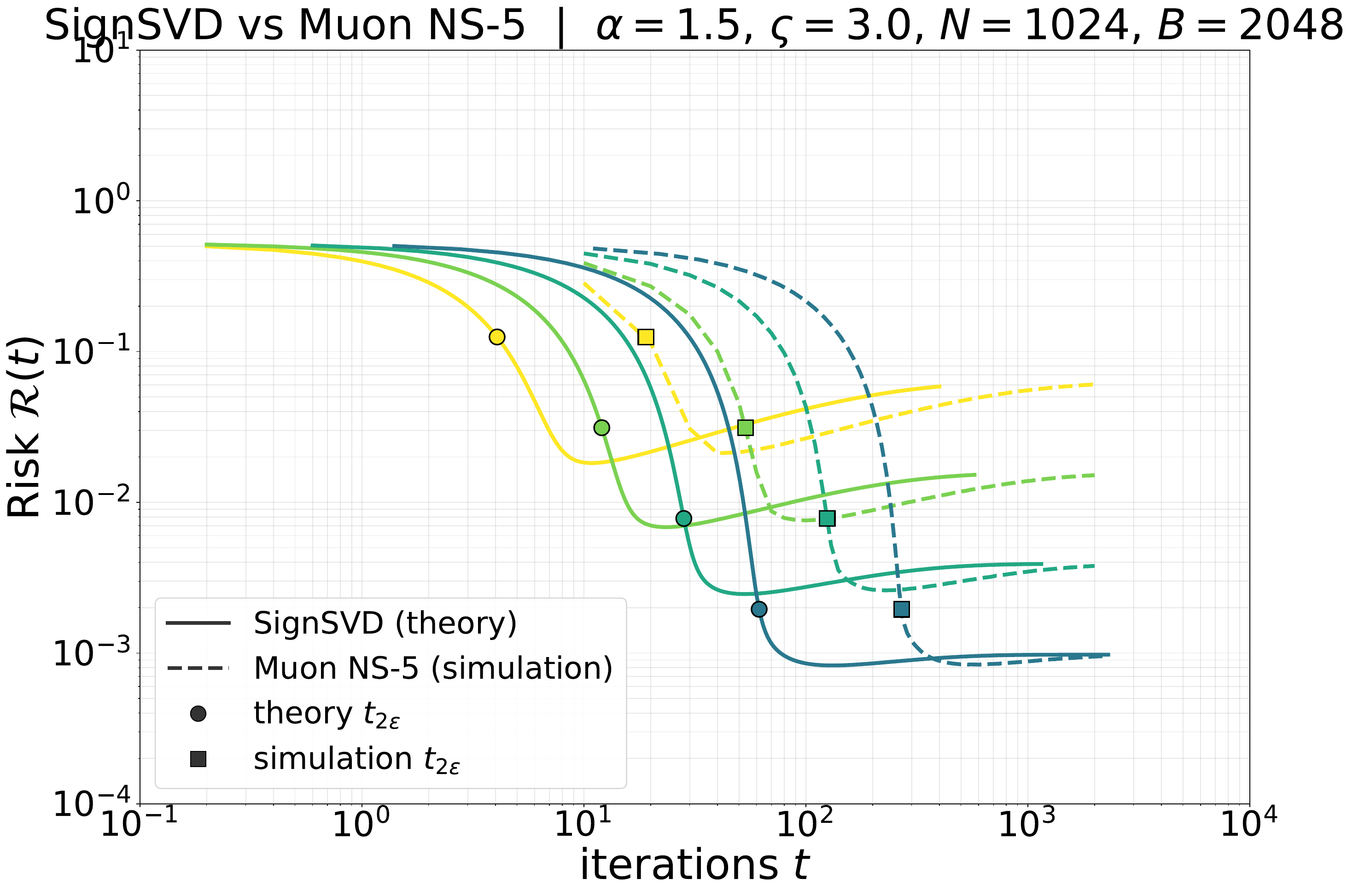}
    \caption{Half-anisotropic, $\alpha=1.5$, $\texp=3.0$.}
  \end{subfigure}
  \caption{\Muon without momentum (Newton--Schulz, dashed) against the deterministic
  time-to-$\epsilon$ model of Sec.~\ref{sec:timetoepsilon} (solid), $N=1024$, $B=2048$, $16$ trials per
  target, colour-coded by $\epsilon\in\{2^{-4},\dots,2^{-14}\}$. Circles: model $t_{2\epsilon}$; squares:
  simulated $t_{2\epsilon}$. The simulations follow the model's $\epsilon$-exponent in every regime.}
  \label{fig:muon-vs-signsvd-tte}
\end{figure}

\subsection{What the theory predicts at $B=2N$}
\label{ssec:mom-theory}

\paragraph{Drift.}
At $B=2N$ the small parameter of \cref{thm:app-os-limit} is
$\epsilon_{\rm os}=3(1-\bmom^2)N/B=1.5(1-\bmom^2)$, which equals $0.29$, $0.15$ and $0.03$ at
$\bmom=0.9$, $0.95$, $0.99$. In this window the theorem gives
\[
\mathfrak d_i=\vartheta_i\sqrt{\frac{B_{\rm eff}}N}\,\mathfrak d^{\rm os}(\mu_i),\qquad
\vartheta_i\in\big[1-\epsilon_{\rm os},(1-\epsilon_{\rm os})^{-1}\big],
\]
with a profile $\mathfrak d^{\rm os}$ that does not depend on $\bmom$: momentum enters the drift
only through the effective batch $B_{\rm eff}=B(1+\bmom)/(1-\bmom)$. We evaluate the zero-ridge
drift kernel of each row exactly (\cref{eqn:app-drift-kernel}, single integral over the
branch). Relative to $\sqrt{B_{\rm eff}/B}$ times the drift at $\bmom=0$, the ratio lies in
$[1.00,1.19]$, $[1.00,1.14]$ and $[0.95,1.02]$ over all rows at $\bmom=0.9$, $0.95$, $0.99$.

\paragraph{Volatility.}
The zero-ridge row volatility is explicit at a finite spectrum: it is $1$ plus twice the sum
over lags of double integrals of the two-time row kernel
(\cref{prop:app-pl-vol-zero,prop:app-pl-response}), and without momentum it equals $1$, since
the polar factor of a full-rank square gradient is orthogonal. With momentum the head rows
have $\mathfrak v_1=19.0$, $38.9$, $198.2$, essentially the factor $B_{\rm eff}/B=19$, $39$, $199$ of a
linear buffer, while the tail rows have less, $\mathfrak v_{1000}=9.2$, $16.7$, $56.7$. Our code
reproduces the isotropic constants of Fig.~\ref{fig:isotropic-main} to $0.2\%$ when all
eigenvalues are equal.

\paragraph{Noise constant.}
The noise constant of the matched-floor protocol is (\cref{eqn:eta-star-tte})
\[
\mathcal S(\bmom)=\sum_{i=1}^N\frac{\mu_i\,\mathfrak v_i}{2\,\mathfrak d_i},
\]
and it does not depend on the target exponent $\texp$, because neither coefficient does. The
theory values are in Table~\ref{tab:mom-noise}.

\subsection{The noise floor}
\label{ssec:mom-floor}

\begin{table}[t]
  \centering
  \caption{Noise constant at $\alpha=1.5$, $N=1024$, $B=2048$. Theory: exact polar factor, zero
  ridge, coefficients of Sec.~\ref{ssec:mom-theory}. Measured: \Muon with five Newton--Schulz
  steps, calibrated $\mathcal S_{\rm eff}$; the three target exponents agree to four significant
  figures at every $\bmom$.}
  \label{tab:mom-noise}
  \begin{tabular}{l|cccc}
    & $\bmom=0$ & $\bmom=0.9$ & $\bmom=0.95$ & $\bmom=0.99$ \\\hline
    theory $\mathcal S(\bmom)$ & $11.99$ & $39.10$ & $54.79$ & $116.56$ \\
    measured $\mathcal S_{\rm eff}$, $\texp=0.7$ & $9.97$ & $36.89$ & $52.48$ & $117.39$ \\
    measured $\mathcal S_{\rm eff}$, $\texp=1.5$ & $9.97$ & $36.91$ & $52.49$ & $117.38$ \\
    measured $\mathcal S_{\rm eff}$, $\texp=3.0$ & $9.97$ & $36.89$ & $52.49$ & $117.22$ \\\hline
    measured / theory & $0.83$ & $0.94$ & $0.96$ & $1.01$
  \end{tabular}
\end{table}

\begin{figure}[t]
  \centering
  \includegraphics[width=0.9\textwidth]{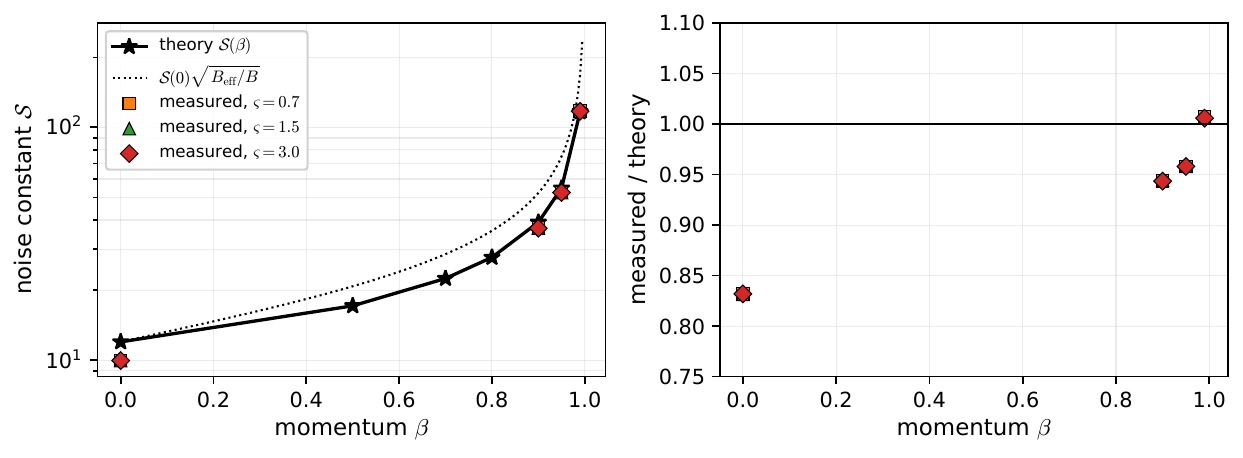}
  \caption{Noise constant against momentum. Left: theory $\mathcal S(\bmom)$ (stars, line), the plain
  rescaling $\mathcal S(0)\sqrt{B_{\rm eff}/B}$ (dotted), and the measured $\mathcal S_{\rm eff}$ for the three
  target exponents, which coincide. Right: measured over theory.}
  \label{fig:mom-noise}
\end{figure}

The theory predicts the noise floor with momentum (Table~\ref{tab:mom-noise},
Fig.~\ref{fig:mom-noise}). For $\bmom\in\{0.9,0.95,0.99\}$ the measured noise constants are within
$6\%$, $4\%$ and $1\%$ of $\mathcal S(\bmom)$, and, as predicted, they do not depend on $\texp$. The
floor at a given step is $(\eta\,\mathcal S(\bmom)/2)^2$, so the noise-floor adjustment for
momentum is the ratio $\mathcal S(\bmom)/\mathcal S(0)$, equal to $3.26$, $4.57$ and $9.72$. It is smaller
than the plain rescaling $\sqrt{B_{\rm eff}/B}=4.36$, $6.24$, $14.1$ by the factors $0.75$, $0.73$
and $0.69$: the drift grows by $\sqrt{B_{\rm eff}/B}$ on every row, but the volatility of the tail rows
grows by less than $B_{\rm eff}/B$, and the tail rows dominate $\mathcal S$.

Without momentum the measured constant is $17\%$ below the theory. Here the gradient is a
square matrix with aspect ratio $N/B=1/2$ and spread-out singular values, and five Newton--Schulz
steps differ from the exact polar factor by the spread of their outputs, the effect identified in
Appendix~\ref{sec:app-nsridge} at $\bmom=0$. With momentum the buffer averages many batches,
its singular values are concentrated, and the difference is small. The empirical prefactor
$\hat\rho\approx0.84$ of an earlier version of this appendix, defined relative to the measured
$\bmom=0$ value, combines this $17\%$ with the factors $0.69$--$0.75$ above.

\subsection{Time to $\epsilon$}
\label{ssec:mom-tte}

\begin{figure}[t]
  \centering
  \includegraphics[width=0.95\textwidth]{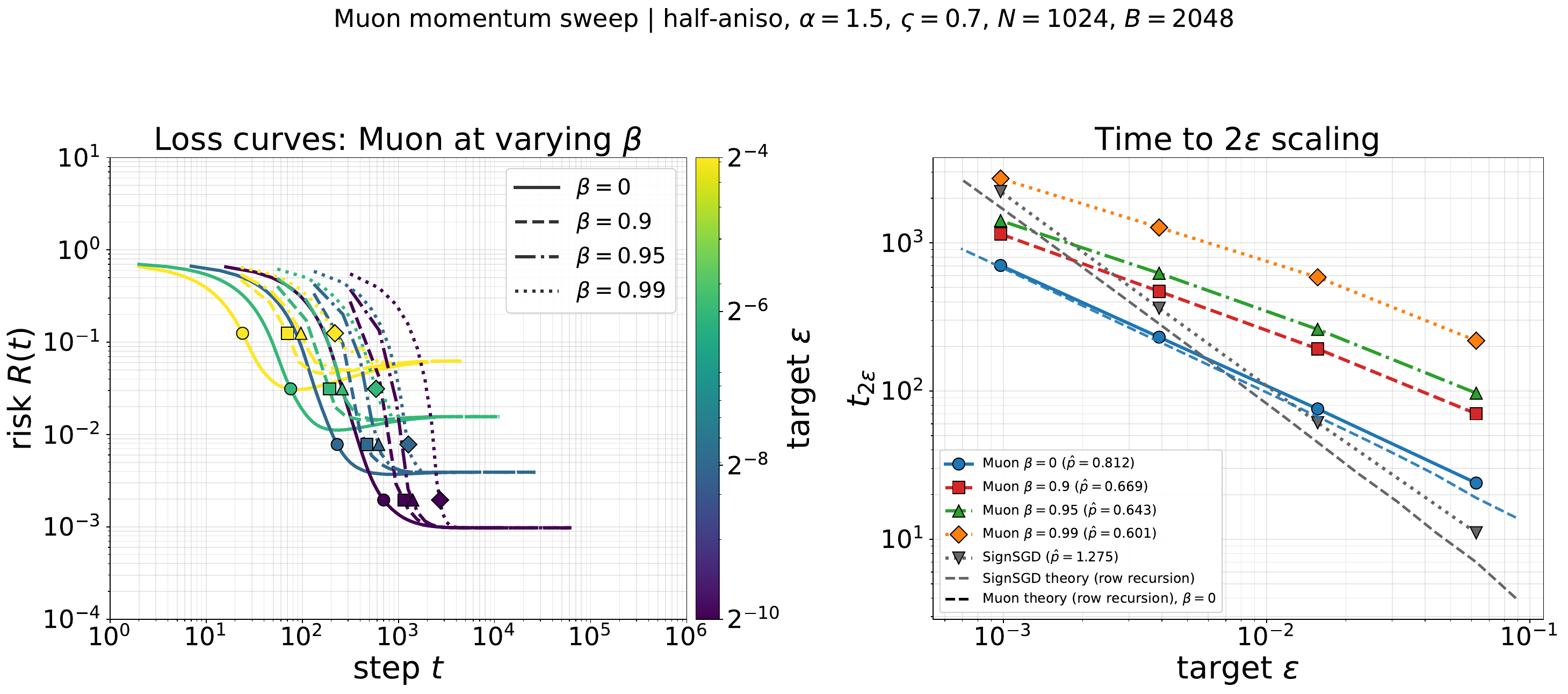}
  \caption{$\alpha=1.5$, $\texp=0.7$, $N=1024$, $B=2048$, $16$ trials. Left: risk at the matched step for each
  target $\epsilon$ (colour) and $\bmom$ (linestyle); markers at the measured $t_{2\epsilon}$. Right: $t_{2\epsilon}$
  against $\epsilon$. Blue dashed: the row recursion with the theory coefficients of
  Sec.~\ref{ssec:mom-theory} at $\bmom=0$, at its own matched step (no line is drawn for $\bmom>0$; see
  Sec.~\ref{ssec:mom-tte}). Grey dashed: the same recursion with the \SignSGD
  kernels of Table~\ref{tab:power-law-kernels}. The asymptotic exponents are $0.625$ (\SignSVD) and $1.25$ (\SignSGD).}
  \label{fig:muon-beta-sweep-A}
\end{figure}

\begin{figure}[t]
  \centering
  \includegraphics[width=0.95\textwidth]{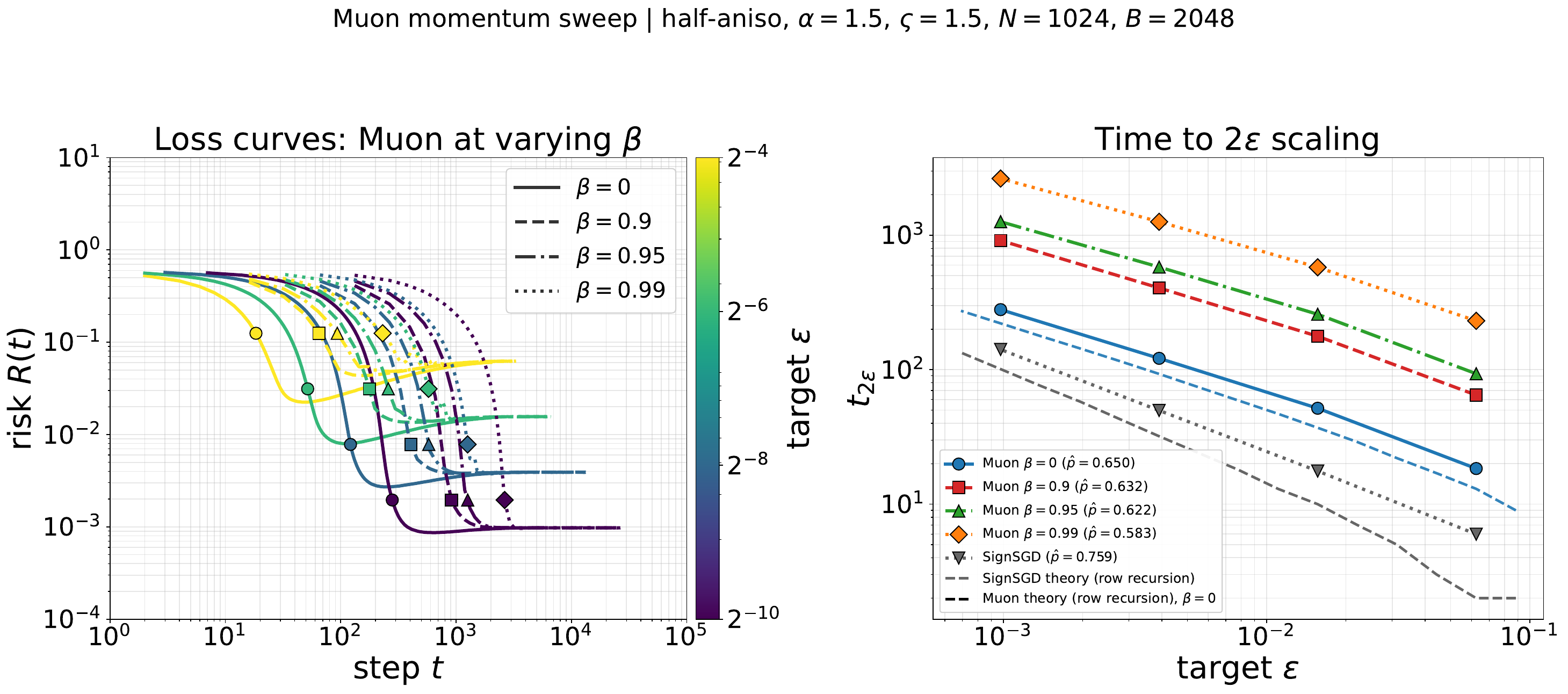}
  \caption{$\alpha=1.5$, $\texp=1.5$; conventions as in Fig.~\ref{fig:muon-beta-sweep-A}. Asymptotic exponents $0.5$
  (\SignSVD) and $0.75$ (\SignSGD).}
  \label{fig:muon-beta-sweep-B}
\end{figure}

\begin{figure}[t]
  \centering
  \includegraphics[width=0.95\textwidth]{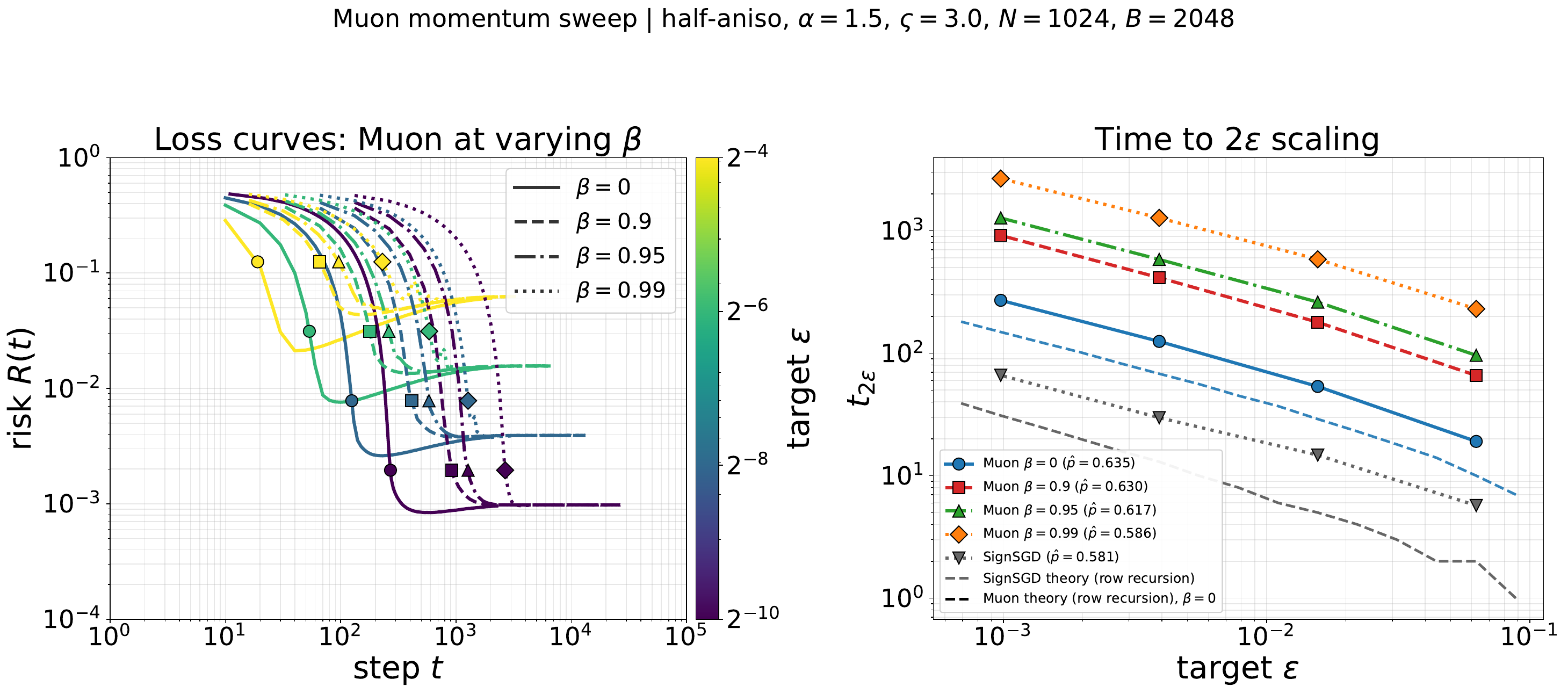}
  \caption{$\alpha=1.5$, $\texp=3.0$; conventions as in Fig.~\ref{fig:muon-beta-sweep-A}. Both algorithms are
  saturated, with asymptotic exponent $0.5$.}
  \label{fig:muon-beta-sweep-C}
\end{figure}

\paragraph{The exponent.}
The $\epsilon$-exponent of the time to reach a target is set by the power-law orders in $i$ of
$\mathfrak d_i$ and $\mathfrak v_i$ (Thm.~\ref{thm:tte-half-aniso}). Momentum multiplies the drift of every
row by $\sqrt{B_{\rm eff}/B}$ up to a bounded factor, and the volatility by factors bounded above and
below uniformly in $i$ (between $9.1$ and $19$ at $\bmom=0.9$, between $55$ and $198$ at $\bmom=0.99$), so it does not change these orders, and the exponent is the
same for every $\bmom$. The measured slopes (Figs.~\ref{fig:muon-beta-sweep-A}--\ref{fig:muon-beta-sweep-C},
right panels) agree: they lie within $10$--$30\%$ of the \SignSVD exponent for every $\bmom$, do not
depend systematically on $\bmom$, and show the same finite-$N$ deviation already present at $\bmom=0$.

\paragraph{The time constants.}
Without momentum the theory follows the simulations (blue dashed lines in the right panels: the row
recursion with the coefficients of Sec.~\ref{ssec:mom-theory} at its own matched step): within
$0.8$--$1.06$ at $\texp=0.7$ and $0.7$--$0.84$ at $\texp=1.5$, and about $0.55$ at $\texp=3.0$, where the gap
includes the Newton--Schulz offset of Sec.~\ref{ssec:mom-floor}. With momentum the theory predicts
a curve close to the $\bmom=0$ one, while the simulations are slower by large factors ($2$--$8$ at
$\bmom=0.9$, $5$--$30$ at $\bmom=0.99$), so the panels show no theory line for $\bmom>0$.
Fig.~\ref{fig:mom-lag} shows the gap on the risk itself.

\begin{figure}[t]
  \centering
  \includegraphics[width=0.55\textwidth]{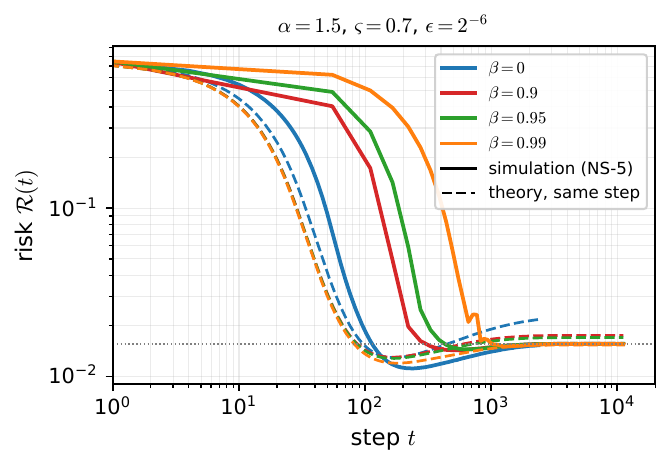}
  \caption{Risk at $\alpha=1.5$, $\texp=0.7$, $\epsilon=2^{-6}$: simulations (solid) against the row recursion with the
  theory coefficients at the simulations' own step (dashed). With momentum the simulated descent is
  slower than the theory, increasingly with $\bmom$; the floors agree to the accuracy of
  Table~\ref{tab:mom-noise}.}
  \label{fig:mom-lag}
\end{figure}

\paragraph{Why the momentum transient is outside the theory.}
The gain of the drift by $\sqrt{B_{\rm eff}/B}$ (Sec.~\ref{ssec:mom-theory}) is a linear response: the
polar factor of a buffer that is mostly noise responds linearly to the signal in it, and averaging
more batches raises the signal-to-noise ratio. It requires the noise to dominate the buffer
(\cref{rem:app-os-linear}, which asks for $B_{\rm eff}\ll N^2$ for a spread-out error). Here
$B_{\rm eff}/N^2=0.04$ at $\bmom=0.9$ and $0.39$ at $\bmom=0.99$, and the initial error is concentrated: at
$\texp=0.7$ the top mode carries about two thirds of the weighted risk. Early in the descent the buffer
is therefore close to its signal $\Sigmaout\Dt$, whose polar factor has bounded entries, and the extra
signal-to-noise ratio buys no extra speed. Near the floor the error is small and spread out, the noise
dominates again, and the linear response holds; this is why the noise floor and the exponent are
predicted.

\begin{table}[t]
  \centering
  \caption{Steps to $2\epsilon$ at $\alpha=1.5$, $\texp=0.7$, $\epsilon=2^{-6}$, with the step and initial error of the ablation.
  Gaussian gradient model $G_t=\Sigmaout\Dt+\sqrt{2\cR(t)/B}\,\Sigmaout^{1/2}W_t$ in the eigenbasis; the
  fresh buffer has the noise level of the momentum buffer but is formed at the current state. One
  run per entry.}
  \label{tab:mom-arms}
  \begin{tabular}{l|c|cc|cc|c}
    & theory & \multicolumn{2}{c|}{fresh buffer} & \multicolumn{2}{c|}{momentum buffer} & \Muon run \\
    & & polar & NS-5 & polar & NS-5 & \\\hline
    $\bmom=0$    & $58$ & $68$  & $71$  & --- & --- & $76$ \\
    $\bmom=0.9$  & $51$ & $120$ & $160$ & $141$  & $179$ & $220$ \\
    $\bmom=0.99$ & $50$ & $357$ & $482$ & $1039$ & $787$ & $605$
  \end{tabular}
\end{table}

A controlled comparison separates this effect from two others (Table~\ref{tab:mom-arms}). A
Gaussian model of the gradient reproduces the simulations to within about $30\%$. In it, a fresh
buffer at the current state with the exact polar factor, which has neither a stale buffer nor
Newton--Schulz, is already $2.4$ and $7$ times slower than the theory at $\bmom=0.9$ and $0.99$; this is
the loss of linear response. Replacing the exact polar factor by five Newton--Schulz steps slows the
descent by a further $1.3$--$1.35$: the effective ridge $\varrho_X\norm{M_t}_\F^2$ of
Appendix~\ref{sec:app-nsridge} grows with the signal in $\norm{M_t}_\F$ and suppresses the noise
directions through which the tail rows descend. Using the actual momentum buffer, whose signal lags
the state, slows it by $1.1$--$1.2$ at $\bmom=0.9$ and by $1.6$--$2.9$ at $\bmom=0.99$; with this stale buffer
Newton--Schulz is faster than the exact polar factor, because its ridge down-weights the old
directions, as in Appendix~\ref{sec:app-nsridge}.

\paragraph{Takeaway.}
At $B=2N$ the theory predicts how momentum changes the noise floor of \Muon, through the explicit
noise constant $\mathcal S(\bmom)$, to within a few percent, and it predicts that the $\epsilon$-exponent of the
time to target does not change, as observed. The transient with momentum is outside the theory at
these settings: the buffer is signal-dominated early in the descent, so the linear-response gain of
the drift is not realized, and the stale buffer and the Newton--Schulz ridge add smaller slowdowns. At
matched floor, momentum therefore costs time in this ablation.

\clearpage

\clearpage
\section{Roadmap of the results and their proofs}\label{sk:ch-roadmap}
The preceding sections give additional related work, numerical methods, finite-ridge experiments
and momentum ablations. The sections that follow present the results and their proofs. This
proof appendix is written for a reader who has finished the main text. It restates, verbatim,
every appendix result that the main text cites, and follows each with a paragraph on what the
result says and which claim of the main text it supports, and, for theorems, propositions,
lemmas and corollaries, a sketch of the proof. The line printed after each restated result gives
its number in the supplement, which contains the complete appendix with full proofs, in sections
SA to SJ; references such as Lemma SE.4 point into the supplement. The statements of the
supplement are formalized in Lean, except the remarks and the numerical statements. The standard
Gaussian results they use are proved in the accompanying foundations library, listed in
\cref{sec:supp-lean} of the supplement; the line after
each restated result names the corresponding Lean declarations and their status.

The proof appendix has three parts. \Cref{sk:ch-model} fixes the model, the algorithm and the
proportional regime, which are standing assumptions for everything that follows.
\Cref{sk:ch-engines} gives an informal overview of six proof strategies and how they fit together:
the deterministic equivalent for the Gram resolvent of the momentum buffer
(\cref{sk:E1}), the expected regularized update (\cref{sk:E2}), the volatility and the lag kernel
(\cref{sk:E3}), the passage from the algorithm to the deterministic dynamics (\cref{sk:E4}), the
power-law and large-batch asymptotics (\cref{sk:E5}), and the Volterra analysis behind the phase
map (\cref{sk:E6}). The seven remaining chapters state the precise results and their formalization
status, each opening with the setting its results share.

The first four strategies form a chain, and together they carry the central claim of the main text,
\cref{thm:deterministic-dynamics}. \Cref{sk:E1} computes the resolvent of the Gram matrix of the
momentum buffer at a fixed state through the fixed-point system of \cref{fig:fp-system}.
\Cref{sk:E2} integrates that resolvent into the mean of the regularized update, which gives the drift
kernel \cref{eqn:dij}, and \cref{sk:E3} computes the covariance of consecutive updates, which gives the
volatility kernel. \Cref{sk:E4} compares the algorithm with the deterministic recursion
\cref{eqn:fQ-recursion} over blocks of steps, using the two kernels and the fact that a diffuse error
stays diffuse. \Cref{sk:ch-dynamics} restates this comparison and its half-anisotropic form, which is
the row recursion \cref{eqn:ri-recursion}; it rests on all four strategies. It comes first because the
later chapters evaluate its coefficients: \cref{sk:ch-drift} restates the fixed-point system and the
drift kernel (\cref{sk:E1,sk:E2}), and \cref{sk:ch-volatility} reads the volatility kernel for the
half-anisotropic rows and without momentum (\cref{sk:E1,sk:E2,sk:E3}).

The next three chapters turn the kernels into the quantitative claims of the main text.
\Cref{sk:ch-isotropic} treats identity covariances, where the fixed-point system becomes one scalar
equation and the dynamics closes on the risk; it supplies the drift and volatility of
\cref{sec:isotropic}, the learning-rate ratio and the speedup, at large and at small batch. It uses
the first four strategies and, for its limits in the aspect ratio, \cref{sk:E5}. \Cref{sk:ch-powerlaw}
computes the half-anisotropic coefficients at zero ridge, for power-law spectra at thin batch and for
every spectrum at oversampled batch, which gives \cref{tab:power-law-kernels} and the square-root law
of \cref{thm:large-batch}; it is \cref{sk:E5} applied to the formulas of \cref{sk:E1,sk:E2,sk:E3}.
\Cref{sk:ch-phase} takes the kernels as fixed coefficients of the row recursion and derives the times
to accuracy, their ratio and the phase diagram of \cref{sec:timetoepsilon}; it rests on \cref{sk:E6},
and its isotropic comparison also on \cref{sk:E1}.

The last chapter, \cref{sk:ch-signsgd}, concerns the comparator. Its results are one-step Gaussian
computations for the entrywise sign and use none of these six strategies. They give the \SignSGD kernels of
\cref{sec:isotropic,sec:half-aniso}, on which the \SignSGD arrays of \cref{sk:ch-phase} are modelled.

\section{Model, algorithm and standing assumptions}\label{sk:ch-model}
The model is the matrix linear regression of \cref{sec:model_set_up}: a
parameter $W$ of size $\Nout\times\Nin$, a teacher $\Wstar$, independent
Gaussian features with covariances $\Sigmaout=\sum_i\mu_i\,u_i\tensor u_i$
and $\Sigmain=\sum_j\lambda_j\,v_j\tensor v_j$, and the risk
$\cR(W)=\tfrac12\norm{\Sigmaout^{1/2}(W-\Wstar)\Sigmain^{1/2}}_\F^2=\tfrac12\sum_{i,j}\mu_i\lambda_j\,\CMscr Q_{ij}$.
The minibatch gradient \cref{eqn:minibatch-gradient} at the error
$\Dt=\Wt-\Wstar$ is written $\Gt^\circ$; the superscript $^\circ$ marks physical
objects, as opposed to the normalized ones divided by $\sqrt{2\cR(t)}$. For
$\ridge>0$ the regularized polar map is $\Qridge\ridge(K)=(KK^\T+\ridge\Id)^{-1/2}K$
(\cref{eqn:app-setup-Q}).
The three statements below are standing for the whole appendix: the first
fixes the law of the data, the second the algorithm whose risk the dynamics
tracks, and the third the limit in which every asymptotic statement is made.
The random-matrix results add one condition on the state, that the error be
spectrally diffuse, which the dynamical theorem shows persists, from a diffuse
initial state, over the time window it covers.

\begin{assumption}[Data-target normalization]\label{assum:data-target}
The input and output features are independent and Gaussian,
$\Xin\sim\mathcal N(0,\Sigmain)$ and $\Xout\sim\mathcal N(0,\Sigmaout)$, with
positive semidefinite covariances $\Sigmain\in\R^{\Nin\times\Nin}$ and
$\Sigmaout\in\R^{\Nout\times\Nout}$. They are represented in whitened form as
$\Xin=\Sigmain^{1/2}w$ and $\Xout=\Sigmaout^{1/2}z$ with
$w\sim\mathcal N(0,\Id_{\Nin})$ and $z\sim\mathcal N(0,\Id_{\Nout})$
independent; the whitened columns of a batch are written $(w_a,z_a)_{a\le B}$.
\end{assumption}

\skpointer{assum:data-target}{Lean: proved: \texttt{\detokenize{Muon}\detokenize{.}\allowbreak{}\detokenize{Setup}\detokenize{.}\allowbreak{}\detokenize{assum}\detokenize{_}\allowbreak{}\detokenize{data}\detokenize{_}\allowbreak{}\detokenize{target}}, \texttt{\detokenize{Muon}\detokenize{.}\allowbreak{}\detokenize{Setup}\detokenize{.}\allowbreak{}\detokenize{dataTarget}}, +2 more}
This is the data model of \cref{sec:model_set_up}. In the whitened form the
residual of a sample is $d=z^\T\At w$ with the rotated error
$\At=\Sigmaout^{1/2}\Dt\Sigmain^{1/2}$, and $\cR(t)=\tfrac12\norm{\At}_\F^2$
(\cref{eqn:app-setup-whitened}). Every random-matrix statement uses the
normalized residual $d/\sqrt{2\cR(t)}=z^\T\Ahat w$, which has unit variance and
is close to a standard Gaussian when the normalized error $\Ahat$ has small
operator norm (\cref{lem:app-setup-residual}). Gaussianity is used for these
computations (Wick pairings and Stein's identity); the exact one-step risk
recursion needs only second moments.

\begin{definition}[The algorithm]\label{def:app-setup-algorithm}
Fix a momentum $\bmom\in[0,1)$, a normalized ridge $\ridge>0$, a
deterministic step sequence $(\lr_t)_{t\ge0}$ with $\lr_t\ge0$, an initial
parameter $W_0$ and an initial buffer $M_{-1}^\circ\in\R^{\Nout\times\Nin}$,
both given before the first batch is drawn (the default is $M_{-1}^\circ=0$).
Let $t_{\mathrm{stop}}\defeq\inf\{t\ge0:\cR(t)=0\}\in\{0,1,\dots\}\cup\{\infty\}$.
The algorithm is, for $t\ge0$,
\begin{equation}\label{eqn:app-setup-algorithm}
\begin{gathered}
M_t^\circ=\bmom\,M_{t-1}^\circ+(1-\bmom)\,\Gt^\circ,\qquad
W_{t+1}=\Wt-\lr_t\,\Gtilde^\circ,\\
\Gtilde^\circ\defeq\begin{cases}\Qridge{\ridgephys_t}(M_t^\circ),&t<t_{\mathrm{stop}},\\ 0,&t\ge t_{\mathrm{stop}},\end{cases}
\qquad \ridgephys_t\defeq2\cR(t)\,\ridge .
\end{gathered}
\end{equation}
\end{definition}

\skpointer{def:app-setup-algorithm}{Lean: proved: \texttt{\detokenize{Muon}\detokenize{.}\allowbreak{}\detokenize{Setup}\detokenize{.}\allowbreak{}\detokenize{def}\detokenize{_}\allowbreak{}\detokenize{app}\detokenize{_}\allowbreak{}\detokenize{setup}\detokenize{_}\allowbreak{}\detokenize{algorithm}}, \texttt{\detokenize{Muon}\detokenize{.}\allowbreak{}\detokenize{Setup}\detokenize{.}\allowbreak{}\detokenize{AlgParams}}, +17 more}
This is \Muon$_\varphi$ of \cref{sec:spectral_algs} with the regularized polar
map $\Qridge\ridge$, applied at the physical ridge
$2\cR(t)\ridge$. Since $\Qridge{2\cR\ridge}(\sqrt{2\cR}\,K)=\Qridge\ridge(K)$
(\cref{prop:app-setup-scaling}), the update is the regularized polar factor,
at the fixed normalized ridge $\ridge$, of the buffer divided by
$\sqrt{2\cR(t)}$; this is why the laws computed later do not depend on the
risk. The stopping time only matters with momentum, where the buffer can be
nonzero at a state of zero risk. A heavy-ball buffer at ridge $\ridge'$ is the
same algorithm at ridge $(1-\bmom)^2\ridge'$ (\cref{rem:app-setup-heavyball}),
and no statement concerns $\ridge\downarrow0$ or Newton--Schulz iterations.

\begin{assumption}[Proportional regime]\label{assum:proportional-regime}
The dimensions $\Nin$, $\Nout$, $B$ tend to infinity jointly, with aspect
ratios converging to constants:
\[
\frac{\Nin}{B} \to \gin \in (0,\infty), \qquad
\frac{\Nout}{B} \to \gout \in (0,\infty).
\]
The covariance operator norms remain uniformly bounded: there exists a
constant $C > 0$, independent of $\Nin$ and $\Nout$, such that
\[
\|\Sigmain\|_{\mathrm{op}} \;\leq\; C, \qquad
\|\Sigmaout\|_{\mathrm{op}} \;\leq\; C.
\]
We refer to the joint limit as the \emph{proportional regime}.
\end{assumption}

\skpointer{assum:proportional-regime}{Lean: proved: \texttt{\detokenize{Muon}\detokenize{.}\allowbreak{}\detokenize{Setup}\detokenize{.}\allowbreak{}\detokenize{assum}\detokenize{_}\allowbreak{}\detokenize{proportional}\detokenize{_}\allowbreak{}\detokenize{regime}}, \texttt{\detokenize{Muon}\detokenize{.}\allowbreak{}\detokenize{ProportionalRegime}}}
This is the proportional limit of \cref{thm:deterministic-dynamics}; the main
text's $\Nin=\Nout=N$ is the case $\gin=\gout$. No lower bound on the
covariances and no rank condition is assumed. The random-matrix results of
the supplement are proved under the stronger \cref{assum:app-setup-regime}, for Gaussian data in
the frozen-state model of \cref{sk:E1}:
explicit bounds $\max(\Nin,\Nout)\le\gamma_\star B$ and
$\norm{\Sigmain},\norm{\Sigmaout}\le C_\Sigma$ along the sequence, a normalized
ridge in a fixed compact subset of $(0,\infty)$, and a spectrally diffuse
normalized error, $\norm{\Ahat}\to0$ with $\norm{\Ahat}_\F=1$.

\section{Proof strategies}\label{sk:ch-engines}
This section gives an informal overview of the proof strategies and how they fit together.
The arguments are grouped by the methods they share: the resolvent calculation leads to the
drift and volatility kernels, which enter the comparison with deterministic dynamics;
asymptotic analysis of these kernels and a Volterra equation then give the predictions for
learning rates and times to accuracy. The following sections are organized by result. They
state the precise hypotheses and conclusions, identify the corresponding Lean formalizations
and their status, and give a proof sketch for each result. The strategies described here
explain the connections between those proofs; the supplement contains the complete proofs.

\subsection{Deterministic equivalent for the EMA Gram resolvent}\label{sk:E1}
Every kernel of the spectral family is a spectral integral against the resolvent of the
Gram matrix of the momentum buffer, and this engine computes that resolvent
(\cref{sec:app-de}). It works in the frozen-state model of
\cref{def:app-setup-frozen}, in which the state is held fixed and every
gradient in the buffer is recomputed at it with its own independent batch.
The normalized buffer is then
\[
\Mt=\sum_{j\ge0}\emaw j\,G_{t+1-j},\qquad \emaw j=(1-\bmom)\bmom^j,\qquad
G_{t+1-j}=\frac1B\sum_{a=1}^Bd_{ja}\,\Sigmaout^{1/2}(z_{ja}\tensor w_{ja})\Sigmain^{1/2},
\]
with independent standard Gaussian vectors $z_{ja},w_{ja}$ and residuals
$d_{ja}=z_{ja}^\T\Ahat w_{ja}$, where $\Ahat=\Sigmaout^{1/2}\Dt\Sigmain^{1/2}/\sqrt{2\cR(t)}$
is the normalized error, of Frobenius norm one. The engine assumes the
proportional regime and that the error is diffuse, $\norm{\Ahat}\to0$: no
singular direction of the error carries a fixed fraction of the risk. Its
output is that, uniformly for $s$ in compact subsets of $(0,\infty)$ and in
operator norm (\cref{thm:app-de-gram}),
\[
\EE\big(\Mt\Mt^\T+s\Id\big)^{-1}\longrightarrow\Rbar(-s)=\big(s\Id+\tilde s_1(-s)\,\Sigmaout\big)^{-1},
\]
with the same statement for $\Mt^\T\Mt$ and the limit $\Rtbar(-s)=\big(s\Id+s_1(-s)\Sigmain\big)^{-1}$, where
$\tilde s_1,s_1$ solve the weighted fixed-point system of
\cref{def:app-de-branch}; normalized traces against bounded test matrices
converge in mean square (\cref{cor:app-de-trace}).

The resolvent equation is first written exactly through one deleted sample.
The buffer is a sum of rank-one terms, one per sample,
$\Mt=\sum_{j,a}\mathsf h_{ja}\,\mathsf u_{ja}\tensor\mathsf v_{ja}$ with
$\mathsf h_{ja}=\emaw jd_{ja}$, $\mathsf u_{ja}=\Sigmaout^{1/2}z_{ja}/\sqrt B$ and
$\mathsf v_{ja}=\Sigmain^{1/2}w_{ja}/\sqrt B$, so the identity
$\Id-sR=R\Mt\Mt^\T$ for $R=R(-s)$ is a sum over samples of
$(\mathsf hR\Mt\mathsf v)\tensor\mathsf u$. Removing one sample changes $\Mt$
by a rank-one matrix, and the Woodbury formula for a $2\times2$ block
computes the change exactly (\cref{lem:app-de-deletion}). With $M_0$ the
buffer without the sample, $R_0$ its resolvent and $\mathsf X_0=R_0M_0$,
\[
\mathsf h\,R\Mt\mathsf v=-\alpha_3\,R_0\mathsf u+\alpha_1\,\mathsf X_0\mathsf v,\qquad
\alpha_1=\frac{\mathsf h(1+\mathsf h\mathsf C)}{\Lambda},\quad
\alpha_3=-\frac{s\mathsf h^2\mathsf V}{\Lambda},\quad
\Lambda=(1+\mathsf h\mathsf C)^2+s\mathsf h^2\mathsf U\mathsf V,
\]
where $\mathsf U=\mathsf u^\T R_0\mathsf u$, $\mathsf V=\mathsf v^\T\widetilde R_0\mathsf v$
and $\mathsf C=\mathsf u^\T\mathsf X_0\mathsf v$. The deleted resolvents do not
depend on the deleted sample, so after taking expectations each term is a
Gaussian computation in one sample.

The coefficients then become functions of the residual of the deleted sample
alone. The scalars $\mathsf U$ and $\mathsf V$ are quadratic forms of a
Gaussian vector independent of $R_0$, so they concentrate, with errors
$O(B^{-1/2})$, at the expected normalized traces
$\sigma_B=\frac1B\Tr(\Sigmaout\EE R)$ and $\tilde\sigma_B$, and $\mathsf C$
concentrates at $0$ (\cref{lem:app-de-scalars}); the traces themselves
fluctuate by $O(B^{-1/2})$, by the Efron--Stein inequality
(\cref{lem:app-de-variance}). For a sample of age $j$ this gives, in mean
square (\cref{lem:app-de-replace}),
\[
\alpha_1\approx\frac{\emaw j\,d}{1+\emaw j^2q_B\,d^2},\qquad
\alpha_3\approx-\frac{s\,\emaw j^2\,\tilde\sigma_B\,d^2}{1+\emaw j^2q_B\,d^2},\qquad q_B=s\,\sigma_B\tilde\sigma_B ,
\]
and the age of a sample enters only through its squared weight in the
denominator.

Gaussian integration produces the weight function, and diffuseness removes
the cross term. In the term carrying $\alpha_3$ one needs
$\EE[\psi(d)\,z\tensor z]$ with $\psi(d)=d^2/(1+xd^2)$ and $x=\emaw j^2q_B$.
Stein's identity gives the exact formula
\[
\EE\big[(z\tensor z-\Id)\psi(d)\big]=\Ahat\,\EE\big[(w\tensor w)\psi''(d)\big]\,\Ahat^\T,
\]
whose Frobenius norm is at most $6\norm{\Ahat}$ (\cref{lem:app-de-stein}).
The residual $d$ is asymptotically standard Gaussian because
$\norm{\Ahat}\to0$ (\cref{lem:app-setup-residual}), so $\EE\psi(d)$ tends to
$\Jfun(x)$.
In the term carrying $\alpha_1$ the response is
$\EE[f(d)\,w\tensor z]=\Ahat^\T\EE f'(d)+O(\norm{\Ahat}^2)$, of Frobenius norm
of order one but of operator norm at most $\norm{\Ahat}$, which vanishes.
Summing over ages gives, in operator norm (\cref{eqn:app-de-closure-error}),
\[
\Id-s\,\EE R-a_B\,\EE R\,\Sigmaout\longrightarrow0,\qquad
a_B=s\,\tilde\sigma_B\,\Fw(q_B),\qquad \Fw(q)=\sum_{j\ge0}\emaw j^2\,\Jfun(\emaw j^2q),
\]
that is, $\EE R\approx(s\Id+a_B\Sigmaout)^{-1}$. Taking the trace against
$\Sigmaout$ closes the loop: $(\sigma_B,\tilde\sigma_B,a_B,c_B)$, with $c_B$
the input-side analogue of $a_B$, solves the weighted system up to vanishing
residuals.

An approximate root is close to the root. The weighted system has exactly
one nonnegative solution, and it is stable under approximate roots uniformly
over spectral measures of bounded mass and support
(\cref{thm:app-de-branch}(iii)). Hence $a_B\to\tilde s_1$, and the resolvent
identity gives $\EE R\to\Rbar$. The Gaussian limit of $d$ and the replacement
of the coefficients are qualitative, no rate is assumed for $\norm{\Ahat}\to0$,
and the stability is proved by compactness, so the engine gives no rate. The system
depends on the whole weight sequence: with $\sw p=\sum_j\emaw j^p$, the EMA weights with $\bmom>0$ have
$\sw4/\sw2^{\,3}=(1+\bmom)^2/(1+\bmom^2)>1$, so no single effective batch size
$B/\sw2$ reproduces $\Fw$ beyond its value at $q=0$ (\cref{rem:app-de-effbatch}).

\subsection{Drift: the expected regularized update}\label{sk:E2}
This engine computes the mean of the update in the frozen-state model
(\cref{sec:app-drift}), which is the input of every drift kernel. It works
in the setting of \cref{sk:E1} and uses its output. It produces a
deterministic matrix $\Gtildebar_\ridge$, whose entries in the eigenbases of the
covariances are multiples of the corresponding entries of $\Dt$, with, for normalized ridges in a fixed
interval $[\ridge_-,\ridge_+]\subset(0,\infty)$ (\cref{thm:app-drift-mean}),
\[
\sup_{\ridge\in[\ridge_-,\ridge_+]}\norm{\EE\,\Qridge\ridge(\Mt)-\Gtildebar_\ridge}_\F\longrightarrow0 .
\]
The norm is the unnormalized Frobenius norm, the norm in which
\cref{eqn:app-drift-tested} controls the drift errors of all modes together.

The update is an integral of the mixed resolvent along the negative axis.
For a singular value $\sigma_k$ of $\Mt$, the scalar integral
$\frac2\pi\int_0^\infty\sigma_k(\sigma_k^2+\ridge+u^2)^{-1}du$ equals
$\sigma_k(\sigma_k^2+\ridge)^{-1/2}$, the filter of the regularized polar map,
so (\cref{lem:app-setup-contour})
\[
\EE\,\Qridge\ridge(\Mt)=\frac2\pi\int_0^\infty\mathcal H(\ridge+u^2)\,du,\qquad
\mathcal H(s)\defeq\EE\big[(\Mt\Mt^\T+s\Id)^{-1}\Mt\big].
\]
It suffices to find a deterministic equivalent of $\mathcal H(s)$ in
Frobenius norm, uniformly for $s$ in compact sets, and a bound that controls
the integral at large $s$.

The mixed resolvent solves a linear response equation. Deleting one sample
as in \cref{sk:E1}, now through the identity
$\mathsf hR\mathsf u=\alpha_1R_0\mathsf u-\alpha_2\mathsf X_0\mathsf v$ with
$\alpha_2=\mathsf h^2\mathsf U/\Lambda$, writes $\mathcal H$ exactly as a sum
over ages of a source term carrying $\alpha_1$ and a feedback term carrying
$\alpha_2$ (\cref{lem:app-drift-exact}). In the source term the Gaussian
response of a sample with residual function $f$ is
\[
\EE\big[f(d)\,z\tensor w\big]=\Ahat\,\EE f'(d)+O\big(\norm{\Ahat}^2\big),
\]
and multiplying by $\Sigmaout^{1/2}$ and $\Sigmain^{1/2}$ turns $\Ahat$ into the
mean gradient
\[
\EE\Gt=\Sigmaout\Dt\Sigmain/\sqrt{2\cR(t)}.
\]
By Stein's identity
$\EE f'(d)\to\Jfun(\emaw j^2q)$, so the weights now enter to the first power,
through $\Lw$. In the feedback term $\alpha_2$ is replaced by its conditional
mean, whose expected sum over ages, $b_B(s)$, tends to $\sigma\Fw(q)$. The
result is the response equation (\cref{prop:app-drift-response})
\[
\mathcal H(s)\big(\Id+b_B(s)\,\Sigmain\big)=\Lw(q)\,\Rbar(-s)\,\EE\Gt+o(1)\quad\text{in Frobenius norm.}
\]
Since $s(\Id+\sigma\Fw(q)\Sigmain)=s\Id+s_1\Sigmain=\Rtbar(-s)^{-1}$ by the
fixed-point system, the candidate below solves the same equation exactly with
$b_B$ replaced by its limit, and comparing the two equations shows that
$\mathcal H(s)$ is close to it (\cref{thm:app-drift-mixed}(i)):
\[
T(s)=s\,\Lw\big(q(-s)\big)\,\Rbar(-s)\,\EE\Gt\,\Rtbar(-s).
\]
Because the Frobenius norm is not normalized by the dimension, every
replacement must cost $o(1)$ there, which is more than the trace statements
of \cref{sk:E1} give. The delicate replacement is the covariance between the
conditional mean of $\alpha_2$ and the deleted mixed resolvent. It is expanded
along a martingale over the other samples; replacing one of them changes the
buffer by rank two, which moves the conditional mean by $O(\emaw j^2/B)$
(\cref{lem:app-drift-rank}), and the covariance is
$O(\emaw j^2B^{-1/2})$ in Frobenius norm (\cref{lem:app-drift-covariance}).

The integral is split at a cutoff. The exact identity, without any
asymptotic replacement, bounds $\norm{\mathcal H(s)}_\F\le C/s$ for $s\ge1$ at
every dimension and for every state, and $\norm{T(s)}_\F\le C_\Sigma/s$
(\cref{thm:app-drift-mixed}(ii)). With
$\Gtildebar_\ridge=\frac2\pi\int_0^\infty T(\ridge+u^2)\,du$ this gives, for
every $\mathsf R\ge1$,
\[
\norm{\EE\,\Qridge\ridge(\Mt)-\Gtildebar_\ridge}_\F\le\frac2\pi\Big[\mathsf R\sup_{\ridge_-\le s\le\ridge_++\mathsf R^2}\norm{\mathcal H(s)-T(s)}_\F+\frac{C}{\mathsf R}\Big],
\]
and the dimension limit followed by $\mathsf R\to\infty$ gives the output.

The candidate acts entrywise in the eigenbases. Since $\Rbar(-s)$ and
$\Rtbar(-s)$ are functions of $\Sigmaout$ and $\Sigmain$,
\[
u_i^\T T(s)v_j=\frac{s\,\Lw(q(-s))\,\mu_i\lambda_j}{(s+\tilde s_1(-s)\mu_i)(s+s_1(-s)\lambda_j)}\cdot\frac{u_i^\T\Dt v_j}{\sqrt{2\cR(t)}},
\]
and integrating at $s=\ridge+u^2$ gives
$u_i^\T\Gtildebar_\ridge v_j=\mathfrak d(\mu_i,\lambda_j)\,u_i^\T\Dt v_j/\sqrt{\cR(t)}$
with the kernel \cref{eqn:app-drift-kernel}. The same argument applied to the
buffer truncated to its first $J$ ages shows that, after the dimension limit,
ages beyond $J$ contribute a vanishing amount as $J\to\infty$
(\cref{prop:app-drift-truncation}); the dynamics uses this.

\subsection{Volatility and the lag kernel}\label{sk:E3}

The volatility is computed in the frozen-state model of
\cref{def:app-setup-frozen}: the error is held at a state whose normalized
error $\Ahat$ has vanishing operator norm, and the momentum buffer
$M_{t'}=\sum_{j\ge0}\emaw jG_{t'+1-j}$ is a stationary average of independent
batches at that state. The updates
${\mathsf S}_{t'}=\Qridge\ridge(M_{t'})$ are then stationary, and ${\mathsf S}^{\mathrm c}_{t'}$ is
the update minus its mean. For matrices $F,G$ diagonal in the eigenbases,
$C_\ell[F,G]=B^{-1}\EE\Tr[F{\mathsf S}^{\mathrm c}_tG\,{\mathsf S}^{\mathrm c\,\T}_{t+\ell}]$ is
the covariance of two updates $\ell$ steps apart; with $F=u_i\tensor u_i$ and
$G=v_j\tensor v_j$ it is that of the entries of mode $(i,j)$. The engine
produces the forcing measure $\zeta_\ridge$. Its atom at a point $(\mu,\lambda)$ is the
deterministic long-run covariance of the update tested on the eigenspaces of $\mu$ and
$\lambda$, and each mode $(i,j)$ contributes the mass $\mathfrak v_{ij}/B$ at
$(\mu_i,\lambda_j)$. The engine shows that the covariance measure of $T$ consecutive
updates, divided by $BT$, is within an explicit error of its long-run value and within
$o(1)$ of $\zeta_\ridge$, and that $\mathfrak v_{ij}$ is a series of
nonnegative double integrals of a closed-form kernel.

\begin{sketch}
Over a block of $T$ steps the updates accumulate their lag covariances, which
decay geometrically in the lag. Expanding the square of
$S_T=\sum_{t'\le T}{\mathsf S}^{\mathrm c}_{t'}$ over pairs of times, stationarity
makes all pairs at lag $\ell$ equal. This gives an exact identity, and the
geometric decay bounds its distance to the long-run variance
$\Gamma[F,G]=C_0[F,G]+2\sum_{\ell\ge1}C_\ell[F,G]$
(\cref{thm:app-vol-lag-reduction}(i),(iii)):
\[
\begin{aligned}
\frac1{BT}\,\EE\Tr\big[F\,S_T\,G\,S_T^\T\big]&=C_0[F,G]+2\sum_{\ell=1}^{T-1}\Big(1-\frac\ell T\Big)C_\ell[F,G]\\
&=\Gamma[F,G]+O\Big(\frac{\bmom^T}{1-\bmom}+\frac{\bmom}{T(1-\bmom)^2}\Big).
\end{aligned}
\]
The decay comes from the recursion
$M_{t+\ell}=\bmom^\ell\Mt+\sum_{r<\ell}\emaw rG_{t+\ell+1-r}$. Replacing $\Mt$ in it by
a copy built from independent batches gives a buffer with the law of
$M_{t+\ell}$, independent of $\Mt$, at distance $\bmom^\ell$ times the
difference of two independent buffers. The regularized map is
$\ridge^{-1/2}$-Lipschitz in Frobenius norm, so $\abs{C_\ell}\le\mathsf K_\star\bmom^\ell$. The error term of the display holds also in total variation for the covariance measure of $T$
consecutive updates (\cref{thm:app-vol-joint-forcing}(iv)), the form the
dynamics use.

At each fixed lag, the covariance is unchanged in the limit when every
residual $d=z^\T\Ahat w$, which depends on its own sample, is replaced by an
independent Gaussian weight. Call a \emph{filter} an odd function applied to the
singular values of a buffer: $\Qridge\ridge$ is the filter $x/\sqrt{x^2+\ridge}$, and
it is an average of the resolvent filters $x/(x^2+s)$, which give
$(KK^\T+s\Id)^{-1}K$ (\cref{lem:app-setup-contour}). Adding one sample is a
rank-one change of each buffer. By the Woodbury formula, with a $2\times2$ matrix to invert, the change of a
tested product of two resolvents is a function of $d$ and of quadratic forms
that concentrate at normalized traces, so it depends on the sample only
through the law of $d$, which is asymptotically Gaussian because
$\norm{\Ahat}\to0$. Summed over the samples, one at a time, the errors vanish
(\cref{lem:app-vol-replacement}).

In the resulting model, which has no state, the lag covariances have
deterministic limits in which only tree-shaped Gaussian pairings survive. For polynomial filters the tested product of the two buffers is
a polynomial in Gaussian variables, and its expectation is a sum over pairings, each gluing the cyclic product into
a connected graph. Pairings whose graph is a tree
contribute order one and the others $O(1/B)$ (\cref{lem:app-vol-tree-closure}). The
tree values are products of normalized traces $B^{-1}\sum_i\mu_i^p$,
$B^{-1}\sum_j\lambda_j^p$ and of moments of the weights, hence deterministic, and
polynomial approximation extends the limit to all odd bounded continuous
filters (\cref{lem:app-vol-tree-extension}). The resolvent filters give the lag kernel
$\mathfrak k_\ell[F,G](s,u)$, and the filter of $\Qridge\ridge$ gives the deterministic
covariance $\overline C_\ell[F,G](\ridge,\ridge)$, the double integral of
$\frac4{\pi^2}\mathfrak k_\ell[F,G](\ridge+x^2,\ridge+y^2)$ over $x,y>0$
(\cref{thm:app-vol-lag-kernel,thm:app-vol-closed-filtered}). With the geometric
decay, $\Gamma$ is within $o(1)$ of $\overline\Gamma=\overline C_0+2\sum_{\ell\ge1}\overline C_\ell$,
without a rate (\cref{thm:app-vol-integrated}(iv)).

The trees can be summed, and the lag kernel is an explicit expression in the
solution of the fixed-point system of \cref{fig:fp-system} at two spectral parameters
and in two scalar functions of the momentum. In every
tree, the subtrees hanging off the path from the vertex carrying the test $F$ to the
vertex carrying the test $G$ involve one buffer only, and they sum to its one-time
deterministic equivalent (\cref{sk:E1}): the branch
$(\sigma,\tilde\sigma,\tilde s_1,s_1,q)$ at $-s$ for the earlier buffer and, primed,
at $-u$ for the later one (\cref{lem:app-vol-closed-onebuffer}). Only batches on
the path are shared; a batch of age $j$ enters the two buffers with weights
$\emaw j$ and $\bmom^\ell\emaw j$. The sum over paths is a geometric series of
$2\times2$ matrices, and, first for large $s,u$ and then by analyticity for all
$s,u>0$ (\cref{thm:app-vol-closed-kernel}),
\[
\begin{gathered}
\mathfrak k_\ell[F,G](s,u)=\frac{su\,b_\ell\,\Tr\big(F\Sigmaout\Rbar(-s)\Rbar(-u)\big)\Tr\big(G\Sigmain\Rtbar(-s)\Rtbar(-u)\big)}{B^2\,\Delta_\ell(s,u)},\\
b_\ell=\sum_{j\ge0}\EE\frac{\bmom^\ell\emaw j^2g^2}{(1+q\emaw j^2g^2)(1+\bmom^{2\ell}q'\emaw j^2g^2)},
\end{gathered}
\]
where $\Rbar(-s)=(s\Id+\tilde s_1\Sigmaout)^{-1}$, $\Rtbar(-s)=(s\Id+s_1\Sigmain)^{-1}$, and
$\Delta_\ell>0$ is the determinant of the $2\times2$ series, built from the branch, the traces of the same form with $F=\Sigmaout$ and $G=\Sigmain$,
$b_\ell$ and a companion scalar $e_\ell$. The lag
enters only through $b_\ell$ and $e_\ell$, and $b_\ell$ carries the factor $\bmom^\ell$.
With spectral projections as tests the kernel factors into multiplicities and
a nonnegative mode kernel, $\mathfrak k_\ell[P_\mu,P'_\lambda](s,u)=r_\mu t_\lambda B^{-2}\,\mathfrak h_\ell(\mu,\lambda;s,u)$ for the spectral projections $P_\mu,P'_\lambda$ of $\Sigmaout,\Sigmain$ with ranks $r_\mu,t_\lambda$, and the atom of $\zeta_\ridge$ at
$(\mu_i,\lambda_j)$ splits into equal mode coefficients
(\cref{cor:app-vol-closed-modes}, \cref{prop:app-vol-disintegration})
\[
\mathfrak v_{ij}(\ridge)=\frac1B\Big(\mathfrak c_0(\mu_i,\lambda_j;\ridge)+2\sum_{\ell\ge1}\mathfrak c_\ell(\mu_i,\lambda_j;\ridge)\Big),
\]
where $\mathfrak c_\ell(\mu,\lambda;\ridge)\ge0$ is the double integral of
$\frac4{\pi^2}\mathfrak h_\ell(\mu,\lambda;\ridge+x^2,\ridge+y^2)$ over $x,y>0$.
\end{sketch}

\subsection{From the algorithm to the deterministic dynamics}\label{sk:E4}

This is the argument of \cref{thm:app-dyn-main}, sketched in full after that
theorem. It takes in the drift kernel $\mathfrak d$ (\cref{sk:E2}), the
forcing measure $\zeta$ with the covariance of consecutive updates
(\cref{sk:E3}), and an initial error that is \emph{diffuse}: the normalized
error $\widehat A_0$, of Frobenius norm one, has vanishing operator norm. It
produces the comparison, up to a time $T_\star$, of the spectral energy
measure $\chi(\Delta_t)$, which places $\CMscr Q_{ij}$ at $(\mu_i,\lambda_j)$,
with the solution of the deterministic measure equation
(\cref{eqn:app-dyn-ode-measure}).

The argument has four ideas. The algorithm is compared, over blocks of
$L=\lfloor B^{1/8}\rfloor$ steps, with the model in which the state is held
at its value at the start of the block and the momentum buffer is a
stationary average of independent batches; because a test of the energy measure is a
quadratic form in the error, its expected change over a block has an exact
expansion (\cref{eqn:app-dyn-block-identity}). The two leading terms of that
expansion are the drift and the volatility, and because the block is long the
momentum correlations of consecutive updates are summed in full. The motion of
the state inside a block changes the expected update only through a causal
response, which Stein's identity reduces to projections that are small
because the error is diffuse (\cref{prop:app-dyn-stationary-tangent}).
Finally, diffuseness persists, by an operator-norm bound on the mean update
that holds at every state and a matrix martingale inequality
(\cref{lem:app-dyn-diffuse}). Gronwall's inequality and routine estimates
close the comparison.

\subsection{Power-law and large-batch asymptotics}\label{sk:E5}

This engine computes the half-anisotropic coefficients $\mathfrak d_i$ and
$\mathfrak v_i$ when $\Sigmain=\Id_N$ and $\Nout=\Nin=N$, in two regimes: the
thin batch $N/B\to\infty$ with a power-law spectrum, and the oversampled batch
$B_{\rm eff}>3(1+\bmom)^2N$ with an arbitrary spectrum, where only the drift is computed. It takes in the scalar
fixed-point system of the deterministic equivalent (\cref{sk:E1}) and the
integral formulas for the drift and the volatility (\cref{sk:E2,sk:E3}), and
it produces explicit profiles. At each finite dimension the coefficients are
deterministic functions of the spectrum, of $B$ and of $\bmom$; the zero-ridge
coefficient is their limit as $\ridge\downarrow0$, and nothing is claimed about
the algorithm at zero ridge.

With an isotropic input side, the fixed-point system reduces to a single
scalar, the participation $m$. Put
\[
m_N(\Lambda)=\frac1B\sum_{i=1}^N\frac{\Lambda\mu_i}{1+\Lambda\mu_i},\qquad
\Lambda_N\defeq m_N^{-1},\qquad
f_i(m)=\frac{\Lambda_N(m)\mu_i}{1+\Lambda_N(m)\mu_i}.
\]
Each eigenvalue enters with a soft indicator of $\mu_i\ge1/\Lambda$, so
$Bm_N(\Lambda)$ is a soft count of the output directions above the level
$1/\Lambda$, and $\sum_if_i(m)=Bm$. On the branch at the spectral parameter
$-s$ the participation equals $r_\bmom(q)\defeq q\Fw(q)$, which counts, up to a bounded error,
the past batches that the normalized update whitens (\cref{lem:app-pl-weights}(iii)), and $s$ fixes $m$ through the exact relation
$s=(\gamma-m)m/(Q_m\Lambda_N(m))$, with $Q_m=r_\bmom^{-1}(m)$; this is a
decreasing bijection of $(0,\gamma)$ onto $(0,\infty)$
(\cref{prop:app-pl-branch}). In this coordinate the drift is a single integral
and the zero-lag volatility is $f_i$ itself (\cref{eqn:app-pl-drift-single}):
\[
\mathfrak d_i(\ridge)=\frac{\sqrt2}{\pi\gamma}\int_0^\infty
\frac{\Lw(Q_{m(\ridge+u^2)})}{\Fw(Q_{m(\ridge+u^2)})}\,f_i\big(m(\ridge+u^2)\big)\,\dif u,
\qquad \volraw_i(\ridge)=f_i\big(m(\ridge)\big).
\]
The volatility is a double integral over two participations with an explicit
lag kernel (\cref{eqn:app-pl-vol-zero}). With momentum $r_\bmom$ is unbounded,
so $m\to\gamma$ as $s\downarrow0$, and at zero ridge the zero-lag coefficient of
every row is one, $\volraw_i(0)=1$ (\cref{prop:app-pl-drift-zero,prop:app-pl-vol-zero}).

The drift integral splits at participation of order one into a head, where
few past batches are whitened, and a tail, where many are. In the head $q$ is
small, the weight functions take their small-argument values $\Fw\approx\sw2$
and $\Lw\approx1$, and the integrand is a Lorentzian in $u$ of width
$\sqrt{\gamma\sw2\mu_i}$. When the row is resolved, that is
$m_N(1/\mu_i)\to0$, the head part of the drift is
(\cref{lem:app-pl-head})
\[
\sqrt2\,\mathfrak d_i^{\le M}\sim\sqrt{\frac{\mu_i}{\gamma\sw2}}=\sqrt{\frac{\mu_iB_{\rm eff}}N}.
\]
In the tail the weight function $\sqrt q\,\Lw(q)$ oscillates with the
geometric spacing of the weights and has no limit, but its average in the
participation is exact. With $\mathcal A(m)\defeq\sqrt{Q_m}\Lw(Q_m)/E(m)$ and
$E(m)\defeq Q_mr_\bmom'(Q_m)$ (\cref{lem:app-pl-average}),
\[
\sup_{t\ge m_0}\Big|\int_{m_0}^t\big(\mathcal A(m)-\sqrt{2\pi}\big)\,\dif m\Big|\le C_{\bmom,m_0}.
\]
In the variable $q$ each past batch contributes
$\int_0^\infty x^{-1/2}\Jfun(x)\,\dif x=\pi\,\EE\abs g=\sqrt{2\pi}$ once it is
whitened, and $m$ counts the whitened batches up to a bounded error. An
integration by parts then replaces $\mathcal A$ by $\sqrt{2\pi}$ against any
slowly varying weight (\cref{eqn:app-pl-average-ibp}). This is where the constant
$c_0=\sqrt{2\pi}/\pi=\EE\abs g$ of the profiles comes from.

For a power law the participation is a Riemann sum of a continuum function,
and the small-argument law of that function separates the cases $\alpha<1$,
$\alpha=1$ and $\alpha>1$. With $V=c_N\Lambda N^{-\alpha}$ one has
$m_N(\Lambda)=\gamma S_N(V)$, where
$S_N(V)=N^{-1}\sum_jV/(V+(j/N)^\alpha)$ is within $1/N$ of
$S_\alpha(V)=\int_0^1V(V+t^\alpha)^{-1}\dif t$, so that
$\Lambda_N(m)=(N^\alpha/c_N)V_N(m/\gamma)$ with $V_N=S_N^{-1}$
(\cref{eqn:app-pl-Lambda-V}, \cref{lem:app-pl-finite-inverse}). As $V\downarrow0$
(\cref{lem:app-pl-continuum}),
\[
S_\alpha(V)\sim\frac V{1-\alpha}\ (\alpha<1),\qquad S_1(V)\sim V\log\frac1V,\qquad
S_\alpha(V)\sim K_\alpha V^{1/\alpha}\ (\alpha>1),
\]
with $K_\alpha=\pi/(\alpha\sin(\pi/\alpha))$.
For $\alpha>1$ this says that a participation $m$ resolves about $Bm/K_\alpha$
rows: $c_N\Lambda_N(m)\sim(Bm/K_\alpha)^\alpha$ when $Bm\to\infty$ and
$m/\gamma\to0$ (\cref{lem:app-pl-finite-inverse}(iv)). The amplitude $c_N$
multiplies the drift by $\sqrt{c_N}$ and leaves the volatility unchanged
(\cref{lem:app-pl-amplitude}), so the proofs take $c_N=1$. A uniform
statement over the rows reduces, by passing to subsequences, to row sequences
along which $i/N$ and the relevant ratios converge.

In the oversampled regime the branch is the momentum-free profile equation up
to a factor that stays in $[1-\epsilon_{\rm os},1]$, with
$\epsilon_{\rm os}=3(1+\bmom)^2N/B_{\rm eff}$. In the rescaled variables
$\hat s=s/(\sw2\gin)$ and $b=\tilde s_1(-s)/(\sw2\gin)$, where $\gin=N/B$, the branch satisfies
exactly (\cref{lem:app-os-reduction})
\[
\frac b{\widehat F}+m_{\hat\nu}(\hat s/b)=1,\qquad
\widehat F=\frac{\Fw(q)}{\sw2}\in(0,1],\qquad
m_\nu(c)=\int\frac{x\,\nu(\dif x)}{c+x},
\]
where $\hat\nu$ is the empirical spectral distribution of $\Sigmaout$; this is the
same soft count as above, since $m_N(\Lambda)=\gamma\,m_{\hat\nu}(1/\Lambda)$. The
coupling gives $q\Fw(q)<\gin$, which bounds $q$, and the Gaussian inequality
$\Jfun(x)\ge1/(1+3x)$ (Cauchy--Schwarz with $\EE g^4=3$) then keeps $\widehat F$
and $\Lw$ above $1-\epsilon_{\rm os}$ for every $s$ (\cref{lem:app-os-scalar}). A
monotone comparison places $b$ between $\widehat Fb_0$ and $b_0$, where $b_0$ is
the root at $\widehat F=1$, and hence the drift integrand within the factors
$(1-\epsilon_{\rm os})^{\pm1}$ of the integrand of the profile $\mathfrak d^{\rm os}$
(\cref{lem:app-os-compare}, \cref{thm:app-os-limit}). The substitution
$c=\hat s/b_0$ turns the profile equation into
$\Upsilon_\nu(c)\defeq c(1-m_\nu(c))=\hat s$, and one integration by parts gives
the exact representation (\cref{lem:app-os-repr})
\[
\mathfrak d^{\rm os}_\nu(\mu)=\frac{\sqrt2}\pi\,\mu\int_0^\infty\frac{\sqrt{\Upsilon_\nu(c)}}{(c+\mu)^2}\,\dif c
=\sqrt{\frac\mu2}\;\chi_\nu(\mu),\qquad
\chi_\nu(\mu)=\int_0^\infty\varpi(y)\sqrt{1-m_\nu(\mu y)}\,\dif y,
\]
with $\varpi(y)=\frac2\pi\sqrt y\,(1+y)^{-2}$.
Here $\varpi$ is a probability density and $1-m_\nu(c)$ is a soft fraction of
the spectrum below $c$, so the efficiency $\chi_\nu(\mu)\in(0,1]$ averages the
square root of the fraction of the spectrum below the level of the mode, on a
logarithmic scale around it (\cref{rem:app-os-probabilistic}). The results on
head, packed and tail modes are bounds on this average.

\subsection{Volterra analysis and the phase map}\label{sk:E6}

This is the argument of \cref{thm:app-phase-map}, and the common part of every
result of \cref{sk:ch-phase}. It takes in the deterministic row model of that
chapter: the row recursion \cref{eqn:ri-recursion} in continuous time, with
weights $\mu_i$, drifts $\mathfrak d_i$, volatilities $\mathfrak v_i$ and
initial row energies that are fixed numbers (\cref{def:app-phase-model}), run
at the matched step $\lr=2\sqrt\epsilon/\mathcal S$. It produces an exact
formula for the time $t_{T\epsilon}$ as a prefactor times a dimensionless
integral, lower and upper bounds for that integral, and, for the power-law
arrays, its order in each range of exponents.

\begin{sketch}
The risk of the model solves a scalar equation once time is measured in units
of the effective step. In the row equations the drift of every row is divided
by $\sqrt{\fRisk}$, so the natural time is
$s=\clock(t)=\int_0^t\lr\,\fRisk(t')^{-1/2}\,\dif t'$. In this time each row
solves a linear equation driven by $\sqrt Q$, where $Q(s)$ is the risk at time
$s$, and summing the rows with weights $\mu_i/2$ gives exactly
(\cref{thm:app-phase-volterra,thm:app-phase-clock})
\[
\begin{gathered}
Q=F+k*\sqrt Q,\qquad
t(s)=\frac1\lr\int_0^s\sqrt{Q(v)}\,\dif v,\\
F(s)=\tfrac12\sum_i\mu_i\fQ_i(0)e^{-2\mathfrak d_is},\qquad
k(s)=\tfrac\lr2\sum_i\mu_i\mathfrak v_ie^{-2\mathfrak d_is}.
\end{gathered}
\]
The forcing $F$ is the risk the model would have without noise, and the kernel
$k$, of total mass $\lr\mathcal S/2$, feeds the volatility back in. Monotone
iteration gives a solution, unique by the comparison principle below, and its limit is
$\lr^2\mathcal S^2/4$, which
is $\epsilon$ at the matched step; there the kernel has mass $\sqrt\epsilon$.

The time to a level is the integral of $\sqrt Q$ up to the first passage, and it
is squeezed between the forcing and any function that dominates the equation.
Since $k*\sqrt Q\ge0$, the risk is at least the forcing, so it cannot reach a level
before the forcing does, and $t_a\ge\lr^{-1}\int_0^{s_F}\sqrt F$, with $s_F$ the first
time $F\le a$ (\cref{thm:app-phase-hitting}). Conversely, a positive function $U$
with $U\ge F+k*\sqrt U$ lies above $Q$: where $Q>U$ the difference is bounded by a
fixed multiple of its own convolution with $k$; since $k$ has small mass on short intervals, this forces it to vanish step
by step (\cref{lem:app-phase-comparison}).

For the power-law arrays a change of time scale separates the dimension from the
accuracy. With the time scale $T_0^j$ the decay rates $2\mathfrak d_iT_0^j$ become
$2i^{-q_j}$, and $Y_j(u)=Q_j(T_0^ju)$ solves
$Y_j=f_{N,q_j}+\sqrt\epsilon\,g_j*\sqrt{Y_j}$, where
$f_{N,q}(u)=\frac12\sum_{i\le N}i^{-p}e^{-2ui^{-q}}$ and $g_j$ is a decreasing kernel of
mass one; neither depends on $B$ (\cref{prop:app-phase-prefactors}). If $\hat u_j$ is
the first time $Y_j\le T\epsilon$, then exactly (\cref{eqn:app-phase-time,eqn:app-phase-ST0})
\[
t^j_{T\epsilon}=\frac{\mathcal S_jT_0^j}{2\sqrt\epsilon}\,\mathsf I_j,\qquad
\mathsf I_j\defeq\int_0^{\hat u_j}\sqrt{Y_j(u)}\,\dif u,\qquad
\frac{\mathcal S_{\text{s-SGD}}T_0^{\text{s-SGD}}}{\mathcal S_{\text{s-SVD}}T_0^{\text{s-SVD}}}=D_N .
\]

The order of the clock integral $\mathsf I_j$ is decided by whether $\sqrt{f_{N,q_j}}$ is
integrable uniformly in $N$. The forcing lies between two power laws
(\cref{lem:app-phase-sums}), the lower one for $N\ge4$ and $u\le(N/4)^q$:
\[
A_p(1+u)^{-2r_q}\le f_{N,q}(u)\le C_{p,q}^2(1+u)^{-2r_q},\qquad r_q=\frac{p-1}{2q} .
\]
Write $\mathsf I_j=\int_0^{\hat u_j}\sqrt{f_{N,q_j}}+\mathcal E_j$, where $\mathcal E_j\ge0$ is the
noise part. When $r_j\le1$, suppose the first passage occurs before a horizon
$H\le(N/4)^{q_j}$. There the lower envelope holds, and $f_{N,q_j}(\hat u_j)\le T\epsilon$ gives
$1+\hat u_j\ge(A_p/T\epsilon)^{1/(2r_j)}$. Integrating the envelopes then gives (\cref{lem:app-phase-clock-integral})
\[
\int_0^{\hat u_j}\!\!\sqrt{f_{N,q_j}}\ \begin{cases}
\text{lies between }c\,\epsilon^{-(1-r_j)/(2r_j)}\text{ and }C_{p,q_j}I_{r_j}(H),&r_j<1,\\[2pt]
=\dfrac{\log(1/\epsilon)}{2\sqrt{8q_j}}+O\big(1+\log_+((1+H)\sqrt\epsilon)+HN^{-q_j}\big),&r_j=1,\\[4pt]
\longrightarrow J(p,q_j)\defeq\int_0^\infty\sqrt{f_{\infty,q_j}(u)}\,\dif u\quad(N\to\infty,\ \epsilon\downarrow0),&r_j>1,
\end{cases}
\]
with $I_r(x)=\int_0^x(1+w)^{-r}\dif w$, so that $C_{p,q}I_r(H)$ is of order $\epsilon^{-(1-r)/(2r)}$
when $H$ is of order $\epsilon^{-1/(2r)}$. Since $r_{\text{s-SVD}}=(p-1)/\alpha$ and $r_{\text{s-SGD}}=(p-1)/(2\alpha)$, the
boundary $r_j=1$ is $\texp=1$ for \SignSVD and $\texp=\alpha+1$ for \SignSGD.

The noise part is controlled by comparison with $(h+\sqrt\epsilon)^2$, and checking
the kernel condition is what produces the windows. Take for $h$ a power-law envelope
of $\sqrt{f_{N,q_j}}$ that decreases to $\sqrt\epsilon$ at a horizon $H$. If
$g_j*h\le2h$ or $g_j*h\le\sqrt\epsilon$ on $[0,H]$, then $(h+\sqrt\epsilon)^2$ dominates
the equation, the first passage occurs before $H$, and $\mathcal E_j$ is at most a fixed
multiple of $\epsilon^{(r_0-1)/(2r_0)}$, where $r_0\le r_j$ is the exponent of the envelope
and $r_0>1$ exactly when $r_j>1$ (\cref{lem:app-phase-square}). For $r_j<1$ this is of
the order of the main term, so these models get orders; for $r_j>1$ it tends to zero,
and the time has an equivalent. (For \SignSVD at $\alpha=2$, $\texp>1$,
\cref{lem:app-phase-log-global} adds a correction term to the square and bounds the noise part by a
constant that tends to zero.) The kernel condition holds by small local mass of a
decreasing kernel, by an integrable power tail, or by a small kernel height, each
on a range of times that the window of the theorem keeps above the horizon (see the
note to \cref{thm:app-phase-map}).
\end{sketch}

\section{Dynamics: from the algorithm to the deterministic recursion}\label{sk:ch-dynamics}
This chapter connects the algorithm to the deterministic recursion
\cref{eqn:fQ-recursion}. The algorithm runs at the step $\lr=\macrostep/B$, so a unit of
the macroscopic time $\macrotime=t/B$ is $B$ steps. The error is $\Dt=\Wt-\Wstar$, and its
normalized form $\widehat A_t=\Sigmaout^{1/2}\Dt\Sigmain^{1/2}/\sqrt{2\cR(t)}$ has Frobenius norm one. With
$\mathfrak d_{ij}=\mathfrak d(\mu_i,\lambda_j)$ the drift kernel (\cref{cor:app-drift-kernel}) and $\mathfrak v_{ij}$ the
volatility kernel, which is $B$ times the mass of the forcing measure $\zeta$ at
$(\mu_i,\lambda_j)$ divided equally among the modes there (\cref{sk:E3}), the deterministic mode-energy equations
\cref{eqn:app-dyn-ode}, divided by $B^2$, and the two spectral energy measures are
\[
\begin{gathered}
\frac{\dif\fQ_{ij}}{\dif\macrotime}=-\frac{2\macrostep\,\mathfrak d_{ij}}{\sqrt{\fRisk}}\,\fQ_{ij}+\frac{\macrostep^2}{B}\,\mathfrak v_{ij},\qquad
\fRisk=\frac12\sum_{i,j}\mu_i\lambda_j\fQ_{ij},\\
\chi^{\mathrm{act}}(\macrotime)=\sum_{i,j}\CMscr Q_{ij}(\lfloor B\macrotime\rfloor)\,\delta_{(\mu_i,\lambda_j)},\quad
\chi^{\mathrm{ode}}(\macrotime)=\sum_{i,j}\fQ_{ij}(\macrotime)\,\delta_{(\mu_i,\lambda_j)},
\end{gathered}
\]
with $\fQ_{ij}(0)=\CMscr Q_{ij}(0)$; the measures live on $[0,C_\Sigma]^2$. The norm $\norm{m}_\BL$ of a signed
measure is the supremum of $\abs{\int\phi\,\dif m}$ over tests with $\norm\phi_\infty\le1$ that are
$1$-Lipschitz for the $\ell^1$ distance. Let $\gamma_\star$, $C_\Sigma$, $E_\star$ and $r_0\le r_1$ be the constants fixed in \cref{thm:app-dyn-main}: bounds on the aspect ratios, on the covariances, on $\norm{\Delta_0}_\F^2$, and on the initial risk from below and above. The horizon $T_\star=\min(r_0,1)/(8(1+C_E+C_{\cR}))$
(\cref{eqn:app-dyn-Tstar}) is explicit, with $C_E$ and $C_{\cR}$ bounds on the rates of change of
the energy and of the risk along the equations, depending only on
$(\macrostep,\ridge,\bmom,\gamma_\star,C_\Sigma,E_\star,r_0,r_1)$. The theorem compares $\chi^{\mathrm{act}}$ with
$\chi^{\mathrm{ode}}$ for general covariances, by the argument of \cref{sk:E4}; the corollary
sums the equations over the columns when $\Sigmain=\Id$.

\begin{theorem}[Dynamical comparison]\label{thm:app-dyn-main}
Fix $\macrostep>0$, $\ridge>0$, $\bmom\in[0,1)$, $\gamma_\star<\infty$,
$C_\Sigma\ge1$, $E_\star<\infty$, $c_M<\infty$ and $0<r_0\le r_1<\infty$.
Along a sequence of dimensions with $B\to\infty$ assume:
\begin{enumerate}[label=(H\arabic*)]
\item \label{hyp:app-dyn-data} \cref{assum:data-target,assum:proportional-regime}
hold with deterministic covariances, $\max(\Nin,\Nout)\le\gamma_\star B$
and $\norm{\Sigmain},\norm{\Sigmaout}\le C_\Sigma$;
\item \label{hyp:app-dyn-algorithm} the algorithm of
\cref{def:app-setup-algorithm} is run with momentum $\bmom$, normalized
ridge $\ridge$ and step $\lr_t=\macrostep/B$;
\item \label{hyp:app-dyn-initial} $W_0$ is deterministic,
$\norm{\Delta_0}_\F^2\le E_\star$, $r_0\le\cR(0)\le r_1$, and
$\norm{\widehat A_0}\to0$;
\item \label{hyp:app-dyn-buffer} the initial buffer $M^\circ_{-1}$ is
independent of all batches and $\EE\norm{M^\circ_{-1}}_\F^2\le c_MB$.
\end{enumerate}
Let $T_\star$ be given by \cref{eqn:app-dyn-Tstar} and let $(\fQ_{ij})$ be
the solution of \cref{eqn:app-dyn-ode} of \cref{prop:app-dyn-ode}. Then the
following hold.
\begin{enumerate}[label=(\roman*)]
\item With probability tending to one,
\[
t_{\rm stop}>BT_\star,\qquad
\inf_{t\le BT_\star}\cR(t)\ge r_0/2,\qquad
\sup_{t\le BT_\star}\norm{\Dt}_\F^2\le E_\star+1.
\]
\item $\sup_{\macrotime\le T_\star}\norm{\chi^{\mathrm{act}}(\macrotime)-\chi^{\mathrm{ode}}(\macrotime)}_\BL\to0$
in probability.
\item $\sup_{\macrotime\le T_\star}\norm{\mu\lambda\,(\chi^{\mathrm{act}}(\macrotime)-\chi^{\mathrm{ode}}(\macrotime))}_\BL\to0$
and $\sup_{\macrotime\le T_\star}\abs{\cR(\lfloor B\macrotime\rfloor)-\fRisk(\macrotime)}\to0$
in probability.
\item On the event of (i), $\sup_{t\le BT_\star}\norm{\Ahat}\to0$ in probability.
\end{enumerate}
\end{theorem}

\skpointer{thm:app-dyn-main}{Lean: proved: \texttt{\detokenize{Muon}\detokenize{.}\allowbreak{}\detokenize{Dynamics}\detokenize{.}\allowbreak{}\detokenize{thm}\detokenize{_}\allowbreak{}\detokenize{app}\detokenize{_}\allowbreak{}\detokenize{dyn}\detokenize{_}\allowbreak{}\detokenize{main}\detokenize{_}\allowbreak{}\detokenize{i}}, \texttt{\detokenize{Muon}\detokenize{.}\allowbreak{}\detokenize{Dynamics}\detokenize{.}\allowbreak{}\detokenize{thm}\detokenize{_}\allowbreak{}\detokenize{app}\detokenize{_}\allowbreak{}\detokenize{dyn}\detokenize{_}\allowbreak{}\detokenize{main}\detokenize{_}\allowbreak{}\detokenize{ii}}, +4 more; foundation results: \texttt{\detokenize{GaussianOpNorm}}, \texttt{\detokenize{GaussianOpNormL6}}, \texttt{\detokenize{GaussianPoincare}}, \texttt{\detokenize{GaussianQuadFormLp}}}
This theorem is the precise form of \cref{thm:deterministic-dynamics}. At
the step $\lr=\macrostep/B$, the projected risks $\CMscr Q_{ij}$ follow, in
aggregate, the recursion \cref{eqn:fQ-recursion} with the drift kernel of
\cref{cor:app-drift-kernel} and the volatility coefficient of
\cref{prop:app-vol-disintegration}; the equations \cref{eqn:app-dyn-ode} are
that recursion in the macroscopic time $\macrotime=t/B$. ``In aggregate''
refers to the spectral energy measures $\chi^{\mathrm{act}}$ and
$\chi^{\mathrm{ode}}$ defined at the start of this section, which place the
projected risks at the points $(\mu_i,\lambda_j)$. The risk is half the
integral of $\mu\lambda$ against such a measure, so (iii) is the risk
statement of the main text. Nothing is claimed about a single mode.

The kernels are computed in the frozen-state model of
\cref{def:app-setup-frozen}, in which the momentum buffer is a stationary
average of independent batches, all evaluated at one fixed state. That
computation assumes the state is \emph{diffuse}: the normalized error
$\widehat A_t$ has vanishing operator norm, so
no single singular direction carries a fixed fraction of the risk. Part (iv)
says that this property, assumed at time zero in \ref{hyp:app-dyn-initial},
persists along the trajectory. This is what allows the frozen-state kernels
to be used there.

\begin{sketch}
The algorithm is compared, over blocks of $L=\lfloor B^{1/8}\rfloor$ steps,
with the frozen-state model at the state reached at the start of the block.
Over a block the state moves by at most a constant times $L/\sqrt B$ in
Frobenius norm (\cref{lem:app-dyn-movement}), so we hold it at its
block-start value $D$. We run the same batches at the fixed state $D$, with
the momentum buffer started from an independent stationary past. The
resulting updates $V_r$, $r<L$, are those of the frozen-state model at $D$
(\cref{lem:app-dyn-chains}), and $\overline V$ denotes their common mean. The
spectral energy measure is tested by $X(\Delta)=\int f\tensor g\,\dif\chi(\Delta)$
for bounded $f,g$. This is a quadratic form in $\Delta$, so the expected change
of $X$ over a block has an exact expansion. Write
$\mathcal WX=f(\Sigmaout)\,X\,g(\Sigmain)$, let ${\mathsf S}_{t+r}$ be the actual
updates of the block starting at step $t$, and put
\[
X(\Delta)=\langle\Delta,\mathcal W\Delta\rangle,\qquad
S=\sum_{r<L}V_r,\qquad
\mathcal E=\sum_{r<L}{\mathsf S}_{t+r}-S .
\]
Then, exactly (\cref{eqn:app-dyn-block-identity}),
\[
\begin{aligned}
\EE_t\big[X(\Delta_{t+L})-X(D)\big]
={}&-2\lr\Big\langle\mathcal WD,\sum_{r<L}\EE_t{\mathsf S}_{t+r}\Big\rangle
+\lr^2L^2\langle\overline V,\mathcal W\overline V\rangle\\
&+\lr^2LB\int f\tensor g\,\dif\mathfrak F_L(D)
+\lr^2\,\EE_t\big[2\langle S,\mathcal W\mathcal E\rangle+\langle\mathcal E,\mathcal W\mathcal E\rangle\big],
\end{aligned}
\]
where $\mathfrak F_L(D)$ is the covariance measure of the sum of $L$
consecutive updates of the frozen-state model
(\cref{thm:app-vol-joint-forcing}(iv)).

The two leading terms of this expansion are the drift and the volatility of
the main text. Replace $\EE_t{\mathsf S}_{t+r}$ by $\overline V$ in the first
term. The deterministic equivalent of \cref{cor:app-drift-kernel} gives
$u_i^\T\overline Vv_j=\mathfrak d_{ij}\,u_i^\T Dv_j/\sqrt{\cR(D)}$ up to an
error of vanishing Frobenius norm. The first term is therefore
$-2\lr L\,\cR(D)^{-1/2}\int f\tensor g\,\mathfrak d\,\dif\chi(D)$ up to
$o(L/B)$. The third term carries the covariance of a sum of $L$ consecutive
updates. Momentum correlates these updates, and because the block is long the
correlations are summed in full: $\mathfrak F_L(D)$ is within
$O(\bmom^L+L^{-1})$ in total variation of the integrated lag covariance. That
covariance is close to the forcing measure $\zeta$, which carries the mass
$\mathfrak v_{ij}/B$ at $(\mu_i,\lambda_j)$
(\cref{thm:app-vol-joint-forcing}(iii),(iv), \cref{prop:app-vol-disintegration}).
The third term is therefore $(L/B)\,\macrostep^2\int f\tensor g\,\dif\zeta$ up
to $o(L/B)$. The second term, the square of the mean step, is only
$O(L^2/B^2)$. Over a block, then, the expected change of $X$ is $L/B$ times
the test of the measure $\mathsf F_B(\chi(D))$ below. Summed over the blocks,
this reproduces the deterministic equations, which in measure form read
exactly (\cref{eqn:app-dyn-ode-measure})
\[
\chi^{\mathrm{ode}}(\macrotime)=\chi^{\mathrm{ode}}(0)+\int_0^\macrotime\mathsf F_B\big(\chi^{\mathrm{ode}}(\sigma)\big)\,\dif\sigma,\qquad
\mathsf F_B(\chi)=-\frac{2\macrostep}{\sqrt{\fRisk(\chi)}}\,\mathfrak d\,\chi+\macrostep^2\zeta,
\]
with $\fRisk(\chi)$ half the integral of $\mu\lambda$ against $\chi$, and
$\mathfrak d\,\chi$ the measure with density $\mathfrak d$ with respect to
$\chi$.

The motion of the state within a block does not change the expected update
at leading order, because the error is diffuse. Apart from the effect of
the batches carried in the buffer from before the block, replacing
$\EE_t{\mathsf S}_{t+r}$ by $\overline V$ neglects, to first order in $\lr$,
the response $\lr\,\EE\dot V_r$. Here $\dot V_r$ is the derivative of the
frozen-state update with respect to the step size, at step size zero
(\cref{lem:app-dyn-tangent}). This derivative is causal: in $\dot V_r$ the
batch of step $t+j$ acts only on the updates $V_k$, $k<j$, which are built
from earlier batches and are independent of it. Removing each sample of the
acting batch from the buffer and integrating it out by Stein's identity
(\cref{prop:app-dyn-reduction}) leaves mean products of the projections
$\langle\Sigmaout D\Sigmain,V_k\rangle$ with scalars of the reduced buffer. The
product of their means is bounded without diffuseness (\cref{lem:app-dyn-mean}).
Their covariance is bounded through the change caused by a single sample,
which for such a projection is $O(\delta(D))$ in $L^2$, where
$\delta(D)=\norm{\widehat A(D)}$ is the operator norm of the normalized error
(\cref{prop:app-dyn-scalar}). The ridge also moves with the risk, and its
contribution to $\dot V_r$ is bounded in the same way
(\cref{lem:app-dyn-ridge}). For a test matrix $\Pi$ of bounded Frobenius
norm the result is (\cref{prop:app-dyn-stationary-tangent})
\[
\lr\,\big|\EE\langle\Pi,\dot V_r\rangle\big|\le C\Big\{\frac{r}{\sqrt B}+\frac1B+\delta(D)\big(\norm\Pi+\delta(D)\big)\Big\}.
\]
The relevant test is $\Pi=\mathcal WD$, whose operator norm is bounded but
not small, so the response vanishes only because $\delta(D)\to0$.

The error stays diffuse, by a drift bound and a matrix martingale for the
weighted state $Y_b=\Sigmaout^{1/2}\Delta_{t_b}\Sigmain^{1/2}$ at the start
$t_b$ of block $b$. The drift is controlled by an operator-norm bound on the
mean update that holds at every state of bounded risk, diffuse or not, so the argument is
not circular (\cref{lem:app-dyn-mean}(iv)):
\[
\norm{\Sigmaout^{1/2}\overline V\,\Sigmain^{1/2}}\le C\norm{\Sigmaout^{1/2}D\,\Sigmain^{1/2}}+Cr_B,\qquad r_B\to0 .
\]
The centered increments of $Y_b$ have operator norm $O(L/B)$ in dimension
$O(B)$, so their partial sums vanish by a matrix martingale inequality
(\cref{lem:app-dyn-matrix-martingale}). A discrete Gronwall inequality then
bounds $\max_b\norm{Y_b}$ by $e^{C\macrostep(T+1)}(\norm{Y_0}+o_B)$, with
$o_B\to0$ in probability, and $\norm{Y_0}\to0$ by \ref{hyp:app-dyn-initial}.

The remaining errors are routine, with two exceptions. During the first
$\lceil C_0\log(2B)\rceil$ steps of a block, the momentum buffer still
remembers the previous block, so these updates are not close to the
stationary model. They are compared with the model started that many steps
earlier, and their total cost is of order $\log^2(2B)/L$
(\cref{lem:app-dyn-prefix}). The deterministic equivalents for the drift and
the volatility come without a rate. They are made uniform over diffuse
states by a subsequence argument (\cref{lem:app-dyn-uniform}), and this is
one reason the theorem gives no rate. The fluctuations of the block
increments form a martingale whose maximum vanishes
(\cref{lem:app-dyn-martingale}). Energy and risk then move at bounded rates,
so they stay close to their initial values up to $T_\star$
(\cref{lem:app-dyn-interval}), which is (i). Gronwall's inequality, with the
Lipschitz continuity of $\mathsf F_B$ (\cref{lem:app-dyn-field}), and routine
estimates give (ii) and (iii).
\end{sketch}

When $\Sigmain=\Id_{\Nin}$, the projected row risk is $\CMscr Q_i(t)=\norm{\Dt^\T u_i}^2=\sum_j\CMscr Q_{ij}(t)$, and the
row coefficient $\mathfrak v_i(\ridge)$ of \cref{eqn:app-vol-rows} is the deterministic
counterpart, in aggregate over the rows, of the long-run variance of the row ${\mathsf S}^\T u_i$ of the update: $B$ times
the deterministic integrated covariance $\overline\Gamma[u_i\tensor u_i,\Id](\ridge,\ridge)$ of \cref{sk:E3}.

\begin{corollary}[Half-anisotropic input]\label{cor:app-dyn-half}
Assume the hypotheses of \cref{thm:app-dyn-main} with $\Sigmain=\Id_{\Nin}$.
Let $\mathfrak d_i\defeq\mathfrak d(\mu_i,1)$ and let $\mathfrak v_i\defeq\mathfrak v_i(\ridge)$
be the row coefficient of \cref{eqn:app-vol-rows}. Then the row equations
\begin{equation}\label{eqn:app-dyn-ode-rows}
\frac{\dif}{\dif\macrotime}\big(B\fQ_i(\macrotime)\big)=-\frac{2\macrostep\,\mathfrak d_i}{\sqrt{\fRisk(\macrotime)}}\,B\fQ_i(\macrotime)+\macrostep^2\mathfrak v_i,\qquad
\fRisk(\macrotime)=\frac12\sum_i\mu_i\fQ_i(\macrotime),\qquad
\fQ_i(0)=\norm{\Delta_0^\T u_i}^2,
\end{equation}
have a unique solution on $[0,T_\star]$, given by $\fQ_i=\sum_j\fQ_{ij}$
with $(\fQ_{ij})$ the solution of \cref{eqn:app-dyn-ode}. The row measures
$\chi^{\mathrm{act}}_{\mathrm{row}}(\macrotime)\defeq\sum_i\CMscr Q_i(\lfloor B\macrotime\rfloor)\,\delta_{\mu_i}$
and $\chi^{\mathrm{ode}}_{\mathrm{row}}(\macrotime)\defeq\sum_i\fQ_i(\macrotime)\,\delta_{\mu_i}$
on $[0,C_\Sigma]$ satisfy
$\sup_{\macrotime\le T_\star}\norm{\chi^{\mathrm{act}}_{\mathrm{row}}(\macrotime)-\chi^{\mathrm{ode}}_{\mathrm{row}}(\macrotime)}_\BL\to0$
in probability, and the same holds for the measures weighted by $\mu$.
\end{corollary}

\skpointer{cor:app-dyn-half}{Lean: proved: \texttt{\detokenize{Muon}\detokenize{.}\allowbreak{}\detokenize{Dynamics}\detokenize{.}\allowbreak{}\detokenize{cor}\detokenize{_}\allowbreak{}\detokenize{app}\detokenize{_}\allowbreak{}\detokenize{dyn}\detokenize{_}\allowbreak{}\detokenize{half}\detokenize{_}\allowbreak{}\detokenize{i}}, \texttt{\detokenize{Muon}\detokenize{.}\allowbreak{}\detokenize{Dynamics}\detokenize{.}\allowbreak{}\detokenize{cor}\detokenize{_}\allowbreak{}\detokenize{app}\detokenize{_}\allowbreak{}\detokenize{dyn}\detokenize{_}\allowbreak{}\detokenize{half}\detokenize{_}\allowbreak{}\detokenize{ii}}, +1 more}
The corollary is the dynamical statement behind the half-anisotropic section
of the main text, where $\Sigmain=\Id_N$ and only the output covariance is
anisotropic. There the natural unknowns are the projected row risks
$\fQ_i=\sum_j\fQ_{ij}$, the deterministic counterparts of
$\CMscr Q_i$, and the risk
$\fRisk=\frac12\sum_i\mu_i\fQ_i$ involves only them. The corollary says that the
row risks obey a closed system, the row recursion \cref{eqn:ri-recursion} in
the macroscopic time $\macrotime=t/B$: dividing \cref{eqn:app-dyn-ode-rows} by $B$
and taking $B$ steps of size $\lr=\macrostep/B$ per unit of time, the change over
one step is $-2\lr\,\mathfrak d_i\fQ_i/\sqrt{\fRisk}+\lr^2\mathfrak v_i$. The algorithm follows
this system in aggregate over the rows: the row measure, which places
$\CMscr Q_i$ at $\mu_i$, is close in the bounded-Lipschitz distance to its
deterministic counterpart, and its $\mu$-weighted version gives the risk. As in
\cref{thm:app-dyn-main}, nothing is claimed about a single row.

\begin{sketch}
With an isotropic input side, both kernels are constant along each row, so
summing the mode equations over the columns closes them on the rows. The drift
kernel $\mathfrak d_{ij}=\mathfrak d(\mu_i,\lambda_j)$ is evaluated at $\lambda_j=1$ for every $j$, so
$\mathfrak d_{ij}=\mathfrak d_i$ (\cref{cor:app-drift-halfaniso}). The mode volatilities of a row are
equal, $\mathfrak v_{ij}=\mathfrak v_i/\Nin$, and sum to the row coefficient $\mathfrak v_i$
(\cref{prop:app-vol-half-recovery}). Summing \cref{eqn:app-dyn-ode} over $j$ and
dividing by $B$ therefore gives, exactly,
\[
\frac{\dif}{\dif\macrotime}\Big(B\sum_j\fQ_{ij}\Big)
=-\frac{2\macrostep\,\mathfrak d_i}{\sqrt{\fRisk}}\,B\sum_j\fQ_{ij}+\macrostep^2\sum_j\mathfrak v_{ij}
=-\frac{2\macrostep\,\mathfrak d_i}{\sqrt{\fRisk}}\,B\fQ_i+\macrostep^2\mathfrak v_i ,
\]
with $\fRisk=\frac12\sum_{i,j}\mu_i\fQ_{ij}=\frac12\sum_i\mu_i\fQ_i$ and
$\fQ_i(0)=\sum_j(u_i^\T\Delta_0v_j)^2=\norm{\Delta_0^\T u_i}^2$ by Parseval's identity in the basis
$(v_j)$ (\cref{lem:app-setup-decomposition}).

The row system has a unique solution on $[0,T_\star]$. The mode system has a solution on
$[0,T_\star]$ (\cref{prop:app-dyn-ode}), so its row sums solve the row system there.
The row system is locally Lipschitz where its risk is positive, so this solution
is its only one (Picard--Lindel\"of).

The row measures are images of the mode measures, and the projection
$(\mu,\lambda)\mapsto\mu$ preserves the bounded-Lipschitz norm on measures carried by
the segment $\{\lambda=1\}$. Both $\chi^{\mathrm{act}}(\macrotime)$ and $\chi^{\mathrm{ode}}(\macrotime)$
live on that segment, and their images under the projection are the row
measures, since $\CMscr Q_i=\sum_j\CMscr Q_{ij}$. A test on the rectangle restricts to a
test $\phi(\cdot,1)$ of the same size and Lipschitz constant, and a test $\psi$ of $\mu$
extends to the test $\psi(\mu)$ on the rectangle, so the two norms coincide.
\Cref{thm:app-dyn-main}(ii) therefore gives the row statement, and
\cref{thm:app-dyn-main}(iii) gives the $\mu$-weighted one, because $\mu\lambda=\mu$ on
the segment. All the probabilistic work is that of \cref{thm:app-dyn-main},
sketched in \cref{sk:E4}: the block comparison of the algorithm with the
frozen-state model, the identification of the drift and the volatility, and the
persistence of diffuseness.
\end{sketch}

\section{Drift}\label{sk:ch-drift}
The drift kernel $\mathfrak d_{ij}$ of \cref{eqn:fQ-recursion} is computed in
the frozen-state model of \cref{def:app-setup-frozen}: the state $\Dt$ is held
fixed, the momentum buffer is replaced by $\Mt=\sum_{j\ge0}\emaw j\,G_{t+1-j}$,
with weights $\emaw j=(1-\bmom)\bmom^j$, of independent minibatch gradients all
recomputed at $\Dt$ and divided by $\sqrt{2\cR(t)}$, and the update is
$\Gtilde=\Qridge\ridge(\Mt)$. Its drift on the mode $(i,j)$ is
$\cD_{ij}(t)=(u_i^\T\Dt v_j)\,\EE[u_i^\T\Gtilde v_j]$ (\cref{eqn:app-setup-DV}).
\Cref{assum:app-setup-regime} is the proportional regime in this model, with
a normalized ridge in a fixed interval $[\ridge_-,\ridge_+]\subset(0,\infty)$ and
a diffuse normalized error, $\norm{\Ahat}\to0$. For a nonnegative summable
weight sequence $w$, not all zero, and a standard Gaussian $g$, put
\[
\Jfun(x)=\EE\frac{g^2}{1+xg^2},\qquad
F_w(q)=\sum_{j\ge0}w_j^2\,\Jfun(w_j^2q),\qquad
L_w(q)=\sum_{j\ge0}w_j\,\Jfun(w_j^2q),
\]
written $\Fw,\Lw$ for the EMA weights. The weighted fixed-point system and
its unique solution give the deterministic equivalent of the Gram resolvent
of the buffer (\cref{sk:E1}), from which the expected update (\cref{sk:E2})
and the drift kernel are read off.

In the next definition $w$ is a weight sequence and $F_w$ the function of the
chapter introduction, and the spectral measures of the covariances are
$\nu_{\mathrm{out}}=\frac1B\sum_i\delta_{\mu_i}$ and
$\nu_{\mathrm{in}}=\frac1B\sum_j\delta_{\lambda_j}$
(\cref{eqn:app-setup-spectral-measures}). The spectral parameter $z=-s$ is
negative.

\begin{definition}[Weighted fixed-point system]\label{def:app-de-branch}
Let $(\nu,\tilde\nu)$ be finite positive Borel measures on $[0,\infty)$ with
compact support, let $w$ be as above, and let $z=-s$ with $s>0$. The weighted
fixed-point system at $z$ is the system of equations in the four unknowns
$(\sigma,\tilde\sigma,\tilde s_1,s_1)\in[0,\infty)^4$
\begin{equation}\label{eqn:app-de-system}
\begin{aligned}
\sigma&=\int\frac{x}{\tilde s_1x-z}\,\nu(dx),\qquad&
\tilde\sigma&=\int\frac{x}{s_1x-z}\,\tilde\nu(dx),\\
q&=-z\,\sigma\tilde\sigma,\qquad&
\tilde s_1&=-z\,\tilde\sigma\,F_w(q),\qquad s_1=-z\,\sigma\,F_w(q).
\end{aligned}
\end{equation}
For the spectral measures of \cref{eqn:app-setup-spectral-measures} the first two
equations read
\[
\sigma=\frac1B\Tr\big(\Sigmaout(\tilde s_1\Sigmaout-z\Id_{\Nout})^{-1}\big),\qquad
\tilde\sigma=\frac1B\Tr\big(\Sigmain(s_1\Sigmain-z\Id_{\Nin})^{-1}\big),
\]
and we define the deterministic equivalents
\begin{equation}\label{eqn:app-de-Rbar}
\Rbar(z)\defeq\big(\tilde s_1(z)\,\Sigmaout-z\,\Id_{\Nout}\big)^{-1},\qquad
\Rtbar(z)\defeq\big(s_1(z)\,\Sigmain-z\,\Id_{\Nin}\big)^{-1},
\end{equation}
where $(\sigma,\tilde\sigma,\tilde s_1,s_1)(z)$ is the solution given by
\cref{thm:app-de-branch}.
\end{definition}

\skpointer{def:app-de-branch}{Lean: proved: \texttt{\detokenize{Muon}\detokenize{.}\allowbreak{}\detokenize{DetEquiv}\detokenize{.}\allowbreak{}\detokenize{IsBranchSol}}, \texttt{\detokenize{Muon}\detokenize{.}\allowbreak{}\detokenize{DetEquiv}\detokenize{.}\allowbreak{}\detokenize{branchSol}}, +5 more}
For the EMA weights and the spectral measures of the covariances this is the system of
\cref{fig:fp-system}; the definition allows any weight sequence and any pair of measures. The unknowns
$\sigma,\tilde\sigma$ are the deterministic values of the normalized traces
$\frac1B\Tr(\Sigmaout R(z))$ and $\frac1B\Tr(\Sigmain\widetilde R(z))$ of the
Gram resolvents of the normalized buffer, and $\tilde s_1,s_1$ are the
scalings with which $\Sigmaout$ and $\Sigmain$ enter their deterministic
equivalents $\Rbar,\Rtbar$. Multiplying the last two equations by $\sigma$ and
$\tilde\sigma$ gives the coupling $\tilde s_1\sigma=s_1\tilde\sigma=qF_w(q)$ of
\cref{fig:fp-system}. The system depends on the covariances only through
their eigenvalues, and at $\bmom=0$ the weight function is $F_w=\Jfun$.

In the next theorem $\sw2=\sum_jw_j^2$, and $\mathcal K$ is the class of pairs
$(\nu,\tilde\nu)$ of positive measures on $[0,C_\Sigma]$, each of total mass at
most $\gamma_\star$ (\cref{eqn:app-de-class}); it contains
$(\nu_{\mathrm{out}},\nu_{\mathrm{in}})$ at every dimension under
\cref{assum:app-setup-regime}.

\begin{theorem}[The weighted branch]\label{thm:app-de-branch}
Let $w$ be nonnegative, summable and not all zero, and let $C_\Sigma,\gamma_\star>0$.
\begin{enumerate}[label=(\roman*)]
\item (Existence and uniqueness.) For every $(\nu,\tilde\nu)\in\mathcal K$ and
every $s>0$ the system \cref{eqn:app-de-system} at $z=-s$ has exactly one
solution in $[0,\infty)^4$. It satisfies
\begin{equation}\label{eqn:app-de-box}
\sigma,\tilde\sigma\le\frac{\gamma_\star C_\Sigma}{s},\qquad
\tilde s_1,s_1\le\gamma_\star C_\Sigma\,\sw2,\qquad q\le\frac{(\gamma_\star C_\Sigma)^2}{s}.
\end{equation}
\item (Continuity.) Equip $\mathcal K$ with the topology of weak convergence
against $C([0,C_\Sigma])$ in each coordinate. The solution map
$(s,\nu,\tilde\nu)\mapsto(\sigma,\tilde\sigma,\tilde s_1,s_1)$ is continuous on
$(0,\infty)\times\mathcal K$, and for every compact $I\subset(0,\infty)$ it is
uniformly continuous in $s\in I$, uniformly over $(\nu,\tilde\nu)\in\mathcal K$.
\item (Stability under approximate roots.) Let $I\subset(0,\infty)$ be compact
with $s_-\defeq\min I$. For every $\epsilon>0$ there is $\delta>0$, depending
only on $(\epsilon,I,C_\Sigma,\gamma_\star,w)$, with the following property.
If $(\nu,\tilde\nu)\in\mathcal K$, $s\in I$, and
$(\sigma',\tilde\sigma',a',c')\in[0,\gamma_\star C_\Sigma/s_-]^2\times[0,\gamma_\star C_\Sigma\sw2]^2$
satisfy, with $q'=s\sigma'\tilde\sigma'$,
\[
\begin{gathered}
\Big|\sigma'-\int\frac{x}{s+a'x}\,\nu(dx)\Big|\le\delta,\qquad
\Big|\tilde\sigma'-\int\frac{x}{s+c'x}\,\tilde\nu(dx)\Big|\le\delta,\\
\abs{a'-s\tilde\sigma'F_w(q')}\le\delta,\qquad
\abs{c'-s\sigma'F_w(q')}\le\delta,
\end{gathered}
\]
then $\abs{\sigma'-\sigma}+\abs{\tilde\sigma'-\tilde\sigma}+\abs{a'-\tilde s_1}+\abs{c'-s_1}\le\epsilon$,
where $(\sigma,\tilde\sigma,\tilde s_1,s_1)$ is the solution of
\cref{eqn:app-de-system} at $z=-s$ for $(\nu,\tilde\nu)$.
\item (Smoothness in $s$.) Fix $(\nu,\tilde\nu)\in\mathcal K$ with
$\int x\,\nu(dx)>0$ and $\int x\,\tilde\nu(dx)>0$. Then
$s\mapsto(\sigma,\tilde\sigma,\tilde s_1,s_1,q)(-s)$ is $C^\infty$ on $(0,\infty)$,
and $\abs{q'(s)}\le q(s)/s$.
\end{enumerate}
\end{theorem}

\skpointer{thm:app-de-branch}{Lean: proved: \texttt{\detokenize{Muon}\detokenize{.}\allowbreak{}\detokenize{DetEquiv}\detokenize{.}\allowbreak{}\detokenize{thm}\detokenize{_}\allowbreak{}\detokenize{app}\detokenize{_}\allowbreak{}\detokenize{de}\detokenize{_}\allowbreak{}\detokenize{branch}\detokenize{_}\allowbreak{}\detokenize{i}}, \texttt{\detokenize{Muon}\detokenize{.}\allowbreak{}\detokenize{DetEquiv}\detokenize{.}\allowbreak{}\detokenize{thm}\detokenize{_}\allowbreak{}\detokenize{app}\detokenize{_}\allowbreak{}\detokenize{de}\detokenize{_}\allowbreak{}\detokenize{branch}\detokenize{_}\allowbreak{}\detokenize{i}\detokenize{_}\allowbreak{}\detokenize{weight}}, +3 more}
The fixed-point system of \cref{fig:fp-system} has exactly one nonnegative
solution on the negative axis, for every momentum and every spectrum, so the
kernels built from it are well defined. Since $\mathcal K$ contains the
spectral measures at every dimension, continuity uniform over $\mathcal K$
makes the dimension dependence of the kernels harmless, and (iii) identifies
the limit of the random traces in \cref{sk:E1}.

\begin{sketch}
The system reduces exactly to one equation in $q$. If $\nu$ or $\tilde\nu$
has no mass on $(0,C_\Sigma]$ the solution is explicit, so assume both have. Put
$H(x)=\int x\lambda/(s+x\lambda)\,\nu(d\lambda)$, which increases from $0$ to
the mass $p$ of $\nu$ on $(0,C_\Sigma]$, and let $X$ be its inverse; define
$\widetilde H,\widetilde X,\tilde p$ from $\tilde\nu$ in the same way. The
coupling reads $H(\tilde s_1)=\widetilde H(s_1)=r(q)$ with $r(q)=qF_w(q)$,
which is strictly increasing from $0$. Hence
\[
\tilde s_1=X(r(q)),\qquad s_1=\widetilde X(r(q)),\qquad \sigma=\frac{r(q)}{\tilde s_1},\qquad \tilde\sigma=\frac{r(q)}{s_1},
\]
and the equation $q=s\sigma\tilde\sigma$ becomes
(\cref{eqn:app-de-crossing})
\[
\Phi(q)\defeq\frac{X(r(q))}{r(q)}\cdot\frac{\widetilde X(r(q))}{r(q)}-\frac sq=0,\qquad r(q)<\min(p,\tilde p);
\]
conversely every root of $\Phi$ gives a solution. Since $H(x)/x$ decreases,
$X(\rho)/\rho$ increases, so $\Phi$ is strictly increasing. It tends to
$-\infty$ as $q\downarrow0$, where the product stays bounded, and it becomes
positive at the right end, where either one factor blows up or $s/q\to0$. So $\Phi$ has exactly one root. The bounds follow from
$x/(s+ax)\le C_\Sigma/s$ and $F_w\le\sw2$.

Continuity and stability follow from uniqueness by compactness. The class
$\mathcal K$ is weakly compact (\cref{lem:app-de-compact}), so a sequence of
approximate roots that stays away from the exact roots would have a limit
solving the limit system and differing from its unique solution. This is why
(iii) has no rate. The implicit function theorem, applied to the root
equation written in logarithmic form, whose $q$-derivative is at least $1/q$,
gives (iv).
\end{sketch}

With the solution $(\sigma,\tilde\sigma,\tilde s_1,s_1,q)$ of the weighted system
for the EMA weights and $\EE\Gt=\Sigmaout\Dt\Sigmain/\sqrt{2\cR(t)}$, the
candidate update and the kernel function are
(\cref{eqn:app-drift-H,eqn:app-drift-Gbar,eqn:app-drift-kernel})
\[
\begin{gathered}
\Gtildebar_\ridge=\frac2\pi\int_0^\infty T(\ridge+u^2)\,du,\qquad
T(s)=s\,\Lw\big(q(-s)\big)\,\Rbar(-s)\,\EE\Gt\,\Rtbar(-s),\\
\mathfrak d(\mu,\lambda)=\frac{\sqrt2}{\pi}\int_0^\infty
\frac{z\,\Lw(q(-z))\,\mu\lambda}{\big(z+\tilde s_1(-z)\mu\big)\big(z+s_1(-z)\lambda\big)}\,du,
\qquad z=\ridge+u^2 .
\end{gathered}
\]
Here, as in \cref{eqn:dij}, the letter $z$ changes sign: it is the positive
number $\ridge+u^2$, and the solution of the system is evaluated at the
negative spectral parameter $-z$.

\begin{corollary}[The drift kernel]\label{cor:app-drift-kernel}
Under \cref{assum:app-setup-regime}, let $\Gtilde=\Qridge\ridge(\Mt)$ be the
frozen-state update and put $(E_\ridge)_{ij}\defeq u_i^\T\big(\EE\Gtilde-\Gtildebar_\ridge\big)v_j$.
\begin{enumerate}[label=(\roman*)]
\item For every pair $(i,j)$ and every $\ridge\in[\ridge_-,\ridge_+]$,
\begin{equation}\label{eqn:app-drift-projected}
\EE\big[u_i^\T\Gtilde v_j\big]=\frac{\sqrt2\,\mathfrak d(\mu_i,\lambda_j)}{\sqrt{2\cR(t)}}\,(u_i^\T\Dt v_j)+(E_\ridge)_{ij}
=\frac{\mathfrak d(\mu_i,\lambda_j)}{\sqrt{\cR(t)}}\,(u_i^\T\Dt v_j)+(E_\ridge)_{ij},
\end{equation}
and $\sup_{\ridge\in[\ridge_-,\ridge_+]}\big(\sum_{i,j}(E_\ridge)_{ij}^2\big)^{1/2}=\varepsilon_B\to0$.
Consequently the frozen-state drift \cref{eqn:app-setup-DV} of the update
$\Gtilde$ is
\begin{equation}\label{eqn:app-drift-Dij}
\cD_{ij}(t)=\frac{\mathfrak d_{ij}}{\sqrt{\cR(t)}}\,\CMscr Q_{ij}(t)+(u_i^\T\Dt v_j)(E_\ridge)_{ij},
\qquad \mathfrak d_{ij}\defeq\mathfrak d(\mu_i,\lambda_j),
\end{equation}
and for every bounded function $\phi$ on $[0,\infty)^2$,
\begin{equation}\label{eqn:app-drift-tested}
\Big|\sum_{i,j}\phi(\mu_i,\lambda_j)\,(u_i^\T\Dt v_j)\,(E_\ridge)_{ij}\Big|\le\norm\phi_\infty\,\norm{\Dt}_\F\,\varepsilon_B .
\end{equation}
\item In the frozen-state model with $\max(\Nout,\Nin)\le\gamma_\star B$ and
$\norm{\Sigmaout},\norm{\Sigmain}\le C_\Sigma$, at every finite dimension and for
every $\ridge>0$, the function $\mathfrak d$ is
continuous on $[0,\infty)^2$, vanishes when $\mu=0$ or $\lambda=0$, and satisfies
\begin{equation}\label{eqn:app-drift-kernel-bounds}
0\le\mathfrak d(\mu,\lambda)\le\frac{\mu\lambda}{\sqrt{2\ridge}},\qquad
\abs{\partial_\mu\mathfrak d(\mu,\lambda)}\le\frac{\lambda}{\sqrt{2\ridge}},\qquad
\abs{\partial_\lambda\mathfrak d(\mu,\lambda)}\le\frac{\mu}{\sqrt{2\ridge}},
\end{equation}
with one-sided derivatives at $0$; in particular, for $\ridge\in[\ridge_-,\ridge_+]$,
$\mathfrak d$ is Lipschitz on $[0,C_\Sigma]^2$ with constant $C_\Sigma/\sqrt{2\ridge_-}$
in each variable.
\end{enumerate}
\end{corollary}

\skpointer{cor:app-drift-kernel}{Lean: proved: \texttt{\detokenize{Muon}\detokenize{.}\allowbreak{}\detokenize{Drift}\detokenize{.}\allowbreak{}\detokenize{Erho}}, \texttt{\detokenize{Muon}\detokenize{.}\allowbreak{}\detokenize{Drift}\detokenize{.}\allowbreak{}\detokenize{dIdx}}, +5 more}
This corollary identifies the drift kernel of \cref{eqn:fQ-recursion} and
justifies \cref{eqn:dij} at a positive ridge. Under
\cref{assum:app-setup-regime}, that is, in the frozen-state model with a
diffuse error, the expected update projected on the eigenpair $(u_i,v_j)$ is
$\mathfrak d_{ij}\,u_i^\T\Dt v_j/\sqrt{\cR(t)}$ up to an error $(E_\ridge)_{ij}$, so the drift of the mode is
the projected risk times the kernel divided by the square root of the risk.
The error matrix $E_\ridge$ is small
in Frobenius norm over all pairs together, not mode by mode; this is the form
\cref{eqn:app-drift-tested} in which the dynamics uses it. Part (ii), valid at
every dimension, gives the Lipschitz bound used by the deterministic
equations. Nothing is claimed as $\ridge\downarrow0$: the zero-ridge kernel
of the main text is the formal specialization.

\begin{sketch}
The expected update has the deterministic equivalent $\Gtildebar_\ridge$ in
Frobenius norm (\cref{thm:app-drift-mean}, \cref{sk:E2}). Writing the
regularized polar map as the integral
$\frac2\pi\int_0^\infty(\Mt\Mt^\T+\ridge+u^2)^{-1}\Mt\,du$ of mixed resolvents,
removing one sample at a time from the buffer and integrating it out by
Stein's identity gives a linear equation whose solution is $T(s)$ up to $o(1)$
in Frobenius norm (\cref{thm:app-drift-mixed}), and a bound of order $1/s$ at
large $s$, valid at every dimension, controls the tail of the integral.

The candidate acts entrywise in the eigenbases, with the kernel as multiplier.
Since $\Rbar(-s)$ is a function of $\Sigmaout$ and $\Rtbar(-s)$ of $\Sigmain$,
and $u_i^\T\Sigmaout\Dt\Sigmain v_j=\mu_i\lambda_j\,u_i^\T\Dt v_j$,
\[
u_i^\T T(s)v_j=\frac{s\,\Lw(q(-s))\,\mu_i\lambda_j}{\big(s+\tilde s_1(-s)\mu_i\big)\big(s+s_1(-s)\lambda_j\big)}\cdot\frac{u_i^\T\Dt v_j}{\sqrt{2\cR(t)}} .
\]
Integrating at $s=\ridge+u^2$ against $\frac2\pi\,du=\sqrt2\cdot\frac{\sqrt2}{\pi}\,du$
gives exactly $u_i^\T\Gtildebar_\ridge v_j=\mathfrak d_{ij}\,u_i^\T\Dt v_j/\sqrt{\cR(t)}$,
which is \cref{eqn:app-drift-projected}. Because the bases are orthonormal,
$\sum_{i,j}(E_\ridge)_{ij}^2=\norm{\EE\Gtilde-\Gtildebar_\ridge}_\F^2\le\varepsilon_B^2$.
Multiplying by $u_i^\T\Dt v_j$ gives \cref{eqn:app-drift-Dij}, and the
Cauchy--Schwarz inequality over the pairs gives \cref{eqn:app-drift-tested}.

The bounds of (ii) are read off the integrand. Since $\tilde s_1,s_1\ge0$
and $0\le\Lw\le1$, the integrand of \cref{eqn:app-drift-kernel} lies between
$0$ and $\mu\lambda/z$, and its $\mu$-derivative
$z^2\Lw\lambda/((z+\tilde s_1\mu)^2(z+s_1\lambda))$ lies between $0$ and
$\lambda/z$. With
\[
\frac{\sqrt2}{\pi}\int_0^\infty\frac{du}{\ridge+u^2}=\frac{1}{\sqrt{2\ridge}},
\]
this gives \cref{eqn:app-drift-kernel-bounds}; the $\lambda$-derivative is
symmetric.
\end{sketch}

In the next corollary $\CMscr Q_i=\sum_j\CMscr Q_{ij}=\norm{\Dt^\T u_i}^2$ is
the projected row risk, and its drift is the row sum $\sum_j\cD_{ij}(t)$
(\cref{eqn:app-setup-row}). The letter $z=\ridge+u^2$ is positive, as in
\cref{cor:app-drift-kernel}, and $\rho_i$ in (iv) is a measure, not the
ridge $\ridge$.

\begin{corollary}[Half-anisotropic input]\label{cor:app-drift-halfaniso}\label{rem:app-drift-halfaniso}
Under \cref{assum:app-setup-regime}, assume $\Sigmain=\Id_{\Nin}$, let
$\ridge\in[\ridge_-,\ridge_+]$, write $z=\ridge+u^2$, and put
$\mathfrak d_i\defeq\mathfrak d(\mu_i,1)$.
\begin{enumerate}[label=(\roman*)]
\item For every pair $(i,j)$ the kernel is $\mathfrak d_{ij}=\mathfrak d_i$, and
\[
\mathfrak d_i=\frac{\sqrt2\,B}{\pi\,\Nin}\int_0^\infty\frac{z\,\Lw(q(-z))\,\tilde\sigma(-z)\,\mu_i}{z+\tilde s_1(-z)\mu_i}\,du .
\]
\item The row drift of \cref{eqn:app-setup-row} is, with $E_\ridge$ as in
\cref{cor:app-drift-kernel},
\[
\sum_j\cD_{ij}(t)=\frac{\mathfrak d_i}{\sqrt{\cR(t)}}\,\CMscr Q_i(t)+\sum_j(u_i^\T\Dt v_j)(E_\ridge)_{ij}.
\]
\item At $\bmom=0$,
\[
\mathfrak d_i=\frac{\sqrt2\,B}{\pi\,\Nin}\int_0^\infty\frac{\tilde s_1(-z)\mu_i}{z+\tilde s_1(-z)\mu_i}\,du .
\]
\item At $\bmom=0$, if $\rho_i$ is a probability measure on $[0,\infty)$ with
\[
\big(z+\tilde s_1(-z)\mu_i\big)^{-1}=\int\frac{\rho_i(dx)}{x+z}\qquad(z>0),
\]
then
\[
\mathfrak d_i=\frac{B}{\sqrt2\,\Nin}\int_{[0,\infty)}\frac{x}{\sqrt{x+\ridge}}\,\rho_i(dx).
\]
\end{enumerate}
\end{corollary}

\skpointer{cor:app-drift-halfaniso}{Lean: proved: \texttt{\detokenize{Muon}\detokenize{.}\allowbreak{}\detokenize{Drift}\detokenize{.}\allowbreak{}\detokenize{dIdx}}, \texttt{\detokenize{Muon}\detokenize{.}\allowbreak{}\detokenize{Drift}\detokenize{.}\allowbreak{}\detokenize{cor}\detokenize{_}\allowbreak{}\detokenize{app}\detokenize{_}\allowbreak{}\detokenize{drift}\detokenize{_}\allowbreak{}\detokenize{halfaniso}\detokenize{_}\allowbreak{}\detokenize{i}}, +4 more}
This is the drift of the half-anisotropic setting of \cref{sec:half-aniso}.
When $\Sigmain=\Id$ every $\lambda_j$ equals $1$, so the kernel
$\mathfrak d_{ij}$ does not depend on $j$, and the projected row risks
$\CMscr Q_i$ have the drift of
\cref{eqn:ri-recursion} with $\mathfrak d_i$ the common value of the
$\mathfrak d_{ij}$, not their sum. Part (i) writes $\mathfrak d_i$ through the
input trace $\tilde\sigma$, and (iii) simplifies it at $\bmom=0$. Part (iv)
writes it as a moment of a row spectral measure $\rho_i$; setting $\ridge=0$
there formally gives the $\sqrt x$-moment
$\mathfrak d_i=\frac{B}{\sqrt2\,\Nin}\int\sqrt x\,\rho_i(dx)$, in the
normalization of the main-text kernel \cref{eqn:dij}. Neither the existence of
$\rho_i$ nor the limit $\ridge\downarrow0$ is
proved; the phase-comparison chapter adopts the zero-ridge kernels as
assumptions.

\begin{sketch}
With identity input covariance the input trace equation is algebraic. The
second equation of \cref{eqn:app-de-system} becomes
$\tilde\sigma(-s)=(\Nin/B)/(s+s_1(-s))$, so the input factor of the kernel
\cref{eqn:app-drift-kernel} at $\lambda=1$ is
\[
\frac{1}{z+s_1(-z)}=\frac{B\,\tilde\sigma(-z)}{\Nin},
\]
and substituting it gives (i). Part (ii) is the sum over $j$ of
\cref{eqn:app-drift-Dij}, since $\mathfrak d_{ij}=\mathfrak d_i$ and
$\CMscr Q_i=\sum_j\CMscr Q_{ij}$; its error terms, summed over the rows
against a bounded function $\psi(\mu_i)$, are controlled by
\cref{eqn:app-drift-tested} with $\phi(\mu,\lambda)=\psi(\mu)$.

At $\bmom=0$ the source factor equals the coupling. There is one weight, so
$\Lw=\Fw=\Jfun$, and with $q(-z)=z\sigma\tilde\sigma$ and the coupling
$\tilde s_1\sigma=q\Jfun(q)$ of \cref{fig:fp-system},
\[
z\,\Jfun\big(q(-z)\big)\,\tilde\sigma(-z)=\frac{q\Jfun(q)}{\sigma}=\tilde s_1(-z),
\]
which turns (i) into (iii).

Part (iv) integrates a Stieltjes representation in $u$. If
$(z+\tilde s_1(-z)\mu_i)^{-1}=\int\rho_i(dx)/(x+z)$ with $\rho_i$ a probability
measure, then the integrand of (iii) is
$\tilde s_1\mu_i/(z+\tilde s_1\mu_i)=\int x\,\rho_i(dx)/(x+z)$, and the scalar
integral
\[
\frac2\pi\int_0^\infty\frac{du}{x+\ridge+u^2}=\frac{1}{\sqrt{x+\ridge}}
\]
gives (iv) after exchanging the two integrals.
\end{sketch}

\section{Volatility}\label{sk:ch-volatility}
The volatility is the noise that the updates inject into each mode. It is
computed in the frozen-state model of \cref{sk:E3}, where the updates
${\mathsf S}_{t'}(\ridge)=\Qridge\ridge(M_{t'})$ are stationary and ${\mathsf S}^{\mathrm c}_{t'}(\ridge)$ is the update minus its
mean. With momentum, consecutive buffers share most of their batches, so over
the $B$ steps of a unit of macroscopic time the quadratic terms accumulate as a
long-run variance. For symmetric $F,G$ the lag covariances and their integrated
sum are
\[
\begin{gathered}
C_\ell[F,G](\ridge,\ridge')=\frac1B\EE\Tr\big[F\,{\mathsf S}^{\mathrm c}_t(\ridge)\,G\,{\mathsf S}^{\mathrm c}_{t+\ell}(\ridge')^\T\big],\\
\Gamma[F,G](\ridge,\ridge')=C_0[F,G](\ridge,\ridge')+\sum_{\ell\ge1}\big(C_\ell[F,G](\ridge,\ridge')+C_\ell[F,G](\ridge',\ridge)\big),
\end{gathered}
\]
and $C_\ell=C_\ell[\Sigmaout,\Sigmain]$, $\Gamma=\Gamma[\Sigmaout,\Sigmain]$. Their deterministic equivalents
$\overline C_\ell$, $\overline\Gamma$ are given by the lag kernels $\mathfrak k_\ell$ of \cref{sk:E3}. With $P_\mu$, $P'_\lambda$
the spectral projections of $\Sigmaout,\Sigmain$ and $r_\mu,t_\lambda$ their ranks, the forcing measure and the
mode and row coefficients are
\[
\begin{gathered}
\zeta_\ridge=\sum_{\mu,\lambda}\overline\Gamma[P_\mu,P'_\lambda](\ridge,\ridge)\,\delta_{(\mu,\lambda)},\qquad
\mathfrak v_{ij}(\ridge)=\frac{B\,\zeta_\ridge(\{(\mu_i,\lambda_j)\})}{r_{\mu_i}t_{\lambda_j}},\\
\mathfrak v_i(\ridge)=B\,\overline\Gamma[u_i\tensor u_i,\Id](\ridge,\ridge)\quad\text{when }\Sigmain=\Id,
\end{gathered}
\]
and $\mathfrak v_{ij}$ is the volatility kernel of \cref{eqn:fQ-recursion}. The two results of
this chapter read $\mathfrak v_{ij}$ in the two cases the main text uses: the
half-anisotropic rows, and no momentum.

\begin{proposition}[Recovery of the rows]\label{prop:app-vol-half-recovery}
If $\Sigmain=\Id_{\Nin}$, then
$\zeta_\ridge=B^{-1}\sum_i\mathfrak v_i(\ridge)\,\delta_{(\mu_i,1)}$, and
$\mathfrak v_{ij}(\ridge)=\mathfrak v_i(\ridge)/\Nin$, so that $\sum_j\mathfrak v_{ij}(\ridge)=\mathfrak v_i(\ridge)$ for
every $i$.
\end{proposition}

\skpointer{prop:app-vol-half-recovery}{Lean: proved: \texttt{\detokenize{Muon}\detokenize{.}\allowbreak{}\detokenize{Volatility}\detokenize{.}\allowbreak{}\detokenize{prop}\detokenize{_}\allowbreak{}\detokenize{app}\detokenize{_}\allowbreak{}\detokenize{vol}\detokenize{_}\allowbreak{}\detokenize{half}\detokenize{_}\allowbreak{}\detokenize{recovery}}}
The proposition is the relation $\mathfrak v_i=\sum_j\mathfrak v_{ij}$ that the main text
uses in the row recursion \cref{eqn:ri-recursion} of the half-anisotropic
setting $\Sigmain=\Id_{\Nin}$. There the volatility of row $i$ is the long-run
variance of the whole row ${\mathsf S}^\T u_i$ of the update, and its deterministic
counterpart $\mathfrak v_i(\ridge)$ is defined directly from the row test. The proposition says that this row coefficient is
the sum of the mode coefficients of the row, and that all $\Nin$ modes of a row
carry the same share $\mathfrak v_i/\Nin$. It is an exact identity between
deterministic quantities at every finite dimension; no limit is taken.
With the closed form of \cref{sk:E3} it gives the explicit row coefficient
$\mathfrak v_i=(\Nin/B)\big(\mathfrak c_0(\mu_i,1;\ridge)+2\sum_{\ell\ge1}\mathfrak c_\ell(\mu_i,1;\ridge)\big)$
(\cref{eqn:app-vol-special-rows}), whose same-step term is $\volraw_i(\ridge)$ (\cref{cor:app-vol-marginals}),
the deterministic equivalent of the expected row energy of the update (\cref{thm:app-vol-half-same}).

\begin{sketch}
The deterministic covariances are linear in each of the two tests, so the
mass of the forcing measure at a point is the sum of the row coefficients of
the rows sitting there. The deterministic lag covariance $\overline C_\ell[F,G]$ is a
limit of expected traces $B^{-1}\EE\Tr[F{\mathsf S}^{\mathrm c}_tG\,{\mathsf S}^{\mathrm c\,\T}_{t+\ell}]$, and in the
pairing expansion of \cref{sk:E3} each test enters every tree through exactly
one factor. Hence $\overline C_\ell$ and its lag sum $\overline\Gamma$ are linear in $F$ and in
$G$. When $\Sigmain=\Id$ the input side has the single eigenvalue $1$, whose
spectral projection is $\Id$, with multiplicity $\Nin$. The atom of
$\zeta_\ridge$ at $(\mu,1)$ is therefore, splitting the projection $P_\mu$ onto the
eigenvalue $\mu$ of $\Sigmaout$ into its rows,
\[
\zeta_\ridge\big(\{(\mu,1)\}\big)=\overline\Gamma[P_\mu,\Id](\ridge,\ridge)=\sum_{i:\,\mu_i=\mu}\overline\Gamma[u_i\tensor u_i,\Id](\ridge,\ridge)=\frac1B\sum_{i:\,\mu_i=\mu}\mathfrak v_i(\ridge),
\]
where the last step is the definition of $\mathfrak v_i$. This is the first claim, $\zeta_\ridge=B^{-1}\sum_i\mathfrak v_i\,\delta_{(\mu_i,1)}$.

The row coefficient depends on the row only through its eigenvalue, which
turns the atom into equal mode coefficients. With the test $u_i\tensor u_i$, the
row index enters every tree value only through one normalized trace
$B^{-1}\mu_i^p$ with $p\ge1$, so $\mathfrak v_i$ is a function of $\mu_i$ alone
(\cref{thm:app-vol-mode-kernels}(iv)). The atom at $(\mu,1)$ is thus $r_\mu\mathfrak v_i/B$
for any row $i$ with $\mu_i=\mu$, where $r_\mu$ is the multiplicity of $\mu$. The
mode coefficient divides the atom equally among the $r_\mu\Nin$ modes $(i,j)$
that sit at $(\mu,1)$ (\cref{prop:app-vol-disintegration}), so
\[
\mathfrak v_{ij}(\ridge)=\frac{B}{r_\mu\Nin}\cdot\frac{r_\mu\,\mathfrak v_i(\ridge)}B=\frac{\mathfrak v_i(\ridge)}{\Nin},\qquad
\sum_{j=1}^{\Nin}\mathfrak v_{ij}(\ridge)=\mathfrak v_i(\ridge).
\]
\end{sketch}

Here $r_\bmom(q)=q\Fw(q)$, and $b_\ell,e_\ell$, $\mathfrak k_\ell=\mathfrak k_\ell[\Sigmaout,\Sigmain]$ and the mode integrals
$\mathfrak c_\ell(\mu,\lambda;\ridge)$ are the closed-form objects of \cref{sk:E3}; $\overline{\mathfrak C}_\ell(\ridge,\ridge)$ is the
measure with atoms $\overline C_\ell[P_\mu,P'_\lambda](\ridge,\ridge)$. The same-step row and column energies are
\[
\volraw_i(\ridge)=\frac{\tilde s_1(-\ridge)\,\mu_i}{\ridge+\tilde s_1(-\ridge)\,\mu_i},\qquad
\widetilde{\mathfrak v}^{(0)}_j(\ridge)=\frac{s_1(-\ridge)\,\lambda_j}{\ridge+s_1(-\ridge)\,\lambda_j},
\]
and $\overline Z_{\mathfrak k}(\ridge)$ is $\frac4{\pi^2}$ times the double integral of the lag-zero kernel
$\mathfrak k_0(\ridge+x^2,\ridge+y^2)$ over $x,y>0$ (\cref{eqn:app-vol-Zk}, where the kernel is written in the variables of one resolvent pair and identified with $\mathfrak k_0$ by \cref{thm:app-vol-double-same}(i)).

\begin{remark}[No momentum]\label{rem:app-vol-nomomentum}
At $\bmom=0$ the only weight is $\emaw0=1$, the buffer is the current
gradient, $F_0=L_0=\Jfun$, $r_0'(q)=\EE\,g^2(1+qg^2)^{-2}$ and
$F_0'(q)=-\EE\,g^4(1+qg^2)^{-2}$, and the branch is the no-momentum system of
\cref{rem:app-de-nomomentum}. In the closed form, $b_\ell$ carries the factor
$\bmom^\ell$, so $b_\ell=0$ and $\mathfrak k_\ell=0$ for every $\ell\ge1$; the
same-step scalars are the divided differences
$b_0=\big(r_0(q)-r_0(q')\big)/(q-q')$ and $e_0=-\big(F_0(q)-F_0(q')\big)/(q-q')$. All
lag covariances and their deterministic equivalents vanish for $\ell\ge1$
(\cref{thm:app-vol-lag-reduction}(ii), \cref{thm:app-vol-integrated}(ii)), so
$\Gamma=C_0$, $\overline\Gamma=\overline C_0$, the forcing measure is
$\zeta_\ridge=\overline{\mathfrak C}_0(\ridge,\ridge)$, and
\[
\mathfrak v_{ij}(\ridge)=\frac1B\,\mathfrak c_0(\mu_i,\lambda_j;\ridge),\qquad
\sum_{j=1}^{\Nin}\mathfrak v_{ij}(\ridge)=\volraw_i(\ridge),\qquad
\sum_{i=1}^{\Nout}\mathfrak v_{ij}(\ridge)=\widetilde{\mathfrak v}^{(0)}_j(\ridge),
\]
by \cref{cor:app-vol-closed-modes,cor:app-vol-marginals}, for every pair of
covariances. For $\Sigmain=\Id$ this is $\mathfrak v_i(\ridge)=\volraw_i(\ridge)$ of
\cref{thm:app-vol-mode-kernels}(iv), with the spectral form of
\cref{rem:app-vol-half-spectral}. The covariance-weighted same-step total is
$\overline C_0(\ridge,\ridge)=\overline Z_{\mathfrak k}(\ridge)$ of \cref{thm:app-vol-double-same},
with the kernel \cref{eqn:app-vol-kernel} formed from $F_0'$ and $r_0'$.
\end{remark}

\skpointer{rem:app-vol-nomomentum}{Lean: not formalized (remark)}
Without momentum the buffer is the current gradient, so updates at different
steps are built from disjoint batches and are independent; every lag
covariance vanishes and the volatility is the same-step variance of the
update. In the closed form of \cref{sk:E3} this is the factor $\bmom^\ell$ carried by
$b_\ell$, and the same-step scalars become divided differences of the
no-momentum functions $r_0$ and $F_0$. Summed over a row, the same-step term is
$\volraw_i(\ridge)$ (\cref{cor:app-vol-marginals}), the deterministic equivalent of the expected
row energy $\EE\norm{{\mathsf S}^\T u_i}^2$: the exact identity
${\mathsf S}{\mathsf S}^\T=\Id-\ridge(\Mt\Mt^\T+\ridge\Id)^{-1}$ reduces it to the deterministic equivalent of
the Gram resolvent (\cref{thm:app-vol-half-same}). This is the
coefficient behind the main text's statement that without momentum, at
$B\ge N$, the update is orthogonal and $\mathfrak v_i=1$. The remark itself is at a
positive ridge. As $\ridge\downarrow0$ the integrand of the row energy tends to the
indicator of $(0,\infty)$, whose integral against the spectral measure $\rho_i$ of
$u_i$ for $\Mt\Mt^\T$ is $\rho_i((0,\infty))=1-\norm{P_0u_i}^2$, with $P_0$ the projection on the
left null space of $\Mt$ (\cref{rem:app-vol-half-spectral}); this is $1$ when the gradient has
full row rank. The appendix makes no statement at $\ridge=0$.

\section{The isotropic case}\label{sk:ch-isotropic}
In this chapter the data are isotropic, $\Sigmain=\Id_{\Nin}$ and $\Sigmaout=\Id_{\Nout}$, as in
\cref{sec:isotropic}; $\gin=\Nin/B$ and $\gout=\Nout/B$ are the two aspect ratios, $\gamma_\wedge$ and
$\gamma_\vee$ their minimum and maximum, and $\gamma=N/B$ in the square case. Every eigenvalue is one, so
every mode has the same drift kernel $\mathfrak d_{\mathrm{iso}}(\ridge)$ and the same share
$\mathfrak v_{\mathrm{iso}}(\ridge)/(\Nout\Nin)$ of the volatility. The fixed-point system of
\cref{fig:fp-system} at $z=-s$ reduces to one scalar equation for its root $q=q(-s)$
(\cref{prop:app-iso-closure}):
\[
\begin{gathered}
s\,q=\big(\gout-r_\bmom(q)\big)\big(\gin-r_\bmom(q)\big),\\
r_\bmom(q)\defeq q\Fw(q)=\sum_{j\ge0}\Phi(\emaw j^2q),\qquad
\Phi(v)\defeq\EE\frac{vg^2}{1+vg^2}\quad(g\sim\mathcal N(0,1)).
\end{gathered}
\]
The value $\psi(s)\defeq r_\bmom(q(-s))<\gamma_\wedge$ is the participation level: $B\psi(\ridge)$ is the
deterministic equivalent of the squared Frobenius norm of the update $\Qridge\ridge(M_t)$, which for an
exact polar factor would be its rank (\cref{thm:app-iso-volatility}(i)). The drift is the integral
\cref{eqn:app-iso-drift} of $q\Lw(q)$ along the root, written $\mathfrak d_{\mathrm{iso}}(\ridge;\bmom,\gin,\gout)$ when the
arguments matter; its value at zero ridge is
\[
C_\bmom(\gin,\gout)\defeq\frac{\sqrt2}{\pi\,\gin\gout}\int_0^\infty q(-u^2)\,\Lw\big(q(-u^2)\big)\,\dif u,\qquad
C_\bmom(\gamma)\defeq C_\bmom(\gamma,\gamma),\qquad C\defeq C_0,
\]
and $\sw k\defeq\sum_{j\ge0}\emaw j^k$, which is $(1-\bmom)^k/(1-\bmom^k)$ for $\bmom>0$, so that $B/\sw2=B_{\rm eff}$. The chapter closes the
dynamics on the risk, computes the two coefficients of \cref{eqn:comparison-general-form} at zero ridge at large
and small batch, and ends with a small ridge at small batch.

\begin{proposition}[Zero-ridge isotropic volatility without momentum]\label{prop:app-iso-volatility-zero}\label{rem:app-iso-volatility-zero}
Let $\bmom=0$ and $m_\gamma\defeq\min\{1,\gin,\gout\}$. Then
$\psi(\ridge)\to m_\gamma$ as $\ridge\downarrow0$, and
\[
\lim_{\ridge\downarrow0}\mathfrak v_{\mathrm{iso}}(\ridge)=B\,m_\gamma=\min\{B,\Nin,\Nout\}.
\]
\end{proposition}

\skpointer{prop:app-iso-volatility-zero}{Lean: proved: \texttt{\detokenize{Muon}\detokenize{.}\allowbreak{}\detokenize{Isotropic}\detokenize{.}\allowbreak{}\detokenize{prop}\detokenize{_}\allowbreak{}\detokenize{app}\detokenize{_}\allowbreak{}\detokenize{iso}\detokenize{_}\allowbreak{}\detokenize{volatility}\detokenize{_}\allowbreak{}\detokenize{zero}}; foundation results: \texttt{\detokenize{GaussianOpNorm}}}
Without momentum the zero-ridge volatility is $\min\{B,\Nin,\Nout\}$, the rank of one minibatch gradient.
This is the value $\mathfrak v_0=\min\{B,N\}$ of \cref{fig:isotropic-main}, and it gives the learning-rate ratio
$N/\sqrt B$ at $\bmom=0$ and $B\le N$ in \cref{eqn:lr-ratio}. For an exact polar factor the squared Frobenius
norm is the rank, $\norm{\Qridge0(\Gt)}_\F^2=\rank\Gt$ (\cref{lem:app-setup-filter}(i)), so the value is the
expected one; the proposition shows that the deterministic coefficients of the regularized algorithm have this
limit as $\ridge\downarrow0$. No statement is made about the algorithm at zero ridge.

\begin{sketch}
Without momentum the volatility is the same-time value $B\psi(\ridge)$, exactly. The updates of two different
steps are built from independent batches, so the lag covariances of the volatility vanish
(\cref{thm:app-vol-integrated}(ii)), and the integrated coefficient is the expected squared Frobenius norm of
one update. Its deterministic equivalent is read off the anchors of the scalar equation
(\cref{eqn:app-iso-vol-nomomentum}):
\[
\mathfrak v_{\mathrm{iso}}(\ridge)=B\,\psi(\ridge)=B\,\Phi\big(q(-\ridge)\big)=B\big(\gout-\ridge\,\sigma(-\ridge)\big).
\]
At $\bmom=0$ there is a single weight $\emaw0=1$, so the participation level is one Gaussian average,
$\psi=\Phi(q)$.

The participation level has two ceilings. Both factors on the right of the scalar equation are positive,
so $\psi<\min(\gin,\gout)$ (\cref{prop:app-iso-closure}(ii)). And $\Phi<1$, because the integrand
$vg^2/(1+vg^2)$ is below one (\cref{lem:app-iso-gaussian}(ii)): one batch contributes at most one unit of
participation. Hence $\psi(\ridge)<m_\gamma=\min\{1,\gin,\gout\}$.

As the ridge vanishes, the level reaches the lower of the two ceilings. Suppose that $\psi(\ridge_k)\le
m_\gamma-\delta$ along a sequence $\ridge_k\downarrow0$. Then both factors of the scalar equation stay at least
$\delta$, and the root is forced to infinity:
\[
q(-\ridge_k)=\frac{\big(\gout-\psi(\ridge_k)\big)\big(\gin-\psi(\ridge_k)\big)}{\ridge_k}\ge\frac{\delta^2}{\ridge_k}\longrightarrow\infty .
\]
But $\Phi(v)\to1$ as $v\to\infty$, so $\psi(\ridge_k)=\Phi(q(-\ridge_k))\to1\ge m_\gamma$, a contradiction. Hence
$\psi(\ridge)\to m_\gamma$, and $\mathfrak v_{\mathrm{iso}}(\ridge)\to B\,m_\gamma=\min\{B,\Nin,\Nout\}$ because
$B\gin=\Nin$ and $B\gout=\Nout$. With momentum the ceiling one disappears, since $r_\bmom$ is unbounded, and
the same-time total tends to $\min\{\Nin,\Nout\}$ instead (\cref{prop:app-iso-vol-momentum}(i)).
\end{sketch}

\begin{definition}[Zero-ridge isotropic volatility]\label{def:app-iso-vol-zero}
Let $\bmom\in[0,1)$ and let $\mathfrak v_{\mathrm{iso}}(\ridge)$ be the number of
\cref{thm:app-iso-volatility}(ii) at momentum $\bmom$. The zero-ridge isotropic
volatility at momentum $\bmom$ is
\[
\mathfrak v_{\mathrm{iso},\bmom}\defeq\lim_{\ridge\downarrow0}\mathfrak v_{\mathrm{iso}}(\ridge),
\]
the limit being taken at this fixed $\bmom$, whenever it exists.
\end{definition}

\skpointer{def:app-iso-vol-zero}{Lean: proved: \texttt{\detokenize{Muon}\detokenize{.}\allowbreak{}\detokenize{Isotropic}\detokenize{.}\allowbreak{}\detokenize{def}\detokenize{_}\allowbreak{}\detokenize{app}\detokenize{_}\allowbreak{}\detokenize{iso}\detokenize{_}\allowbreak{}\detokenize{vol}\detokenize{_}\allowbreak{}\detokenize{zero}}, \texttt{\detokenize{Muon}\detokenize{.}\allowbreak{}\detokenize{Isotropic}\detokenize{.}\allowbreak{}\detokenize{isoVol}}, +1 more; foundation results: \texttt{\detokenize{GaussianOpNorm}}}
The definition names the volatility $\mathfrak v_\bmom$ of the main text's scalar recursion
\cref{eqn:comparison-general-form} at zero ridge, written $\mathfrak v_{\mathrm{iso},\bmom}$ here. It is the limit
of the deterministic coefficient of \cref{thm:app-iso-volatility}(ii) as $\ridge\downarrow0$ at a fixed
momentum, it depends on $B$, $\Nin$ and $\Nout$, and it says nothing about the algorithm at zero ridge. At
$\bmom=0$ the limit is $\min\{B,\Nin,\Nout\}$ (\cref{prop:app-iso-volatility-zero}). For $\bmom>0$ the
supplement proves that the limit exists in the square case, where it equals $N(1+2\sum_{\ell\ge1}\overline
C_{1,\ell}(0))$ with $\overline C_{1,\ell}(0)$ the correlation of two updates $\ell$ steps apart
(\cref{prop:app-iso-vol-momentum}(ii)).

\begin{proposition}[Scalar reduction]\label{prop:app-iso-scalar}
Let $\Sigmain=\Id_{\Nin}$ and $\Sigmaout=\Id_{\Nout}$, let $\bmom\in[0,1)$,
$\ridge>0$ and $\macrostep>0$, and write $\mathfrak d_{\mathrm{iso}}$ and
$\mathfrak v_{\mathrm{iso}}=\mathfrak v_{\mathrm{iso}}(\ridge)$ for the
numbers of \cref{thm:app-iso-drift,thm:app-iso-volatility}.
\begin{enumerate}[label=(\roman*)]
\item The deterministic mode-energy equations \cref{eqn:app-dyn-ode} close
on the risk: if $(\fQ_{ij})$ solves them on an interval on which $\fRisk>0$,
then $\fRisk=\frac12\sum_{i,j}\fQ_{ij}$ solves
\begin{equation}\label{eqn:app-iso-scalar}
\frac{\dif\fRisk}{\dif\macrotime}=-2\macrostep\,\mathfrak d_{\mathrm{iso}}\sqrt{\fRisk}+\frac{\macrostep^2}{2B}\,\mathfrak v_{\mathrm{iso}},\qquad
\fRisk(0)=\cR(0),
\end{equation}
and \cref{eqn:app-iso-scalar} has at most one solution with $\fRisk>0$ on
any interval $[0,T]$.
\item Under \cref{assum:app-setup-regime} and for $\ridge\in[\ridge_-,\ridge_+]$, one step of
size $\lr\ge0$ from a frozen state, $\Delta^+=\Dt-\lr\,\Qridge\ridge(\Mt)$, satisfies
\begin{equation}\label{eqn:app-iso-onestep}
\begin{gathered}
\EE\,\cR(\Delta^+)-\cR(t)=-2\lr\,\mathfrak d_{\mathrm{iso}}\sqrt{\cR(t)}+\frac{\lr^2B}{2}\,\psi(\ridge)+r_t,\\
\abs{r_t}\le\lr\sqrt{2\cR(t)}\,\varepsilon_B+\frac{\lr^2B}{2}\,\gamma_\star\ridge_+e_B ,
\end{gathered}
\end{equation}
with $\varepsilon_B$ of \cref{thm:app-drift-mean} and $e_B$ of \cref{thm:app-vol-half-same}.
\end{enumerate}
\end{proposition}

\skpointer{prop:app-iso-scalar}{Lean: proved: \texttt{\detokenize{Muon}\detokenize{.}\allowbreak{}\detokenize{Isotropic}\detokenize{.}\allowbreak{}\detokenize{IsIsoScalarOde}}, \texttt{\detokenize{Muon}\detokenize{.}\allowbreak{}\detokenize{Isotropic}\detokenize{.}\allowbreak{}\detokenize{prop}\detokenize{_}\allowbreak{}\detokenize{app}\detokenize{_}\allowbreak{}\detokenize{iso}\detokenize{_}\allowbreak{}\detokenize{scalar}\detokenize{_}\allowbreak{}\detokenize{i}}, +2 more; foundation results: \texttt{\detokenize{GaussianOpNorm}}}
With isotropic data the deterministic mode-energy equations close on the risk. Part (i) is the
continuous-time form of the main text's recursion \cref{eqn:comparison-general-form}: with $\lr=\macrostep/B$
and $t=B\macrotime$ it reads $\dif\fRisk/\dif t=-2\lr\,\mathfrak d_{\mathrm{iso}}\sqrt\fRisk+\frac{\lr^2}2\mathfrak
v_{\mathrm{iso}}$, so the drift and volatility of the main text are $\mathfrak d=\mathfrak d_{\mathrm{iso}}$ and
$\mathfrak v=\mathfrak v_{\mathrm{iso}}$. Part (ii) is the balance for one step from a fixed state. Its quadratic term is the same-time value
$B\psi(\ridge)$, which equals $\mathfrak v_{\mathrm{iso}}$ at $\bmom=0$ and differs from it by the lag terms for $\bmom>0$;
over many steps the lag terms are added (\cref{cor:app-iso-dynamics}).

\begin{sketch}
Every mode has the same drift kernel, and the expected update is a multiple of the error itself. At
$\mu_i=\lambda_j=1$ the denominator of the drift kernel \cref{eqn:dij} is
$(z+\tilde s_1)(z+s_1)=z\gin\gout/q$ by the anchors of the scalar equation, so the kernel loses its dependence
on the mode and becomes an integral along the root (\cref{thm:app-iso-drift}(i),(iii)):
\[
\mathfrak d_{\mathrm{iso}}=\frac{\sqrt2}{\pi\,\gin\gout}\int_0^\infty q(-z)\,\Lw\big(q(-z)\big)\,\dif u\quad(z=\ridge+u^2),\qquad
\EE\,\Qridge\ridge(\Mt)=\frac{\mathfrak d_{\mathrm{iso}}}{\sqrt{\cR(t)}}\,\Dt+E_\ridge ,
\]
with $\norm{E_\ridge}_\F\le\varepsilon_B\to0$. The volatility is shared equally, $\mathfrak v_{ij}=\mathfrak
v_{\mathrm{iso}}/(\Nout\Nin)$ (\cref{thm:app-iso-volatility}(ii)). Since $\mu_i\lambda_j=1$, the risk is
$\fRisk=\frac12\sum_{i,j}\fQ_{ij}$, and summing the mode-energy equations \cref{eqn:app-dyn-ode} over the pairs gives
\[
\frac{\dif}{\dif\macrotime}\big(2B^2\fRisk\big)=-\frac{2\macrostep\,\mathfrak d_{\mathrm{iso}}}{\sqrt\fRisk}\,2B^2\fRisk+\macrostep^2B\,\mathfrak v_{\mathrm{iso}} ,
\]
which is \cref{eqn:app-iso-scalar} after division by $2B^2$. The right side is continuously differentiable in
$\fRisk>0$, so the Picard--Lindel\"of theorem gives uniqueness. This is (i).

For one step from a fixed state the change of the risk is exactly a linear term, the pairing of the error with
the mean update, plus half the second moment of the update (\cref{prop:app-drift-onestep}). By the display above, the linear term is
$-\lr\langle\Dt,\EE\Qridge\ridge(\Mt)\rangle=-2\lr\,\mathfrak d_{\mathrm{iso}}\sqrt{\cR(t)}$ up to
$\lr\sqrt{2\cR(t)}\,\varepsilon_B$, by $\norm{\Dt}_\F^2=2\cR(t)$ and the Cauchy--Schwarz inequality. For the
quadratic term, the regularized polar map of a matrix $K$ satisfies, exactly
(\cref{eqn:app-vol-half-identity}),
\[
\norm{\Qridge\ridge(K)}_\F^2=\Tr\,KK^\T\big(KK^\T+\ridge\Id\big)^{-1}=\Nout-\ridge\,\Tr\big(KK^\T+\ridge\Id\big)^{-1} .
\]
For the normalized buffer, the deterministic equivalent of the resolvent trace divided by $B$ is
$\sigma(-\ridge)$, so the second moment is $B(\gout-\ridge\,\sigma(-\ridge))=B\psi(\ridge)$ up to $B\gamma_\star\ridge_+e_B$
(\cref{thm:app-vol-half-same}, \cref{thm:app-iso-volatility}(i)). This is (ii).
\end{sketch}

\begin{corollary}[Isotropic dynamics]\label{cor:app-iso-dynamics}
Let $\Sigmain=\Id_{\Nin}$ and $\Sigmaout=\Id_{\Nout}$, fix
$\macrostep>0$, $\ridge>0$, $\bmom\in[0,1)$, $\gamma_\star<\infty$, $E_\star<\infty$,
$c_M<\infty$ and $0<r_0\le r_1<\infty$, and along a sequence of dimensions
with $B\to\infty$ assume:
\begin{enumerate}[label=(\alph*)]
\item \cref{assum:proportional-regime} holds and $\max(\Nin,\Nout)\le\gamma_\star B$;
\item the algorithm of \cref{def:app-setup-algorithm} is run with momentum
$\bmom$, normalized ridge $\ridge$ and step $\lr_t=\macrostep/B$;
\item $W_0$ is deterministic, $\norm{\Delta_0}_\F^2\le E_\star$,
$r_0\le\cR(0)\le r_1$ and $\norm{\Delta_0}/\norm{\Delta_0}_\F\to0$;
\item the initial buffer $M^\circ_{-1}$ is independent of all batches and
$\EE\norm{M^\circ_{-1}}_\F^2\le c_MB$.
\end{enumerate}
Let $T_\star$ be given by \cref{eqn:app-dyn-Tstar} at $C_\Sigma=1$. Then
\cref{eqn:app-iso-scalar} has a unique solution $\fRisk$ on $[0,T_\star]$, and
\begin{enumerate}[label=(\roman*)]
\item with probability tending to one, $t_{\rm stop}>BT_\star$ and
$\inf_{t\le BT_\star}\cR(t)\ge r_0/2$;
\item $\sup_{\macrotime\le T_\star}\abs{\cR(\lfloor B\macrotime\rfloor)-\fRisk(\macrotime)}\to0$ in probability;
\item on the event of (i), $\sup_{t\le BT_\star}\norm{\Dt}/\norm{\Dt}_\F\to0$ in probability.
\end{enumerate}
\end{corollary}

\skpointer{cor:app-iso-dynamics}{Lean: proved: \texttt{\detokenize{Muon}\detokenize{.}\allowbreak{}\detokenize{Isotropic}\detokenize{.}\allowbreak{}\detokenize{cor}\detokenize{_}\allowbreak{}\detokenize{app}\detokenize{_}\allowbreak{}\detokenize{iso}\detokenize{_}\allowbreak{}\detokenize{dynamics}\detokenize{_}\allowbreak{}\detokenize{scalar}}, \texttt{\detokenize{Muon}\detokenize{.}\allowbreak{}\detokenize{Isotropic}\detokenize{.}\allowbreak{}\detokenize{cor}\detokenize{_}\allowbreak{}\detokenize{app}\detokenize{_}\allowbreak{}\detokenize{iso}\detokenize{_}\allowbreak{}\detokenize{dynamics}\detokenize{_}\allowbreak{}\detokenize{i}}, +3 more; foundation results: \texttt{\detokenize{GaussianOpNorm}}, \texttt{\detokenize{GaussianOpNormL6}}, \texttt{\detokenize{GaussianPoincare}}, \texttt{\detokenize{GaussianQuadFormLp}}}
The corollary is the precise form of the main text's claim that, with isotropic data, the risk of \Muon follows
the scalar recursion \cref{eqn:comparison-general-form}. At the step $\lr=\macrostep/B$ and up to the macroscopic
time $T_\star$, the risk at step $\lfloor B\macrotime\rfloor$ is uniformly close to the solution of
\cref{eqn:app-iso-scalar}, whose coefficients are the integrated volatility $\mathfrak v_{\mathrm{iso}}(\ridge)$ and the drift
$\mathfrak d_{\mathrm{iso}}(\ridge)$ at the fixed ridge $\ridge>0$. Hypothesis (c) asks that the initial error be diffuse:
its operator norm is a vanishing fraction of its Frobenius norm, so no single singular direction carries a fixed
fraction of the risk. Part (iii) says that this persists. The time $T_\star>0$ does not depend on the dimension. The corollary does not
reach the zero-ridge coefficients $C_\bmom$ and $\mathfrak v_{\mathrm{iso},\bmom}$, which enter only as limits of the
coefficients.

\begin{sketch}
The corollary is \cref{thm:app-dyn-main} at identity covariances. The covariance bound of
\ref{hyp:app-dyn-data} holds with $C_\Sigma=1$, the weighted pairing is the Frobenius pairing, and the
normalized error is $\Dt/\norm{\Dt}_\F$, so the diffuseness of \ref{hyp:app-dyn-initial} is the ratio
$\norm{\Delta_0}/\norm{\Delta_0}_\F$ of hypothesis (c) (\cref{eqn:app-iso-frobenius}). The mode-energy equations have a
unique solution on $[0,T_\star]$ with risk at least $r_0/2$ (\cref{prop:app-dyn-ode}), and by
\cref{prop:app-iso-scalar}(i) its risk solves \cref{eqn:app-iso-scalar}, which has only one positive solution. The
spectral energy measure is the single atom $\norm{\Delta}_\F^2\,\delta_{(1,1)}=2\cR(\Delta)\,\delta_{(1,1)}$, so the
measure statement of the theorem is the risk statement (ii).

In the isotropic case the block argument behind the theorem (\cref{sk:E4}) is a computation for the risk alone.
The algorithm is compared, over blocks of $L=\lfloor B^{1/8}\rfloor$ steps, with the model in which the state is
held at its value $D$ at the start of the block and the momentum buffer is a stationary average of independent
batches. Test the exact block expansion \cref{eqn:app-dyn-block-identity} of \cref{lem:app-dyn-block} against the
constant function. The test is then $\norm\Delta_\F^2=2\cR(\Delta)$. Its drift term is the pairing of $D$ with the
mean update, which is $\mathfrak d_{\mathrm{iso}}D/\sqrt{\cR(D)}$ up to a vanishing error (\cref{thm:app-iso-drift}(iii)).
Its volatility term is the total mass $\mathfrak v_{\mathrm{iso}}/B$ of the forcing measure (\cref{eqn:app-iso-vol-modes}).
Halving gives
\[
\EE_t\big[\cR(\Delta_{t+L})\big]-\cR(D)=\frac LB\Big(-2\macrostep\,\mathfrak d_{\mathrm{iso}}\sqrt{\cR(D)}+\frac{\macrostep^2}{2B}\,\mathfrak v_{\mathrm{iso}}\Big)+\frac LB\,\epsilon_t,
\]
where the mean of $\abs{\epsilon_t}$ tends to zero, uniformly over the blocks that start at a diffuse state whose
risk and energy stay within the bounds of the hypotheses (proof of \cref{lem:app-dyn-block}). The main term is
$L/B$ times the right side of \cref{eqn:app-iso-scalar}. Its volatility is the integrated coefficient, not the
same-time value $B\psi(\ridge)$ of \cref{eqn:app-iso-onestep}. A block is long compared with the memory
$1/(1-\bmom)$ of the buffer, so the covariance of the sum of its $L$ updates contains the correlations of all
lags (\cref{thm:app-vol-joint-forcing}(iv)). The first steps of a block, during which the buffer still remembers the
previous block, cost a total of order $\log^2(2B)/L$, and the fluctuations of the block increments form a
martingale whose maximum vanishes.

The error stays diffuse, which is what keeps the mean update close to $\mathfrak d_{\mathrm{iso}}D/\sqrt{\cR(D)}$
along the trajectory. The operator norm of the mean update is at most a constant times the operator norm of the
normalized error, plus a vanishing term, at every state, diffuse or not (\cref{lem:app-dyn-mean}(iv)), so the
argument is not circular. The centered increments of
the error over a block have operator norm $O(L/B)$, and their partial sums vanish in operator norm by a matrix
martingale inequality (\cref{lem:app-dyn-matrix-martingale}). A discrete Gronwall inequality then bounds
$\max_{t\le BT_\star}\norm{\Dt}$ by a constant times $\norm{\Delta_0}$ plus a vanishing term, while the Frobenius
norm stays of order one because the risk stays above $r_0/2$; this is (iii). Gronwall's inequality for
\cref{eqn:app-iso-scalar} and routine estimates give (i) and (ii).
\end{sketch}

\begin{proposition}[The zero-ridge constant]\label{prop:app-iso-zero}
Let $\bmom\in[0,1)$ and $\gin,\gout>0$. Put $K_0\defeq1$, and for $\bmom>0$
let $K_\bmom$ be the unique $q>0$ with $q\Fw(q)=\min(\gin,\gout)$.
\begin{enumerate}[label=(\roman*)]
\item For every $t>0$,
\[
0<q(-t)\,\Lw\big(q(-t)\big)\le\min\Big(K_\bmom,\frac{\gin\gout}{t}\Big),
\]
so the integral \cref{eqn:app-iso-C} converges and $0<C_\bmom(\gin,\gout)<\infty$.
Moreover $C_\bmom(\gin,\gout)=C_\bmom(\gout,\gin)$, and
\begin{equation}\label{eqn:app-iso-C-t}
C_\bmom(\gin,\gout)=\frac{1}{\sqrt2\,\pi\,\gin\gout}\int_0^\infty t^{-1/2}\,q(-t)\,\Lw\big(q(-t)\big)\,\dif t .
\end{equation}
\item Let $\mathfrak d_{\mathrm{iso}}(\ridge)$ be the kernel \cref{eqn:app-iso-drift}
at momentum $\bmom$, for the measures \cref{eqn:app-iso-measures}. Then, at this
fixed $\bmom$, $\mathfrak d_{\mathrm{iso}}(\ridge)\to C_\bmom(\gin,\gout)$ as $\ridge\downarrow0$.
\end{enumerate}
\end{proposition}

\skpointer{prop:app-iso-zero}{Lean: proved: \texttt{\detokenize{Muon}\detokenize{.}\allowbreak{}\detokenize{Isotropic}\detokenize{.}\allowbreak{}\detokenize{prop}\detokenize{_}\allowbreak{}\detokenize{app}\detokenize{_}\allowbreak{}\detokenize{iso}\detokenize{_}\allowbreak{}\detokenize{zero}\detokenize{_}\allowbreak{}\detokenize{i}}, \texttt{\detokenize{Muon}\detokenize{.}\allowbreak{}\detokenize{Isotropic}\detokenize{.}\allowbreak{}\detokenize{prop}\detokenize{_}\allowbreak{}\detokenize{app}\detokenize{_}\allowbreak{}\detokenize{iso}\detokenize{_}\allowbreak{}\detokenize{zero}\detokenize{_}\allowbreak{}\detokenize{ii}}}
The proposition makes sense of the zero-ridge drift $C_\bmom(\gin,\gout)$ of \cref{eqn:app-iso-C}, which is the drift
$\mathfrak d_\bmom$ of the main text's recursion \cref{eqn:comparison-general-form} and the constant of
\cref{eqn:C-integral}: the integral converges, the constant is positive and symmetric in the two aspect ratios,
and it is the limit of the drift kernel $\mathfrak d_{\mathrm{iso}}(\ridge)$ as $\ridge\downarrow0$ at a fixed momentum.
Like every zero-ridge statement of the chapter, (ii) concerns the deterministic kernel and not the algorithm at
zero ridge. The form \cref{eqn:app-iso-C-t} in the spectral parameter $t=u^2$ is the one used for the batch
limits. At $\bmom=0$ the integrand is the participation level, $q\Jfun(q)=\psi$, so $C_0$ is an average over the
spectral scale of the effective rank of the update divided by $B$.

\begin{sketch}
The integrand is bounded near the origin and decays like $1/t$, which gives convergence. Since $\Jfun\le1$ and
the weights sum to one, $\Lw\le1$ and $q\Lw(q)\le q$. Two bounds on the root follow from the scalar equation.
The trace equations give $\sigma(-t)\le\gout/t$ and $\tilde\sigma(-t)\le\gin/t$, so
$q(-t)=t\,\sigma(-t)\tilde\sigma(-t)\le\gin\gout/t$. And the participation level stays below
$\min(\gin,\gout)$, so for $\bmom>0$ the root stays below the root $K_\bmom$ at which $r_\bmom$ reaches that level,
because $r_\bmom$ is increasing; at $\bmom=0$ the integrand is $\Phi(q)<1$ directly. Together,
\[
0<q(-t)\,\Lw\big(q(-t)\big)\le\min\Big(K_\bmom,\frac{\gin\gout}t\Big),
\]
so the integrand of \cref{eqn:app-iso-C} at $t=u^2$ is at most $\min(K_\bmom,\gin\gout/u^2)$, which is
integrable. The scalar equation is symmetric in $(\gin,\gout)$, so the root and the constant are too, and the
substitution $t=u^2$ gives \cref{eqn:app-iso-C-t}.

The ridge enters only as a shift of the spectral parameter, so the limit is a limit under the integral. By
\cref{thm:app-iso-drift}(i) the kernel at ridge $\ridge$ is the same integral with the root evaluated at
$-(\ridge+u^2)$. The root is continuous in the spectral parameter (\cref{thm:app-de-branch}(iv)) and $\Lw$ is
continuous, so the integrand converges pointwise as $\ridge\downarrow0$. The bound above at $t=\ridge+u^2$
dominates it by the same function $\min(K_\bmom,\gin\gout/u^2)$, and dominated convergence gives (ii). The limit is taken at a fixed $\bmom$; at zero ridge the limits $\bmom\downarrow0$ and
$\ridge\downarrow0$ do not commute for $\gamma>1$ (\cref{prop:app-iso-beta-zero}).
\end{sketch}

For a finite measure $\nu$ on $[0,\infty)$ write $S_\nu(-t)\defeq\int\nu(\dif x)/(x+t)$ for $t>0$. At $\bmom=0$ the root
solutions $\sigma(-t)$, $\tilde\sigma(-t)$ of the fixed-point system of \cref{fig:fp-system} are the transforms of a unique
pair of measures $\mathfrak m_{\mathrm{out}},\mathfrak m_{\mathrm{in}}$ on $[0,\infty)$ of masses $\gout,\gin$, the isotropic measure
pair. For $r>0$, $\mathsf M_r$ is the Marchenko--Pastur law with an atom
$(1-1/r)_+$ at zero and density $\sqrt{(b_r-x)(x-a_r)}/(2\pi x)$ on $(a_r,b_r)$, where
$a_r=(1-\sqrt r)^2/r$ and $b_r=(1+\sqrt r)^2/r$.

\begin{theorem}[Small aspect ratios]\label{thm:app-iso-small}
Let $\gin^{(j)},\gout^{(j)}\to0$ with $r_j\defeq\gout^{(j)}/\gin^{(j)}\to r\in(0,\infty)$,
and let $(\mathfrak m_{\mathrm{out}}^{(j)},\mathfrak m_{\mathrm{in}}^{(j)})$ be the isotropic
measure pair at $(\gin^{(j)},\gout^{(j)})$ of \cref{prop:app-iso-measure}.
Let $\mathsf P_j$ be the image of the probability measure
$\mathfrak m_{\mathrm{out}}^{(j)}/\gout^{(j)}$ under $\lambda\mapsto\lambda/\gout^{(j)}$,
and write $\mathsf m_j(t)\defeq S_{\mathsf P_j}(-t)$.
\begin{enumerate}[label=(\roman*)]
\item For every $t>0$,
\begin{equation}\label{eqn:app-iso-small-residual}
\Big|t\,\mathsf m_j(t)^2+\Big(t+\frac1{r_j}-1\Big)\mathsf m_j(t)-1\Big|\le\frac{3\gout^{(j)}}{r_j^2t^2}.
\end{equation}
\item $\mathsf P_j\to\mathsf M_r$ weakly, and $\int\sqrt x\,\mathsf P_j(\dif x)\to J_r\defeq\int\sqrt x\,\mathsf M_r(\dif x)>0$.
\item The constant satisfies
\begin{equation}\label{eqn:app-iso-small}
\begin{gathered}
C\big(\gin^{(j)},\gout^{(j)}\big)=\frac{r_j}{\sqrt{2\gout^{(j)}}}\int\sqrt x\,\mathsf P_j(\dif x)
\sim\frac{r}{\sqrt{2\gout^{(j)}}}\,J_r,\\
J_r=\int_{a_r}^{b_r}\frac{\sqrt{(b_r-x)(x-a_r)}}{2\pi\sqrt x}\,\dif x .
\end{gathered}
\end{equation}
In particular $J_1=8/(3\pi)$, and $C(\gamma)\sim8/(3\pi\sqrt{2\gamma})$ as $\gamma\downarrow0$.
\end{enumerate}
\end{theorem}

\skpointer{thm:app-iso-small}{Lean: proved: \texttt{\detokenize{Muon}\detokenize{.}\allowbreak{}\detokenize{Isotropic}\detokenize{.}\allowbreak{}\detokenize{thm}\detokenize{_}\allowbreak{}\detokenize{app}\detokenize{_}\allowbreak{}\detokenize{iso}\detokenize{_}\allowbreak{}\detokenize{small}\detokenize{_}\allowbreak{}\detokenize{i}}, \texttt{\detokenize{Muon}\detokenize{.}\allowbreak{}\detokenize{Isotropic}\detokenize{.}\allowbreak{}\detokenize{thm}\detokenize{_}\allowbreak{}\detokenize{app}\detokenize{_}\allowbreak{}\detokenize{iso}\detokenize{_}\allowbreak{}\detokenize{small}\detokenize{_}\allowbreak{}\detokenize{ii}}, +1 more}
The theorem gives the zero-ridge drift without momentum at large batch, $N\ll B$. In the square case
$C(\gamma)\sim\frac{8}{3\pi}(2\gamma)^{-1/2}$, that is $\mathfrak d_0\approx\frac8{3\pi}\sqrt{B/(2N)}$, the
large-batch limit of the approximate form of $\mathfrak d_0$ in \cref{eqn:C-integral}. The constant $J_r$ is the mean
square root of the Marchenko--Pastur law $\mathsf M_r$ of ratio $r=\gout/\gin$, with its atom at zero when
$r>1$, so the drift constant is a mean singular value. Part (i) is a
quantitative form of the convergence that (ii) states.

\begin{sketch}
Without momentum the drift constant is a half-moment of a spectral measure. The isotropic measure pair exists
and is unique (\cref{prop:app-iso-measure}): it is obtained as a limit of the expected spectral measures of the
Gram matrices of one minibatch gradient, and the transforms identify the limit. The integrand of the drift is then
$\psi(t)=\int\frac{\lambda}{\lambda+t}\,\mathfrak m_{\mathrm{out}}(\dif\lambda)$, and the elementary integral
$\int_0^\infty t^{-1/2}\frac{x}{x+t}\,\dif t=\pi\sqrt x$ turns the form \cref{eqn:app-iso-C-t} into a mean of
square roots, that is of singular values (\cref{prop:app-iso-spectral}):
\[
C(\gin,\gout)=\frac1{\sqrt2\,\gin\gout}\int\sqrt\lambda\;\mathfrak m_{\mathrm{out}}(\dif\lambda)
=\frac{r_j}{\sqrt{2\gout}}\int\sqrt x\;\mathsf P_j(\dif x),
\]
the second form after the rescaling $\lambda=\gout x$ that defines $\mathsf P_j$.

At small aspect ratio the closure becomes the Marchenko--Pastur quadratic. The transform of $\mathsf P_j$ is
$\mathsf m_j(t)=\sigma(-\gout t)$, and both transforms are bounded by their mass over the spectral parameter,
so the argument $v_j$ of $\Phi$ in the closure is at most $\gout/(r_jt)$, which is small. Without momentum the
closure reads $\gout(1-t\,\mathsf m_j)=\Phi(v_j)$, and $\Phi$ is linear up to a quadratic error, $0\le
v-\Phi(v)\le3v^2$ (\cref{lem:app-iso-gaussian}(iii)). Inserting $v_j$ and dividing by $\gout$ leaves the quadratic
equation of \cref{eqn:app-iso-small-residual} with the residual $3\gout/(r_j^2t^2)$. This is (i).

The rescaled laws converge to $\mathsf M_r$, and so do their half-moments. For each $t$, every limit point of
$\mathsf m_j(t)$ is a positive root of $t\,m^2+(t+\frac1r-1)m-1=0$, and that root is the transform of $\mathsf M_r$
(\cref{lem:app-iso-mp}). The first moments are exactly $1/r_j$, because the first moment of $\mathfrak m_{\mathrm{out}}$
is $\gin\gout$ (\cref{lem:app-iso-moments}(ii)). Convergence of the transforms on the negative axis then gives
weak convergence, and the bounded first moments give convergence of the half-moments
(\cref{lem:app-iso-transforms}(i)). This is (ii), and (iii) follows from the display. At $r=1$,
$J_1=\frac1{2\pi}\int_0^4\sqrt{4-x}\,\dif x=\frac8{3\pi}$.
\end{sketch}

The linearized closure replaces the participation count $r_\bmom(q)$ by $q$. Its root is
$Y_{\gin,\gout}(s)=2\gin\gout/\big(\gin+\gout+s+\sqrt{(\gin+\gout+s)^2-4\gin\gout}\big)$, the root below $\gamma_\wedge$ of
$sY=(\gout-Y)(\gin-Y)$, and its drift is
\[
\mathfrak d_{\mathrm{lin}}(\gin,\gout;\ridge)\defeq\frac{\sqrt2}{\pi\,\gin\gout}\int_0^\infty Y_{\gin,\gout}(\ridge+v^2)\,\dif v\qquad(\ridge\ge0),
\]
the drift integral \cref{eqn:app-iso-drift} with $q\Lw(q)$ replaced by the linear root. At $\ridge=0$,
$\mathfrak d_{\mathrm{iso}}(0;\bmom,\gin,\gout)\defeq C_\bmom(\gin,\gout)$.

\begin{theorem}[Small aspect ratios with momentum]\label{thm:app-iso-small-momentum}
Let $\bmom\in[0,1)$ and $\gin,\gout>0$, and assume
$\varepsilon_\bmom\defeq6\gamma_\wedge(1-\bmom^2)\le\frac12$.
\begin{enumerate}[label=(\roman*)]
\item For every $\ridge\ge0$ there is $\vartheta\in\big[1-\varepsilon_\bmom,(1-\varepsilon_\bmom)^{-1/2}\big]$ with
\[
\mathfrak d_{\mathrm{iso}}(\ridge;\bmom,\gin,\gout)=\vartheta\;\mathfrak d_{\mathrm{lin}}(\sw2\gin,\sw2\gout;\ridge).
\]
\item With $r\defeq\gout/\gin$ and $J_r$ of \cref{thm:app-iso-small},
\[
\mathfrak d_{\mathrm{lin}}(\gin,\gout;0)=\frac{r\,J_r}{\sqrt{2\gout}},\qquad
\mathfrak d_{\mathrm{lin}}(\gamma,\gamma;0)=\frac{8}{3\pi\sqrt{2\gamma}} .
\]
In particular, in the square case $C_\bmom(\gamma)=8\vartheta/(3\pi\sqrt{2\sw2\gamma})$
with $\vartheta$ of (i) at $\ridge=0$.
\item If $6\gamma_\wedge\le\frac12$, then for every $\bmom\in[0,1)$
\[
\big(1-6\gamma_\wedge\big)^{3/2}\le\frac{C_\bmom(\gin,\gout)}{\sw2^{-1/2}\,C_0(\gin,\gout)}\le\big(1-6\gamma_\wedge\big)^{-3/2}.
\]
\item For every $\ridge\ge0$,
\[
(1-\varepsilon_\bmom)\big(1-6\sw2\gamma_\wedge\big)^{1/2}\le
\frac{\mathfrak d_{\mathrm{iso}}(\ridge;\bmom,\gin,\gout)}{\mathfrak d_{\mathrm{iso}}(\ridge;0,\sw2\gin,\sw2\gout)}\le
(1-\varepsilon_\bmom)^{-1/2}\big(1-6\sw2\gamma_\wedge\big)^{-1}.
\]
\end{enumerate}
\end{theorem}

\skpointer{thm:app-iso-small-momentum}{Lean: proved: \texttt{\detokenize{Muon}\detokenize{.}\allowbreak{}\detokenize{Isotropic}\detokenize{.}\allowbreak{}\detokenize{thm}\detokenize{_}\allowbreak{}\detokenize{app}\detokenize{_}\allowbreak{}\detokenize{iso}\detokenize{_}\allowbreak{}\detokenize{small}\detokenize{_}\allowbreak{}\detokenize{momentum}\detokenize{_}\allowbreak{}\detokenize{i}}, \texttt{\detokenize{Muon}\detokenize{.}\allowbreak{}\detokenize{Isotropic}\detokenize{.}\allowbreak{}\detokenize{thm}\detokenize{_}\allowbreak{}\detokenize{app}\detokenize{_}\allowbreak{}\detokenize{iso}\detokenize{_}\allowbreak{}\detokenize{small}\detokenize{_}\allowbreak{}\detokenize{momentum}\detokenize{_}\allowbreak{}\detokenize{ii}}, +2 more}
At large batch, momentum acts on the drift as the effective batch $B_{\rm eff}=B/\sw2=B(1+\bmom)/(1-\bmom)$ of the
main text: up to the factor $\vartheta$, the drift is that of the linearized closure at the aspect ratios
$(\sw2\gin,\sw2\gout)=(\Nin/B_{\rm eff},\Nout/B_{\rm eff})$. In the square case
$C_\bmom(\gamma)\approx\frac8{3\pi}\sqrt{B_{\rm eff}/(2N)}$, the large-batch form of $\mathfrak d_\bmom$ in
\cref{eqn:C-integral}. The small parameter
is $\gamma_\wedge(1-\bmom^2)$, of the order of $N/B_{\rm eff}$, so the theorem also covers a fixed aspect ratio as
$\bmom\uparrow1$.

\begin{sketch}
At small aspect ratio the root stays where the participation count is linear, with slope
$\sw2=\sum_j\emaw j^2$. Expanding each term $\Phi(\emaw j^2q)$ to second order gives, for all $q\ge0$
(\cref{eqn:app-iso-small-expansions}),
\[
\sw2q\Big(1-3\frac{\sw4}{\sw2}\,q\Big)\le r_\bmom(q)\le\sw2q,\qquad
1-3\sw3q\le\Lw(q)\le1 .
\]
The level $r_\bmom(q)$ stays below $\gamma_\wedge$, and $r_\bmom(2\gamma_\wedge/\sw2)\ge\gamma_\wedge$ once
$\varepsilon_\bmom\le\frac12$, so the root stays below $\bar q=2\gamma_\wedge/\sw2$. At such roots the two
expansions are accurate to relative errors $6\gamma_\wedge\sw4/\sw2^2$ and $6\gamma_\wedge\sw3/\sw2$, both at most
$\varepsilon_\bmom$, because $\sw4/\sw2^2=(1-\bmom^2)/(1+\bmom^2)$.

A larger participation count produces a smaller root, so the root is squeezed between two linear roots. The
right side $(\gout-r)(\gin-r)$ of the scalar equation decreases in the count $r$, and comparing the two sides
shows that raising the count lowers the root (\cref{lem:app-iso-comparison}). For the linear count $r(q)=cq$ the
root is $c^{-1}Y(s/c)$, where $Y(s)$ is the explicit root of $sY=(\gout-Y)(\gin-Y)$. The two expansions give
\[
\frac1{\sw2}\,Y\Big(\frac s{\sw2}\Big)\le q(-s)\le\frac1c\,Y\Big(\frac sc\Big),\qquad c=\sw2\Big(1-6\gamma_\wedge\frac{\sw4}{\sw2^2}\Big).
\]
Inserting this sandwich and the bounds on $\Lw$ in the drift integral \cref{eqn:app-iso-drift}, and rescaling
$u=\sqrt{\sw2}\,v$, gives (i): the factor $\sw2$ in the count is absorbed by multiplying both aspect ratios by
$\sw2$, because $Y_{t\gin,t\gout}(ts)=t\,Y_{\gin,\gout}(s)$ (\cref{eqn:app-iso-dlin-scale}).

The linear drift at zero ridge is a half-moment of the Marchenko--Pastur law, as in \cref{thm:app-iso-small}. The
rescaled linear root $y(x)=Y(\gout x)/\gout$ solves $xy=(1-y)(\frac1r-y)$, which is the equation satisfied by
$1-x\,m_r(x)$ for the transform $m_r$ of $\mathsf M_r$, and the half-moment formula gives
$\mathfrak d_{\mathrm{lin}}(\gin,\gout;0)=rJ_r/\sqrt{2\gout}$. In the square case $y(v^2)=e^{-2\theta}$ at
$v=2\sinh\theta$, so $\int_0^\infty y(v^2)\,\dif v=\int_0^\infty(e^{-\theta}+e^{-3\theta})\,\dif\theta=\frac43$, which
gives $8/(3\pi\sqrt{2\gamma})$. This is (ii). Parts (iii) and (iv) divide (i) at momentum $\bmom$ by (i) at momentum
zero, at the aspect ratios $(\gin,\gout)$ and $(\sw2\gin,\sw2\gout)$ respectively.
\end{sketch}

For $n\ge0$, $0\le z\le1$ and $\bmom\in(0,1)$ put
\[
k_n\defeq\frac1{2n+1}+\frac1{2n+3},\qquad
\Omega(z)\defeq\frac4{\pi^2}\sum_{n\ge0}k_n^2\,z^{2n+1},\qquad
V_\bmom\defeq1+2\sum_{\ell\ge1}\Omega(\bmom^\ell).
\]

\begin{lemma}[The volatility factor]\label{lem:app-iso-V}
$\Omega(1)=1$, and for $\bmom\in(0,1)$
\[
V_\bmom=\frac4{\pi^2}\sum_{n\ge0}k_n^2\,\frac{1+\bmom^{2n+1}}{1-\bmom^{2n+1}},\qquad
1<V_\bmom<\frac1{\sw2}=\frac{1+\bmom}{1-\bmom},
\]
\[
\lim_{\bmom\uparrow1}\sw2V_\bmom=\frac14+\frac1{\pi^2}+\frac{7\zeta(3)}{2\pi^2}=0.77759\ldots .
\]
\end{lemma}

\skpointer{lem:app-iso-V}{Lean: proved: \texttt{\detokenize{Muon}\detokenize{.}\allowbreak{}\detokenize{Isotropic}\detokenize{.}\allowbreak{}\detokenize{lem}\detokenize{_}\allowbreak{}\detokenize{app}\detokenize{_}\allowbreak{}\detokenize{iso}\detokenize{_}\allowbreak{}\detokenize{V}\detokenize{_}\allowbreak{}\detokenize{i}}, \texttt{\detokenize{Muon}\detokenize{.}\allowbreak{}\detokenize{Isotropic}\detokenize{.}\allowbreak{}\detokenize{lem}\detokenize{_}\allowbreak{}\detokenize{app}\detokenize{_}\allowbreak{}\detokenize{iso}\detokenize{_}\allowbreak{}\detokenize{V}\detokenize{_}\allowbreak{}\detokenize{ii}}, +2 more}
$V_\bmom$ is the factor by which momentum multiplies the zero-ridge volatility at large batch:
$\mathfrak v_{\mathrm{iso},\bmom}\approx NV_\bmom$ (\cref{thm:app-iso-small-volatility}), where $\Omega(\bmom^\ell)$ is, at large batch,
the correlation of two updates $\ell$ steps apart. The lemma places $V_\bmom$ strictly between $1$, the value without
correlations, and $1/\sw2=B_{\rm eff}/B$, the value if successive updates were correlated like successive buffers,
by $\bmom^\ell$. The main text writes $c_\bmom=\sw2V_\bmom$, so that $\mathfrak v_\bmom\approx c_\bmom NB_{\rm eff}/B$ in
\cref{eqn:C-integral}; the lemma gives $c_\bmom<1$ and the limit $0.7776\ldots$ as $\bmom\uparrow1$.

\begin{sketch}
The coefficients of $\Omega$ form a probability distribution on the odd powers. With $x=2n+1$, the square
$k_n^2=x^{-2}+2/(x(x+2))+(x+2)^{-2}$ sums over odd $x$ to $\frac{\pi^2}8+1+(\frac{\pi^2}8-1)=\frac{\pi^2}4$, by
$\sum_{x\ \mathrm{odd}}x^{-2}=\pi^2/8$ and a telescoping sum, so $\Omega(1)=1$. Summing the geometric series in
the lag, $1+2\sum_{\ell\ge1}\bmom^{\ell x}=(1+\bmom^x)/(1-\bmom^x)$, shows that $V_\bmom$ is the average of
$(1+\bmom^x)/(1-\bmom^x)$ over odd $x$ under these weights. Each value lies in $(1,(1+\bmom)/(1-\bmom)]$, with
equality on the right only at $x=1$, which gives the two bounds.

The limit is an average of $1/x$. As $\bmom\uparrow1$, $\sw2(1+\bmom^x)/(1-\bmom^x)\to1/x$, with the bound $2$, so
dominated convergence gives $\lim\sw2V_\bmom=\frac4{\pi^2}\sum_nk_n^2/(2n+1)$. Partial fractions and
$\sum_{x\ \mathrm{odd}}x^{-3}=\frac78\zeta(3)$ evaluate the sum as $\frac78\zeta(3)+\frac{\pi^2}{16}+\frac14$, and
multiplying by $4/\pi^2$ gives $\frac14+\frac1{\pi^2}+\frac{7\zeta(3)}{2\pi^2}$.
\end{sketch}

In the square case $\Nin=\Nout=N$ every row has the same lag coefficients $\overline C_{1,\ell}(\ridge)$: the
deterministic covariance, in units of one row, of two updates $\ell$ steps apart (\cref{eqn:app-pl-Cil} at
$\mu_1=\dots=\mu_N=1$). Their limits $\overline C_{1,\ell}(0)$ as $\ridge\downarrow0$ exist
(\cref{prop:app-iso-vol-momentum}(ii)).

\begin{theorem}[Small-aspect integrated volatility]\label{thm:app-iso-small-volatility}
Let $\bmom\in(0,1)$, $\Nin=\Nout=N$ and $\gamma=N/B\le(1+\bmom^2)/96$. Then for every $\ell\ge1$
\[
\big|\overline C_{1,\ell}(0)-\Omega(\bmom^\ell)\big|\le\frac{96\gamma}{1+\bmom^2}\,\Omega(\bmom^\ell),
\]
and the zero-ridge volatility of \cref{def:app-iso-vol-zero} satisfies
\[
\Big|\frac{\mathfrak v_{\mathrm{iso},\bmom}}N-V_\bmom\Big|\le\frac{96\gamma}{1+\bmom^2}\,\big(V_\bmom-1\big),\qquad
\Big|\frac{\mathfrak v_{\mathrm{iso},\bmom}}{N\,V_\bmom}-1\Big|\le96\gamma .
\]
\end{theorem}

\skpointer{thm:app-iso-small-volatility}{Lean: proved: \texttt{\detokenize{Muon}\detokenize{.}\allowbreak{}\detokenize{Isotropic}\detokenize{.}\allowbreak{}\detokenize{thm}\detokenize{_}\allowbreak{}\detokenize{app}\detokenize{_}\allowbreak{}\detokenize{iso}\detokenize{_}\allowbreak{}\detokenize{small}\detokenize{_}\allowbreak{}\detokenize{volatility}\detokenize{_}\allowbreak{}\detokenize{i}}, \texttt{\detokenize{Muon}\detokenize{.}\allowbreak{}\detokenize{Isotropic}\detokenize{.}\allowbreak{}\detokenize{thm}\detokenize{_}\allowbreak{}\detokenize{app}\detokenize{_}\allowbreak{}\detokenize{iso}\detokenize{_}\allowbreak{}\detokenize{small}\detokenize{_}\allowbreak{}\detokenize{volatility}\detokenize{_}\allowbreak{}\detokenize{ii}}; foundation results: \texttt{\detokenize{GaussianOpNorm}}}
At large batch, in the square case, the zero-ridge volatility with momentum is $NV_\bmom$ up to a relative error
$96\gamma$, uniformly in $\bmom$ when $\gamma\le1/96$. This is the large-batch form
$\mathfrak v_\bmom\approx c_\bmom NB_{\rm eff}/B$ of \cref{eqn:C-integral}, with $c_\bmom=\sw2V_\bmom$ (\cref{lem:app-iso-V}). The
correlation of two updates $\ell$ steps apart is $\Omega(\bmom^\ell)$, not the correlation $\bmom^\ell$ of the buffers.
Since $V_\bmom<1/\sw2$, momentum amplifies the volatility by less than the effective batch
$B_{\rm eff}=B/\sw2$ by which it acts on the drift (\cref{thm:app-iso-small-momentum}); \cref{rem:app-iso-momentum}
draws the consequence.

\begin{sketch}
The volatility is a sum over lags of explicit double integrals. By \cref{prop:app-iso-vol-momentum}(ii),
$\mathfrak v_{\mathrm{iso},\bmom}/N=1+2\sum_{\ell\ge1}\overline C_{1,\ell}(0)$, where the term $1$ is the same-time value at
zero ridge and $\overline C_{1,\ell}(0)$ is the correlation of two updates $\ell$ steps apart. In the square case this
correlation is a double integral over two spectral parameters of a kernel built from the two roots
$q=q(-x^2)$ and $q'=q(-y^2)$ (\cref{lem:app-iso-lag}):
\[
\overline C_{1,\ell}(0)=\frac4{\pi^2}\int_0^\infty\!\!\int_0^\infty\frac{b_\ell\,q\,q'}{\gamma^3\,\Delta_\ell}\,\dif x\,\dif y,\qquad
\Delta_\ell=\Big(1-\frac{qq'e_\ell}\gamma\Big)^2-\frac{qq'b_\ell^2}{\gamma^2}.
\]
Here $b_\ell=\bmom^\ell\,(r_\bmom(q)-r_\bmom(\bmom^{2\ell}q'))/(q-\bmom^{2\ell}q')$ is $\bmom^\ell$ times a divided difference of the
participation count, close to $\bmom^\ell\sw2$ when both roots are small, and $e_\ell=O(\bmom^{2\ell}\sw4)$.

At small aspect ratio the integrand is a fixed function of the linear root, up to a relative error of order
$\gamma$. Rescale both spectral parameters by $\sqrt{\sw2\gamma}$. By the sandwich of
\cref{thm:app-iso-small-momentum}, the rescaled roots $\sw2q/\gamma$ lie within a factor $(1-\delta)^{-1}$ of
$y(\xi^2)$, where $y$ is the root of $xy=(1-y)^2$ and $\delta=12\gamma(1-\bmom^2)/(1+\bmom^2)$; the normalized
$b_\ell/(\bmom^\ell\sw2)$ lies in $[1-\delta,1]$, and $qq'e_\ell/\gamma\le\delta$. With $r=\bmom^\ell$ and
$\mathsf Y_0=y(\xi^2)y(\xi'^2)\le1$, the integrand is increasing in each of these three quantities, so it is
squeezed around $\mathfrak f\defeq r\mathsf Y_0/(1-r^2\mathsf Y_0)$:
\[
\mathfrak f\,\Big(1-\frac{3\delta}{1-\bmom^2}\Big)\le\text{integrand}\le\mathfrak f\,\Big(1+\frac{8\delta}{1-\bmom^2}\Big),\qquad
\frac{8\delta}{1-\bmom^2}=\frac{96\gamma}{1+\bmom^2}.
\]
The perturbations are measured against the gap $1-r^2\mathsf Y_0\ge1-\bmom^2$ in the denominator, which is why
$\delta$ is divided by $1-\bmom^2$. The hypothesis $\gamma\le(1+\bmom^2)/96$ is exactly $4\delta\le(1-\bmom^2)/2$, which
keeps the perturbed denominator above half the gap.

The limit kernel is a geometric series of products, and each product integrates to $k_n$. Expanding
$\mathfrak f=\sum_{n\ge0}r^{2n+1}\mathsf Y_0^{n+1}$, the double integral factorizes into squares of one-dimensional
integrals. With $\xi=2\sinh\theta$ the linear root is $y(\xi^2)=e^{-2\theta}$, so
\[
\int_0^\infty y(\xi^2)^{n+1}\,\dif\xi=\int_0^\infty e^{-2(n+1)\theta}\,2\cosh\theta\,\dif\theta=\frac1{2n+1}+\frac1{2n+3}=k_n ,
\]
and $\frac4{\pi^2}\iint\mathfrak f=\Omega(\bmom^\ell)$. This is the lag bound. Summing it over $\ell$ against
$V_\bmom=1+2\sum_{\ell\ge1}\Omega(\bmom^\ell)$ gives the first volatility bound, and dividing by $V_\bmom\ge1$ gives the
second.
\end{sketch}

The threshold aspect ratio is $\gamma^\circ_\bmom\defeq2+1/(1-\bmom^2)$, and $\mathsf K^\pm_\bmom$ are explicit
constants that depend only on $\bmom$.

\begin{theorem}[Zero ridge, large aspect ratios]\label{thm:app-iso-zero-ridge}
Let $\bmom\in(0,1)$ and $\gin,\gout>0$, put $\iota\defeq\gamma_\wedge/\gamma_\vee\in(0,1]$ and
\[
\mathsf F(\iota)\defeq\frac1\iota\int_0^\iota\sqrt{(1-x)(\iota-x)}\,\dif x .
\]
\begin{enumerate}[label=(\roman*)]
\item The zero-ridge constant satisfies
\[
\Big|C_\bmom(\gin,\gout)-\frac{\mathsf F(\iota)}{\sqrt\pi}\Big|\le\frac{2.32}{(1-\bmom)\sqrt{\gin\gout}},
\]
where $\mathsf F(1)=\frac12$ and, for $\iota<1$,
\[
\mathsf F(\iota)=\frac1\iota\Big[\frac{(1+\iota)\sqrt\iota}4+\frac{(1-\iota)^2}8\log\frac{1-\sqrt\iota}{1+\sqrt\iota}\Big].
\]
In particular, in the square case,
\[
\Big|C_\bmom(\gamma)-\frac1{2\sqrt\pi}\Big|\le\frac{2.32}{(1-\bmom)\gamma}\qquad(\gamma>0).
\]
\item Let $\Nin=\Nout=N$ and $\gamma=N/B\ge\gamma^\circ_\bmom$. The zero-ridge volatility of
\cref{def:app-iso-vol-zero} satisfies, with $\mathsf K^\pm_\bmom$ of \cref{eqn:app-iso-Kpm},
\[
\frac{1-2\mathsf K^-_\bmom}\gamma\le\frac{\mathfrak v_{\mathrm{iso},\bmom}}{N\gamma}-\frac12\le\frac{1+2\mathsf K^+_\bmom}\gamma .
\]
In particular $\mathfrak v_{\mathrm{iso},\bmom}=\frac{N^2}{2B}\big(1+O_\bmom(B/N)\big)$, and
$\mathfrak v_{\mathrm{iso},\bmom}/(N\gamma)\to\frac12$ as $\gamma\to\infty$ at every fixed $\bmom\in(0,1)$.
\end{enumerate}
\end{theorem}

\skpointer{thm:app-iso-zero-ridge}{Lean: proved: \texttt{\detokenize{Muon}\detokenize{.}\allowbreak{}\detokenize{Isotropic}\detokenize{.}\allowbreak{}\detokenize{thm}\detokenize{_}\allowbreak{}\detokenize{app}\detokenize{_}\allowbreak{}\detokenize{iso}\detokenize{_}\allowbreak{}\detokenize{zero}\detokenize{_}\allowbreak{}\detokenize{ridge}\detokenize{_}\allowbreak{}\detokenize{i}}, \texttt{\detokenize{Muon}\detokenize{.}\allowbreak{}\detokenize{Isotropic}\detokenize{.}\allowbreak{}\detokenize{thm}\detokenize{_}\allowbreak{}\detokenize{app}\detokenize{_}\allowbreak{}\detokenize{iso}\detokenize{_}\allowbreak{}\detokenize{zero}\detokenize{_}\allowbreak{}\detokenize{ridge}\detokenize{_}\allowbreak{}\detokenize{ii}}, +1 more; foundation results: \texttt{\detokenize{GaussianOpNorm}}}
The theorem gives the isotropic kernels of \Muon with momentum at large aspect ratio $\gamma=N/B$.
The drift does not decay: in the square case
$C_\bmom(\gamma)\to1/(2\sqrt\pi)$, with no logarithm and no dependence on $\bmom$, whereas without
momentum $C_0(\gamma)\sim1/(\sqrt\pi\gamma)$ (\cref{thm:app-iso-large}). The volatility grows linearly,
$\mathfrak v_{\mathrm{iso},\bmom}=\frac{N^2}{2B}(1+O_\bmom(B/N))$. The two kernels give the limit $1/(2\sqrt\pi)$
of \cref{fig:isotropic-main} and, inserted in \cref{eqn:lr-ratio}, the learning-rate ratio
$\sqrt\pi N^2\mathfrak d_\bmom/(\sqrt B\,\mathfrak v_\bmom)=\sqrt B\,(1+O_\bmom(B/N))$. Both kernels are
deterministic limits as $\ridge\downarrow0$ at fixed $\bmom$; nothing is claimed for the algorithm at zero ridge.

\begin{sketch}
With momentum, the update approaches full rank as $\ridge\downarrow0$. In the scalar equation of the
introduction to \cref{sk:ch-isotropic}, each term $\Phi(\emaw j^2q)$ of $r_\bmom(q)$ lies in $(0,1)$ and is close
to one when $\emaw j^2q\gg1$. Hence $r_\bmom(q)$ counts the past batches in the buffer, of age $j$ and
weight $\emaw j$, that the normalized update whitens, and the scalar equation keeps this count below
$\gamma_\wedge$. With momentum the count grows without bound in $q$: for $(1-\bmom)^2q\ge1$
(\cref{lem:app-iso-count}(i)),
\[
r_\bmom(q)=\frac{\log\big((1-\bmom)^2q\big)}{2\log(1/\bmom)}+O\Big(\frac1{1-\bmom}\Big).
\]
As $\ridge\downarrow0$ the root grows, so the participation level $\psi$ tends to its ceiling $\gamma_\wedge$. We
write $\psi(0)=\gamma_\wedge$, and the squared Frobenius norm $B\psi$ of the update tends to the full rank
$\min\{\Nin,\Nout\}$ (\cref{prop:app-iso-vol-momentum}(i)). Without momentum there is one age, $r_0=\Phi<1$,
and $\psi(\ridge)\to\min\{1,\gamma_\wedge\}$; then $B\psi$ tends to $\min\{B,\Nin,\Nout\}$, the rank of one gradient
(\cref{prop:app-iso-volatility-zero}).

The drift is, up to an error of order $1/((1-\bmom)\sqrt{\gin\gout})$, an explicit function of the level. The
drift \cref{eqn:app-iso-drift} is an integral of $q\Lw(q)$ along the root. Let $Q_m$ be the root at level $m$.
The scalar equation inverts exactly as (\cref{lem:app-iso-participation})
\[
s=\frac{\mathsf W(m)^2}{Q_m},\qquad \mathsf W(m)\defeq\sqrt{(\gin-m)(\gout-m)},
\]
so the integral can be taken over the level $m\in(0,\psi(\ridge))$ instead of the spectral parameter.
After one integration by parts the integrand contains $\Lambda_\bmom(Q_m)$, where
$\Lambda_\bmom(q)\defeq\int_0^qx^{-1/2}\Lw(x)\,\dif x$. Every whitened age contributes the same amount to
$\Lambda_\bmom$, and the error comes from the few ages near the threshold $\emaw j^2q\approx1$
(\cref{lem:app-iso-weights}(iv), \cref{lem:app-iso-constants}):
\[
\int_0^\infty y^{-1/2}\Jfun(y)\,\dif y=\sqrt{2\pi},\qquad
\Lambda_\bmom(Q_m)=\sqrt{2\pi}\,m+O\Big(\frac1{1-\bmom}\Big).
\]
This gives the master estimate (\cref{thm:app-iso-master}),
\[
\Big|\mathfrak d_{\mathrm{iso}}(\ridge)-\frac1{\sqrt\pi\,\gin\gout}\int_0^{\psi(\ridge)}\mathsf W(m)\,\dif m\Big|
<\frac{2.32}{(1-\bmom)\sqrt{\gin\gout}} .
\]
At $\ridge=0$ the drift is $C_\bmom$ and the level is $\gamma_\wedge$, and the substitution $m=\gamma_\vee x$
evaluates the main term exactly as $\mathsf F(\iota)/\sqrt\pi$; in the square case $\mathsf F(1)=\frac12$. This is
part (i). The drift stays of order one because the integral runs over all $\gamma_\wedge$ directions.

With momentum the volatility is a sum over lags, and each lag coefficient is an exact average over
the level. Successive updates are correlated, and the zero-ridge volatility is
(\cref{prop:app-iso-vol-momentum}(ii))
\[
\mathfrak v_{\mathrm{iso},\bmom}=N\Big(1+2\sum_{\ell\ge1}\overline C_{1,\ell}(0)\Big),
\]
where $\overline C_{1,\ell}(0)$ is the correlation of two updates $\ell$ steps apart, a double integral of
the closed-form kernel of the lag-$\ell$ covariance (\cref{lem:app-iso-lag}). Part (ii) says that the lag sum is $\gamma/4$
up to $O_\bmom(1)$ (\cref{thm:app-iso-vol-rate}). To see this, write the double integral in the two
roots $y$ and $z_\ell\defeq\bmom^{-2\ell}z$; multiplying a root by $\bmom^{-2\ell}$ shifts every age of the
buffer by $\ell$ places. Let $\mathsf A(y)\defeq\gamma-r_\bmom(y)$ be the residual count, the directions not
yet whitened at the root $y$, and $\mathsf j(y)\defeq\mathsf A(y)+2y\,r_\bmom'(y)$. Averaging over the
intermediate level $\eta$ between $y$ and $z$ gives the exact identity
(\cref{prop:app-iso-vol-representation,lem:app-iso-averaging,eqn:app-iso-T1-sum})
\[
\begin{gathered}
\sum_{\ell\ge1}\overline C_{1,\ell}(0)=\frac1\gamma\int_0^\gamma\EE_\star\big[\Psi(y,z)\big]\,\dif\eta,\\
\Psi(y,z)=\mathbf 1\{y<Q_\gamma\}\,\frac{\mathsf j(y)\,\mathsf T(z)}{\gamma^2\widehat\Delta(y,z)},\qquad
\mathsf T(z)=\sum_{\ell\ge1}\mathbf 1\{z_\ell<Q_\gamma\}\,\mathsf j(z_\ell).
\end{gathered}
\]
Here $\widehat\Delta$ is the determinant that appears in that closed-form kernel, and $\EE_\star$ averages over how far
$y$ and $z$ sit from the root at level $\eta$, on a logarithmic scale, under a fixed probability law
with density proportional to $1/\sinh$ of half the total offset.

At level $\eta$ the summed integrand is approximately half the residual count, $(\gamma-\eta)/2$,
whose average over $\eta\in(0,\gamma)$ is $\gamma/4$. On the diagonal
$\gamma^2\widehat\Delta(y,y)=\mathsf A(y)\,\mathsf j(y)$ (\cref{eqn:app-iso-diagonal}), so the first factor of
$\Psi$ is $1/\mathsf A(y)$. A shift by $\ell$ ages raises the count by about $\ell$, so the lag $\ell$
contributes about $\mathsf A(z)-\ell$ to $\mathsf T(z)$ while $\ell$ is below the residual count, and nothing
beyond it (\cref{lem:app-iso-shifts}). Precisely, with $\mathsf A_\vee$ the larger of the two residual counts,
$E_z$ the error of the counting formula at $z$, and $\bar{\mathsf e}_\bmom$ a bound on the slope
$y\,r_\bmom'(y)$ (\cref{eqn:app-iso-T1-upper-T,eqn:app-iso-T1-lower-T,lem:app-iso-slope}),
\[
\frac{(\mathsf A(z)-1)_+^2}2\le\mathsf T(z)\le\frac{(\mathsf A_\vee+E_z)^2}2+2\bar{\mathsf e}_\bmom(\mathsf A_\vee+E_z).
\]
The correction in (ii) is of relative order $B/N$, with a constant that is bounded but not identified,
because these errors do not vanish. The integrand differs from $(\gamma-\eta)/2$ by terms linear or
quadratic in the logarithmic offset and in $E_z$, and these terms have bounded averages. The lag sum is
therefore $\gamma/4+O_\bmom(1)$, and the relative error is $O_\bmom(1/\gamma)=O_\bmom(B/N)$ (proof of
\cref{thm:app-iso-vol-rate}, \cref{lem:app-iso-low-levels}). The
hypothesis $\gamma\ge\gamma^\circ_\bmom$ keeps the residual count at least $\Phi(1)$ wherever $\Psi>0$, and hence
$\widehat\Delta$ away from zero there (\cref{lem:app-iso-shifts}(v), \cref{lem:app-iso-kernel-bounds}(v)).
\end{sketch}

The fixed-ridge participation count and the constants of the counting formula are
\[
\begin{gathered}
\widehat{\mathfrak p}_\bmom(\ridge)\defeq\frac{\log\big((1-\bmom)^2\gin\gout/\ridge\big)}{2\log(1/\bmom)},\qquad
\mathsf c_\bmom\defeq\frac{\sqrt{\pi/2}}{1-\bmom},\qquad
\mathsf c'_\bmom\defeq1+\frac1{1-\bmom^2},\\
\mathsf c''_\bmom\defeq\frac1{\log(1/\bmom)}+\mathsf c_\bmom+\mathsf c'_\bmom ,
\end{gathered}
\]
and $\mathsf k_\bmom\defeq\mathsf e^-_\bmom(e^2-1)e^{-3}/(\pi^2\mathsf e^+_\bmom)$, with
$\mathsf e^-_\bmom\defeq\bmom^2\Phi'(1)$ and $\mathsf e^+_\bmom\defeq\mathsf c_\bmom+1/(1-\bmom^2)$.

\begin{theorem}[Fixed ridge, large aspect ratios]\label{thm:app-iso-fixed-ridge}
Let $\bmom\in(0,1)$, $\ridge>0$ and $\gin,\gout>0$, write $\widehat{\mathfrak p}=\widehat{\mathfrak p}_\bmom(\ridge)$,
and assume
\[
\widehat{\mathfrak p}\ge\frac{2\mathsf c'_\bmom}{\log(1/\bmom)\,\gamma_\wedge},\qquad
\widehat{\mathfrak p}+\mathsf c'_\bmom\le\frac{\gamma_\wedge}2 .
\]
Put $e_\bmom\defeq\max\big(\mathsf c'_\bmom,\ \mathsf c_\bmom+2(\widehat{\mathfrak p}+\mathsf c'_\bmom)/(\log(1/\bmom)\gamma_\wedge)\big)$. Then
\[
\big|\psi(\ridge)-\widehat{\mathfrak p}\big|\le e_\bmom,\qquad
\Big|\sqrt{\gin\gout}\;\mathfrak d_{\mathrm{iso}}(\ridge)-\frac{\widehat{\mathfrak p}}{\sqrt\pi}\Big|
\le\frac1{\sqrt\pi}\Big(e_\bmom+\frac{(\widehat{\mathfrak p}+\mathsf c'_\bmom)^2}{\gamma_\wedge}\Big)+\frac{2.32}{1-\bmom}.
\]
If moreover $\Nin=\Nout=N$ and $\gamma=N/B$, then
\[
\mathfrak v_{\mathrm{iso}}(\ridge)\ge B\,\psi(\ridge)+\mathsf k_\bmom B\Big(1-\frac{\psi(\ridge)}\gamma\Big)^2\big(\psi(\ridge)-\mathsf c''_\bmom-\mathsf c'_\bmom-1\big)_+^2 .
\]
\end{theorem}

\skpointer{thm:app-iso-fixed-ridge}{Lean: proved: \texttt{\detokenize{Muon}\detokenize{.}\allowbreak{}\detokenize{Isotropic}\detokenize{.}\allowbreak{}\detokenize{thm}\detokenize{_}\allowbreak{}\detokenize{app}\detokenize{_}\allowbreak{}\detokenize{iso}\detokenize{_}\allowbreak{}\detokenize{fixed}\detokenize{_}\allowbreak{}\detokenize{ridge}\detokenize{_}\allowbreak{}\detokenize{i}}, \texttt{\detokenize{Muon}\detokenize{.}\allowbreak{}\detokenize{Isotropic}\detokenize{.}\allowbreak{}\detokenize{thm}\detokenize{_}\allowbreak{}\detokenize{app}\detokenize{_}\allowbreak{}\detokenize{iso}\detokenize{_}\allowbreak{}\detokenize{fixed}\detokenize{_}\allowbreak{}\detokenize{ridge}\detokenize{_}\allowbreak{}\detokenize{ii}}}
The theorem describes what a fixed positive ridge does at small batch with momentum. The main text uses the
regularized polar map as a model of $K$ Newton--Schulz steps, with an effective ridge of order $a^{-2K}$, and
cites this theorem for the effect of a small ridge. A past batch of age $j$ enters the buffer with weight
$\emaw j=(1-\bmom)\bmom^j$, and the ridge whitens it only if its weight is large enough; the number of ages that pass
is about $\widehat{\mathfrak p}$. The theorem says that the
participation level $\psi(\ridge)$ is this count up to a bounded error, that the drift is about
$\widehat{\mathfrak p}/\sqrt{\pi\gin\gout}$, which decays like $\log\gamma/\gamma$ in the square case instead of
reaching the zero-ridge value $1/(2\sqrt\pi)$, and that the volatility is at least of order $B\psi(\ridge)^2$. The
two hypotheses ask that the count be past the first age and well below its ceiling $\gamma_\wedge$.

\begin{sketch}
The participation level counts the ages whose weight passes the threshold set by the root. The terms
$\Phi(\emaw j^2q)$ of $r_\bmom(q)$ are within $\sqrt{\pi/2}\,(\emaw j^2q)^{-1/2}$ of one for the ages with
$\emaw j^2q\ge1$, and at most $\emaw j^2q$ for the others; both corrections decay geometrically with the
distance of the age to the threshold. Hence, for $(1-\bmom)^2q\ge1$ (\cref{lem:app-iso-count}(i)),
\[
\mathsf n_\bmom(q)-\mathsf c_\bmom<r_\bmom(q)\le\mathsf n_\bmom(q)+\mathsf c'_\bmom,\qquad
\mathsf n_\bmom(q)=\frac{\log\big((1-\bmom)^2q\big)}{2\log(1/\bmom)} ,
\]
where $\mathsf n_\bmom(q)$ is, up to one, the number of ages with $\emaw j^2q\ge1$.

At a fixed ridge the root is about $\gin\gout/\ridge$, and the count at that root is $\widehat{\mathfrak p}$. Write
$H=\psi(\ridge)$ and $\mathsf W(m)=\sqrt{(\gin-m)(\gout-m)}$. The scalar equation at $s=\ridge$ says that the root
is $\mathsf W(H)^2/\ridge$, so
\[
\mathsf n_\bmom\big(q(-\ridge)\big)=\widehat{\mathfrak p}+\frac{\log\big(\mathsf W(H)/\sqrt{\gin\gout}\big)}{\log(1/\bmom)}\le\widehat{\mathfrak p} .
\]
The upper count then gives $H\le\widehat{\mathfrak p}+\mathsf c'_\bmom\le\gamma_\wedge/2$. Below half the ceiling,
$\mathsf W(H)\ge\sqrt{\gin\gout}\,(1-H/\gamma_\wedge)$, so the logarithm costs at most $2H/(\gamma_\wedge\log(1/\bmom))$, and
the lower count gives $H\ge\widehat{\mathfrak p}-e_\bmom$. The case $(1-\bmom)^2q(-\ridge)<1$, in which not even the
newest batch passes the threshold, would force a ridge that the first hypothesis excludes.

The drift is a fixed function of the level, so each whitened age contributes the same amount. By the master
estimate (\cref{thm:app-iso-master}, explained in the note on \cref{thm:app-iso-zero-ridge}),
\[
\Big|\mathfrak d_{\mathrm{iso}}(\ridge)-\frac1{\sqrt\pi\,\gin\gout}\int_0^{H}\mathsf W(m)\,\dif m\Big|<\frac{2.32}{(1-\bmom)\sqrt{\gin\gout}} .
\]
For $m\le H$ the integrand lies between $\sqrt{\gin\gout}(1-m/\gamma_\wedge)$ and $\sqrt{\gin\gout}$, so the main term times
$\sqrt{\gin\gout}$ is $H/\sqrt\pi$ up to $H^2/(\sqrt\pi\gamma_\wedge)$. With the bound on $H$ this gives the drift bound.

The volatility exceeds its same-time value by the correlations of updates that whiten common ages. By
\cref{prop:app-iso-vol-momentum}(ii), $\mathfrak v_{\mathrm{iso}}(\ridge)=B\psi(\ridge)+2N\sum_{\ell\ge1}\overline C_{1,\ell}(\ridge)$.
The lag lower bound (\cref{lem:app-iso-lag-lower}) bounds the determinant of the lag kernel by one, writes the
double integral in the two roots, and keeps only pairs of roots past the threshold and within a factor $e^2$ of
each other, where the count is accurate. A shift of $\ell$ ages raises the count by about $\ell$, so the lag $\ell$
is seen only at levels that stay below $\psi(\ridge)$ after the shift, a range of length about $\psi(\ridge)-\ell$:
\[
\overline C_{1,\ell}(\ridge)\ge\frac{\mathsf k_\bmom}{\gamma^3}\,\big(\gamma-\psi(\ridge)\big)^2\big(\psi(\ridge)-\ell-\mathsf c''_\bmom-\mathsf c'_\bmom\big)_+ .
\]
Summing over $\ell$, with $\sum_{\ell\ge1}(X-\ell)_+\ge(X-1)_+^2/2$, gives the volatility bound. At leading order the
integrated volatility is $B\psi(\ridge)^2$ (\cref{cor:app-iso-fixed-ridge}(iii)).
\end{sketch}

For $0<\ridge<1$ the participation count of the ridge is $\mathfrak p_\bmom(\ridge)\defeq\log(1/\ridge)/(2\log(1/\bmom))$,
up to one the number of ages $j\ge0$ with $\bmom^{2j}\ge\ridge$, and the crossover function is
\[
\mathcal P(x)\defeq\tfrac12-\tfrac12(1-x)^2+(1-x)^2\log(1-x)\quad(0\le x<1),\qquad\mathcal P(1)\defeq\tfrac12 .
\]

\begin{theorem}[Crossover]\label{thm:app-iso-crossover}
Let $\bmom\in(0,1)$ be fixed, let $\Nin=\Nout=N$ and $\gamma=N/B\to\infty$, and let
$\ridge_\gamma\in(0,1)$ satisfy $\mathfrak p_\bmom(\ridge_\gamma)/\gamma\to c\in(0,\infty)$. Then
$\psi(\ridge_\gamma)/\gamma\to c\wedge1$, the same-time total satisfies
$B\psi(\ridge_\gamma)/N\to c\wedge1$,
\[
\mathfrak d_{\mathrm{iso}}(\ridge_\gamma)\longrightarrow\Xi(c)\defeq\frac{(c\wedge1)\big(1-\tfrac12(c\wedge1)\big)}{\sqrt\pi},
\]
and, with $\mathcal P$ of \cref{eqn:app-iso-crossover-P},
\[
\frac{\mathfrak v_{\mathrm{iso}}(\ridge_\gamma)}{N\gamma}\longrightarrow\mathcal P(c\wedge1).
\]
\end{theorem}

\skpointer{thm:app-iso-crossover}{Lean: proved: \texttt{\detokenize{Muon}\detokenize{.}\allowbreak{}\detokenize{Isotropic}\detokenize{.}\allowbreak{}\detokenize{thm}\detokenize{_}\allowbreak{}\detokenize{app}\detokenize{_}\allowbreak{}\detokenize{iso}\detokenize{_}\allowbreak{}\detokenize{crossover}}}
The theorem connects the fixed-ridge regime with the zero-ridge regime at small batch, in the square case. The
ridge lets about $\mathfrak p_\bmom(\ridge)$ past batches be whitened, while there are only
$\gamma=N/B$ directions to fill. When the ridge shrinks with $\gamma$ so that $\mathfrak p_\bmom(\ridge_\gamma)\approx c\gamma$,
the update whitens the fraction $c\wedge1$ of the directions, and the drift and the volatility have explicit limits
in this fraction. As $c\to0$ they reduce to the fixed-ridge laws, $\Xi(c)\sim c/\sqrt\pi$ and $\mathcal P(c)\sim c^2$,
and for $c\ge1$, that is $\ridge\lesssim\bmom^{2\gamma}$, they equal the zero-ridge limits $1/(2\sqrt\pi)$ and $\frac12$ of
\cref{thm:app-iso-zero-ridge}. The main text uses these laws at the ridge calibrated to Newton--Schulz, where they
predict the NS-5 speedups of \cref{fig:isotropic-speedup}.

\begin{sketch}
The whitened fraction $\psi(\ridge_\gamma)/\gamma$ tends to $c\wedge1$, by the counting argument of
\cref{thm:app-iso-fixed-ridge}. Write $H=\psi(\ridge_\gamma)$ and $\mathfrak p=\mathfrak p_\bmom(\ridge_\gamma)$. The root at the
ridge is $(\gamma-H)^2/\ridge_\gamma$, and the count of ages that pass its threshold, $\mathsf n_\bmom$ of that theorem's
note, brackets $H$ (\cref{lem:app-iso-count}(i)):
\[
\mathfrak p+\frac{\log\big((1-\bmom)(\gamma-H)\big)}{\log(1/\bmom)}-\mathsf c_\bmom<H\le\mathfrak p+\frac{\log\big((1-\bmom)\gamma\big)}{\log(1/\bmom)}+\mathsf c'_\bmom .
\]
The corrections are logarithmic in $\gamma$, hence negligible against $\mathfrak p\approx c\gamma$, unless $\gamma-H$ is
small. If $c<1$, the upper bound keeps $H$ below $(1-\delta)\gamma$, and then the lower bound gives $H\ge c\gamma-o(\gamma)$.
If $c\ge1$ and $H\le(1-\delta)\gamma$ along a subsequence, the lower bound gives $H>(1-\delta/2)\gamma$, a contradiction; so
$H/\gamma\to1$. When the ridge would admit more ages than there are directions, all directions are whitened, as at zero
ridge.

In the square case the drift is a quadratic function of the whitened fraction, up to $O(1/\gamma)$. The master
estimate (\cref{thm:app-iso-master}(ii)) with $\mathsf W(m)=\gamma-m$ reads
\[
\Big|\mathfrak d_{\mathrm{iso}}(\ridge)-\Xi_0\Big(\frac{\psi(\ridge)}\gamma\Big)\Big|<\frac{2.32}{(1-\bmom)\gamma},\qquad
\Xi_0(x)=\frac{x-\frac12x^2}{\sqrt\pi},
\]
because $\gamma^{-2}\int_0^\psi(\gamma-m)\,\dif m=\psi/\gamma-\frac12(\psi/\gamma)^2$. Continuity of $\Xi_0$ gives the drift limit
$\Xi(c)=\Xi_0(c\wedge1)$. The order-one quantity is the drift itself, so $\gamma\,\mathfrak d_{\mathrm{iso}}(\ridge_\gamma)\to\infty$.

The lag sum at a positive ridge is the zero-ridge computation with the levels cut at $\psi$ instead of $\gamma$
(\cref{thm:app-iso-vol-ridge}). The lag coefficients have the same exact representation in the participation levels as
at zero ridge (\cref{prop:app-iso-vol-representation}; see the note on \cref{thm:app-iso-zero-ridge}), with the levels
restricted to $(0,\psi)$ and an extra factor $\mathsf u_\ridge(y)=(1-\ridge y/\mathsf A(y)^2)^{-1/2}\ge1$ at each root, where
$\mathsf A(y)=\gamma-r_\bmom(y)$ is the residual count. At level $\eta$ there are $X=\psi-\eta$ shifts left before the cut and
$x=\gamma-\eta$ residual directions. The lag $\ell<X$ contributes about $x-\ell$. Dividing the sum of these
contributions by the residual count leaves $\mathsf H(\eta)=X-X^2/(2x)$. Its average is exactly
\[
\frac1\gamma\int_0^{\psi}\Big(\psi-\eta-\frac{(\psi-\eta)^2}{2(\gamma-\eta)}\Big)\dif\eta=\frac\gamma2\,\mathcal P\Big(\frac\psi\gamma\Big),
\]
and at $\psi=\gamma$ it is the value $\gamma/4$ of the zero-ridge case (\cref{thm:app-iso-vol-rate}). The errors of the counting formula cost
$O_\bmom((1+\psi)/\gamma)$ in the lag sum. So does the ridge factor, which is close to one except for roots within a
bounded logarithmic distance of the top root $q(-\ridge)$; the shifted roots $\bmom^{-2\ell}z$, equally spaced in logarithmic scale with step
$2\log(1/\bmom)$, meet that window about once per step, about $\psi$ times in all (\cref{lem:app-iso-ridge-excess}). Hence
$\mathfrak v_{\mathrm{iso}}(\ridge)/N=\psi/\gamma+\gamma\,\mathcal P(\psi/\gamma)+O_\bmom((1+\psi)/\gamma)$, and dividing by
$\gamma$ and using the continuity of $\mathcal P$ gives the volatility limit $\mathcal P(c\wedge1)$.
\end{sketch}

\begin{remark}[Momentum as a change of batch and of step]\label{rem:app-iso-momentum}
In the oversampled square case $N<B$, with $\gamma=N/B$ small, the small-aspect theorems describe
momentum at the level of the scalar equation \cref{eqn:app-iso-scalar}, read
with its formal zero-ridge coefficients. For every step $\macrostep$, with
$\macrostep'\defeq\macrostep/\sqrt{\sw2}=\macrostep\sqrt{(1+\bmom)/(1-\bmom)}$,
\[
-2\macrostep\,C_\bmom\sqrt{\fRisk}+\frac{\macrostep^2}{2B}\,\mathfrak v_{\mathrm{iso},\bmom}
=-2\macrostep'\big(\sqrt{\sw2}\,C_\bmom\big)\sqrt{\fRisk}+\frac{\macrostep'^2}{2B}\big(\sw2\,\mathfrak v_{\mathrm{iso},\bmom}\big).
\]
On the drift, $\sqrt{\sw2}\,C_\bmom(\gamma)=C_0(\gamma)(1+O(\gamma))$ by
\cref{thm:app-iso-small-momentum}(iii): momentum acts as the effective batch
$B/\sw2=B(1+\bmom)/(1-\bmom)$. On the volatility the factor is $V_\bmom$, not
$1/\sw2$: $\mathfrak v_{\mathrm{iso},0}=N$ and $\mathfrak v_{\mathrm{iso},\bmom}=NV_\bmom(1+O(\gamma))$
(\cref{thm:app-iso-small-volatility}), so that
$\sw2\,\mathfrak v_{\mathrm{iso},\bmom}=\sw2V_\bmom\,\mathfrak v_{\mathrm{iso},0}(1+O(\gamma))$ with
$\sw2V_\bmom\in(\sw2,1)$ (\cref{lem:app-iso-V}). Momentum at step $\macrostep$ is
therefore no momentum at step $\macrostep'$ with the volatility multiplied by
$\sw2V_\bmom<1$. The step-rescaling reading is only approximate in the
oversampled isotropic case: momentum carries less noise than the rescaled
step, by the factor $\sw2V_\bmom$, which is $0.826\ldots$ at $\bmom=\frac12$ and tends
to $0.7776\ldots$ as $\bmom\uparrow1$. The main text writes $B_{\rm eff}\defeq B/\sw2$ for
the effective batch and $c_\bmom\defeq\sw2V_\bmom$ for this factor, so that
$\mathfrak v_{\mathrm{iso},\bmom}=c_\bmom NB_{\rm eff}/B\,\big(1+O(\gamma)\big)$. The factor $1/\sw2$ for the volatility of
\cref{cor:app-pl-endpoints}(ii) belongs to a different regime, the resolved rows
of a power-law spectrum with $\gamma\to\infty$; that regime and the oversampled
isotropic one are disjoint, and neither statement specializes to the other.
\end{remark}

\skpointer{rem:app-iso-momentum}{Lean: not formalized (remark)}
At large batch, momentum at step $\macrostep$ is no momentum at the larger step
$\macrostep'=\macrostep\sqrt{B_{\rm eff}/B}$, with the drift $C_0$ of no momentum and the volatility of no momentum multiplied
by $c_\bmom=\sw2V_\bmom<1$, up to factors $1+O(\gamma)$. The displayed identity is algebra; its content is the two large-batch laws
$\sqrt{\sw2}\,C_\bmom=C_0(1+O(\gamma))$ (\cref{thm:app-iso-small-momentum}(iii)) and
$\mathfrak v_{\mathrm{iso},\bmom}=NV_\bmom(1+O(\gamma))$ with $\mathfrak v_{\mathrm{iso},0}=N$ for $N<B$
(\cref{thm:app-iso-small-volatility,prop:app-iso-volatility-zero}). This is the reading behind $B_{\rm eff}$ and
$c_\bmom$ in \cref{fig:isotropic-main} and in the large-batch discussion of the speedup in \cref{sec:isotropic}. It
concerns the scalar equation with its formal zero-ridge coefficients, and the factor $1/\sw2$ of the power-law
regime (\cref{cor:app-pl-endpoints}(ii)) belongs to a disjoint regime.

\section{Power-law spectra and the large-batch regime}\label{sk:ch-powerlaw}
This chapter computes the coefficients $\mathfrak d_i$ and $\mathfrak v_i$ of the
row recursion \cref{eqn:ri-recursion} in the half-anisotropic setting
$\Sigmain=\Id_N$, $\Nout=\Nin=N$, at zero normalized ridge (except \cref{cor:app-os-update}, which concerns the expected update at a positive ridge). At each finite
dimension these are deterministic functions of the spectrum
$\mu_1\ge\dots\ge\mu_N$, of $B$ and of $\bmom$, obtained as limits
$\ridge\downarrow0$ of the coefficients of the drift and volatility chapters; no
result concerns the algorithm at zero ridge. Two regimes are treated with one
set of tools (\cref{sk:E5}). In the thin-batch regime the first eight results
locate the crossover between resolved and unresolved rows and give both
coefficients on either side of it, which is \cref{tab:power-law-kernels}. In the
oversampled regime the last six results concern the drift alone, for a general spectrum (the last one specializes to power laws), and give the
square-root law of \cref{thm:large-batch} and the modes that escape it.

The thin-batch results share the following setup. \Cref{assum:app-pl-regime} fixes
$\alpha>0$ and $\bmom\in(0,1)$ and lets $N,B\to\infty$ with $\gamma=N/B\to\infty$, with
$\mu_i=c_Ni^{-\alpha}$ for amplitudes $c_N>0$. The weights are $\emaw j=(1-\bmom)\bmom^j$,
$\sw p=\sum_j\emaw j^p$, so $\sw2=(1-\bmom)/(1+\bmom)$, and $c_0=\EE\abs g=\sqrt{2/\pi}$ for a standard Gaussian $g$. With
$\Jfun(x)=\EE[g^2/(1+xg^2)]$ put
\[
\Fw(q)=\sum_j\emaw j^2\Jfun(\emaw j^2q),\qquad \Lw(q)=\sum_j\emaw j\Jfun(\emaw j^2q),
\]
\[
r_\bmom(q)=q\Fw(q),\quad Q_m=r_\bmom^{-1}(m),\quad E(m)=Q_mr_\bmom'(Q_m),
\]
and the participation $m_N(\Lambda)=B^{-1}\sum_i\Lambda\mu_i/(1+\Lambda\mu_i)$ with inverse
$\Lambda_N:(0,\gamma)\to(0,\infty)$. The constants are
\[
a_N=\sqrt{c_N}N^{-\alpha/2},\quad K_\alpha=\frac{\pi}{\alpha\sin(\pi/\alpha)},\quad i_{\mathrm c}=\frac B{K_\alpha},
\]
\[
b_N=\sqrt{c_N}\big(\gamma i_{\mathrm c}^\alpha\big)^{-1/2},\quad d_\alpha=\frac{c_0\pi}{\alpha\cos(\pi/(2\alpha))},
\]
the last four used only for $\alpha>1$. The zero-ridge drift $\mathfrak d_i(0)$, row lag
coefficients $\overline C_{i,\ell}(0)$ and integrated volatility
$\mathfrak v_i(0)=\sum_{\ell\in\mathbb Z}\overline C_{i,\ell}(0)$ (with $\overline C_{i,-\ell}=\overline C_{i,\ell}$) are the limits as $\ridge\downarrow0$ of
the drift coefficient $\mathfrak d_i(\ridge)=\mathfrak d(\mu_i,1)$ and of the row coefficients of
the volatility chapter, at each fixed dimension (\cref{def:app-pl-coefficients}).

The oversampled results share the following objects. The effective batch is
$B_{\rm eff}=B/\sw2$, the small parameter is $\epsilon_{\rm os}=3(1-\bmom^2)\Nin/B$, and
$\hat\nu=\hat\nu_{\mathrm{out}}=\Nout^{-1}\sum_k\delta_{\mu_k}$ is the empirical spectral distribution of $\Sigmaout$, with
shape ratio $r=\Nout/\Nin$. For a probability measure $\nu$ on $[0,\infty)$ let
$m_\nu(c)=\int x\,\nu(\dif x)/(c+x)$, and let $b_0(\hat s)$ be the unique root in $(0,1]$ of
$b+r\,m_\nu(\hat s/b)=1$. Then
\[
\mathfrak d^{\rm os}_{\nu,r}(\mu;\hat\ridge)=\frac{\sqrt2}\pi\int_0^\infty\frac{b_0(\hat s)\,\mu}{\hat s+b_0(\hat s)\,\mu}\,\dif v\quad(\hat s=\hat\ridge+v^2),
\]
\[
\chi_\nu(\mu)=\int_0^\infty\varpi(y)\sqrt{1-m_\nu(\mu y)}\,\dif y,\qquad \varpi(y)=\frac2\pi\frac{\sqrt y}{(1+y)^2},
\]
and $\mathfrak d^{\rm os}_\nu(\mu)=\mathfrak d^{\rm os}_{\nu,1}(\mu;0)$. The zero-ridge coefficient $\mathfrak d_i(0)$ is the
limit of $\mathfrak d_i(\ridge)$ as $\ridge\downarrow0$, which exists for every spectrum
(\cref{thm:app-os-limit}(ii)).

For participations $m_1,m_2>0$, with $q_k=Q_{m_k}$, an integer lag $\ell\ge0$ and a standard
Gaussian $g$, the functions referred to in the definition as those of
\cref{eqn:app-pl-participation-kernel} are
\[
\mathsf D_\ell(m_1,m_2)=\sum_{a\ge0}\EE\frac{\bmom^\ell\sqrt{q_1q_2}\,\emaw a^2g^2}{(1+q_1\emaw a^2g^2)(1+\bmom^{2\ell}q_2\emaw a^2g^2)},
\]
\[
\mathsf M_\ell(m_1,m_2)=\sum_{a\ge0}\EE\frac{\bmom^{2\ell}q_1q_2\,\emaw a^4g^4}{(1+q_1\emaw a^2g^2)(1+\bmom^{2\ell}q_2\emaw a^2g^2)},
\]
and for a negative lag $-\ell$ they are those of the lag $\ell$ at $(m_2,m_1)$.

\begin{definition}[Profiles]\label{def:app-pl-profiles}
Let $S_\alpha(V)\defeq\int_0^1V(V+t^\alpha)^{-1}\dif t$ for $V>0$, a strictly increasing
bijection onto $(0,1)$, and $V_\alpha\defeq S_\alpha^{-1}$. For $0<x\le1$ and $0<p<1$ put
\[
f_x(p)\defeq\frac{V_\alpha(p)}{V_\alpha(p)+x^\alpha},\qquad U_\alpha(p)\defeq\frac{V_\alpha(p)}{p(1-p)} .
\]
The bulk profiles are
\begin{equation}\label{eqn:app-pl-bulk-profiles}
\begin{gathered}
\Psi_\alpha(x)\defeq c_0\int_0^1\sqrt{\frac{1-p}{p\,V_\alpha(p)}}\;f_x(p)\,\dif p,\\
C_{\alpha,x}(\tau)\defeq\int_0^{1-\tau}\frac{f_x(p+\tau)-f_x(p)}{\tau}\sqrt{\frac{U_\alpha(p)}{U_\alpha(p+\tau)}}\,\dif p\quad(0<\tau<1),\\
\Gamma_\alpha(x)\defeq2\int_0^1C_{\alpha,x}(\tau)\,\dif\tau .
\end{gathered}
\end{equation}
For $\alpha>1$, the crossover profiles are, for $y,X>0$,
\begin{equation}\label{eqn:app-pl-Phi}
\begin{aligned}
\Phi_{\alpha,\bmom}(y)\defeq\frac1\pi\int_0^\infty&\sqrt q\,\Lw(q)\,r_\bmom(q)^{-\frac{\alpha+1}2}\\
&\times\frac{y\,r_\bmom(q)^\alpha}{1+y\,r_\bmom(q)^\alpha}\Big(1+(\alpha-1)\frac{q\,r_\bmom'(q)}{r_\bmom(q)}\Big)\frac{\dif q}{q},
\end{aligned}
\end{equation}
\begin{equation}\label{eqn:app-pl-G}
\begin{aligned}
G_{\alpha,\bmom}(X)\defeq\frac1{\pi^2}\sum_{\ell\in\mathbb Z}\int_0^\infty\!\!\int_0^\infty
&\frac{\sqrt{\zeta_X(u)\big(1-\zeta_X(u)\big)\zeta_X(v)\big(1-\zeta_X(v)\big)}}{\sqrt{uv}}\\
&\times\frac{\mathsf D_\ell(u,v)\,\vartheta_\infty(u)\vartheta_\infty(v)}{\Delta_\ell^\infty(u,v)}\,\dif u\,\dif v,
\end{aligned}
\end{equation}
where $\mathsf D_\ell$ and $\mathsf M_\ell$ are the functions of
\cref{eqn:app-pl-participation-kernel}, which depend on $(m_1,m_2)=(u,v)$ only
through $(Q_u,Q_v)$, and
\[
\zeta_X(u)\defeq\frac{u^\alpha}{X^\alpha+u^\alpha},\qquad
\vartheta_\infty(u)\defeq\frac1{E(u)}+\frac{\alpha-1}u,\]
\[
t_\alpha(u,v)\defeq\frac{v^\alpha u-u^\alpha v}{v^\alpha-u^\alpha},\qquad
\Delta_\ell^\infty\defeq1-\frac{t_\alpha(u,v)\,\mathsf M_\ell(u,v)}{uv},
\]
with $t_\alpha(u,u)\defeq(\alpha-1)u/\alpha$, and with the convention of
\cref{eqn:app-pl-participation-kernel} for negative lags.
\end{definition}

\skpointer{def:app-pl-profiles}{Lean: proved: \texttt{\detokenize{Muon}\detokenize{.}\allowbreak{}\detokenize{PowerLaw}\detokenize{.}\allowbreak{}\detokenize{profU}}, \texttt{\detokenize{Muon}\detokenize{.}\allowbreak{}\detokenize{PowerLaw}\detokenize{.}\allowbreak{}\detokenize{profF}}, +10 more}
These are the profiles of \cref{tab:power-law-kernels}. The function $S_\alpha$ is
the continuum limit of the participation of a power law: a participation fraction
$p=m/\gamma$ corresponds to the level $V_\alpha(p)$, and $f_x(p)$ is the share of
the row at position $x=i/N$ that is whitened at that level (\cref{sk:E5}). The drift
profile $\Psi_\alpha$ integrates this share against the weight $c_0\sqrt{(1-p)/(pV_\alpha)}$
that the averaged momentum weights produce, and $C_{\alpha,x}(\tau)$ is the limiting
correlation of two updates of row $x$ that are $\tau\gamma$ steps apart, so that
$\Gamma_\alpha$ is the integrated volatility per unit of $\gamma$. Neither depends on $\bmom$.
For $\alpha>1$ the crossover profiles $\Phi_{\alpha,\bmom}$ and $G_{\alpha,\bmom}$ describe the rows
near the crossover index $i_{\mathrm c}=B/K_\alpha$, and they keep the full momentum
weight functions.

\begin{proposition}[The profiles]\label{prop:app-pl-profiles}
\leavevmode
\begin{enumerate}[label=(\roman*)]
\item $\Psi_\alpha$ and $\Gamma_\alpha$ are finite, positive and continuous on
$(0,1]$, $\Psi_\alpha$ is bounded below by a positive constant there, and as $x\downarrow0$
\begin{equation}\label{eqn:app-pl-bulk-small}
\begin{gathered}
\Psi_\alpha(x)\sim\begin{cases}\dfrac{c_0}{\sqrt{1-\alpha}}\log(x^{-\alpha}),&\alpha<1,\\[6pt]
\dfrac23c_0\big[\log(1/x)\big]^{3/2},&\alpha=1,\\[6pt]
d_\alpha\sqrt{K_\alpha}\,x^{-(\alpha-1)/2},&\alpha>1,\end{cases}
\\
\Gamma_\alpha(x)\sim\begin{cases}\dfrac{x^\alpha\big[\log(x^{-\alpha})\big]^2}{1-\alpha},&\alpha<1,\\[6pt]
\dfrac49x\big[\log(1/x)\big]^3,&\alpha=1,\\[6pt]
\dfrac{4K_\alpha^2}{\alpha-1}\,x,&\alpha>1.\end{cases}
\end{gathered}
\end{equation}
\item For $\alpha>1$, $\Phi_{\alpha,\bmom}$ and $G_{\alpha,\bmom}$ are finite, positive and
continuous on $(0,\infty)$, and
\begin{equation}\label{eqn:app-pl-crossover-ends}
\begin{gathered}
\Phi_{\alpha,\bmom}(y)\sim\sqrt{y/\sw2}\ (y\to\infty),\quad
\Phi_{\alpha,\bmom}(y)\sim d_\alpha\,y^{\frac{\alpha-1}{2\alpha}}\ (y\downarrow0),\\
G_{\alpha,\bmom}(X)\to\frac1{\sw2}\ (X\downarrow0),\quad
G_{\alpha,\bmom}(X)\sim\frac{4K_\alpha}{\alpha-1}X\ (X\to\infty).
\end{gathered}
\end{equation}
\end{enumerate}
\end{proposition}

\skpointer{prop:app-pl-profiles}{Lean: proved: \texttt{\detokenize{Muon}\detokenize{.}\allowbreak{}\detokenize{PowerLaw}\detokenize{.}\allowbreak{}\detokenize{prop}\detokenize{_}\allowbreak{}\detokenize{app}\detokenize{_}\allowbreak{}\detokenize{pl}\detokenize{_}\allowbreak{}\detokenize{profiles}\detokenize{_}\allowbreak{}\detokenize{i}}, \texttt{\detokenize{Muon}\detokenize{.}\allowbreak{}\detokenize{PowerLaw}\detokenize{.}\allowbreak{}\detokenize{prop}\detokenize{_}\allowbreak{}\detokenize{app}\detokenize{_}\allowbreak{}\detokenize{pl}\detokenize{_}\allowbreak{}\detokenize{profiles}\detokenize{_}\allowbreak{}\detokenize{ii}}}
Part (i) is the small-$x$ table to the right of \cref{tab:power-law-kernels}, with
its constants: as $x\downarrow0$ the unresolved drift profile grows only logarithmically in
$1/x$ for $\alpha\le1$ and like $x^{-(\alpha-1)/2}$ for $\alpha>1$, while the unresolved volatility per unit
of $\gamma$ vanishes like $x^\alpha$ times logarithms for $\alpha\le1$ and linearly in $x$ for $\alpha>1$. Part (ii) gives the two ends of
the crossover profiles for $\alpha>1$; they are the resolved and unresolved laws of
\cref{thm:app-pl-drift,thm:app-pl-volatility}, which the profiles join. All
statements concern explicit functions and involve no dimension.

\begin{sketch}
The drift profile is dominated, as $x\downarrow0$, by small participation fractions,
where the inverse $V_\alpha$ has the power laws of \cref{sk:E5}. For $\alpha<1$,
$V_\alpha(p)\sim(1-\alpha)p$ turns the integrand, without the factor $c_0$, into
$(1-\alpha)^{-1/2}(p+x^\alpha/(1-\alpha))^{-1}$ on a fixed small interval, whose integral is $\log(x^{-\alpha})(1+o(1))$; at $\alpha=1$ the
extra logarithm of $V_1(p)\sim p/\log(1/p)$ gives the power $3/2$. For $\alpha>1$ the
substitution $p=K_\alpha xu$ isolates a convergent integral (\cref{lem:app-pl-Psi}),
\[
\Psi_\alpha(x)\sim c_0\sqrt{K_\alpha}\,x^{-(\alpha-1)/2}\int_0^\infty\frac{u^{(\alpha-1)/2}}{1+u^\alpha}\,\dif u
=d_\alpha\sqrt{K_\alpha}\,x^{-(\alpha-1)/2}.
\]

The volatility profile is a positive double integral, and for $\alpha\le1$ it is
asymptotically the square of a single integral. With
$A_x(p)=\sqrt{f_x(p)(1-f_x(p))(1-p)/p}$, a partial-fraction identity in the variable
of the spectrum gives exactly (\cref{lem:app-pl-matching})
\[
\Gamma_\alpha(x)=2\int_{0<p<q<1}\frac{A_x(p)A_x(q)}{\Delta_{\mathrm c}(p,q)}\,\dif p\,\dif q,\qquad
\Big(\int_0^1A_x\Big)^2=\frac{x^\alpha\Psi_\alpha(x)^2}{c_0^2},
\]
where $\Delta_{\mathrm c}\le1$ is the continuum determinant of the lag kernel. Without the
determinant the double integral is the square on the right. For $\alpha\le1$ the
integral of $A_x$ is spread over a range of scales of $p$ whose logarithmic length
tends to infinity, so most of the double integral comes from pairs with $q>Rp$ and
$p,q<\eta$, where $1-\Delta_{\mathrm c}\le\eta+1/R$; the diagonal strip $q\le Rp$ and the
upper end $p\to1$ are of lower order. This gives
$\Gamma_\alpha(x)\sim x^\alpha\Psi_\alpha(x)^2/c_0^2$ and the volatility equivalents of (i). For
$\alpha>1$ the mass of $A_x$ sits at $p\asymp x$, the determinant matters, and the
substitution $p=K_\alpha xr$ turns $\Gamma_\alpha(x)/(K_\alpha x)$ into a double integral over the lag and the
participation, with value $4K_\alpha/(\alpha-1)$ (\cref{lem:app-pl-supervol}).

The ends of the crossover profiles in (ii) come from the two mechanisms of \cref{sk:E5}.
As $y\to\infty$ only small participations count in $\Phi_{\alpha,\bmom}(y)$ and the head Lorentzian gives
$\sqrt{y/\sw2}$; as $y\downarrow0$ the substitution $m=y^{-1/\alpha}w$ moves it to large participation,
where averaging gives the integral above (\cref{lem:app-pl-super}). For $G_{\alpha,\bmom}(X)$ the
rescaling $m=Xr$ sends every lag term to $\bmom^{\abs\ell}$ as $X\downarrow0$, with sum $1/\sw2$, and
as $X\to\infty$ the concentration of the lag kernel on $m_2=m_1+\ell$
(\cref{lem:app-pl-localization}) gives the double integral $4K_\alpha/(\alpha-1)$
(\cref{lem:app-pl-supervol}). Finiteness, positivity and continuity follow from majorants
that are locally uniform in the argument.
\end{sketch}

\begin{theorem}[Drift crossover]\label{thm:app-pl-drift}
Assume \cref{assum:app-pl-regime} and let $i=i_N\to\infty$ be a row sequence.
\begin{enumerate}[label=(\roman*)]
\item Let $0<\alpha\le1$ and put
\begin{equation}\label{eqn:app-pl-icd}
\begin{gathered}
d\defeq\frac{c_0}{\sqrt{1-\alpha}},\quad i_{\mathrm c,\mathfrak d}\defeq\Big[\frac{BN^{\alpha-1}}{\sw2d^2\log^2\gamma}\Big]^{1/\alpha},\quad
a_{\mathfrak d}\defeq d\,a_N\log\gamma\quad(\alpha<1);\\
d\defeq\frac23c_0,\quad i_{\mathrm c,\mathfrak d}\defeq\frac{B}{\sw2d^2\log^3\gamma},\quad
a_{\mathfrak d}\defeq d\,a_N(\log\gamma)^{3/2}\quad(\alpha=1).
\end{gathered}
\end{equation}
If $i/i_{\mathrm c,\mathfrak d}\to x\in(0,\infty)$, then
$\sqrt2\,\mathfrak d_i(0)/a_{\mathfrak d}\to1+x^{-\alpha/2}$. If $i/i_{\mathrm c,\mathfrak d}\to0$,
then $\sqrt2\,\mathfrak d_i(0)\sim\sqrt{\mu_i/(\gamma\sw2)}$. If $i/i_{\mathrm c,\mathfrak d}\to\infty$,
then $\sqrt2\,\mathfrak d_i(0)\sim a_N\Psi_\alpha(i/N)$, and if moreover $i/N\to0$,
\[
\sqrt2\,\mathfrak d_i(0)\sim d\,a_N\log\big((N/i)^\alpha\big)\ \ (\alpha<1),\qquad
\sqrt2\,\mathfrak d_i(0)\sim d\,a_N\big[\log(N/i)\big]^{3/2}\ \ (\alpha=1).
\]
\item Let $\alpha>1$. If $i/i_{\mathrm c}\to x\in(0,\infty)$, then
$\sqrt2\,\mathfrak d_i(0)/b_N\to\Phi_{\alpha,\bmom}(x^{-\alpha})$. If $i/i_{\mathrm c}\to0$, then
$\sqrt2\,\mathfrak d_i(0)\sim\sqrt{\mu_i/(\gamma\sw2)}$. If $i/i_{\mathrm c}\to\infty$, then
$\sqrt2\,\mathfrak d_i(0)\sim a_N\Psi_\alpha(i/N)$, and if moreover $i/N\to0$,
\[
\sqrt2\,\mathfrak d_i(0)\sim d_\alpha\sqrt{c_NK_\alpha}\,N^{-1/2}\,i^{-(\alpha-1)/2}.
\]
\end{enumerate}
\end{theorem}

\skpointer{thm:app-pl-drift}{Lean: proved: \texttt{\detokenize{Muon}\detokenize{.}\allowbreak{}\detokenize{PowerLaw}\detokenize{.}\allowbreak{}\detokenize{thm}\detokenize{_}\allowbreak{}\detokenize{app}\detokenize{_}\allowbreak{}\detokenize{pl}\detokenize{_}\allowbreak{}\detokenize{drift}\detokenize{_}\allowbreak{}\detokenize{i}}, \texttt{\detokenize{Muon}\detokenize{.}\allowbreak{}\detokenize{PowerLaw}\detokenize{.}\allowbreak{}\detokenize{thm}\detokenize{_}\allowbreak{}\detokenize{app}\detokenize{_}\allowbreak{}\detokenize{pl}\detokenize{_}\allowbreak{}\detokenize{drift}\detokenize{_}\allowbreak{}\detokenize{ii}}}
This theorem gives the drift column of \cref{tab:power-law-kernels} for $\bmom>0$,
with the crossover between its two rows. Rows below the crossover index are
resolved: $\sqrt2\,\mathfrak d_i(0)\sim\sqrt{\mu_i/(\gamma\sw2)}$, which is
$\mathfrak d_i\sim\sqrt{\mu_iB_{\rm eff}/(2N)}$, the square-root law of the large-batch
section. Rows above it are not whitened individually, and their drift
$\sqrt2\,\mathfrak d_i(0)\sim a_N\Psi_\alpha(i/N)$, with $a_N=\sqrt{c_N}N^{-\alpha/2}$, is a
profile of $i/N$ that does not depend on $\bmom$; for $i\ll N$ it decays only
logarithmically ($\alpha\le1$) or like $i^{-(\alpha-1)/2}$ ($\alpha>1$), far more slowly than
$\sqrt{\mu_i}$. For $\alpha\le1$ the two laws add at the crossover, so that
$\sqrt2\,\mathfrak d_i(0)/a_{\mathfrak d}\to1+x^{-\alpha/2}$; for $\alpha>1$ the crossover is
described by the profile $\Phi_{\alpha,\bmom}$, which keeps the momentum. This sketch also
proves the drift half of the uniform approximation behind \cref{cor:app-pl-endpoints}.

\begin{sketch}
The drift is the sum of a head and a tail, each with an explicit limit
(\cref{sk:E5}). Changing variables from $u$ to the participation $m$ gives
(\cref{lem:app-pl-drift-participation})
\[
\sqrt2\,\mathfrak d_i(0)=\frac1{\pi\gamma}\int_0^\gamma\mathcal A(m)\,\mathcal B_i(m)\,E(m)\vartheta_N(m)\,\dif m,
\]
\[
\mathcal B_i(\gamma p)=a_N\sqrt{\frac{1-p}{p\,V_N(p)}}\,\frac{V_N(p)}{V_N(p)+x^\alpha},\quad x=\frac iN,
\]
where $\vartheta_N=-\frac{\dif}{\dif m}\log s_N$ comes from the Jacobian of $u^2=s_N(m)$, and
$E\vartheta_N$ is within $C(1/m+1/(\gamma-m))$ of one for $m\ge M$. On $m>M$ the averaging of \cref{sk:E5} replaces
$\mathcal A$ by $\sqrt{2\pi}$, the factor $E\vartheta_N$ is replaced by one, and $V_N$ by
$V_\alpha$, all with relative error $C/M+o(1)$ uniformly in the row; the tail is
therefore $a_N\Psi_\alpha^{[M/\gamma]}(x)$, the profile with the integral restricted to
$p\ge M/\gamma$ (\cref{lem:app-pl-tail}). The head $m\le M$ is at most
$C_M\sqrt{\mu_i/\gamma}$ for every row and equals $\sqrt{\mu_i/(\gamma\sw2)}(1+o(1))$ on
resolved rows (\cref{lem:app-pl-head}). For bulk rows, $i/N\to x_0>0$, the head is
$o(a_N)$ and $\sqrt2\,\mathfrak d_i(0)\sim a_N\Psi_\alpha(x_0)$.

For $\alpha\le1$ the crossover is the row at which head and tail are equal. With
$c_N=1$, write the head in units of $a_N$ as
\[
W\defeq\frac{\sqrt{\mu_i/(\gamma\sw2)}}{a_N}=\sw2^{-1/2}e^{(T-\log\gamma)/2},\qquad
T=\log(x^{-\alpha})\ (\alpha<1),\quad T=\log(1/x)\ (\alpha=1),
\]
and recall $\Psi_\alpha(x)\sim dT^\nu$ with $\nu=1$ for $\alpha<1$ and $\nu=3/2$ for $\alpha=1$
(\cref{prop:app-pl-profiles}). The cutoff $p\ge M/\gamma$ removes a part of the profile
that is negligible as long as $W\lesssim(\log\gamma)^\nu$, and when $W\gg(\log\gamma)^\nu$ both
profiles are $o(W)$ (\cref{lem:app-pl-cutoff}). In every case, after $M\to\infty$,
\[
\sqrt2\,\mathfrak d_i(0)=a_N\big(W+\Psi_\alpha(x)\big)\big(1+o(1)\big),
\]
which is the drift approximant of \cref{eqn:app-pl-approx-sub}. The index
$i_{\mathrm c,\mathfrak d}$ is where $W=d(\log\gamma)^\nu$: at $i/i_{\mathrm c,\mathfrak d}\to x_0$ one has
$W=d(\log\gamma)^\nu x_0^{-\alpha/2}(1+o(1))$ and $T=\log\gamma+O(\log\log\gamma)$, so both terms are
multiples of $a_{\mathfrak d}=d\,a_N(\log\gamma)^\nu$, which gives $1+x_0^{-\alpha/2}$. On the rows
$i\asymp i_\#$, where $i_\#=(c_N\resscale)^{1/\alpha}$ is the no-momentum threshold of \cref{prop:app-pl-scales}, the drift exceeds the no-momentum drift scale $(\gamma\resscale)^{-1/2}$ by a
factor of order $\log\gamma$, which comes from the growth $\log Q_m\sim2\abs{\log\bmom}\,m$ of the
momentum participation (\cref{rem:app-pl-oldscale}); this is why
$i_{\mathrm c,\mathfrak d}$ lies below the no-momentum threshold $i_\#$ by a factor of order
$(\log\gamma)^{-2/\alpha}$ (\cref{prop:app-pl-scales}). The main-text display \cref{eqn:i-sharp}, which
also writes $i_\#$, is for $\bmom>0$ and $\alpha\le1$ of the order of $i_{\mathrm c,\mathfrak d}$ itself; in
this chapter $i_\#$ always means the no-momentum threshold.

For $\alpha>1$ the crossover sits at participation of order one, where the weights
cannot be averaged, so its profile keeps the momentum. At $i\asymp B$ put
$y=(i_{\mathrm c}/i)^\alpha$. The scaling $c_N\Lambda_N(m)\sim(Bm/K_\alpha)^\alpha$ of \cref{sk:E5}
gives $\Lambda_N(m)\mu_i\to ym^\alpha$ and $m\Lambda_N'(m)/\Lambda_N(m)\to\alpha$, so that on compact sets
of $m$ the drift integrand, divided by $b_N$, converges to that of
(\cref{eqn:app-pl-Phi-u})
\[
\Phi_{\alpha,\bmom}(y)=\frac1\pi\int_0^\infty\mathcal A(m)E(m)\vartheta_\infty(m)\,m^{-\frac{\alpha+1}2}\,\frac{ym^\alpha}{1+ym^\alpha}\,\dif m,
\]
while the head bounds and the tail bound $CK_\alpha^{-\alpha/2}M^{-(\alpha-1)/2}$ make small and
large $m$ negligible (\cref{lem:app-pl-super}). On the two windows the head and the tail are
estimated directly, as for $\alpha\le1$ (\cref{lem:app-pl-super}), and the identities
$b_N\sqrt y=\sqrt{\mu_i/\gamma}$ and
$b_Nd_\alpha y^{(\alpha-1)/(2\alpha)}=a_Nd_\alpha\sqrt{K_\alpha}(i/N)^{-(\alpha-1)/2}$ show that the results
match the two endpoints of $\Phi_{\alpha,\bmom}$ (\cref{prop:app-pl-profiles}(ii)).
\end{sketch}

\begin{theorem}[Volatility crossover and bulk profile]\label{thm:app-pl-volatility}
Assume \cref{assum:app-pl-regime} and let $i=i_N$ be a row sequence.
\begin{enumerate}[label=(\roman*)]
\item (Bulk.) If $i/N\to x\in(0,1]$ and $\ell=\ell_N$ are integers with
$\ell/\gamma\to\tau\in(0,1)$, then
\[
\overline C_{i,\ell}(0)\longrightarrow C_{\alpha,x}(\tau),\qquad \frac{\mathfrak v_i(0)}{\gamma}\longrightarrow\Gamma_\alpha(x),
\]
locally uniformly in $x\in(0,1]$ and $\tau\in(0,1)$.
\item Let $0<\alpha\le1$, let $i\to\infty$, and put
\begin{equation}\label{eqn:app-pl-icv}
i_{\mathrm c,\mathfrak v}\defeq c_0^{2/\alpha}\,i_{\mathrm c,\mathfrak d}=\begin{cases}N\big[(1-\alpha)/(\sw2\gamma\log^2\gamma)\big]^{1/\alpha},&\alpha<1,\\ 9B/(4\sw2\log^3\gamma),&\alpha=1.\end{cases}
\end{equation}
If $i/i_{\mathrm c,\mathfrak v}\to y\in(0,\infty)$, then
$\sw2\,\mathfrak v_i(0)\to1+2c_0y^{\alpha/2}+y^\alpha$, uniformly for $y$ in compact subsets of
$(0,\infty)$. If $i/i_{\mathrm c,\mathfrak v}\to0$, then $\mathfrak v_i(0)\to1/\sw2$. If
$i/i_{\mathrm c,\mathfrak v}\to\infty$, then $\mathfrak v_i(0)\sim\gamma\Gamma_\alpha(i/N)$, and if moreover
$i/N\to0$,
\[
\mathfrak v_i(0)\sim\frac{\gamma(i/N)^\alpha\big[\log((N/i)^\alpha)\big]^2}{1-\alpha}\quad(\alpha<1),\qquad
\mathfrak v_i(0)\sim\frac49\,\frac iB\,\big[\log(N/i)\big]^3\quad(\alpha=1).
\]
\item Let $\alpha>1$ and $i\to\infty$. If $i/i_{\mathrm c}\to X_0\in(0,\infty)$, then
$\mathfrak v_i(0)\to G_{\alpha,\bmom}(X_0)$. If $i/i_{\mathrm c}\to0$, then $\mathfrak v_i(0)\to1/\sw2$. If
$i/i_{\mathrm c}\to\infty$, then $\mathfrak v_i(0)\sim\gamma\Gamma_\alpha(i/N)$, and if moreover $i/N\to0$,
\[
\mathfrak v_i(0)\sim\frac{4K_\alpha^2}{\alpha-1}\,\frac iB .
\]
\end{enumerate}
\end{theorem}

\skpointer{thm:app-pl-volatility}{Lean: proved: \texttt{\detokenize{Muon}\detokenize{.}\allowbreak{}\detokenize{PowerLaw}\detokenize{.}\allowbreak{}\detokenize{thm}\detokenize{_}\allowbreak{}\detokenize{app}\detokenize{_}\allowbreak{}\detokenize{pl}\detokenize{_}\allowbreak{}\detokenize{volatility}\detokenize{_}\allowbreak{}\detokenize{i}\detokenize{_}\allowbreak{}\detokenize{part}}, \texttt{\detokenize{Muon}\detokenize{.}\allowbreak{}\detokenize{PowerLaw}\detokenize{.}\allowbreak{}\detokenize{thm}\detokenize{_}\allowbreak{}\detokenize{app}\detokenize{_}\allowbreak{}\detokenize{pl}\detokenize{_}\allowbreak{}\detokenize{volatility}\detokenize{_}\allowbreak{}\detokenize{iii}\detokenize{_}\allowbreak{}\detokenize{part}}, +3 more; foundation results: \texttt{\detokenize{GaussianOpNorm}}}
This theorem gives the volatility column of \cref{tab:power-law-kernels} for
$\bmom>0$. On resolved rows $\mathfrak v_i(0)\to1/\sw2=B_{\rm eff}/B$: the row is fully
whitened at every step, and (as the proof shows for $\alpha>1$) momentum correlates the updates $\ell$ steps apart by
$\bmom^{\abs\ell}$. On unresolved rows $\mathfrak v_i(0)\sim\gamma\Gamma_\alpha(i/N)$, which grows with
$\gamma=N/B$, and part (i) identifies the correlation of updates $\tau\gamma$ steps apart in
a bulk row as $C_{\alpha,x}(\tau)$. For $\alpha\le1$ the crossover profile
$1+2c_0y^{\alpha/2}+y^\alpha$ is not a sum of the two endpoint laws: it is the second moment
of a sum of two correlated terms. For $\alpha>1$ the crossover profile is
$G_{\alpha,\bmom}$. For $\alpha\le1$ the crossover index differs from the drift index by the
fixed factor $c_0^{2/\alpha}$. This sketch also proves the volatility half of the uniform
approximation behind \cref{cor:app-pl-endpoints}.

\begin{sketch}
The zero-ridge volatility is a positive double integral over two participations,
summed over the lags (\cref{prop:app-pl-vol-zero}):
\[
\mathfrak v_i(0)=\frac1{\pi^2}\sum_{\ell\in\mathbb Z}\int_0^\gamma\!\!\int_0^\gamma W_i(m_1)W_i(m_2)\,\frac{\mathsf D_\ell(m_1,m_2)}{\Delta^N_\ell(m_1,m_2)}\,\dif m_1\,\dif m_2 .
\]
It is obtained from the closed-form lag kernel of \cref{sk:E3} by writing the
regularized filter as an integral of resolvents and changing variables to the
participation. Here $W_i$ is the weight of the row, $\mathsf D_\ell$ correlates two
buffers $\ell$ steps apart through the past batches they share, and the determinant
$\Delta^N_\ell\in(0,1]$ of the kernel is close to one when the smaller participation is
much smaller than the larger one and than $\gamma$ (\cref{lem:app-pl-determinant}).

In the bulk, the lag $\ell$ couples the participation $m$ to $m+\ell$, because
shifting the buffer by $\ell$ steps moves every past batch by $\ell$ places. Exactly,
$\mathsf D_\ell=(r_\bmom(q_1e^t)-r_\bmom(q_1))/(2\sinh(t/2))$ with $t=\log(\bmom^{2\ell}q_2/q_1)$, and
for large $u$, $r_\bmom(e^u)$ is within $Ce^{-u/2}$ of a function that grows by exactly one when $u$ grows by $2\abs{\log\bmom}$. The kernel $\mathsf D_\ell/(\pi^2E(m_1)E(m_2))$ therefore
concentrates on the shell $m_2=m_1+\ell+O(1)$ and carries mass one per unit of $m_1$,
by the identity (\cref{lem:app-pl-localization})
\[
\int_{\mathbb R}\frac{t}{2\sinh(t/2)}\,\dif t=\pi^2 .
\]
In the coordinates $m=\gamma p$ and $\ell=\tau\gamma$ the row factors converge to the function $A_x(p)$ and
the determinant to its continuum limit $\Delta_{\mathrm c}(p,p+\tau)$, both of the sketch for
\cref{prop:app-pl-profiles}, so
$\overline C_{i,\ell}(0)\to C_{\alpha,x}(\tau)$ (the identity of \cref{lem:app-pl-matching}). The
lag sum is then a Riemann sum with mesh $1/\gamma$; lags beyond $(1+\eta)\gamma$ contribute
$O(\gamma\bmom^{\ell-\gamma})$ each, because the later buffer then shares no whitened batch,
and $\mathfrak v_i(0)/\gamma\to2\int_0^1C_{\alpha,x}=\Gamma_\alpha(x)$.

For $\alpha\le1$ the mesoscopic rows are governed by the coefficient with every
determinant replaced by one, which is a perfect square. With one Gaussian $g$
shared by the two buffers and
$\mathcal Z_g(q)=\sqrt q\sum_a\emaw a\,g/(1+\emaw a^2qg^2)$, the lag sum factorizes as
$\sum_\ell\mathsf D_\ell(m_1,m_2)=\EE[\mathcal Z_g(Q_{m_1})\mathcal Z_g(Q_{m_2})]$
(\cref{lem:app-pl-age}). The flat coefficient is therefore $\EE\,\mathcal V_i(g)^2$ with
$\mathcal V_i(g)=\pi^{-1}\int_0^\gamma W_i\,\mathcal Z_g(Q_m)\,\dif m$, and it lies below
$\mathfrak v_i(0)$. The head of $\mathcal V_i$ is linear in $g$, and in the tail
$\mathcal Z_g/E$ averages to $\pi\sign(g)$, as $\mathcal A$ does for the drift, so
(\cref{lem:app-pl-vector})
\[
\Big\|\mathcal V_i(g)-\frac g{\sqrt{\sw2}}-b\,\sign(g)\Big\|_{L^2(g)}=o(1+b),\qquad
b=\frac{\sqrt\gamma\,x^{\alpha/2}\Psi_\alpha(x)}{c_0}.
\]
With $\EE g^2=\EE\sign(g)^2=1$ and $\EE[g\sign(g)]=c_0$ the flat coefficient is
$1/\sw2+(2c_0/\sqrt{\sw2})\,b+b^2$ up to $o((1+b)^2)$. The exact coefficient exceeds
it by $o((1+b)^2)$: off the diagonal the determinant is close to one, the strip
$m_1\asymp m_2$ carries little mass, and the top of the participation range
contributes $O(\gamma x^\alpha)$ (\cref{lem:app-pl-strip,lem:app-pl-upper,lem:app-pl-meso}).
At $i/i_{\mathrm c,\mathfrak v}\to y$ one has $b\to y^{\alpha/2}/\sqrt{\sw2}$, which gives
$\sw2\mathfrak v_i(0)\to1+2c_0y^{\alpha/2}+y^\alpha$.

For $\alpha>1$ the rows near $i_{\mathrm c}$ have a limit profile in which nothing is
averaged. With $X=i/i_{\mathrm c}$ and $m$ of order one, the participation scaling of
\cref{sk:E5} gives $f_i(m)\to\zeta_X(m)$, the normalized sum $\mathsf T_N$ becomes the
integral $t_\alpha(m_1,m_2)$, and $\Delta^N_\ell\to\Delta^\infty_\ell$, so $\mathfrak v_i(0)\to G_{\alpha,\bmom}(X)$
(\cref{lem:app-pl-supervol}). Along rows with $X\to0$ or $X\to\infty$, the rescalings that give the ends of
$G_{\alpha,\bmom}$ (\cref{prop:app-pl-profiles}(ii)), carried out with the dimension errors, give the resolved
limit $1/\sw2$ and the mesoscopic law $4K_\alpha X/(\alpha-1)=4K_\alpha^2i/((\alpha-1)B)$.
\end{sketch}

\begin{corollary}[Crossover windows and joint endpoints]\label{cor:app-pl-endpoints}
Assume \cref{assum:app-pl-regime}, let $i=i_N\to\infty$ be a row sequence, and
let $i_\star\defeq i_{\mathrm c,\mathfrak d}$ for $\alpha\le1$ and $i_\star\defeq i_{\mathrm c}$ for $\alpha>1$.
\begin{enumerate}[label=(\roman*)]
\item If $\alpha\le1$ and $i/i_\star\to x\in(0,\infty)$, then
\[
\frac{\sqrt2\,\mathfrak d_i(0)}{a_{\mathfrak d}}\to1+x^{-\alpha/2},\qquad
\sw2\,\mathfrak v_i(0)\to1+2x^{\alpha/2}+c_0^{-2}x^\alpha,\qquad
\frac{\mu_i}{c_Ni_\star^{-\alpha}}\to x^{-\alpha}.
\]
\item The joint endpoint equivalents are
\begin{equation}\label{eqn:app-pl-endpoint-table}
\begin{array}{c|c|c}
 & i/i_\star\to0 & i/i_\star\to\infty\\\hline
\sqrt2\,\mathfrak d_i(0) & \sqrt{\mu_i/(\gamma\sw2)} & a_N\Psi_\alpha(i/N)\\[3pt]
\mathfrak v_i(0) & 1/\sw2 & \gamma\,\Gamma_\alpha(i/N)\\[3pt]
\mu_i & c_Ni^{-\alpha}\ \text{(exact)} & c_Ni^{-\alpha}\ \text{(exact)}
\end{array}
\end{equation}
\item One has $i_\star/N\to0$. There are row sequences $i\to\infty$ in both
windows $i/i_\star\to0$ and $i/i_\star\to\infty$ if and only if
\begin{equation}\label{eqn:app-pl-windows}
\frac{BN^{\alpha-1}}{\log^2\gamma}\to\infty\ \ (\alpha<1),\qquad
\frac{B}{\log^3\gamma}\to\infty\ \ (\alpha=1),\qquad
B\to\infty\ \ (\alpha>1),
\end{equation}
and then $1\ll i_\star\ll N$.
\item For $\alpha>1$, equating the two mesoscopic endpoint equivalents gives the
indices $i_{\mathrm c}/(\sw2d_\alpha^2)$ for the drift and
$B(\alpha-1)/(4\sw2K_\alpha^2)$ for the volatility, whose ratio is
$(\alpha-1)d_\alpha^2/(4K_\alpha)$; all indices of the corollary are independent of $c_N$.
\end{enumerate}
\end{corollary}

\skpointer{cor:app-pl-endpoints}{Lean: proved: \texttt{\detokenize{Muon}\detokenize{.}\allowbreak{}\detokenize{PowerLaw}\detokenize{.}\allowbreak{}\detokenize{cor}\detokenize{_}\allowbreak{}\detokenize{app}\detokenize{_}\allowbreak{}\detokenize{pl}\detokenize{_}\allowbreak{}\detokenize{endpoints}\detokenize{_}\allowbreak{}\detokenize{i}}, \texttt{\detokenize{Muon}\detokenize{.}\allowbreak{}\detokenize{PowerLaw}\detokenize{.}\allowbreak{}\detokenize{cor}\detokenize{_}\allowbreak{}\detokenize{app}\detokenize{_}\allowbreak{}\detokenize{pl}\detokenize{_}\allowbreak{}\detokenize{endpoints}\detokenize{_}\allowbreak{}\detokenize{ii}}, +2 more; foundation results: \texttt{\detokenize{GaussianOpNorm}}}
This corollary is the \Muon part of \cref{tab:power-law-kernels} for $\bmom>0$: it reads
the drift and volatility crossovers in one coordinate $i/i_\star$. Below $i_\star$ the row
keeps the resolved kernels of the large-batch section, $\sqrt2\,\mathfrak d_i\sim\sqrt{\mu_i/(\gamma\sw2)}$
and $\mathfrak v_i\to1/\sw2=B_{\rm eff}/B$; above it both kernels follow the bulk profiles of
$i/N$, while the covariance $\mu_i=c_Ni^{-\alpha}$ is exact throughout. Part (iii) says when
both windows contain rows: always for $\alpha>1$, and for $\alpha<1$ only if the batch
exceeds $N^{1-\alpha}$ up to logarithms; the crossover always lies far below $N$. Part (iv)
records that for $\alpha>1$ the points where the endpoint laws are equal differ from
$i_{\mathrm c}$ only by constants, and that no index depends on the amplitude $c_N$.

\begin{sketch}
Both theorems are equivalences along row sequences, and they come from one
approximation that is uniform over all rows (\cref{thm:app-pl-triple}). For
$\alpha\le1$ the approximants are (\cref{eqn:app-pl-approx-sub})
\[
\sqrt2\,\widehat{\mathfrak d}_i=\sqrt{\frac{\mu_i}{\gamma\sw2}}+a_N\Psi_\alpha\Big(\frac iN\Big),\qquad
\widehat{\mathfrak v}_i=\frac1{\sw2}+\frac{2c_0}{\sqrt{\sw2}}\sqrt{\gamma\Gamma_\alpha\Big(\frac iN\Big)}+\gamma\Gamma_\alpha\Big(\frac iN\Big),
\]
and for $\alpha>1$ they are $b_N\Phi_{\alpha,\bmom}((i_{\mathrm c}/i)^\alpha)$ and $G_{\alpha,\bmom}(i/i_{\mathrm c})$,
each multiplied by a factor that tends to one when $i/N\to0$ and that turns the
mesoscopic law into the bulk profile when $i/N$ stays positive
(\cref{eqn:app-pl-approx-super}). For every sequence $n_N\to\infty$,
\[
\sup_{n_N\le i\le N}\max\Big\{\Big|\frac{\mathfrak d_i(0)}{\widehat{\mathfrak d}_i}-1\Big|,\ \Big|\frac{\mathfrak v_i(0)}{\widehat{\mathfrak v}_i}-1\Big|\Big\}\longrightarrow0 .
\]
The drift half is the head-plus-tail argument of \cref{thm:app-pl-drift}. The volatility
half is the argument of \cref{thm:app-pl-volatility}: the quadratic $\widehat{\mathfrak v}_i$
is the flat square of that sketch, with $b$ replaced by $\sqrt{\gamma\Gamma_\alpha(i/N)}$, and the
two agree on mesoscopic rows because $\Gamma_\alpha(x)\sim x^\alpha\Psi_\alpha(x)^2/c_0^2$
(\cref{lem:app-pl-matching}). A failure of uniformity would produce a row sequence along
which a ratio stays away from one, so the uniform statement follows from the
sequential one.

Parts (i) and (ii) are the two theorems in the coordinate $i/i_\star$. For $\alpha\le1$ the
volatility index is $i_{\mathrm c,\mathfrak v}=c_0^{2/\alpha}i_\star$, so $i/i_\star\to x$ means
$i/i_{\mathrm c,\mathfrak v}\to y=c_0^{-2/\alpha}x$, and then $2c_0y^{\alpha/2}=2x^{\alpha/2}$ and
$y^\alpha=c_0^{-2}x^\alpha$, which is (i). In the table the left column is the resolved case of
both theorems, the two indices differing by a fixed factor. In the right column the drift
is the unresolved case of \cref{thm:app-pl-drift}, and the volatility follows from the
uniform approximation, since $\widehat{\mathfrak v}_i\sim\gamma\Gamma_\alpha(i/N)$ beyond $i_\star$:
for $\alpha\le1$ because $\gamma\Gamma_\alpha(i/N)\to\infty$ there, and for $\alpha>1$ by the end
$G_{\alpha,\bmom}(X)\sim4K_\alpha X/(\alpha-1)$ of the crossover profile.

Parts (iii) and (iv) are computations with the indices. The definitions give
\[
\frac{i_\star}N=\big(\gamma\sw2d^2\log^2\gamma\big)^{-1/\alpha}\ (\alpha<1),\quad
\frac{i_\star}N=\big(\gamma\sw2d^2\log^3\gamma\big)^{-1}\ (\alpha=1),\quad
\frac{i_\star}N=\frac1{K_\alpha\gamma}\ (\alpha>1),
\]
which tend to zero. A row sequence $i\to\infty$ with $i/i_\star\to0$ exists exactly when
$i_\star\to\infty$, and $i=\lfloor\sqrt{i_\star}\rfloor$ realizes it; $i=N$ realizes the other window. The
condition $i_\star\to\infty$ is \cref{eqn:app-pl-windows}. For (iv), equating
$\sqrt{\mu_i/(\gamma\sw2)}$ with $d_\alpha\sqrt{c_NK_\alpha}N^{-1/2}i^{-(\alpha-1)/2}$ gives
$i=i_{\mathrm c}/(\sw2d_\alpha^2)$, and equating $1/\sw2$ with $4K_\alpha^2i/((\alpha-1)B)$ gives
$i=B(\alpha-1)/(4\sw2K_\alpha^2)$; $c_N$ cancels in both.
\end{sketch}

\begin{proposition}[Reference scales]\label{prop:app-pl-scales}
Assume \cref{assum:app-pl-regime}. For $\gamma>1$ let $\resscale\defeq\Lambda_N(1)$ be the
unique solution of
\[
\sum_{i=1}^N\frac{\resscale\mu_i}{1+\resscale\mu_i}=B,
\]
and put $i_\#\defeq(c_N\resscale)^{1/\alpha}$. Then
\begin{equation}\label{eqn:app-pl-kappa}
\begin{gathered}
c_N\resscale\sim\begin{cases}(1-\alpha)BN^{\alpha-1},&\alpha<1,\\ B/\log\gamma,&\alpha=1,\\ (B/K_\alpha)^\alpha,&\alpha>1,\end{cases}
\\
\frac{i_{\mathrm c,\mathfrak d}}{i_\#}\sim\begin{cases}\big(\pi/(2\sw2\log^2\gamma)\big)^{1/\alpha},&\alpha<1,\\ 9\pi/(8\sw2\log^2\gamma),&\alpha=1,\end{cases}
\qquad
\frac{i_{\mathrm c}}{i_\#}\to1\ \ (\alpha>1),
\end{gathered}
\end{equation}
and $i_{\mathrm c,\mathfrak v}/i_\#\sim(\sw2\log^2\gamma)^{-1/\alpha}$ for $\alpha<1$ and
$\sim9/(4\sw2\log^2\gamma)$ for $\alpha=1$.
\end{proposition}

\skpointer{prop:app-pl-scales}{Lean: proved: \texttt{\detokenize{Muon}\detokenize{.}\allowbreak{}\detokenize{PowerLaw}\detokenize{.}\allowbreak{}\detokenize{kappaOf}}, \texttt{\detokenize{Muon}\detokenize{.}\allowbreak{}\detokenize{PowerLaw}\detokenize{.}\allowbreak{}\detokenize{iSharp}}, +3 more}
The scalar $\resscale$ is the level at which the soft count of output directions equals
the batch size: $\sum_i\resscale\mu_i/(1+\resscale\mu_i)=B$, the rank of one gradient. It is the
resolution scalar of the no-momentum row of \cref{tab:power-law-kernels}, and
$i_\#=(c_N\resscale)^{1/\alpha}$ is the row at which $\resscale\mu_i=1$. The proposition compares this
threshold with the crossover indices of the momentum coefficients. For $\alpha>1$ the
momentum crossover sits at the threshold, $i_{\mathrm c}\sim i_\#$. For $\alpha\le1$ it sits
below it by a factor of order $(\log\gamma)^{-2/\alpha}$. Two indices therefore carry the name $i_\#$ in
the main text. The threshold of this proposition governs the no-momentum row of the table. The
display \cref{eqn:i-sharp}, which cites this proposition, has the form of the momentum crossover
index, $i_{\mathrm c,\mathfrak d}$ for $\alpha\le1$ and $i_{\mathrm c}$ for $\alpha>1$, up to constant factors; for
$\alpha\le1$ it is thus smaller than the threshold $(c_N\resscale)^{1/\alpha}$ by the logarithmic factor
above, and for $\alpha>1$ the two agree up to constants. In this chapter $i_\#$ always denotes
$(c_N\resscale)^{1/\alpha}$.

\begin{sketch}
The scalar $\resscale$ is the participation coordinate at $m=1$, and the continuum law of
\cref{sk:E5} evaluates it. By definition $\resscale=\Lambda_N(1)$, so $V\defeq c_N\resscale N^{-\alpha}$
solves $S_N(V)=1/\gamma$. The Riemann sum $S_N$ is within $1/N$ of $S_\alpha$, hence
\[
S_\alpha(V)\in\frac1\gamma\Big[1,1+\frac1B\Big],
\]
so $V\to0$ and the small-$V$ inversion of $S_\alpha$ applies
(\cref{lem:app-pl-finite-inverse,lem:app-pl-continuum}). It gives
$c_N\resscale=N^\alpha V\sim N^\alpha(1-\alpha)/\gamma$, $N/(\gamma\log\gamma)$ and $N^\alpha(K_\alpha\gamma)^{-\alpha}$ in the
three cases, which is the first display.

The ratios are then read off from the definitions of the crossover indices. With
$d^2=c_0^2/(1-\alpha)=2/(\pi(1-\alpha))$ for $\alpha<1$ and $d^2=\frac49c_0^2=8/(9\pi)$ for $\alpha=1$,
\[
\frac{i_{\mathrm c,\mathfrak d}^\alpha}{c_N\resscale}\sim\frac{BN^{\alpha-1}}{\sw2d^2\log^2\gamma\cdot(1-\alpha)BN^{\alpha-1}}=\frac{\pi}{2\sw2\log^2\gamma}\quad(\alpha<1),\qquad
\frac{i_{\mathrm c,\mathfrak d}}{i_\#}\sim\frac{9\pi}{8\sw2\log^2\gamma}\quad(\alpha=1),
\]
while $i_{\mathrm c}=B/K_\alpha\sim i_\#$ for $\alpha>1$, and $i_{\mathrm c,\mathfrak v}=c_0^{2/\alpha}i_{\mathrm c,\mathfrak d}$ gives
the volatility ratios. The logarithm for $\alpha\le1$ is not an effective-batch effect: on
the rows $i\asymp i_\#$ the momentum drift exceeds the no-momentum scale
$(\gamma\resscale)^{-1/2}$ by a factor of order $\log\gamma$, because the momentum participation grows like
$\log Q_m\sim2\abs{\log\bmom}\,m$, and a fixed change of the batch size cannot remove a
logarithm (\cref{rem:app-pl-oldscale}).
\end{sketch}

In the remark, $m(s)\in(0,\gamma)$ is the participation of the branch at the spectral
parameter $-s$, the root of $s=(\gamma-m)m/(Q_m\Lambda_N(m))$ (\cref{prop:app-pl-branch}), and
$\volraw_i(\ridge)=\Lambda_N(m(\ridge))\mu_i/(1+\Lambda_N(m(\ridge))\mu_i)$ is the zero-lag volatility
coefficient of row $i$ at the ridge $\ridge$ (\cref{eqn:app-pl-drift-single}).

\begin{remark}[No momentum]\label{rem:app-pl-nomomentum}
At $\bmom=0$ the weights are $\emaw0=1$, so $\Fw=\Lw=\Jfun$ and
$r_0(q)=q\Jfun(q)$ increases from $0$ to $1$, not to $\infty$; the
participation is confined to $m<1$, and at zero ridge $m(s)\to1$ when $\gamma>1$.
The computation of \cref{prop:app-pl-branch}(ii), which uses only the
equations of the branch, then gives $\Lambda_N(m(s))\to\resscale$, the root of
\cref{prop:app-pl-scales}, instead of infinity, and by
\cref{eqn:app-pl-drift-single}
\[
\lim_{\ridge\downarrow0}\volraw_i(\ridge)=\frac{\resscale\mu_i}{1+\resscale\mu_i}=\frac1{1+(i/i_\#)^\alpha},\qquad
\sum_i\lim_{\ridge\downarrow0}\volraw_i(\ridge)=B .
\]
All lag coefficients with $\ell\ge1$ vanish (\cref{thm:app-vol-mode-kernels}(iv)),
so $\mathfrak v_i=\volraw_i$, and the volatility crossover at $\bmom=0$ is this
exact profile, with critical index $i_\#$ and the equivalents
\cref{eqn:app-pl-kappa}. The drift at $\bmom=0$ is the single integral
\cref{eqn:app-pl-drift-single} with $\Lw/\Fw=1$; its two endpoint laws, on
either side of $i_\#$, are \cref{prop:app-pl-nomomentum-endpoints} below. These
no-momentum profiles are obtained by setting $\bmom=0$ before the ridge and
dimension limits. For every fixed $\bmom>0$,
$r_\bmom$ is unbounded, $\volraw_i(0)=1$ for every row
(\cref{prop:app-pl-drift-zero}), and the profiles of
\cref{thm:app-pl-triple} hold; neither family is a limit of the other under an
exchange of $\bmom\downarrow0$ with the ridge and dimension limits.
\end{remark}

\skpointer{rem:app-pl-nomomentum}{Lean: not formalized (remark)}
This remark explains the no-momentum row of \cref{tab:power-law-kernels} and why it
differs from the rows with momentum. Without momentum the buffer is a single gradient,
whose participation cannot exceed one, so at zero ridge the update whitens only the
directions above the level $1/\resscale$: the zero-lag volatility of row $i$ is the soft
indicator $1/(1+(i/i_\#)^\alpha)$, the coefficients sum to $B$, and there are no lag terms.
With any fixed momentum the participation is unbounded, every row has zero-lag
coefficient one at zero ridge, and the crossover moves to the indices of
\cref{thm:app-pl-drift,thm:app-pl-volatility}. The two families are not limits of
each other, because the limit $\bmom\downarrow0$ does not commute with the ridge and dimension limits.

The constant of the proposition is (\cref{eqn:app-pl-A-alpha})
\[
A_\alpha=c_0\quad(0<\alpha\le1),\qquad
A_\alpha=\frac{2\alpha}{\pi(\alpha+1)}\int_0^\infty q^{-3/2}\big(q\Jfun(q)\big)^{\frac{\alpha+1}2}\,\dif q\quad(\alpha>1).
\]

\begin{proposition}[No-momentum endpoints]\label{prop:app-pl-nomomentum-endpoints}
Let $\bmom=0$, let $\alpha>0$ be fixed, and let $N,B\to\infty$ with $\gamma=N/B\to\infty$,
$\Sigmain=\Id_N$, $\Nout=\Nin=N$ and $\mu_i=i^{-\alpha}$. For $\gamma>1$ let $\resscale$ and
$i_\#=\resscale^{1/\alpha}$ be as in \cref{prop:app-pl-scales}. At each such dimension the
limits $\mathfrak d_i(0)\defeq\lim_{\ridge\downarrow0}\mathfrak d_i(\ridge)$ and
$\mathfrak v_i(0)\defeq\lim_{\ridge\downarrow0}\mathfrak v_i(\ridge)$ exist, and:
\begin{enumerate}[label=(\roman*)]
\item for every $i$,
\[
\mathfrak v_i(0)=\frac{\resscale\mu_i}{1+\resscale\mu_i}=\frac1{1+(i/i_\#)^\alpha},\qquad \sum_{i=1}^N\mathfrak v_i(0)=B ;
\]
\item for every row sequence $1\le i=i_N\le N$,
\[
\frac{i}{i_\#}\to0\ \Longrightarrow\ \mathfrak d_i(0)\sim\sqrt{\frac{\mu_i}{2\gamma}},\qquad
\frac{i}{i_\#}\to\infty\ \Longrightarrow\ \mathfrak d_i(0)\sim A_\alpha\,\mu_i\sqrt{\frac{\resscale}{2\gamma}} .
\]
\end{enumerate}
\end{proposition}

\skpointer{prop:app-pl-nomomentum-endpoints}{Lean: proved: \texttt{\detokenize{Muon}\detokenize{.}\allowbreak{}\detokenize{PowerLaw}\detokenize{.}\allowbreak{}\detokenize{aAlpha}}, \texttt{\detokenize{Muon}\detokenize{.}\allowbreak{}\detokenize{PowerLaw}\detokenize{.}\allowbreak{}\detokenize{mOf0}}, +6 more; foundation results: \texttt{\detokenize{GaussianOpNorm}}}
This proposition gives the no-momentum row of \cref{tab:power-law-kernels} and the
statement of the main text that unresolved modes at $\bmom=0$ have the drift of \textsc{SGD}.
The volatility is exact at every dimension: row $i$ carries the share
$\resscale\mu_i/(1+\resscale\mu_i)$ of the $B$ directions that one gradient spans. The drift has two
endpoint laws on either side of the threshold $i_\#$ of \cref{prop:app-pl-scales}. For
$i\ll i_\#$, $\mathfrak d_i(0)\sim\sqrt{\mu_i/(2\gamma)}$ is the square-root law with $B_{\rm eff}=B$. For
$i\gg i_\#$, $\mathfrak d_i(0)\sim A_\alpha\mu_i\sqrt{\resscale/(2\gamma)}$ is linear in $\mu_i$, with the constant
$A_\alpha=c_0$ for $\alpha\le1$ and the integral \cref{eqn:app-pl-A-alpha} for $\alpha>1$.

\begin{sketch}
The participation coordinate of \cref{sk:E5} applies at $\bmom=0$, with the
participation confined to $(0,1)$. The computation of the branch uses only its
equations, so with $m(s)\in(0,1)$ defined by $s_N(m(s))=s$ and $\Lw/\Fw=1$,
\[
\volraw_i(\ridge)=f_i\big(m(\ridge)\big),\qquad
\mathfrak d_i(\ridge)=\frac{\sqrt2}{\pi\gamma}\int_0^\infty f_i\big(m(\ridge+u^2)\big)\,\dif u .
\]
Now $r_0(q)=q\Jfun(q)$ increases only to one, so $Q_m\to\infty$ as $m\uparrow1$ while $\Lambda_N(m)$
stays bounded, and $s_N(m)\to0$; hence $m(s)\to1$ and $\Lambda_N(m(s))\uparrow\Lambda_N(1)=\resscale$ as
$s\downarrow0$. Without momentum successive buffers are independent, so all lag
coefficients vanish (\cref{thm:app-vol-mode-kernels}(iv)) and
$\mathfrak v_i(\ridge)=\volraw_i(\ridge)\to\resscale\mu_i/(1+\resscale\mu_i)$, whose sum over $i$ is
$Bm_N(\resscale)=B$; this is (i). The drift integrand is nonincreasing in $u$ with total
variation at most one, and an integration by parts gives
$\abs{\mathfrak d_i(\ridge)-\mathfrak d_i(0)}\le\sqrt2\sqrt\ridge/(\pi\gamma)$.

On resolved rows the drift is the head Lorentzian. If $i/i_\#\to0$ then
$\resscale\mu_i\to\infty$, and for a fixed small $\delta$ the range $m\le\delta$ covers all
$u\ge\sqrt{s_N(\delta)}=o(\sqrt{\gamma\mu_i})$. There $\tilde s_1=\gamma(1+O(\delta)+o(1))$ and
\[
\int_{\sqrt{s_N(\delta)}}^\infty\frac{\tilde s_1\mu_i}{u^2+\tilde s_1\mu_i}\,\dif u=\frac\pi2\sqrt{\gamma\mu_i}\,\big(1+O(\delta)+o(1)\big),
\]
which after the factor $\sqrt2/(\pi\gamma)$ and $\delta\downarrow0$ gives $\sqrt{\mu_i/(2\gamma)}$.

On unresolved rows the share is linear in $\mu_i$, and the drift is a fixed integral
over the participation. If $i/i_\#\to\infty$ then $\resscale\mu_i\to0$, and since
$\Lambda_N\le\resscale$ one has $f_i=\Lambda_N\mu_i(1+O(\resscale\mu_i))$ uniformly. Changing variables to
$m$ gives
\[
\frac{\mathfrak d_i(0)}{\mu_i\sqrt{\resscale/(2\gamma)}}=\big(1+o(1)\big)\frac1\pi\int_0^1\sqrt{\frac{\Lambda_N(m)}{\resscale}}\sqrt{1-\frac m\gamma}\sqrt{\frac m{Q_m}}\,\vartheta_N(m)\,\dif m .
\]
For fixed $m$, $\Lambda_N(m)/\resscale\to m^\nu$ and $m\Lambda_N'/\Lambda_N\to\nu$, with $\nu=1$ for
$\alpha\le1$ and $\nu=\alpha$ for $\alpha>1$, by the continuum laws of \cref{sk:E5}. Dominated
convergence then gives an integral in $q=Q_m$ that equals
$\pi^{-1}\int_0^\infty q^{-1/2}\Jfun(q)\,\dif q=\sqrt{2\pi}/\pi=c_0$ when $\nu=1$; for $\nu=\alpha$ an
integration by parts turns it into
$\frac{2\alpha}{\pi(\alpha+1)}\int_0^\infty q^{-3/2}(q\Jfun(q))^{(\alpha+1)/2}\,\dif q=A_\alpha$. This is (ii).
\end{sketch}

In the corollary, $\Qridge\ridge(\Mt)$ is the regularized update and $[\ridge_-,\ridge_+]\subset(0,\infty)$ the
ridge range of \cref{assum:app-setup-regime}, and $\vartheta_\ridge$ is the ratio of the drift
coefficient to its oversampled profile, $\vartheta_\ridge(\mu)=\sqrt{\sw2\gin}\,\mathfrak d_\ridge(\mu)/\mathfrak d^{\rm os}_{\hat\nu,r}(\mu;\ridge/(\sw2\gin))$
for $\mu>0$ and $\vartheta_\ridge(0)=1$ (\cref{thm:app-os-limit}(iv)), applied to $\Sigmaout$ by functional calculus.

\begin{corollary}[Expected regularized update]\label{cor:app-os-update}
Assume \cref{assum:app-setup-regime} with $\Sigmain=\Id_{\Nin}$, and assume that
$\epsilon_{\rm os}=3(1-\bmom^2)\Nin/B\le\bar\epsilon$ along the sequence of dimensions for a
fixed $\bar\epsilon<1$. For $\ridge\in[\ridge_-,\ridge_+]$ let
$\vartheta_\ridge$ be the function of \cref{thm:app-os-limit}(iv) at the
current dimension. Then
\begin{equation}\label{eqn:app-os-update}
\sup_{\ridge\in[\ridge_-,\ridge_+]}\Big\|\EE\,\Qridge\ridge(\Mt)
-\sqrt{\frac{B_{\rm eff}}{\Nin}}\;\vartheta_\ridge(\Sigmaout)\,
\mathfrak d^{\rm os}_{\hat\nu_{\mathrm{out}},r}\Big(\Sigmaout;\frac{\ridge}{\sw2\gin}\Big)\,
\frac{\Dt}{\sqrt{\cR(t)}}\Big\|_\F\longrightarrow0,
\end{equation}
and $\operatorname{spec}\vartheta_\ridge(\Sigmaout)\subset\big[1-\epsilon_{\rm os},(1-\epsilon_{\rm os})^{-1}\big]$
at every dimension and for every $\ridge$.
\end{corollary}

\skpointer{cor:app-os-update}{Lean: proved: \texttt{\detokenize{Muon}\detokenize{.}\allowbreak{}\detokenize{Oversampled}\detokenize{.}\allowbreak{}\detokenize{OsUpdateHyp}}, \texttt{\detokenize{Muon}\detokenize{.}\allowbreak{}\detokenize{Oversampled}\detokenize{.}\allowbreak{}\detokenize{osUpdateTarget}}, +2 more}
This corollary is the precise form of the main-text display
$W_{t+1}-W_t\approx-\eta\sqrt{B_{\rm eff}/(2N)}\,\sqrt{\mathcal K}\,(W_t-W^*)/\sqrt{\fRisk(t)}$ of
\cref{sec:half-aniso-large-batch}. When $\epsilon_{\rm os}=3(1-\bmom^2)\Nin/B$ stays below a fixed $\bar\epsilon<1$, the expected regularized
update is, up to an error of vanishing Frobenius norm, a function of $\Sigmaout$
applied to the normalized error $\Dt/\sqrt{\cR(t)}$: the factor $\sqrt{B_{\rm eff}/\Nin}$
times the oversampled profile $\mathfrak d^{\rm os}$ evaluated at $\Sigmaout$, times an
operator $\vartheta_\ridge(\Sigmaout)$ that is within $\epsilon_{\rm os}/(1-\epsilon_{\rm os})$ of the
identity at every dimension. On packed modes and at zero rescaled ridge the
profile is within constants of $\sqrt{\mu/2}$ (\cref{cor:app-os-packed}), which is the
square-root preconditioner of the main text. The statement holds at ridges in a
fixed compact subset of $(0,\infty)$, in the proportional regime; it says nothing about
the algorithm at zero ridge.

\begin{sketch}
For an isotropic input side the deterministic equivalent of the expected update is
an explicit function of $\Sigmaout$ applied to the error. The drift theorem gives
$\EE\,\Qridge\ridge(\Mt)$ up to a vanishing Frobenius error as an integral, over the spectral
parameters $s=\ridge+u^2$, of a candidate built from the deterministic resolvents
(\cref{thm:app-drift-mean}, \cref{sk:E2}). With $\Sigmain=\Id$ these resolvents are
$(s\Id+\tilde s_1\Sigmaout)^{-1}$ and $(s+s_1)^{-1}\Id$, so the candidate is the scalar kernel
$\mathsf k(s;\cdot)$ evaluated at $\Sigmaout$ and applied to $\Dt/\sqrt{2\cR(t)}$, and the integral
is exactly (\cref{eqn:app-os-Gbar})
\[
\Gtildebar_\ridge=\mathfrak d_\ridge(\Sigmaout)\,\frac{\Dt}{\sqrt{\cR(t)}} .
\]

The oversampled law rewrites $\mathfrak d_\ridge(\Sigmaout)$ at every finite dimension, with a
relative error that depends on nothing but $\epsilon_{\rm os}$. By the sandwich of
\cref{sk:E5}, applied eigenvalue by eigenvalue (\cref{thm:app-os-limit}(iv)),
\[
\mathfrak d_\ridge(\Sigmaout)=\sqrt{\frac{B_{\rm eff}}{\Nin}}\;\vartheta_\ridge(\Sigmaout)\,\mathfrak d^{\rm os}\Big(\Sigmaout;\frac{\ridge}{\sw2\gin}\Big),\qquad
\operatorname{spec}\vartheta_\ridge(\Sigmaout)\subset\big[1-\epsilon_{\rm os},(1-\epsilon_{\rm os})^{-1}\big].
\]
Inserting this into the previous display gives the corollary. The bound on $\vartheta_\ridge$
holds at every dimension, uniformly in the spectrum, the eigenvalue and the ridge,
so the dimension limit of the drift theorem is taken with a factor that is already
controlled and no exchange of limits occurs. Because the error is relative to the
profile, the law is accurate also on modes where $\mathfrak d^{\rm os}(\mu)\ll\sqrt\mu$.

The profile sees the ridge only through the rescaled ridge $\hat\ridge=\ridge/(\sw2\gin)$,
and this decides when the square-root law is visible. The profile is at most
$\min\{\sqrt{\mu/2},\,\mu/\sqrt{2\hat\ridge}\}$ (\cref{lem:app-os-profile}), so at a fixed normalized
ridge and $\gin\downarrow0$ it becomes linear in $\mu$; the square-root law concerns bounded
$\hat\ridge$, in particular the zero-ridge coefficient of \cref{cor:app-os-packed}.
\end{sketch}

\begin{proposition}[Head modes]\label{prop:app-os-head}
\begin{enumerate}[label=(\roman*)]
\item For every probability measure $\nu$, every $\mu>0$ and every $\delta>0$,
\[
\frac{\delta}{2(1+\sqrt\delta)^2}\,\nu\big([\delta\mu,\infty)\big)\le1-\chi_\nu(\mu)\le\delta+\nu\big([\delta\mu,\infty)\big).
\]
\item For every sequence of probability measures $\nu_N$ and numbers
$\mu_N>0$, one has $\mathfrak d^{\rm os}_{\nu_N}(\mu_N)/\sqrt{\mu_N/2}\to1$ if and
only if $\nu_N([\delta\mu_N,\infty))\to0$ for every $\delta>0$.
\end{enumerate}
\end{proposition}

\skpointer{prop:app-os-head}{Lean: proved: \texttt{\detokenize{Muon}\detokenize{.}\allowbreak{}\detokenize{Oversampled}\detokenize{.}\allowbreak{}\detokenize{prop}\detokenize{_}\allowbreak{}\detokenize{app}\detokenize{_}\allowbreak{}\detokenize{os}\detokenize{_}\allowbreak{}\detokenize{head}\detokenize{_}\allowbreak{}\detokenize{i}}, \texttt{\detokenize{Muon}\detokenize{.}\allowbreak{}\detokenize{Oversampled}\detokenize{.}\allowbreak{}\detokenize{prop}\detokenize{_}\allowbreak{}\detokenize{app}\detokenize{_}\allowbreak{}\detokenize{os}\detokenize{_}\allowbreak{}\detokenize{head}\detokenize{_}\allowbreak{}\detokenize{ii}}}
This proposition supports the statement of \cref{sec:half-aniso-large-batch} that for
modes high in the spectrum the constant of the square-root law tends to one. By the
representation of \cref{sk:E5}, $\mathfrak d^{\rm os}_\nu(\mu)/\sqrt{\mu/2}$ is the efficiency
$\chi_\nu(\mu)$, and part (i) bounds its defect by the mass of the spectrum above a
fraction $\delta$ of the mode. Part (ii) says that the efficiency tends to one exactly
for head modes, which lie above all but a vanishing fraction of the spectrum.

\begin{sketch}
The defect $1-\chi_\nu(\mu)$ is within a factor two of an exact average. The average
of the soft count against $\varpi$ has a closed form, by partial fractions in $y=w^2$
(\cref{lem:app-os-average}):
\[
\int_0^\infty\varpi(y)\,\frac{x}{\mu y+x}\,\dif y=\frac{x}{(\sqrt\mu+\sqrt x)^2},\quad
m^\varpi_\nu(\mu)\defeq\int_0^\infty\varpi(y)\,m_\nu(\mu y)\,\dif y=\int\frac{x\,\nu(\dif x)}{(\sqrt\mu+\sqrt x)^2}.
\]
The elementary bounds $m/2\le1-\sqrt{1-m}\le m$ on $[0,1]$, integrated against $\varpi$, give
$\frac12m^\varpi_\nu(\mu)\le1-\chi_\nu(\mu)\le m^\varpi_\nu(\mu)$ (\cref{lem:app-os-defect}).

The averaged count is then bounded by splitting the spectrum at $\delta\mu$. Its integrand
$x/(\sqrt\mu+\sqrt x)^2$ is nondecreasing in $x$, smaller than $\delta$ below $\delta\mu$, at most
one above, and at least $\delta/(1+\sqrt\delta)^2$ on $[\delta\mu,\infty)$, which gives (i). Part (ii)
follows by letting $N\to\infty$ and then $\delta\downarrow0$ in the upper bound, and from the lower
bound at each fixed $\delta$ for the converse.
\end{sketch}

A mode $\mu_i>0$ is $(K,\delta)$-packed, for $K\ge1$ and $\delta\in(0,1]$, if
$\#\{k\le N:\mu_k\le K\mu_i\}\ge\delta N$, as in \cref{def:packed} (\cref{def:app-os-packed}).

\begin{corollary}[Packed modes]\label{cor:app-os-packed}
Assume $\Nout=\Nin=N$ and $\epsilon_{\rm os}<1$, and let $\mathfrak d_i(0)$ be the
zero-ridge coefficient of \cref{thm:app-os-limit}(ii). For every mode with
$\mu_i>0$,
\[
\mathfrak d_i(0)\le\frac1{1-\epsilon_{\rm os}}\sqrt{\frac{B_{\rm eff}}{2N}}\,\sqrt{\mu_i},
\]
and if $\mu_i$ is $(K,\delta)$-packed, then
\[
\mathfrak d_i(0)\ge(1-\epsilon_{\rm os})\,C_*\sqrt{\frac{\delta}{1+K}}\,\sqrt{\frac{B_{\rm eff}}{2N}}\,\sqrt{\mu_i},
\qquad C_*=\frac12+\frac{2\log2}\pi .
\]
\end{corollary}

\skpointer{cor:app-os-packed}{Lean: proved: \texttt{\detokenize{Muon}\detokenize{.}\allowbreak{}\detokenize{Oversampled}\detokenize{.}\allowbreak{}\detokenize{cor}\detokenize{_}\allowbreak{}\detokenize{app}\detokenize{_}\allowbreak{}\detokenize{os}\detokenize{_}\allowbreak{}\detokenize{packed}\detokenize{_}\allowbreak{}\detokenize{i}}, \texttt{\detokenize{Muon}\detokenize{.}\allowbreak{}\detokenize{Oversampled}\detokenize{.}\allowbreak{}\detokenize{cor}\detokenize{_}\allowbreak{}\detokenize{app}\detokenize{_}\allowbreak{}\detokenize{os}\detokenize{_}\allowbreak{}\detokenize{packed}\detokenize{_}\allowbreak{}\detokenize{ii}}}
This corollary is \cref{thm:large-batch} of the main text, whose constant $0.94$ is $C_*$ rounded.
When the effective batch exceeds $3(1+\bmom)^2N$, the drift of every mode is at most the
square-root law $\sqrt{B_{\rm eff}/(2N)}\sqrt{\mu_i}$, up to the factor $(1-\epsilon_{\rm os})^{-1}$, and it is
within a constant of that law whenever a fixed fraction $\delta$ of the spectrum lies below a
fixed multiple $K\mu_i$ of the mode. Both bounds concern the deterministic zero-ridge
coefficient.

\begin{sketch}
The oversampled law reduces the coefficient to the profile. At zero ridge
(\cref{thm:app-os-limit}(iii), \cref{sk:E5}),
\[
\mathfrak d_i(0)=\vartheta_i\sqrt{\frac{B_{\rm eff}}N}\;\mathfrak d^{\rm os}_{\hat\nu}(\mu_i)=\vartheta_i\sqrt{\frac{B_{\rm eff}}N}\sqrt{\frac{\mu_i}2}\;\chi_{\hat\nu}(\mu_i),\qquad
\vartheta_i\in\big[1-\epsilon_{\rm os},(1-\epsilon_{\rm os})^{-1}\big],
\]
and $\chi\le1$ gives the upper bound.

The efficiency is at least $C_*$ times the square root of the soft fraction of the
spectrum below the mode. Since $1-m_\nu$ is nondecreasing, and since
$\mu y/(\mu y+x)\ge y\,\mu/(\mu+x)$ for $y\le1$,
\[
1-m_\nu(\mu y)\ge1-m_\nu(\mu)\quad(y\ge1),\qquad 1-m_\nu(\mu y)\ge y\,\big(1-m_\nu(\mu)\big)\quad(y\le1).
\]
Inserted in the definition of $\chi_\nu$, these give
$\chi_\nu(\mu)\ge\sqrt{1-m_\nu(\mu)}\cdot\frac2\pi\big[(\log2-\frac12)+(\frac\pi4+\frac12)\big]=C_*\sqrt{1-m_\nu(\mu)}$,
the two brackets being $\int_0^1y(1+y)^{-2}\dif y$ and $\int_1^\infty\sqrt y(1+y)^{-2}\dif y$
(\cref{thm:app-os-sqrt}(i)). Packing bounds the soft fraction below: $\mu/(\mu+x)\ge1/(1+K)$ for
$x\le K\mu$, so $1-m_{\hat\nu}(\mu_i)\ge\hat\nu([0,K\mu_i])/(1+K)\ge\delta/(1+K)$
(\cref{thm:app-os-sqrt}(ii)), which gives the lower bound.
\end{sketch}

\begin{remark}[Where packing holds]\label{rem:app-os-packed}
The count $\#\{k:\mu_k\le K\mu\}$ is nondecreasing in $\mu$, so if the smallest
positive eigenvalue is $(K,\delta)$-packed, then every mode is. If all
eigenvalues lie in $[\mu_{\min},\kappa\mu_{\min}]$ with $\mu_{\min}>0$, every mode is
$(\kappa,1)$-packed, since $\mu_k\le\kappa\mu_{\min}\le\kappa\mu_i$ for all $k$. If the spectrum
consists of clusters of equal eigenvalues, each holding at least $pN$ of them, every
mode is $(1,p)$-packed, since its own cluster lies in $[0,\mu_i]$. For the geometric
spectrum $\mu_k=q^k$ with fixed $q\in(0,1)$, the smallest eigenvalue $\mu_N$ has only
$1+\lfloor\log K/\log(1/q)\rfloor$ eigenvalues in $[0,K\mu_N]$, a vanishing fraction as $N\to\infty$, so it is
not $(K,\delta)$-packed for any fixed $K,\delta$ once $N$ is large. For the power law
$\mu_k=c_Nk^{-\alpha}$, $1\le k\le N$, every mode is $(2,1-2^{-1/\alpha})$-packed, because
$\mu_k\le2\mu_i$ for every $k\ge2^{-1/\alpha}i$, and there are at least $N-2^{-1/\alpha}i\ge(1-2^{-1/\alpha})N$
such $k$. Conversely, packing is necessary for a lower bound of this form:
by \cref{prop:app-os-tail}(vi), along a sequence of spectra the efficiency
$\chi_{\nu_N}(\mu_N)$ stays bounded below only if, for some $K$, the fraction of
the spectrum in $[0,K\mu_N]$ stays bounded below, and by
\cref{prop:app-os-tail}(ii),(iii) a mode far below most of the spectrum
receives drift of gradient-descent order up to a logarithm.
\end{remark}

\skpointer{rem:app-os-packed}{Lean: proved: \texttt{\detokenize{Muon}\detokenize{.}\allowbreak{}\detokenize{Oversampled}\detokenize{.}\allowbreak{}\detokenize{rem}\detokenize{_}\allowbreak{}\detokenize{app}\detokenize{_}\allowbreak{}\detokenize{os}\detokenize{_}\allowbreak{}\detokenize{packed}\detokenize{_}\allowbreak{}\detokenize{i}}, \texttt{\detokenize{Muon}\detokenize{.}\allowbreak{}\detokenize{Oversampled}\detokenize{.}\allowbreak{}\detokenize{rem}\detokenize{_}\allowbreak{}\detokenize{app}\detokenize{_}\allowbreak{}\detokenize{os}\detokenize{_}\allowbreak{}\detokenize{packed}\detokenize{_}\allowbreak{}\detokenize{ii}}}
This remark supports the examples after \cref{def:packed} in the main text. Packing is
monotone in the mode, so it is enough to check it at the smallest eigenvalue; bounded
condition numbers and spectra made of large clusters are packed with explicit constants,
and every mode of a power law is $(2,1-2^{-1/\alpha})$-packed, because at least that fraction of
the indices lies beyond $2^{-1/\alpha}i$. A geometric spectrum is not packed at its smallest
eigenvalue, since only boundedly many eigenvalues lie within a factor $K$ of it. Packing is
also necessary for a lower bound of the form of \cref{cor:app-os-packed}: by
\cref{prop:app-os-tail}(vi) the efficiency stays bounded below only if a fixed fraction of
the spectrum lies below a fixed multiple $K\mu$ of the mode.

For a probability measure $\nu$, $\ell\ge0$, $\mu>0$ and $\Gamma\ge0$ put
\[
\mathcal I_\nu(\ell)=\int_{(\ell,\infty)}\frac{\nu(\dif x)}x,\qquad \mathcal I_\nu=\mathcal I_\nu(0),\qquad
\mathcal E(\mu,\Gamma)=\frac{\sqrt2}\pi\,\mu\int_0^\infty\frac{\min\{\sqrt c,\,c\sqrt\Gamma\}}{(c+\mu)^2}\,\dif c,
\]
and let $C_*=\frac12+\frac{2\log2}\pi$ (\cref{lem:app-os-tailint,thm:app-os-sqrt}).

\begin{proposition}[Tail modes]\label{prop:app-os-tail}
Let $\nu$ be a probability measure and $\mu>0$.
\begin{enumerate}[label=(\roman*)]
\item (Split bound.) For every $\ell\ge0$ with $\mathcal I_\nu(\ell)<\infty$,
\[
\mathfrak d^{\rm os}_\nu(\mu)\le\sqrt{\nu([0,\ell])}\,\sqrt{\frac\mu2}+\mathcal E\big(\mu,\mathcal I_\nu(\ell)\big).
\]
\item (Gradient-descent-like upper bound.) If $\nu(\{0\})=0$ and $\mathcal I_\nu<\infty$, then
\[
\mathfrak d^{\rm os}_\nu(\mu)\le\frac{\sqrt2}{\pi}\sqrt{\mathcal I_\nu}\;\mu\Big[\log\Big(1+\frac1{\mu\mathcal I_\nu}\Big)+2\Big].
\]
\item (Asymptotics.) If $\nu(\{0\})=0$ and $\mathcal I_\nu<\infty$, then, for fixed $\nu$,
\[
\mathfrak d^{\rm os}_\nu(\mu)\sim\frac{\sqrt2}{\pi}\sqrt{\mathcal I_\nu}\;\mu\log\frac1\mu\qquad(\mu\downarrow0).
\]
\item (Counts.) For every $K>0$,
\[
C_*\sqrt{\frac{\nu([0,K\mu])}{1+K}}\le\chi_\nu(\mu)\le\sqrt{\nu([0,K\mu])}+\frac{2}{\pi\sqrt K}\big[\log(1+K)+2\big].
\]
\item (Universal lower bound.) If $\nu([0,X])=1$ for some $X>0$, then
\[
\mathfrak d^{\rm os}_\nu(\mu)\ge\frac1\pi\,\frac{\mu}{\sqrt X}\Big[\log\Big(1+\frac X\mu\Big)-\frac{X}{X+\mu}\Big]\ge\frac1{2\pi}\,\frac{\mu}{\sqrt X}\,\log\frac X\mu .
\]
\item (Sequences.) For probability measures $\nu_N$ and numbers $\mu_N>0$ put
$\chi_N\defeq\chi_{\nu_N}(\mu_N)$. Then $\liminf_N\chi_N>0$ if and only if
$\liminf_N\nu_N([0,K\mu_N])>0$ for some $K>0$; and $\chi_N\to0$ if and only if
$\nu_N([0,K\mu_N])\to0$ for every $K>0$.
\end{enumerate}
\end{proposition}

\skpointer{prop:app-os-tail}{Lean: proved: \texttt{\detokenize{Muon}\detokenize{.}\allowbreak{}\detokenize{Oversampled}\detokenize{.}\allowbreak{}\detokenize{prop}\detokenize{_}\allowbreak{}\detokenize{app}\detokenize{_}\allowbreak{}\detokenize{os}\detokenize{_}\allowbreak{}\detokenize{tail}\detokenize{_}\allowbreak{}\detokenize{i}}, \texttt{\detokenize{Muon}\detokenize{.}\allowbreak{}\detokenize{Oversampled}\detokenize{.}\allowbreak{}\detokenize{prop}\detokenize{_}\allowbreak{}\detokenize{app}\detokenize{_}\allowbreak{}\detokenize{os}\detokenize{_}\allowbreak{}\detokenize{tail}\detokenize{_}\allowbreak{}\detokenize{ii}}, +5 more}
This proposition supports the last sentence of \cref{sec:half-aniso-large-batch}: a mode
that is not packed, such as an isolated small eigenvalue, can be as slow as under
gradient descent, up to a logarithm. If the spectrum has a finite mean inverse
$\mathcal I_\nu$ and no atom at zero, the profile of a small mode is of order
$\mu\log(1/\mu)$ instead of $\sqrt\mu$, by (ii) and (iii); by (v) it is never smaller than that
for a bounded spectrum. Part (iv) brackets the efficiency by the fraction of the spectrum
in $[0,K\mu]$, and (vi) turns this into the criterion along sequences of spectra.

\begin{sketch}
All bounds come from the first form of the representation of \cref{sk:E5},
\[
\mathfrak d^{\rm os}_\nu(\mu)=\frac{\sqrt2}\pi\,\mu\int_0^\infty\frac{\sqrt{\Upsilon_\nu(c)}}{(c+\mu)^2}\,\dif c,\qquad
\Upsilon_\nu(c)=c\big(1-m_\nu(c)\big)=c^2\int\frac{\nu(\dif x)}{c+x},
\]
in which the mode enters only through the kernel $\mu(c+\mu)^{-2}$, which concentrates
near $c=\mu$. Splitting the spectrum at $\ell$ and using $c/(c+x)\le c/x$ above $\ell$ gives
$1-m_\nu(c)\le\nu([0,\ell])+c\,\mathcal I_\nu(\ell)$, hence
$\sqrt{\Upsilon_\nu(c)}\le\sqrt{c\,\nu([0,\ell])}+\min\{\sqrt c,c\sqrt{\mathcal I_\nu(\ell)}\}$, and the two terms give
the two terms of (i). At $\ell=0$ the first term vanishes, and splitting the remaining
integral at $c=1/\mathcal I_\nu$, where the minimum changes branch, gives the logarithm of
(ii) (\cref{lem:app-os-tailint}).

The matching lower bounds use that $\Upsilon_\nu(c)/c^2$ is nonincreasing in $c$. Below a fixed
$c_1$ this gives $\sqrt{\Upsilon_\nu(c)}\ge c\sqrt{\Upsilon_\nu(c_1)}/c_1$, and the integral of
$c(c+\mu)^{-2}$ over $(0,c_1)$ is at least $\log(1+c_1/\mu)-1$; letting $\mu\downarrow0$ and then $c_1\downarrow0$,
so that $\Upsilon_\nu(c_1)/c_1^2\uparrow\mathcal I_\nu$, gives (iii). If the spectrum lies in $[0,X]$, then
$\Upsilon_\nu(c)\ge c^2/(2X)$ for $c\le X$, and the elementary inequality
$\log(1+k)-k/(1+k)\ge\frac12\log k$ gives (v) (\cref{lem:app-os-log}). Part (iv) combines the
square-root law of \cref{cor:app-os-packed} (\cref{thm:app-os-sqrt}) with (i) at $\ell=K\mu$, where
$\mathcal I_\nu(K\mu)\le1/(K\mu)$, and (vi) follows from (iv) by choosing $K$ large.
\end{sketch}

For a probability measure $\nu$ with $\nu([0,\mu_+])=1$ and a mode $0<\mu\le\mu_+/2$, the
acceleration of \cref{cor:app-os-accel} compares the profile with the gradient-descent
rate $\mu/\mu_+$, and $\varkappa$ is the condition number:
\[
\mathcal A=\frac{\mathfrak d^{\rm os}_\nu(\mu)/\mathfrak d^{\rm os}_\nu(\mu_+)}{\mu/\mu_+},\qquad \varkappa=\frac{\mu_+}\mu,\qquad C_*=\frac12+\frac{2\log2}\pi .
\]

\begin{corollary}[Power-law spectra]\label{cor:app-os-powerlaw}
Let $\alpha>0$, $N\ge1$ and $c_N>0$, let $\mu_k=c_Nk^{-\alpha}$ for
$1\le k\le N$, and put $\nu_N\defeq N^{-1}\sum_{k=1}^N\delta_{\mu_k}$. For
$0<x\le1$ put
\[
\Xi_\alpha(x)\defeq\int_0^\infty\varpi(y)\Big(\int_0^1\frac{y\,u^\alpha}{y\,u^\alpha+x^\alpha}\,\dif u\Big)^{1/2}\dif y .
\]
\begin{enumerate}[label=(\roman*)]
\item For every $1\le i\le N$,
\[
0\le\chi_{\nu_N}(\mu_i)-\Xi_\alpha(i/N)\le N^{-1/2}.
\]
\item If $i\le\varepsilon N$ with $\varepsilon\in(0,1]$, then
$1-2\varepsilon^{\alpha/(1+\alpha)}\le\chi_{\nu_N}(\mu_i)\le1$. Consequently, for every
sequence $\varepsilon_N\to0$,
\[
\max_{i\le\varepsilon_NN}\Big|\frac{\mathfrak d^{\rm os}_{\nu_N}(\mu_i)}{\sqrt{\mu_i/2}}-1\Big|\longrightarrow0 .
\]
\item The function $\Xi_\alpha$ is nonincreasing on $(0,1]$, and
\[
1-\Xi_\alpha(x)\le2x^{\alpha/(1+\alpha)},\qquad \Xi_\alpha(x)\ge\frac{C_*}{\sqrt{2(1+\alpha)}} .
\]
\item For every $1\le i\le N$,
\[
\frac{C_*}{\sqrt{2(1+\alpha)}}\sqrt{\frac{\mu_i}2}\le\mathfrak d^{\rm os}_{\nu_N}(\mu_i)\le\sqrt{\frac{\mu_i}2}.
\]
\item If $N^\alpha\ge2$, the acceleration of \cref{cor:app-os-accel} for
$\nu=\nu_N$, $\mu_+=\mu_1$ and $\mu=\mu_N$, for which $\varkappa=N^\alpha$, satisfies
\[
\frac{C_*}{\sqrt{2(1+\alpha)}}\,N^{\alpha/2}\le\mathcal A\le\frac{\sqrt2}{C_*}\,N^{\alpha/2}.
\]
\end{enumerate}
\end{corollary}

\skpointer{cor:app-os-powerlaw}{Lean: proved: \texttt{\detokenize{Muon}\detokenize{.}\allowbreak{}\detokenize{Oversampled}\detokenize{.}\allowbreak{}\detokenize{plMu}}, \texttt{\detokenize{Muon}\detokenize{.}\allowbreak{}\detokenize{Oversampled}\detokenize{.}\allowbreak{}\detokenize{plSpec}}, +6 more}
This corollary is the large-batch paragraph of \cref{sec:power_law_main}. For a power law
every mode keeps the square-root law up to the constant $C_*/\sqrt{2(1+\alpha)}$, which depends
only on $\alpha$, by (iv); the modes $i\ll N$ have efficiency tending to one, by (ii), which with
the factor $\vartheta_i\to1$ of the oversampled law as $\epsilon_{\rm os}\to0$ gives
$\mathfrak d_i\sim\sqrt{\mu_iB_{\rm eff}/(2N)}$ (\cref{thm:app-os-limit}). By (v) the slowest mode
contracts faster than under gradient descent, relative to the top mode, by a factor of order
$N^{\alpha/2}$, the square root of the condition number. The efficiency itself has the limit
profile $\Xi_\alpha(i/N)$ with error at most $N^{-1/2}$, by (i).

\begin{sketch}
After rescaling the spectrum, the soft fraction is a Riemann sum of a monotone function.
The efficiency is unchanged by $\mu_k\mapsto a\mu_k$ (\cref{lem:app-os-scale}), so one may take
$\mu_k=(k/N)^{-\alpha}$. Then
\[
1-m_{\nu_N}(c)=\frac1N\sum_{k=1}^N\phi_c(k/N),\qquad \phi_c(u)=\frac{c\,u^\alpha}{c\,u^\alpha+1},
\]
and since $\phi_c$ is nondecreasing with values in $[0,1)$ the sum exceeds $\int_0^1\phi_c$ by at
most $1/N$. With $\abs{\sqrt a-\sqrt b}\le\sqrt{\abs{a-b}}$ and $\int\varpi=1$ this gives (i), and
$\Xi_\alpha(x)$ is the efficiency of the law of $U^{-\alpha}$, $U$ uniform on $(0,1)$, at the level $x^{-\alpha}$.

The remaining parts are the head and packing bounds applied to this spectrum. The
fraction of the spectrum above $\delta\mu_i$ is at most $\varepsilon\delta^{-1/\alpha}$ when $i\le\varepsilon N$, and
\cref{prop:app-os-head}(i) with $\delta=\varepsilon^{\alpha/(1+\alpha)}$ gives
$1-\chi_{\nu_N}(\mu_i)\le2\varepsilon^{\alpha/(1+\alpha)}$, which is (ii) and the first bound of (iii). For the
limit law $1-m(x^{-\alpha})=\int_0^1u^\alpha(u^\alpha+x^\alpha)^{-1}\dif u\ge1/(2(1+\alpha))$, and the bound
$\chi\ge C_*\sqrt{1-m}$ of \cref{thm:app-os-sqrt}(i) (sketch of \cref{cor:app-os-packed}) gives the lower bound of (iii) and, with (i),
of (iv). Part (v) is \cref{cor:app-os-accel}(i), which brackets the acceleration by
$\chi_\nu(\mu)\sqrt\varkappa$ and $(\sqrt2/C_*)\chi_\nu(\mu)\sqrt\varkappa$, at the slowest mode, where $\varkappa=N^\alpha$ and $\chi_{\nu_N}(\mu_N)$ lies between $C_*/\sqrt{2(1+\alpha)}$ and one by (iv).
\end{sketch}

\section{The phase map and time to accuracy}\label{sk:ch-phase}
In this chapter the kernels are fixed numbers. The deterministic row model
(\cref{def:app-phase-model}) runs the row recursion \cref{eqn:ri-recursion} in continuous
time at a constant step $\lr$, with weights $\mu_i>0$, drifts $\mathfrak d_i>0$, volatilities
$\mathfrak v_i\ge0$ and initial row energies $\fQ_i(0)$, $i\le N$; its risk is
$\fRisk=\frac12\sum_i\mu_i\fQ_i$, $F_0=\fRisk(0)$, and $\mathcal S$ is as in \cref{eqn:eta-star-tte}.
The coefficients are modelled on the zero-ridge kernels of the main text and are not
proved to be those of the algorithm along the trajectory (\cref{rem:app-phase-adoption}),
so every statement is about the model. The matched step $2\sqrt\epsilon/\mathcal S$ makes the
floor $\fRisk(\infty)$ equal to $\epsilon$, and $t_a=\inf\{t:\fRisk(t)\le a\}$. The power-law arrays
(\cref{def:app-phase-arrays}) have $\mu_i=i^{-\alpha}$, $\fQ_i(0)=i^{-\texp}$, $p=\alpha+\texp>1$,
$B\ge N$, $\gamma=N/B$, $H_N(a)=\sum_{i\le N}i^{-a}$, $\bar\mu=H_N(\alpha)/N$, and
\[
\mathfrak d_i^{\text{s-SVD}}=\sqrt{\mu_iB/N},\quad \mathfrak v_i^{\text{s-SVD}}=1;\qquad
\mathfrak d_i^{\text{s-SGD}}=\mu_i\sqrt{B/(\pi\bar\mu)},\quad
\mathfrak v_i^{\text{s-SGD}}=N\big[1+\tfrac2\pi\big(\tfrac{\mu_i}{\bar\mu}-1\big)\big],
\]
the large-batch \Muon kernels without momentum (\SignSVD; the adopted drift is $\sqrt2$ times the
main text's $\sqrt{\mu_iB/(2N)}$, and with the main text's value every ratio of hitting times
below is halved and no exponent changes, \cref{rem:app-phase-adoption}) and the kernels of
\cref{tab:power-law-kernels}. For $j\in\{\text{s-SVD},\text{s-SGD}\}$, $\mathcal S_j$ is the noise
constant, $T_0^{\text{s-SVD}}=\sqrt\gamma$, $T_0^{\text{s-SGD}}=\sqrt{\pi\bar\mu/B}$, $q_{\text{s-SVD}}=\alpha/2$,
$q_{\text{s-SGD}}=\alpha$, $r_j=(p-1)/(2q_j)$, and $D_N=\pi H_N(\alpha)/H_N(\alpha/2)$. The chapter solves
the model when all drifts are equal, then gives the power-law hitting times and their
ratio in windows of dimension and accuracy, and finally at every accuracy.

In the setting of \cref{prop:app-phase-common} the index set is finite and every row
has the same drift, $\mathfrak d_i=\mathfrak d>0$. The total volatility is
$\mathfrak v_{\rm tot}=\sum_i\mu_i\mathfrak v_i>0$, the volatility $\mathfrak v$ of
\cref{eqn:comparison-general-form}, and $\mathcal S=\mathfrak v_{\rm tot}/(2\mathfrak d)$.

\begin{theorem}[Matched floors]\label{thm:app-phase-matched}
In the setting of \cref{prop:app-phase-common}, let $\epsilon>0$, $T>1$ and
$F_0>T\epsilon$, and take the matched step $\lr=2\sqrt\epsilon/\mathcal S=4\mathfrak d\sqrt\epsilon/\mathfrak v_{\rm tot}$.
Put $X\defeq\sqrt{F_0/\epsilon}>\sqrt T$. Then the floor is $\epsilon$ and
\begin{equation}\label{eqn:app-phase-iso-time}
t_{T\epsilon}=\frac{\mathfrak v_{\rm tot}}{4\mathfrak d^2}\Big[X-\sqrt T+\log\frac{X-1}{\sqrt T-1}\Big].
\end{equation}
If $x\defeq1/X\le\frac12$, then
\begin{equation}\label{eqn:app-phase-iso-error}
\Big|\frac{t_{T\epsilon}}{\frac{\mathfrak v_{\rm tot}}{4\mathfrak d^2}\sqrt{F_0/\epsilon}}-1\Big|\le x\Big[\sqrt T+\abs{\log(\sqrt T-1)}+\log\frac1x+2x\Big].
\end{equation}
For two such models with the same $F_0$, $\epsilon$ and $T$,
\[
\frac{t^{(1)}_{T\epsilon}}{t^{(2)}_{T\epsilon}}=\frac{\mathfrak v^{(1)}_{\rm tot}/(\mathfrak d^{(1)})^2}{\mathfrak v^{(2)}_{\rm tot}/(\mathfrak d^{(2)})^2} .
\]
\end{theorem}

\skpointer{thm:app-phase-matched}{Lean: proved: \texttt{\detokenize{Muon}\detokenize{.}\allowbreak{}\detokenize{Phase}\detokenize{.}\allowbreak{}\detokenize{thm}\detokenize{_}\allowbreak{}\detokenize{app}\detokenize{_}\allowbreak{}\detokenize{phase}\detokenize{_}\allowbreak{}\detokenize{matched}}}
This theorem is the constant-learning-rate part of the main text's statement, after
\cref{eqn:rate-signsgd}, that the \Muon speedup is also the ratio of the times to a common
floor. With a common drift the model closes on the risk, as the recursion
\cref{eqn:comparison-general-form} in continuous time, and at the matched step the time to
$T\epsilon$ is explicit. Its leading term is $\sqrt{F_0/\epsilon}$ times
$\mathfrak v_{\rm tot}/(4\mathfrak d^2)$, half the inverse of the rate $2\mathfrak d^2/\mathfrak v$, so
the ratio of the times of two models is the inverse ratio of their rates, exactly and at
every accuracy.

\begin{sketch}
With a common drift the risk solves a scalar equation, which integrates in closed form:
summing the row equations with weights $\mu_i/2$ turns the drift term into $-2\lr\mathfrak d\sqrt{\fRisk}$, so exactly
(\cref{prop:app-phase-common})
\[
\frac{\dif\fRisk}{\dif t}=-2\lr\mathfrak d\sqrt{\fRisk}+\frac{\lr^2\mathfrak v_{\rm tot}}2 .
\]
In $\mathsf u=\sqrt{\fRisk}$ this reads $\mathsf u'=-a_\circ(\mathsf u-b_\circ)/\mathsf u$, with
$a_\circ=\lr\mathfrak d$ and $b_\circ=\lr\mathcal S/2$. The difference $\mathsf u-b_\circ$ keeps its sign,
so for $\sqrt{F_0}>b_\circ$ the risk decreases to $b_\circ^2$ without reaching it, and
separating variables gives (\cref{eqn:app-phase-implicit})
\[
a_\circ t=\sqrt{F_0}-\sqrt{\fRisk(t)}+b_\circ\log\frac{\sqrt{F_0}-b_\circ}{\sqrt{\fRisk(t)}-b_\circ} .
\]

At the matched step the closed form evaluated at the level $T\epsilon$ is the time. There
$b_\circ=\sqrt\epsilon$ and $a_\circ=4\mathfrak d^2\sqrt\epsilon/\mathfrak v_{\rm tot}$, and since
$F_0>T\epsilon>b_\circ^2$, setting $\sqrt{\fRisk}=\sqrt{T\epsilon}$ gives \cref{eqn:app-phase-iso-time}.
With $X-1=X(1-x)$ the exact relative error is
\[
\frac{t_{T\epsilon}}{\frac{\mathfrak v_{\rm tot}}{4\mathfrak d^2}X}-1
=x\Big[-\sqrt T+\log\frac1x+\log(1-x)-\log(\sqrt T-1)\Big],
\]
and $\abs{\log(1-x)}\le2x$ for $x\le\frac12$ gives \cref{eqn:app-phase-iso-error}. The bracket in
\cref{eqn:app-phase-iso-time} depends only on $F_0$, $\epsilon$ and $T$, so it cancels in a ratio.
\end{sketch}

Let $\mathcal N_B=\EE\big[(\sum_{b\le B}x_b^2y_b^2)^{1/2}\big]$ for independent standard Gaussian
$x_b,y_b$. The isotropic arrays \cref{eqn:app-phase-iso-arrays} have $N$ rows with $\mu_i=1$,
a common initial energy, and
\[
\mathfrak d_i^{\text{s-SVD}}=C(\gamma),\quad \sum_i\mathfrak v_i^{\text{s-SVD}}=m;\qquad
\mathfrak d_i^{\text{s-SGD}}=\mathcal N_B/\sqrt\pi,\quad \sum_i\mathfrak v_i^{\text{s-SGD}}=N^2,
\]
where $C(\gamma)=C_0(\gamma)$ is the zero-ridge drift constant without momentum of
\cref{sk:ch-isotropic}. Put $\rho(\gamma)=\pi\,C(\gamma)^2\,\gamma\max(1,\gamma)$.

\begin{theorem}[Isotropic comparison]\label{thm:app-phase-iso}
Let $N,B\ge1$, $\gamma=N/B$, $m=\min(B,N)$, and take the isotropic arrays
\cref{eqn:app-phase-iso-arrays} with a common initial risk
$F_0=\fRisk^{\text{s-SVD}}(0)=\fRisk^{\text{s-SGD}}(0)$. Choose
$\lr^\star_{\text{s-SVD}}$ and $\lr^\star_{\text{s-SGD}}$ by
\cref{eqn:eta-star-tte}, so that both models have floor $\epsilon$, and fix
$T>1$ with $F_0>T\epsilon$.
\begin{enumerate}[label=(\roman*)]
\item For every such $\epsilon$,
$t^{\text{s-SGD}}_{T\epsilon}/t^{\text{s-SVD}}_{T\epsilon}=\rho(\gamma)\,B/\mathcal N_B^2$.
As $\epsilon/F_0\to0$, uniformly in $N$, $B$ and $C(\gamma)$,
\[
t^{\text{s-SVD}}_{T\epsilon}\sim\frac m{4C(\gamma)^2}\sqrt{\frac{F_0}\epsilon},\qquad
t^{\text{s-SGD}}_{T\epsilon}\sim\frac{\pi N^2}{4\mathcal N_B^2}\sqrt{\frac{F_0}\epsilon} .
\]
\item If in addition $B\to\infty$, with $N$ and $\gamma$ allowed to vary, then
\[
t^{\text{s-SGD}}_{T\epsilon}\sim\frac{\pi N^2}{4B}\sqrt{\frac{F_0}\epsilon},\qquad
\frac{t^{\text{s-SGD}}_{T\epsilon}}{t^{\text{s-SVD}}_{T\epsilon}}\Big/\rho(\gamma)\longrightarrow1 ,
\]
and the ratio tends to $\rho(\gamma_0)$ if $\gamma\to\gamma_0\in(0,\infty)$.
\item For every $\gamma>0$,
\[
\rho(1)\le\rho(\gamma)<\frac{32}{9\pi},\qquad \rho(1)\le\frac{t^{\text{s-SGD}}_{T\epsilon}}{t^{\text{s-SVD}}_{T\epsilon}}\le\frac{\pi^2}4\cdot\frac{32}{9\pi},
\]
with $\rho(\gamma)=\rho(1)$ only at $\gamma=1$; and $\rho(\gamma)\to\frac{32}{9\pi}$ as
$\gamma\downarrow0$ and $\rho(\gamma)\to1$ as $\gamma\to\infty$.
\end{enumerate}
\end{theorem}

\skpointer{thm:app-phase-iso}{Lean: proved: \texttt{\detokenize{Muon}\detokenize{.}\allowbreak{}\detokenize{Phase}\detokenize{.}\allowbreak{}\detokenize{thm}\detokenize{_}\allowbreak{}\detokenize{app}\detokenize{_}\allowbreak{}\detokenize{phase}\detokenize{_}\allowbreak{}\detokenize{iso}}}
At $\bmom=0$ this is the precise form of the statement, after \cref{eqn:rate-signsgd}, that
the \Muon speedup is the ratio of the times to a common floor. The identity in (i) is exact: the
ratio is $\rho(\gamma)$ times $B/\mathcal N_B^2\in[1,\pi^2/4]$, which tends to one at large batch. By (iii), $\rho$ has its minimum
$\rho(1)=\frac14(1+\log(1+\sqrt2)/\sqrt2)^2\approx0.6587$ at the square aspect ratio, where \SignSGD is
faster, tends to $1$ at small batch ($\gamma\to\infty$), and has supremum $32/(9\pi)\approx1.13$ at
large batch ($\gamma\downarrow0$), the values of \cref{eqn:rate-signsgd} at $\bmom=0$. With momentum the
same algebra applies to the pair $(C_\bmom(\gamma),\mathfrak v_{\mathrm{iso},\bmom})$.

\begin{sketch}
Each isotropic array has a common drift, so by \cref{thm:app-phase-matched} and $N^2/(mB)=\gamma\max(1,\gamma)$ (\cref{prop:app-phase-iso-ratio}),
\[
\frac{t^{\text{s-SGD}}_{T\epsilon}}{t^{\text{s-SVD}}_{T\epsilon}}
=\frac{N^2/(\mathcal N_B^2/\pi)}{m/C(\gamma)^2}=\rho(\gamma)\,\frac B{\mathcal N_B^2}.
\]
The Gaussian bounds $\max\{\frac2\pi\sqrt B,\sqrt B-4/\sqrt B\}\le\mathcal N_B\le\sqrt B$
(\cref{prop:app-sgd-batch-coefficient}) give $1\le B/\mathcal N_B^2\le\pi^2/4$, and $B/\mathcal N_B^2\le(1-4/B)^{-2}$ for $B>4$. The
relative error in \cref{thm:app-phase-matched} depends only on $\epsilon/F_0$ and $T$, which gives
(i) and (ii).

The shape of $\rho$ comes from an integral for $C(\gamma)$ over the participation level. Let $\Phi(v)=\EE[vZ^2/(1+vZ^2)]$ with $Z$ standard Gaussian, and $V$ its
inverse. The square-case scalar equation of the isotropic analysis gives the spectral
parameter as a function of the level $y$, $t=(\gamma-y)^2/V(y)$. Changing variables from $t$ to $y$
in the half moment that defines $C(\gamma)$ and integrating by parts gives
(\cref{lem:app-phase-ratio-integral})
\[
C(\gamma)=\frac{\sqrt2\,I(\gamma)}{\pi\gamma^2},\qquad
I(\gamma)=\int_0^{\min(\gamma,1)}(\gamma-y)\,w(y)\,\dif y,\qquad w=V^{-1/2},
\]
where $w_\circ(y)=\sqrt y\,w(y)$ decreases strictly from $1$. For $\gamma\ge1$,
$I(\gamma)=\gamma\mathcal A-\mathcal D$ with $\mathcal A=\int_0^1w=\sqrt{\pi/2}$ and $\mathcal D=\int_0^1yw$, so
$\rho(\gamma)=(1-\mathcal D/(\mathcal A\gamma))^2$ increases to $1$. For $\gamma<1$, $\rho=\frac2\pi\gamma^{-3}I^2$, whose
derivative has the sign of $2\gamma I'-3I$; this is negative, because the constant
$w_\circ(\gamma/3)$ integrates to zero against $(3y-\gamma)y^{-1/2}$ and $w_\circ$ is decreasing:
\[
2\gamma I'(\gamma)-3I(\gamma)=\int_0^\gamma(3y-\gamma)\,y^{-1/2}\big(w_\circ(y)-w_\circ(\gamma/3)\big)\,\dif y<0 .
\]
As $\gamma\downarrow0$, $\gamma^{-3/2}I(\gamma)\to\int_0^1(1-y)y^{-1/2}\dif y=\frac43$, so $\rho\to\frac{32}{9\pi}$
(\cref{thm:app-phase-iso-extrema}); the closed form of $C(1)$ (\cref{thm:app-iso-one}) gives $\rho(1)$.
\end{sketch}

The low-exponent window \cref{eqn:app-phase-height} is the set of $(N,B,\epsilon)$ with
$N\ge4$, $B\ge N$, $0<\epsilon<\epsilon_0$ and, for both $j$,
\[
k_{N,j}\,C_j\,I_{r_j}\big(H_j(\epsilon)\big)\le\sqrt\epsilon,\qquad\text{and}\qquad
H_j(\epsilon)\le(N/4)^{q_j}\ \text{ if }r_j\le1,
\]
where $\epsilon_0$ is a fixed constant depending only on $\alpha$, $\texp$ and $T$, $C_j$ is the constant of the upper power-law bound on the noiseless risk of model $j$ (\cref{lem:app-phase-sums}), depending only on $\alpha$ and $\texp$,
$H_j(\epsilon)=(C_j/\sqrt\epsilon)^{1/r_j}-1$, $I_r(x)=\int_0^x(1+w)^{-r}\,\dif w$, and
$k_{N,\text{s-SVD}}=2H_N(\alpha)/H_N(\alpha/2)$ and
$k_{N,\text{s-SGD}}=2(1-\frac2\pi)H_N(\alpha)/N+4H_N(2\alpha)/(\pi H_N(\alpha))$ are the heights of the
rescaled noise kernels (\cref{eqn:app-phase-heights}).

\begin{theorem}[Time to $T\epsilon$ in the power-law phases]\label{thm:app-phase-map}
Fix $\alpha>0$ and $\texp$ with $\alpha+\texp>1$ and $\texp\notin\{1,\alpha+1\}$,
and fix $T\ge4$. Consider the arrays of \cref{def:app-phase-arrays} at their
matched steps along a sequence of $(N,B,\epsilon)$ with $N\to\infty$,
$\epsilon\downarrow0$ and integers $B\ge N$, such that
\begin{itemize}
\item $\epsilon N^{\alpha+\texp-1}\to\infty$ if $\texp<1$;
\item $\epsilon N\to\infty$ if $\texp>1$ and $\alpha>1$;
\item $(N,B,\epsilon)$ lies in the low-exponent window \cref{eqn:app-phase-height}
for all large indices if $\texp>1$ and $\alpha\le1$.
\end{itemize}
Then there are constants $0<c_-\le c_+$, depending only on $\alpha$, $\texp$ and
$T$, such that for all large indices of the sequence, with $x\asymp y$ meaning
$c_-\le x/y\le c_+$,
\begin{align}
t^{\text{s-SVD}}_{T\epsilon}&\asymp\gamma\,H_N(\alpha/2)\,\epsilon^{-\alpha/[2(\alpha+\texp-1)]}\ \ (\texp<1),&
t^{\text{s-SVD}}_{T\epsilon}&\asymp\sqrt\gamma\,\mathcal S_{\text{s-SVD}}\,\epsilon^{-1/2}\ \ (\texp>1),\label{eqn:app-phase-tsvd}\\
t^{\text{s-SGD}}_{T\epsilon}&\asymp\frac{N^2\pi\bar\mu}B\,\epsilon^{-\alpha/(\alpha+\texp-1)}\ \ (\texp<\alpha+1),&
t^{\text{s-SGD}}_{T\epsilon}&\asymp\sqrt{\frac{\pi\bar\mu}B}\,\mathcal S_{\text{s-SGD}}\,\epsilon^{-1/2}\ \ (\texp>\alpha+1).\label{eqn:app-phase-tsgd}
\end{align}
Moreover $H_N(\alpha/2)\asymp N^{1-\alpha/2}$ for $\alpha<2$,
$H_N(\alpha/2)\asymp\log N$ for $\alpha=2$, and $H_N(\alpha/2)\asymp1$ for $\alpha>2$,
with constants depending only on $\alpha$.
\end{theorem}

\skpointer{thm:app-phase-map}{Lean: proved: \texttt{\detokenize{Muon}\detokenize{.}\allowbreak{}\detokenize{Phase}\detokenize{.}\allowbreak{}\detokenize{thm}\detokenize{_}\allowbreak{}\detokenize{app}\detokenize{_}\allowbreak{}\detokenize{phase}\detokenize{_}\allowbreak{}\detokenize{map}}, \texttt{\detokenize{Muon}\detokenize{.}\allowbreak{}\detokenize{Phase}\detokenize{.}\allowbreak{}\detokenize{thm}\detokenize{_}\allowbreak{}\detokenize{app}\detokenize{_}\allowbreak{}\detokenize{phase}\detokenize{_}\allowbreak{}\detokenize{map}\detokenize{_}\allowbreak{}\detokenize{moreover}}}
This theorem supplies the exponents of \cref{thm:tte-half-aniso}: the
exponent of $1/\epsilon$ in each display is $e_{\text{Muon}}$ or
$e_{\text{s-SGD}}$ of \cref{eqn:tte-exponents}. Like every statement of this chapter, it concerns the model with the
adopted arrays at the matched step, not the algorithm. Because the \Muon drift
scales like $\sqrt{\mu_i}$ and the \SignSGD drift like $\mu_i$, the time
saturates at $\epsilon^{-1/2}$ beyond $\texp=1$ for \SignSVD but only beyond
$\texp=\alpha+1$ for \SignSGD. The prefactors carry all the dependence on $N$
and $B$, and their ratio is $D_N$ of \cref{eqn:tte-ratio}. The limit is joint
in $N$ and $\epsilon$. For $\texp<1$ the window keeps $\epsilon$ above the
accuracy scale $\epsilon_N$. At fixed $N$ the exponent is $\frac12$ in every
phase (\cref{cor:app-phase-fixed}), and \cref{thm:app-phase-frozen} treats
$\epsilon/\epsilon_N\to0$.

\begin{sketch}
The risk of the model solves a scalar equation once time is measured in the
right units. In the row equations the drift of every row is divided by
$\sqrt{\fRisk}$, so the natural time is the accumulated effective step
$s=\clock(t)=\int_0^t\lr\,\fRisk(t')^{-1/2}\,\dif t'$
(\cref{eqn:app-phase-clock}). In this time each row solves a linear equation driven by $\sqrt Q$, and
summing the rows with weights $\mu_i/2$ gives exactly
\[
\begin{gathered}
Q(s)=F(s)+\int_0^sk(s-v)\sqrt{Q(v)}\,\dif v,\\
F(s)=\tfrac12\sum_i\mu_i\fQ_i(0)e^{-2\mathfrak d_is},\qquad
k(s)=\tfrac\lr2\sum_i\mu_i\mathfrak v_ie^{-2\mathfrak d_is},
\end{gathered}
\]
where $Q(s)$ is the risk at time $s$
(\cref{thm:app-phase-volterra,thm:app-phase-clock}). The forcing $F$ is the
risk the model would have without noise, and the kernel $k$ feeds the
volatility back in. The equation has a unique solution, and its limit is
$\lr^2\mathcal S^2/4$, which equals $\epsilon$ at the matched step.

After a change of time scale, the time to $T\epsilon$ factors into a prefactor
that carries $N$ and $B$ and a dimensionless integral. The time scale
$T_0^j$ of \cref{def:app-phase-arrays} makes the rescaled decay rates
$2\mathfrak d_iT_0^j$ equal to $2i^{-q_j}$. The rescaled risk
$Y_j(u)=Q_j(T_0^ju)$ then solves (\cref{prop:app-phase-prefactors})
\[
Y_j=f_{N,q_j}+\sqrt\epsilon\,g_j*\sqrt{Y_j},\qquad
f_{N,q}(u)=\frac12\sum_{i=1}^Ni^{-(\alpha+\texp)}e^{-2ui^{-q}},
\]
where $f_{N,q}$ is the rescaled noiseless risk and $g_j$ is a kernel of total
mass one. Converting back to step counts, the first
time $\hat u_j$ at which $Y_j$ reaches $T\epsilon$ gives exactly
(\cref{eqn:app-phase-time,eqn:app-phase-ST0})
\[
t^j_{T\epsilon}=\frac{\mathcal S_jT_0^j}{2\sqrt\epsilon}\int_0^{\hat u_j}\sqrt{Y_j(u)}\,\dif u,\qquad
\mathcal S_{\text{s-SVD}}T_0^{\text{s-SVD}}=\frac\gamma2H_N(\alpha/2),\qquad
\mathcal S_{\text{s-SGD}}T_0^{\text{s-SGD}}=\frac{\pi N^2\bar\mu}{2B}.
\]

The phase is decided by whether the square root of the noiseless risk is
integrable. The forcing is squeezed between two power
laws (\cref{lem:app-phase-sums}), the lower one valid for $N\ge4$ and
$u\le(N/4)^q$:
\[
A_p\,(1+u)^{-2r_q}\le f_{N,q}(u)\le C_{p,q}^2\,(1+u)^{-2r_q},\qquad
r_q=\frac{p-1}{2q}.
\]
Thus $r_j=r_{q_j}$. If $r_j>1$, then $\sqrt{f_{N,q_j}}$ is integrable uniformly
in $N$ (\cref{eqn:app-phase-Jall}), and the integral in the time formula is of
order one. If $r_j<1$, the integral keeps growing until the first passage. At
that time $f_{N,q}(\hat u)\le T\epsilon$. The lower power law then bounds the
first passage from below,
\[
1+\hat u\ge\Big(\frac{A_p}{T\epsilon}\Big)^{1/(2r)},
\]
and the integral is of order $\epsilon^{-(1-r)/(2r)}$
(\cref{lem:app-phase-clock-integral}). With the
prefactor this gives
\[
\begin{gathered}
t^j_{T\epsilon}\asymp\mathcal S_jT_0^j\,\epsilon^{-1/(2r_j)}\ \ (r_j<1),\qquad
t^j_{T\epsilon}\asymp\mathcal S_jT_0^j\,\epsilon^{-1/2}\ \ (r_j>1),\\
\frac1{2r_{\text{s-SVD}}}=\frac{\alpha}{2(\alpha+\texp-1)},\quad
\frac1{2r_{\text{s-SGD}}}=\frac{\alpha}{\alpha+\texp-1}.
\end{gathered}
\]
The boundary $r_j=1$ is the line $\texp=1$ for \SignSVD and $\texp=\alpha+1$
for \SignSGD, because $q_{\text{s-SVD}}$ is half of $q_{\text{s-SGD}}$.

The noise term is controlled by comparison with a function of the form
$(h+\sqrt\epsilon)^2$, and this comparison is where the windows come from. Any
function that dominates the right side of the rescaled equation when inserted
into it lies above $Y_j$ (\cref{lem:app-phase-comparison}). Take for $h$ a
power-law envelope of $\sqrt{f_{N,q_j}}$ that decreases to $\sqrt\epsilon$ at a
horizon $H$. Then $(h+\sqrt\epsilon)^2$ dominates as soon as the kernel does not
amplify the envelope, that is, $g_j*h\le2h$ or $g_j*h\le\sqrt\epsilon$ on $[0,H]$
(\cref{lem:app-phase-square}). This places the first passage before $H$, which
supplies the upper bound in the previous paragraph. It also bounds the noise
part of the integral by a fixed multiple of $\epsilon^{(r_0-1)/(2r_0)}$, where
$r_0\le r_j$ is the exponent of the envelope and $r_0>1$ exactly when $r_j>1$
(\cref{eqn:app-phase-noise-error}). For $r_j<1$ this bound is of the order of
the main term, and the theorems give orders and not equivalents for these
models; for $r_j>1$ it tends to zero.

The kernel condition holds in each range of exponents, and checking it is what
fixes the windows. For $\texp<1$ (\cref{thm:app-phase-A}) the \SignSGD kernel is decreasing
with mass one, so $g_j*h\le h/(1-r_j)$ (\cref{lem:app-phase-mass}), which is
at most $2h$ because here $r_{\text{s-SGD}}=r_{\text{s-SVD}}/2<\frac12$. The \SignSVD kernel, normalized by $H_N(\alpha/2)$,
has small mass up to times of order $N^{\alpha/2}$ when $\alpha<2$
(\cref{lem:app-phase-svd-small}), a logarithmic substitute at $\alpha=2$
(\cref{lem:app-phase-log}) and an integrable tail when $\alpha>2$
(\cref{lem:app-phase-dilation}). The horizon is then
$H_j\asymp\epsilon^{-1/(2r_j)}$, and since $2r_jq_j=\alpha+\texp-1$ the window
$\epsilon N^{\alpha+\texp-1}\ge E_*$, with $E_*$ a fixed threshold, keeps $H_j$
below $(N/4)^{q_j}$, where the lower power law holds. For $\alpha>1$ and
$\texp\ge1$ (\cref{thm:app-phase-B}) the kernels have integrable power tails
(\cref{lem:app-phase-dilation}), except the \SignSVD kernel for $1<\alpha\le2$.
For $1<\alpha<2$ the normalization $H_N(\alpha/2)\ge N^{1-\alpha/2}$ makes the late part of
the convolution small up to the horizon $\sqrt N$, which the window
$\epsilon N\ge E_*$ guarantees (\cref{lem:app-phase-svd-main}). At $\alpha=2$,
\cref{lem:app-phase-log-global} adds a correction term to the comparison
function, needs no relation between $N$ and $\epsilon$, and bounds the noise
part by a fixed constant that tends to zero as $N\to\infty$ and
$\epsilon\downarrow0$. For $\alpha\le1$ and $\texp\ge1$ (\cref{thm:app-phase-C})
no tail is usable, but the kernels are decreasing with maxima
$k_{N,j}=g_j(0)\to0$ (\cref{prop:app-phase-C-window}), so
\[
(g_j*h_j)(u)\le k_{N,j}\,C_j\,I_{r_j}(H_j)\le\sqrt\epsilon\qquad(0\le u\le H_j),
\]
where $C_j=C_{p,q_j}$ and the second inequality is the low-exponent window \cref{eqn:app-phase-height}.
Elementary harmonic-sum bounds give the orders of $H_N(\alpha/2)$ in the last
assertion (\cref{lem:app-phase-harmonic}).
\end{sketch}

The windows of \cref{thm:app-phase-A,thm:app-phase-B,thm:app-phase-C} are the sets of
$(N,B,\epsilon)$ with $B\ge N$ and $0<\epsilon<\epsilon_0$ in which, respectively, $N\ge N_*$ and
$\epsilon N^{p-1}\ge E_*$ ($\texp<1$); $N\ge N_*$ and $\epsilon N\ge E_*$ ($\alpha>1$, $\texp\ge1$); and
$(N,B,\epsilon)$ lies in the low-exponent window ($\alpha\le1$, $\texp\ge1$). Here $N_*$, $E_*$ and
$\epsilon_0$ are fixed constants depending only on $\alpha$, $\texp$ and $T$. With
$f_{\infty,q}(u)=\frac12\sum_{i\ge1}i^{-p}e^{-2ui^{-q}}$, put
$J(p,q)=\int_0^\infty\sqrt{f_{\infty,q}(u)}\,\dif u$, which is finite when $r_q=(p-1)/(2q)>1$.

\begin{corollary}[Ratio of the hitting times]\label{cor:app-phase-ratio}
In the windows of \cref{thm:app-phase-A} ($\texp<1$), \cref{thm:app-phase-B}
($\alpha>1$, $\texp\ge1$) and \cref{thm:app-phase-C} ($\alpha\le1$, $\texp\ge1$),
with constants depending only on $\alpha$, $\texp$ and $T$,
\begin{equation}\label{eqn:app-phase-ratio}
\frac{t^{\text{s-SGD}}_{T\epsilon}}{t^{\text{s-SVD}}_{T\epsilon}}\asymp D_N\cdot\begin{cases}
\epsilon^{-\alpha/[2(p-1)]},&\texp<1,\\
\epsilon^{-1/2}/\log(1/\epsilon),&\texp=1,\\
\epsilon^{-(\alpha+1-\texp)/[2(p-1)]},&1<\texp<\alpha+1,\\
\log(1/\epsilon),&\texp=\alpha+1,\\
1,&\texp>\alpha+1.
\end{cases}
\end{equation}
Along every sequence with $N\to\infty$ and $\epsilon\downarrow0$ in the window
of \cref{thm:app-phase-B} or \cref{thm:app-phase-C},
\[
\frac{t^{\text{s-SGD}}_{T\epsilon}}{t^{\text{s-SVD}}_{T\epsilon}}\sim\frac{D_N\log(1/\epsilon)}{2\sqrt{8\alpha}\,J(p,\alpha/2)}\quad(\texp=\alpha+1),\qquad
\frac{t^{\text{s-SGD}}_{T\epsilon}}{t^{\text{s-SVD}}_{T\epsilon}}\sim D_N\,\frac{J(p,\alpha)}{J(p,\alpha/2)}\quad(\texp>\alpha+1).
\]
\end{corollary}

\skpointer{cor:app-phase-ratio}{Lean: proved: \texttt{\detokenize{Muon}\detokenize{.}\allowbreak{}\detokenize{Phase}\detokenize{.}\allowbreak{}\detokenize{cor}\detokenize{_}\allowbreak{}\detokenize{app}\detokenize{_}\allowbreak{}\detokenize{phase}\detokenize{_}\allowbreak{}\detokenize{ratio}\detokenize{_}\allowbreak{}\detokenize{i}}, \texttt{\detokenize{Muon}\detokenize{.}\allowbreak{}\detokenize{Phase}\detokenize{.}\allowbreak{}\detokenize{cor}\detokenize{_}\allowbreak{}\detokenize{app}\detokenize{_}\allowbreak{}\detokenize{phase}\detokenize{_}\allowbreak{}\detokenize{ratio}\detokenize{_}\allowbreak{}\detokenize{ii}}, +3 more}
Dividing the two hitting-time orders gives the ratio inside the windows, which is the
windowed form of \cref{eqn:tte-ratio}. Off the lines $\texp=1$ and $\texp=\alpha+1$ the exponent
of $1/\epsilon$ is $\kappa$ of \cref{eqn:tte-exponents}, and on those lines a logarithm appears. The
prefactor is $D_N$ exactly. In the saturated range $\texp>\alpha+1$ the second display gives the
constant $D_NJ(p,\alpha)/J(p,\alpha/2)$, a quotient of the clock integrals of the two noiseless
risks. For $\alpha>2$, $D_N\to\pi\zeta(\alpha)/\zeta(\alpha/2)$, so the ratio has a finite positive limit,
which is Phase E of \cref{fig:tte-phase-diagram}.

\begin{sketch}
In each window both models have the order of
\cref{thm:app-phase-A,thm:app-phase-B,thm:app-phase-C} at once. By \cref{sk:E6}, each time is
$\mathcal S_jT_0^j/(2\sqrt\epsilon)$ times a clock integral of order $\epsilon^{-(1-r_j)/(2r_j)}$,
$\log(1/\epsilon)$ or $1$ as $r_j<1$, $r_j=1$, $r_j>1$, so
\[
t^j_{T\epsilon}\asymp\mathcal S_jT_0^j\,\Pi_{r_j}(\epsilon),\qquad
\Pi_r(\epsilon)=\epsilon^{-1/(2r)},\ \ \epsilon^{-1/2}\log\tfrac1\epsilon,\ \ \epsilon^{-1/2}\quad(r<1,\ r=1,\ r>1).
\]
The prefactors divide to $D_N$ (\cref{eqn:app-phase-DN}); the rest is arithmetic with
$r_{\text{s-SVD}}=(p-1)/\alpha$, $r_{\text{s-SGD}}=(p-1)/(2\alpha)$. For example, for $1<\texp<\alpha+1$ only
\SignSGD has $r_j<1$, and the quotient is $\epsilon^{-\alpha/(p-1)+1/2}=\epsilon^{-(\alpha+1-\texp)/(2(p-1))}$; at
$\texp=1$, where $p-1=\alpha$, it is $\epsilon^{-1}/(\epsilon^{-1/2}\log(1/\epsilon))$.

In the windows of \cref{thm:app-phase-B,thm:app-phase-C} the critical and saturated times
have equivalents, and their quotient gives the constants. For $r_j>1$ the noise part of
the clock integral, the cost $O(N^{-q_j(r_j-1)})$ of truncating the forcing at $N$, and the part of
$J(p,q_j)$ beyond the first passage all vanish, so the clock integral tends to $J(p,q_j)$. For $r_j=1$,
$\sqrt{f_{\infty,q}(u)}=a_q/u+O(u^{-1-1/q})$ with $a_q=(8q)^{-1/2}$, and the first passage satisfies
$\log(1+\hat u)\ge\frac12\log(1/\epsilon)-O(1)$ while the comparison horizon bounds it from above, so
the clock integral is $\frac{a_q}2\log(1/\epsilon)$ up to $o(\log(1/\epsilon))$
(\cref{lem:app-phase-clock-integral}). This gives \cref{eqn:app-phase-equivalents}. At
$\texp=\alpha+1$, \SignSVD is saturated ($r_{\text{s-SVD}}=2$) and \SignSGD is critical, and
\[
\frac{\mathcal S_{\text{s-SGD}}T_0^{\text{s-SGD}}\,\epsilon^{-1/2}\log(1/\epsilon)/(4\sqrt{8\alpha})}{\mathcal S_{\text{s-SVD}}T_0^{\text{s-SVD}}\,J(p,\alpha/2)/(2\sqrt\epsilon)}
=\frac{D_N\log(1/\epsilon)}{2\sqrt{8\alpha}\,J(p,\alpha/2)} .
\]
For $\texp>\alpha+1$ both are saturated, and the quotient is $D_NJ(p,\alpha)/J(p,\alpha/2)$.
\end{sketch}

\begin{proposition}[Limits of the ratio]\label{prop:app-phase-winners}
Write $\mathsf R\defeq t^{\text{s-SGD}}_{T\epsilon}/t^{\text{s-SVD}}_{T\epsilon}$, and consider
sequences with $N\to\infty$, $\epsilon\downarrow0$ and $B_N\ge N$ in the
indicated windows. The constants in $\asymp$ and $O$ depend only on $\alpha$,
$\texp$ and $T$.
\begin{enumerate}[label=(\roman*)]
\item \emph{The range $\texp<1$} (window of \cref{thm:app-phase-A}). With
$\Xi_N\defeq\epsilon N^{p-1}$ and $r_{\text{s-SVD}}=(p-1)/\alpha$,
\[
\mathsf R\asymp\begin{cases}
\Xi_N^{-1/(2r_{\text{s-SVD}})},&0<\alpha<1,\\
(\log N)\,\Xi_N^{-1/(2r_{\text{s-SVD}})},&\alpha=1,\\
N^{\alpha-1}\,\Xi_N^{-1/(2r_{\text{s-SVD}})},&1<\alpha<2,\\
(N/\log N)\,\Xi_N^{-1/(2r_{\text{s-SVD}})},&\alpha=2,\\
\epsilon^{-1/(2r_{\text{s-SVD}})},&\alpha>2.
\end{cases}
\]
Hence $\mathsf R\to0$ for $0<\alpha<1$ along sequences with $\Xi_N\to\infty$, and
$\mathsf R\to\infty$ for $\alpha>2$.
For $\alpha=1$ and $\epsilon_N=N^{-r_{\text{s-SVD}}}(\log N)^{b}$ with $b>0$,
$\mathsf R\to\infty$ if $b<2r_{\text{s-SVD}}$ and $\mathsf R\to0$ if
$b>2r_{\text{s-SVD}}$. For $1<\alpha<2$ and $\epsilon_N=N^{-\omega}$ with
$0<\omega<p-1$, put $\omega_c\defeq r_{\text{s-SVD}}(2-\alpha)\in(0,p-1)$; then
$\mathsf R\to0$ if $\omega<\omega_c$ and $\mathsf R\to\infty$ if $\omega>\omega_c$.
For $\alpha=2$ and $\epsilon_N=(\log N)^{-\omega}$ with $\omega>0$, $\mathsf R\to0$ if
$\omega<2r_{\text{s-SVD}}$ and $\mathsf R\to\infty$ if $\omega>2r_{\text{s-SVD}}$. All
these sequences lie in the window.
\item \emph{$\alpha>2$, $\texp\ge1$} (window of \cref{thm:app-phase-B}).
$\mathsf R\to\infty$ if $\texp\le\alpha+1$, and
$\mathsf R\to\pi\zeta(\alpha)J(p,\alpha)/(\zeta(\alpha/2)J(p,\alpha/2))\in(0,\infty)$ if
$\texp>\alpha+1$; this limit exceeds $\pi/2$ if $\alpha\ge4$.
\item \emph{$1<\alpha\le2$, $\texp\ge1$} (window of \cref{thm:app-phase-B}).
$\mathsf R\to0$ if $\texp>\alpha+1$. If $\texp=\alpha+1$ and $1<\alpha<2$, then
$\mathsf R\to0$. If $\alpha=2$, $\texp=3$ and $\log(1/\epsilon)/\log N\to c_\infty$,
then $0\le c_\infty\le1$ and $\mathsf R\to\pi\zeta(2)c_\infty/(8J(5,1))$.
\item \emph{$0<\alpha\le1$, $\texp\ge1$} (low-exponent window). $\mathsf R\to0$; more
precisely $\mathsf R$ is $O(\log(1/\epsilon)^{-2})$, $O(\epsilon^{(\texp-1)/(p-1)})$,
$O(\sqrt\epsilon\log(1/\epsilon))$ and $O(\sqrt\epsilon)$ for $\texp=1$,
$1<\texp<\alpha+1$, $\texp=\alpha+1$ and $\texp>\alpha+1$ respectively.
\item \emph{$1<\alpha\le2$, $1\le\texp<\alpha+1$} (window of \cref{thm:app-phase-B}). Put
\[
e_\texp\defeq\frac{\alpha+1-\texp}{2(p-1)},\qquad \texp_c(\alpha)\defeq\frac{\alpha^2-2\alpha+3}{3-\alpha} ,
\]
so that $1<\texp_c(\alpha)<\alpha+1$ for $1<\alpha<2$, and $\texp_c(2)=3$.
\begin{enumerate}[label=(\alph*)]
\item If $1<\alpha<2$ and $\texp_c(\alpha)\le\texp<\alpha+1$, then $\mathsf R\to0$ along every
sequence with $N\to\infty$, $\epsilon\downarrow0$ and $\epsilon N\to\infty$.
\item If $1<\alpha<2$ and $1\le\texp<\texp_c(\alpha)$, put
$\omega_c\defeq(2-\alpha)(p-1)/(\alpha+1-\texp)\in(0,1)$. For $\epsilon_N=N^{-\omega}$ with
$0<\omega<1$, $\mathsf R\to0$ if $\omega<\omega_c$ and $\mathsf R\to\infty$ if $\omega>\omega_c$.
For $\epsilon_N=1/\log N$, $\mathsf R\to0$.
\item If $\alpha=2$ and $1\le\texp<3$, then $\mathsf R\to\infty$ for $\epsilon_N=N^{-\omega}$ with
$0<\omega<1$, and $\mathsf R\to0$ for $\epsilon_N=1/\log N$.
\end{enumerate}
All the sequences named in (b) and (c) lie in the window for large $N$.
\end{enumerate}
\end{proposition}

\skpointer{prop:app-phase-winners}{Lean: proved: \texttt{\detokenize{Muon}\detokenize{.}\allowbreak{}\detokenize{Phase}\detokenize{.}\allowbreak{}\detokenize{prop}\detokenize{_}\allowbreak{}\detokenize{app}\detokenize{_}\allowbreak{}\detokenize{phase}\detokenize{_}\allowbreak{}\detokenize{winners}\detokenize{_}\allowbreak{}\detokenize{i}}, \texttt{\detokenize{Muon}\detokenize{.}\allowbreak{}\detokenize{Phase}\detokenize{.}\allowbreak{}\detokenize{prop}\detokenize{_}\allowbreak{}\detokenize{app}\detokenize{_}\allowbreak{}\detokenize{phase}\detokenize{_}\allowbreak{}\detokenize{winners}\detokenize{_}\allowbreak{}\detokenize{i}\detokenize{_}\allowbreak{}\detokenize{zero}}, +13 more}
The ratio law of \cref{cor:app-phase-ratio} is a competition between the dimension factor
$D_N$, which tends to zero for $\alpha\le2$, and the accuracy factor, which grows as $\epsilon\downarrow0$.
This proposition decides it inside the windows: \SignSVD wins for $\alpha>2$ and $\texp\le\alpha+1$,
the row of Phase A in \cref{fig:tte-phase-diagram}, and \SignSGD wins for $\alpha\le1$ and
$\texp\ge1$. Elsewhere the winner depends on how $N$ and $\epsilon$ move together, through a critical
exponent $\omega_c$ and the curve $\texp_c(\alpha)$; these are windowed forms of the accuracy-exponent
map of \cref{cor:app-phase-omega-winners}.

\begin{sketch}
The order of the dimension factor decides each case. Comparing the harmonic sums with
integrals (\cref{lem:app-phase-harmonic}), $D_N$ is of order
$N^{-\alpha/2}$, $N^{-1/2}\log N$, $N^{\alpha/2-1}$, $1/\log N$ for $\alpha<1$, $\alpha=1$, $1<\alpha<2$, $\alpha=2$,
and tends to $\pi\zeta(\alpha)/\zeta(\alpha/2)$ for $\alpha>2$ (\cref{eqn:app-phase-DN-asymp}). For $\alpha>2$ the
accuracy factor tends to infinity unless $\texp>\alpha+1$, which gives (ii); for $\alpha\ge4$ the limit
exceeds $\pi/2$ because $\zeta(\alpha)>1$, $\zeta(\alpha/2)<2$, and $J(p,\alpha)\ge J(p,\alpha/2)$ term by term. In (iii),
$D_N\to0$ and $\log(1/\epsilon)\le\log N$ in the window. For $\texp<1$, write
$\epsilon=\Xi_NN^{-(p-1)}$; since $(p-1)/(2r_{\text{s-SVD}})=\alpha/2$,
\[
\mathsf R\asymp D_N\,\epsilon^{-1/(2r_{\text{s-SVD}})}=D_N\,N^{\alpha/2}\,\Xi_N^{-1/(2r_{\text{s-SVD}})},
\]
and the five rows of (i) are the five orders of $D_NN^{\alpha/2}$; for $1<\alpha<2$ and
$\epsilon_N=N^{-\omega}$ this is $N^{(\omega-\omega_c)/(2r_{\text{s-SVD}})}$.

For $1<\alpha\le2$ and $1\le\texp<\alpha+1$ the sign of one exponent locates the curve $\texp_c$.
These rows of the ratio law read $\mathsf R\asymp D_N\epsilon^{-e_\texp}\ell_\epsilon$, with
$\ell_\epsilon=1/\log(1/\epsilon)$ at $\texp=1$ and $\ell_\epsilon=1$ otherwise, and (\cref{eqn:app-phase-curve})
\[
\frac\alpha2-1+e_\texp=\frac{(3-\alpha)\big(\texp_c(\alpha)-\texp\big)}{2(p-1)} .
\]
For $1<\alpha<2$, $\mathsf R\asymp(\epsilon N)^{-e_\texp}N^{\alpha/2-1+e_\texp}$, which tends to zero when
$\texp\ge\texp_c(\alpha)$ because $\epsilon N\to\infty$; this is (v)(a). Below the curve, at $\epsilon_N=N^{-\omega}$,
the exponent $\frac\alpha2-1+\omega e_\texp$ changes sign at $\omega_c=(1-\frac\alpha2)/e_\texp<1$, which is
(v)(b); (v)(c) is the same with $D_N\asymp1/\log N$. For $\alpha\le1$ and $\texp\ge1$ the window itself
decides: $D_N=\frac\pi2k_{N,\text{s-SVD}}$, and the window condition
$k_{N,\text{s-SVD}}C_{\text{s-SVD}}I_r(H)\le\sqrt\epsilon$ makes this $O(\sqrt\epsilon)$, or
$O(\sqrt\epsilon/\log(1/\epsilon))$ at $\texp=1$, which outweighs the accuracy factor and gives (iv).
\end{sketch}

\begin{remark}[The adopted \SignSGD array and the corrected modes]\label{rem:app-phase-sgd}
The \SignSGD section computes the raw annealed row volatility of the actual
sign update, $\mathfrak v_i^{\rm corr}$, for a Haar eigenbasis independent of the
batch, comparable dimensions and a covariance bounded below
(\cref{thm:app-sgd-modes}). At $\Nout=\Nin=N$, with $\Sigmaout=U\Lambda U^\T$, the
state $\Delta=UD$ and $\cR=\frac12\Tr(\Lambda DD^\T)$, its difference from the
adopted array is
\[
\mathfrak v_i^{\rm corr}-\mathfrak v_i^{\text{s-SGD}}=\frac2\pi\,N\Big[\frac{2\mu_i}{\Tr\Sigmaout}+\frac{(B+2N+3)\,(\Lambda DD^\T\Lambda)_{ii}}{2\cR\,\Tr\Sigmaout}\Big]\ \ge0 ,
\]
where $\mathfrak v_i^{\text{s-SGD}}=N(1-\frac2\pi)+\frac2\pi N^2\mu_i/\Tr\Sigmaout$. The
corrected array depends on the state, its total exceeds $N^2$, and the
power-law family has no covariance lower bound fixed in $N$. The phase
theorems of this subsection concern the adopted array
\cref{eqn:app-phase-arrays}. Replacing it by another nonnegative array with
total $N^2$ and the same drift leaves $\mathcal S_{\text{s-SGD}}$ unchanged, but
not the kernel $g_{\text{s-SGD}}$ on which the joint estimates rest.
\end{remark}

\skpointer{rem:app-phase-sgd}{Lean: not formalized (remark)}
For power-law spectra the main text uses \SignSGD kernels without the nonnegative,
state-dependent term that the \SignSGD analysis finds in the per-row volatility (the
footnote to \cref{eqn:ri-recursion}); this remark gives that term and the reason it is
dropped. The corrected array is proved only for a covariance bounded below, which the
power-law family is not, and it depends on the state, so it does not fit a model with
fixed coefficients. It is at least the adopted array, and its total exceeds $N^2$. Any
array with total $N^2$ and the same drift has the same noise constant
$\mathcal S_{\text{s-SGD}}$, hence the same matched step, but the joint estimates use the shape
of the rescaled kernel $g_{\text{s-SGD}}$ (\cref{sk:E6}), so the phase theorems are statements
about the adopted array only.

\begin{remark}[Momentum arrays]\label{rem:app-phase-momentum}
For $\bmom\in[0,1)$ put $B_{\rm eff}\defeq B(1+\bmom)/(1-\bmom)=B/\sw2$. The
adopted \Muon array with momentum $\bmom$ has the $\mu_i$ and $\fQ_i(0)$ of
\cref{def:app-phase-arrays}, the batch $B_{\rm eff}$ in place of $B$ in the
drift, and the volatility multiplied by $B_{\rm eff}/B$:
\[
\mathfrak d_i^{\bmom}\defeq\sqrt{\mu_i}\sqrt{\frac{B_{\rm eff}}N}=\sqrt{\frac{B_{\rm eff}}B}\;\mathfrak d_i^{\text{s-SVD}},\qquad
\mathfrak v_i^{\bmom}\defeq\frac{B_{\rm eff}}B=\frac{B_{\rm eff}}B\;\mathfrak v_i^{\text{s-SVD}} .
\]
At $\bmom=0$ it is the \SignSVD array. The evidence for it is the
resolved-window endpoint of \cref{cor:app-pl-endpoints}(ii), quoted in
\cref{rem:app-phase-adoption}, which reads
$\sqrt2\,\mathfrak d_i(0)\sim\sqrt{\mu_iB_{\rm eff}/N}$ and
$\mathfrak v_i(0)\to B_{\rm eff}/B$ in this notation.

The momentum array at a step $\lr$ is the \SignSVD array at the step
$\lr\sqrt{B_{\rm eff}/B}$. For every $\lr>0$, the right side of
\cref{eqn:app-phase-rows} satisfies
\[
-\frac{2\lr\,\mathfrak d_i^{\bmom}}{\sqrt{\fRisk}}\,\fQ_i+\lr^2\mathfrak v_i^{\bmom}
=-\frac{2\lr'\,\mathfrak d_i^{\text{s-SVD}}}{\sqrt{\fRisk}}\,\fQ_i+\lr'^2\mathfrak v_i^{\text{s-SVD}},
\qquad \lr'\defeq\lr\sqrt{\frac{B_{\rm eff}}B} .
\]
By \cref{eqn:app-phase-F0S} the noise constant of the momentum array is
$\mathcal S^{\bmom}=\sqrt{B_{\rm eff}/B}\,\mathcal S_{\text{s-SVD}}$, so its matched
step $2\sqrt\epsilon/\mathcal S^{\bmom}$ is mapped to the matched step
$\lr^\star_{\text{s-SVD}}$. At matched steps the two models therefore solve the
same row equations from the same initial values, and they have the same
solution by the uniqueness in \cref{thm:app-phase-clock}. In particular
$t^{\bmom}_{T\epsilon}=t^{\text{s-SVD}}_{T\epsilon}$, and
\cref{thm:app-phase-map}, \cref{cor:app-phase-ratio} and
\cref{prop:app-phase-winners} hold for the momentum array without change.
Written with the noise constant of the momentum array, the saturated
prefactor $\sqrt\gamma\,\mathcal S_{\text{s-SVD}}$ of \cref{eqn:app-phase-tsvd}
is $\sqrt{N/B_{\rm eff}}\,\mathcal S^{\bmom}$.

The oversampled isotropic case refines the volatility factor. There the
volatility with momentum is $c_\bmom B_{\rm eff}/B$ times the volatility
without momentum, with $c_\bmom\defeq\sw2V_\bmom\in(\sw2,1)$, up to a factor
$1+O(\gamma)$ (\cref{rem:app-iso-momentum,lem:app-iso-V}). A factor $c>0$
on every volatility of an array multiplies every matched-step hitting time
by $c$. Indeed, the noise constant becomes $c\,\mathcal S$ and the matched step
becomes $\lr^\star/c$, so the right side of \cref{eqn:app-phase-rows} becomes
$1/c$ times its value for the original array at its matched step. The risk
of the modified model at time $t$ is therefore the original risk at time
$t/c$.
\end{remark}

\skpointer{rem:app-phase-momentum}{Lean: proved: \texttt{\detokenize{Muon}\detokenize{.}\allowbreak{}\detokenize{Phase}\detokenize{.}\allowbreak{}\detokenize{rem}\detokenize{_}\allowbreak{}\detokenize{app}\detokenize{_}\allowbreak{}\detokenize{phase}\detokenize{_}\allowbreak{}\detokenize{momentum}\detokenize{_}\allowbreak{}\detokenize{i}}, \texttt{\detokenize{Muon}\detokenize{.}\allowbreak{}\detokenize{Phase}\detokenize{.}\allowbreak{}\detokenize{rem}\detokenize{_}\allowbreak{}\detokenize{app}\detokenize{_}\allowbreak{}\detokenize{phase}\detokenize{_}\allowbreak{}\detokenize{momentum}\detokenize{_}\allowbreak{}\detokenize{ii}}, +1 more}
This remark is behind the main-text sentence that a constant factor in the volatility
only rescales the time to reach a target, and behind the fact that the power-law
hitting times of \Muon in \cref{fig:tte-four-panel} do not depend on $\bmom$. The adopted
array with momentum multiplies the drift by $\sqrt{B_{\rm eff}/B}$ and the volatility by
$B_{\rm eff}/B$, which is the \SignSVD array at the step $\lr\sqrt{B_{\rm eff}/B}$; this maps the
matched step to the matched step, and uniqueness of the row equations makes the two
trajectories, and hence the hitting times, identical. More generally, a factor $c>0$ on
every volatility multiplies the noise constant by $c$ and divides the matched step by $c$,
so the right side of the row equations at the matched step is divided by $c$, and every
hitting time is multiplied by $c$. In the oversampled isotropic case the volatility with momentum
is $c_\bmom B_{\rm eff}/B$ times the one without, up to a factor $1+O(\gamma)$, with
$c_\bmom=\sw2V_\bmom\in(\sw2,1)$ of \cref{lem:app-iso-V}, the $c_\bmom$ of the main text.

Let $\epsilon_N=N^{-(p-1)}$ be the accuracy scale of \cref{eqn:tte-ratio} and
$x=\epsilon/\epsilon_N$. For an integer $N$ let
$J_N(p,q)=\int_0^\infty\sqrt{f_{N,q}(u)}\,\dif u$ with
$f_{N,q}(u)=\frac12\sum_{i\le N}i^{-p}e^{-2ui^{-q}}$, and let $\mathsf I_j$ be the clock integral of
\cref{sk:E6}, so that $t^j_{T\epsilon}=\mathcal S_jT_0^j\,\mathsf I_j/(2\sqrt\epsilon)$ exactly. Put
$\rho_N=D_N\,J_N(p,\alpha)/J_N(p,\alpha/2)$, the limit of
$t^{\text{s-SGD}}_{T\epsilon}/t^{\text{s-SVD}}_{T\epsilon}$ as $\epsilon\downarrow0$ at fixed $N$
(\cref{cor:app-phase-fixed}).

\begin{theorem}[Frozen regime]\label{thm:app-phase-frozen}
Fix $\alpha>0$, $\texp$ with $p=\alpha+\texp>1$, and $T\ge4$. Consider a sequence
of $(N,B,\epsilon)$ with $N\to\infty$, integers $B\ge N$ and
$x=\epsilon/\epsilon_N\to0$. For all large indices both first hits are attained,
and for both $j$
\[
\frac{\mathsf I_j}{J_N(p,q_j)}\longrightarrow1,\qquad
t^j_{T\epsilon}\sim\frac{\mathcal S_jT_0^j\,J_N(p,q_j)}{2\sqrt\epsilon},\qquad
\frac{t^{\text{s-SGD}}_{T\epsilon}}{t^{\text{s-SVD}}_{T\epsilon}}\sim\rho_N .
\]
Moreover, if $r_j<1$, there are $N_2$ and $c$ such that
\[
\Big|\frac{\mathsf I_j}{J_N(p,q_j)}-1\Big|\le c\,\Big(x^{1/4}+x^{1/2}\log\big(e+x^{-1/2}\big)\Big)^{1/2}
\]
for all $N\ge N_2$, $B\ge N$, $\epsilon<1/(2T)$ and $x\le1$.
\end{theorem}

\skpointer{thm:app-phase-frozen}{Lean: proved: \texttt{\detokenize{Muon}\detokenize{.}\allowbreak{}\detokenize{Phase}\detokenize{.}\allowbreak{}\detokenize{thm}\detokenize{_}\allowbreak{}\detokenize{app}\detokenize{_}\allowbreak{}\detokenize{phase}\detokenize{_}\allowbreak{}\detokenize{frozen}\detokenize{_}\allowbreak{}\detokenize{i}}, \texttt{\detokenize{Muon}\detokenize{.}\allowbreak{}\detokenize{Phase}\detokenize{.}\allowbreak{}\detokenize{thm}\detokenize{_}\allowbreak{}\detokenize{app}\detokenize{_}\allowbreak{}\detokenize{phase}\detokenize{_}\allowbreak{}\detokenize{frozen}\detokenize{_}\allowbreak{}\detokenize{ii}}, +4 more}
Below the accuracy scale $\epsilon_N$, which bounds the risk carried by the modes beyond $N$,
more accuracy no longer changes the comparison. The clock integral of each model freezes at
$J_N(p,q_j)$, the integral of the square root of its noiseless risk, so each time is
$\epsilon^{-1/2}$ times a constant that depends on $N$, and the ratio is asymptotic to its
accuracy-first value $\rho_N$. This is the levelling off of \cref{eqn:tte-ratio} below $\epsilon_N$,
and the tie of Phase D in \cref{fig:tte-phase-diagram} once $\epsilon\lesssim\epsilon_N$. The theorem needs
no window, and for $r_j<1$ it comes with a rate in $x=\epsilon/\epsilon_N$.

\begin{sketch}
At finite $N$ every rate is bounded below, which pins the clock integral to $J_N$ at small
accuracy. The forcing $f_{N,q_j}$ and the kernel $g_j=\sum_iw_ic_ie^{-c_iu}$, whose
weights sum to one, are finite sums of exponentials with rates at least $b_{\min}=2N^{-q_j}$.
The envelope $h=\sqrt{F_0}\,e^{-b_{\min}u/4}$ satisfies $g_j*h\le\frac43h$, so $(h+\sqrt\epsilon)^2$
dominates the equation and the first passage occurs before $(2/b_{\min})\log(F_0/\epsilon)$. After
it, the forcing decays at rate at least $b_{\min}$ from a value at most $T\epsilon$, which bounds
the part of $J_N$ that is missed; and the Cauchy--Schwarz inequality with $\int g_j=1$ gives
$\mathsf I_j\le J_N+\sqrt{\mathsf d\,\mathsf I_j}$. With $\mathsf y_j=\epsilon N^{2q_j}$ this gives
(\cref{lem:app-phase-mixture,eqn:app-phase-FMmodel})
\[
J_N-\sqrt{T\mathsf y_j}\le\mathsf I_j\le\big(\sqrt{J_N}+\sqrt{\mathsf d}\big)^2,\qquad
\mathsf d=\sqrt{\mathsf y_j}\,\log\frac{F_0}\epsilon .
\]

The errors are small relative to $J_N$ when $r_j\ge1$, and a restart removes a logarithm when
$r_j<1$. Since $2q_jr_j=p-1$, $\sqrt{\mathsf y_j}=\sqrt x\,N^{q_j(1-r_j)}$. For $r_j>1$, $J_N\to J(p,q_j)>0$
and $\mathsf y_j\to0$, so both errors vanish. For $r_j=1$, $\mathsf y_j=x$, $J_N\sim\sqrt{q_j/8}\,\log N$
and $\mathsf d=o(\log N)$. For $r_j<1$, $J_N\asymp N^{q_j(1-r_j)}$ (\cref{prop:app-phase-forcing-r}), so the
lower error has relative order $\sqrt x$, but $\mathsf d/J_N$ is of order $\sqrt x\log N$, because
$\log(F_0/\epsilon)$ contains $(p-1)\log N$. The proof therefore first runs the power-law comparison
of \cref{sk:E6} up to a fixed multiple $u_0$ of $N^{q_j}$, where the kernel condition $g_j*h\le2h$
still holds, at a cost $O(x^{1/4}+x^{1/2})N^{q_j(1-r_j)}$. At $u_0$ the memory term $g_j*\sqrt{Y_j}$ is
itself a finite sum of exponentials with the same rates, so the equation starts afresh from
$Y_j(u_0)\le\epsilon(1+cx^{-1/2})^2$, and the bound above now carries
$\log(Y_j(u_0)/\epsilon)=O(\log(e+x^{-1/2}))$ in place of $\log(F_0/\epsilon)$
(\cref{lem:app-phase-restart,lem:app-phase-twophase}). This gives the stated rate. The times
follow from $t^j_{T\epsilon}=\mathcal S_jT_0^j\mathsf I_j/(2\sqrt\epsilon)$, and dividing
$D_N\mathsf I_{\text{s-SGD}}/\mathsf I_{\text{s-SVD}}$ by $\rho_N$ gives the ratio.
\end{sketch}

For model $j$, with $q=q_j$ and $r=r_j$, put
\[
\Upsilon_j(N,\epsilon)=\begin{cases}
\min\big\{\epsilon^{-(1-r)/(2r)},\ N^{q(1-r)}\big\},&r<1,\\
\min\big\{\log(1/\epsilon),\ \log N\big\},&r=1,\\
1,&r>1.
\end{cases}
\]
Let $\mathcal U$ be the set of triples $(\alpha,\texp,j)$ with $p>1$ such that $r_j<1$, or
$j=\text{s-SVD}$ and $\alpha\ne1$, or $j=\text{s-SGD}$ and $\alpha\notin\{\frac12,1\}$. Let
$e_j=\max\{\frac12,\frac1{2r_j}\}$, which are $e_{\text{Muon}}$ and $e_{\text{s-SGD}}$ of
\cref{eqn:tte-exponents}, and $\chi=e_{\text{s-SGD}}-e_{\text{s-SVD}}$, the $\kappa$ of
\cref{thm:tte-half-aniso}. Finally $e_D=-\frac\alpha2,-\frac12,\frac\alpha2-1,0$ for $0<\alpha<1$,
$\alpha=1$, $1<\alpha<2$, $\alpha\ge2$, and $L=[\alpha=1]-[\alpha=2]+[\texp=\alpha+1]-[\texp=1]$, where
$[\cdot]$ is $1$ if the condition holds and $0$ otherwise.

\begin{corollary}[Ratio at every accuracy]\label{cor:app-phase-unified-ratio}
Fix $\alpha>0$, $\texp$ with $\alpha+\texp>1$, and $T\ge4$, and let $N_0$ and
$\epsilon_0$ be thresholds of \cref{thm:app-phase-unified} valid for both
models. If $(\alpha,\texp,\text{s-SVD})$ and $(\alpha,\texp,\text{s-SGD})$ both lie
in $\mathcal U$, then for $N\ge N_0$, $B\ge N$ and $0<\epsilon<\epsilon_0$,
\begin{equation}\label{eqn:app-phase-ratioU}
\frac{t^{\text{s-SGD}}_{T\epsilon}}{t^{\text{s-SVD}}_{T\epsilon}}\asymp
D_N\,\frac{\Upsilon_{\text{s-SGD}}(N,\epsilon)}{\Upsilon_{\text{s-SVD}}(N,\epsilon)} ,
\end{equation}
and if moreover $\texp\notin\{1,\alpha+1\}$, the right side equals
$D_N\max\{\epsilon,\epsilon_N\}^{-\chi}$. If only $(\alpha,\texp,\text{s-SGD})\in\mathcal U$,
the upper half of \cref{eqn:app-phase-ratioU} holds; if only
$(\alpha,\texp,\text{s-SVD})\in\mathcal U$, the lower half holds. For
$\epsilon=N^{-\omega}$ with fixed $\omega>0$,
\[
D_N\,\frac{\Upsilon_{\text{s-SGD}}(N,\epsilon)}{\Upsilon_{\text{s-SVD}}(N,\epsilon)}\asymp N^{e_D+\chi\min\{\omega,p-1\}}\,(\log N)^{L},
\]
with constants depending also on $\omega$.
\end{corollary}

\skpointer{cor:app-phase-unified-ratio}{Lean: proved: \texttt{\detokenize{Muon}\detokenize{.}\allowbreak{}\detokenize{Phase}\detokenize{.}\allowbreak{}\detokenize{cor}\detokenize{_}\allowbreak{}\detokenize{app}\detokenize{_}\allowbreak{}\detokenize{phase}\detokenize{_}\allowbreak{}\detokenize{unified}\detokenize{_}\allowbreak{}\detokenize{ratio}\detokenize{_}\allowbreak{}\detokenize{i}}, \texttt{\detokenize{Muon}\detokenize{.}\allowbreak{}\detokenize{Phase}\detokenize{.}\allowbreak{}\detokenize{cor}\detokenize{_}\allowbreak{}\detokenize{app}\detokenize{_}\allowbreak{}\detokenize{phase}\detokenize{_}\allowbreak{}\detokenize{unified}\detokenize{_}\allowbreak{}\detokenize{ratio}\detokenize{_}\allowbreak{}\detokenize{ii}}, +3 more}
This is the precise form of \cref{thm:tte-half-aniso}. For all large $N$, all $B\ge N$ and all
small $\epsilon$, with no window, the ratio is, when both models lie in $\mathcal U$, $D_N$ times the quotient of the clock orders
$\Upsilon_j$, and off the lines $\texp=1$ and $\texp=\alpha+1$ this is $D_N\max\{\epsilon,\epsilon_N\}^{-\kappa}$.
Above the accuracy scale the ratio grows like $\epsilon^{-\kappa}$ as $\epsilon$ decreases, the main text's
reading of Phase C, and below $\epsilon_N$ it levels off. The set $\mathcal U$ excludes \SignSVD at
$\alpha=1$, $\texp\ge1$ and \SignSGD at $\alpha\in\{\frac12,1\}$, $\texp\ge\alpha+1$, hence the main text's
exclusion of $\alpha\in\{\frac12,1\}$; with one model excluded, one half of the order survives. At
$\epsilon=N^{-\omega}$ the ratio is a power of $N$ times a power of $\log N$, the input of
\cref{thm:app-phase-omega}.

\begin{sketch}
The clock integral of each model is of order $\Upsilon_j$ at every accuracy: it grows as in
\cref{sk:E6} until $\epsilon$ reaches the accuracy scale, then freezes at the order $N^{q_j(1-r_j)}$
of $J_N$ (\cref{thm:app-phase-unified}). The lower bound needs no kernel estimate. Since
$Y_j\ge f_{N,q_j}$, the first passage is no earlier than that of the forcing, and the lower
envelope of the forcing, valid up to $(N/4)^{q_j}$, keeps it above $T\epsilon$ until
$\min\{\sigma_*,(N/4)^{q_j}\}$, where $1+\sigma_*=(A_p/T\epsilon)^{1/(2r_j)}$. Integrating the envelope up to
that time gives the two branches of $\Upsilon_j$, which for $r_j<1$ cross at $\epsilon_N$ because
$q_j(1-r_j)=(p-1)(1-r_j)/(2r_j)$:
\[
\Upsilon_j(N,\epsilon)=\max\{\epsilon,\epsilon_N\}^{-(1-r_j)/(2r_j)}\qquad(r_j<1).
\]

The upper bound needs the kernel controlled up to times of order $N^{q_j}$, and this is where
$\mathcal U$ comes from. For $r_j<1$ the comparison of \cref{sk:E6} covers $\epsilon\ge E_*\epsilon_N$, and
the comparison followed by a restart, as in \cref{thm:app-phase-frozen}, covers the band below. For $r_j\ge1$ the
forcing envelope decays faster than a slow part of the kernel: each kernel in $\mathcal U$ is
bounded by an integrable fast part plus a slow part of height of order $N^{-q}$
(\cref{lem:app-phase-templates}),
\[
g_j(v)\le K_f(1+v)^{-\upsilon}+K_sN^{-q}\Big(\frac{1+v}{N^q}\Big)^{-\mathsf b},\qquad
\upsilon>1,\ 0<\mathsf b<1,\ 0\le v\le N^q-1 .
\]
The slow part acting on the envelope $h$ leaves a remainder $w^2/\sqrt\epsilon$, where $w$ is a
power-law envelope of exponent $\mathsf b/2<1$ that the kernel does not amplify; then
$(h+w+\sqrt\epsilon)^2$ dominates the equation (\cref{lem:app-phase-square-tail}), and a restart as in
\cref{thm:app-phase-frozen} closes the bound. This bound is not available for the \SignSVD
kernel at $\alpha=1$ and the \SignSGD kernel at $\alpha\in\{\frac12,1\}$, the triples excluded from
$\mathcal U$; \SignSVD at $\alpha\ge2$ is treated directly.

The ratio law is then a division. Exactly,
$t^{\text{s-SGD}}_{T\epsilon}/t^{\text{s-SVD}}_{T\epsilon}=D_N\mathsf I_{\text{s-SGD}}/\mathsf I_{\text{s-SVD}}$
(\cref{eqn:app-phase-clockint}), and the lower bound holds for every model, which gives the
one-sided statements. For $r_j\ne1$, $e_j-\frac12$ is $(1-r_j)/(2r_j)$ or $0$ as $r_j<1$ or $r_j>1$, so
$\Upsilon_j=\max\{\epsilon,\epsilon_N\}^{-(e_j-1/2)}$ and the quotient is $\max\{\epsilon,\epsilon_N\}^{-\chi}$. At
$\epsilon=N^{-\omega}$, $\Upsilon_j\asymp N^{\min\{\omega,p-1\}(e_j-1/2)}(\log N)^{[r_j=1]}$ and
$D_N\asymp N^{e_D}(\log N)^{[\alpha=1]-[\alpha=2]}$, which gives the last display.
\end{sketch}

At the accuracy $\epsilon=N^{-\omega}$, with arbitrary integers $B_N\ge N$, write
$\mathsf R=t^{\text{s-SGD}}_{T\epsilon}/t^{\text{s-SVD}}_{T\epsilon}$. With $e_D$, $\chi$ and $L$ as before
\cref{cor:app-phase-unified-ratio}, put $S(\omega)=e_D+\chi\min\{\omega,p-1\}$ and
$\mathsf E=\{\alpha=1,\ \texp\ge1\}\cup\{\alpha=\frac12,\ \texp\ge\frac32\}$.

\begin{theorem}[Accuracy-exponent map]\label{thm:app-phase-omega}
Fix $\alpha>0$, $\texp$ with $p=\alpha+\texp>1$, $\omega>0$ and $T\ge4$. As
$N\to\infty$:
\begin{enumerate}[label=(\roman*)]
\item if $\omega>p-1$, then $\mathsf R\sim\rho_N$, so every conclusion of
\cref{thm:app-phase-accfirst,cor:app-phase-accfirst-winners} holds for $\mathsf R$;
\item if $(\alpha,\texp)\notin\mathsf E$, then $\mathsf R\asymp N^{S(\omega)}(\log N)^L$;
\item if $\alpha=1$ and $1\le\texp<2$, then $\mathsf R\lesssim N^{S(\omega)}(\log N)^L$.
\end{enumerate}
\end{theorem}

\skpointer{thm:app-phase-omega}{Lean: proved: \texttt{\detokenize{Muon}\detokenize{.}\allowbreak{}\detokenize{Phase}\detokenize{.}\allowbreak{}\detokenize{thm}\detokenize{_}\allowbreak{}\detokenize{app}\detokenize{_}\allowbreak{}\detokenize{phase}\detokenize{_}\allowbreak{}\detokenize{omega}\detokenize{_}\allowbreak{}\detokenize{i}}, \texttt{\detokenize{Muon}\detokenize{.}\allowbreak{}\detokenize{Phase}\detokenize{.}\allowbreak{}\detokenize{thm}\detokenize{_}\allowbreak{}\detokenize{app}\detokenize{_}\allowbreak{}\detokenize{phase}\detokenize{_}\allowbreak{}\detokenize{omega}\detokenize{_}\allowbreak{}\detokenize{i}\detokenize{_}\allowbreak{}\detokenize{transfer}}, +2 more}
Setting $\epsilon=N^{-\omega}$ turns the ratio law into one power of $N$, and its exponent
$S(\omega)$ is the competition of the main text's paragraph on which algorithm wins: the
dimension exponent $e_D$, negative for $\alpha<2$, against the accuracy exponent
$\chi\min\{\omega,p-1\}$, which is nondecreasing in $\omega$ until the accuracy reaches $\epsilon_N$ at
$\omega=p-1$ and is constant beyond. Part (ii) holds whenever both models lie in $\mathcal U$,
and (iii) is the one-sided bound for $\alpha=1$, $1\le\texp<2$, where only \SignSGD does. Part
(i) is sharper in the frozen range $\omega>p-1$: the ratio is asymptotic to the
accuracy-first value $\rho_N$, whose limits and orders \cref{thm:app-phase-accfirst}
computes; where the order only says that the two are balanced, this gives the limit, the
tie constants of Phases D and E and of the chord. It is the theorem the main text cites for
Phase B in \cref{fig:tte-phase-diagram}.

\begin{sketch}
Parts (ii) and (iii) are \cref{cor:app-phase-unified-ratio} at $\epsilon=N^{-\omega}$, which lies
below its threshold $\epsilon_0$ for large $N$. Both models lie in $\mathcal U$ exactly when
$(\alpha,\texp)\notin\mathsf E$, and for $\alpha=1$, $1\le\texp<2$ only \SignSGD does, which leaves the upper
half. The exponent is explicit (\cref{eqn:app-phase-chi}):
\[
\begin{aligned}
\chi(p-1)&=\max\Big\{\frac{p-1}2,\alpha\Big\}-\max\Big\{\frac{p-1}2,\frac\alpha2\Big\}\\
&=\frac\alpha2,\ \ \frac{\alpha+1-\texp}2,\ \ 0\quad(\texp<1,\ 1\le\texp\le\alpha+1,\ \texp\ge\alpha+1),
\end{aligned}
\]
so $S(\omega)=e_D+\min\{\omega/(p-1),1\}\,\chi(p-1)$.

Part (i) is \cref{thm:app-phase-frozen}, since $\omega>p-1$ means $x=N^{p-1-\omega}\to0$; what it
needs in addition is the behaviour of $\rho_N=D_NJ_N(p,\alpha)/J_N(p,\alpha/2)$ as $N\to\infty$. For
$r_q<1$ the finite forcing integral grows like a power of $N$ with an explicit constant.
The substitution $u=N^qz$ turns the rescaled forcing into a Riemann sum,
\[
2N^{p-1}f_{N,q}(N^qz)=\frac1N\sum_{i\le N}\Big(\frac iN\Big)^{-p}e^{-2z(i/N)^{-q}}
\longrightarrow\int_0^1t^{-p}e^{-2zt^{-q}}\,\dif t=\frac{(2z)^{-2r_q}}q\,\Gamma(2r_q,2z),
\]
with $\Gamma(a,y)=\int_y^\infty t^{a-1}e^{-t}\dif t$. Halving every rate gives a majorant
$4^{r}C_{p,q}^2z^{-2r}e^{-z}$ that does not depend on $N$ (\cref{lem:app-phase-shift}), so
dominated convergence gives (\cref{prop:app-phase-forcing-r})
\[
\frac{J_N(p,q)}{N^{q(1-r_q)}}\longrightarrow(8q)^{-1/2}\mathcal K(r_q),\qquad
\mathcal K(r)=\int_0^\infty y^{-r}\,\Gamma(2r,y)^{1/2}\,\dif y,\qquad \mathcal K(\tfrac12)=\sqrt{2\pi}.
\]
For $r_q=1$, $J_N(p,q)\sim\sqrt{q/8}\,\log N$, and for $r_q>1$, $J_N(p,q)\to J(p,q)$. Multiplying
the quotient of these by the order of $D_N$ gives each limit and order of
\cref{thm:app-phase-accfirst}. For example, at $\alpha=\texp=1$ the factors
$D_N\sim\frac\pi2N^{-1/2}\log N$ and $J_N(p,1)/J_N(p,\frac12)\sim2\sqrt\pi\,N^{1/2}/\log N$ multiply to
$\pi^{3/2}$.
\end{sketch}

\SignSVD wins, \SignSGD wins, or the two are balanced according as $\mathsf R\to\infty$,
$\mathsf R\to0$ or $\mathsf R\asymp1$. For $1<\alpha<2$ and $\texp<2\alpha-1$, $\omega_c$ is the critical
exponent of the main text: $\omega_c=(2-\alpha)(p-1)/\alpha$ if $\texp<1$ and
$\omega_c=(2-\alpha)(p-1)/(\alpha+1-\texp)$ if $1\le\texp<2\alpha-1$, so that $0<\omega_c<p-1$. The
constants $K_A$ and $K_c$ are the explicit positive limits of $\rho_N$ given by
\cref{thm:app-phase-accfirst} for $0<\alpha<1$, $\texp<1$ and on the chord $\texp=2\alpha-1$,
$1<\alpha<2$.

\begin{corollary}[Phases of the accuracy-exponent map]\label{cor:app-phase-omega-winners}
Under the hypotheses of \cref{thm:app-phase-omega}:
\begin{enumerate}[label=(\alph*)]
\item \emph{Phase A}, $\alpha>2$, $\texp\le\alpha+1$: \SignSVD wins for every $\omega$.
\item \emph{Phase B}, $0<\alpha<1$ and $\texp\ge1$ with $(\alpha,\texp)\notin\mathsf E$, or
$1<\alpha<2$ and $\texp>2\alpha-1$: \SignSGD wins for every $\omega$.
\item \emph{Phase C}, $1<\alpha<2$, $\texp<2\alpha-1$: \SignSGD wins if $\omega<\omega_c$
and \SignSVD wins if $\omega>\omega_c$; at $\omega=\omega_c$ the two are balanced if
$\texp\ne1$, and \SignSGD wins, with order $1/\log N$, if $\texp=1$.
\item \emph{Phase D}, $0<\alpha<1$, $\texp<1$: \SignSGD wins if $\omega<p-1$, and the
two are balanced if $\omega\ge p-1$, with limit $K_A$ if $\omega>p-1$.
\item \emph{Tie}, $\alpha>2$, $\texp>\alpha+1$: balanced for every $\omega$.
\item \emph{Chord}, $1<\alpha<2$, $\texp=2\alpha-1$: \SignSGD wins if $\omega<p-1$, and
the two are balanced if $\omega\ge p-1$, with limit $K_c$ if $\omega>p-1$.
\item \emph{$\alpha=1$}: for $\texp<1$, \SignSGD wins if $\omega<p-1$ and \SignSVD wins,
with order $\log N$, if $\omega\ge p-1$; for $\texp=1$, \SignSGD wins if $\omega<1$,
and $\mathsf R\to\pi^{3/2}$ if $\omega>1$; for $1<\texp<2$, \SignSGD wins for every
$\omega$.
\item \emph{$\alpha=2$}: \SignSVD wins for $\texp<3$, the two are balanced for
$\texp=3$, and \SignSGD wins for $\texp>3$, for every $\omega$.
\end{enumerate}
\end{corollary}

\skpointer{cor:app-phase-omega-winners}{Lean: proved: \texttt{\detokenize{Muon}\detokenize{.}\allowbreak{}\detokenize{Phase}\detokenize{.}\allowbreak{}\detokenize{cor}\detokenize{_}\allowbreak{}\detokenize{app}\detokenize{_}\allowbreak{}\detokenize{phase}\detokenize{_}\allowbreak{}\detokenize{omega}\detokenize{_}\allowbreak{}\detokenize{winners}\detokenize{_}\allowbreak{}\detokenize{a}}, \texttt{\detokenize{Muon}\detokenize{.}\allowbreak{}\detokenize{Phase}\detokenize{.}\allowbreak{}\detokenize{cor}\detokenize{_}\allowbreak{}\detokenize{app}\detokenize{_}\allowbreak{}\detokenize{phase}\detokenize{_}\allowbreak{}\detokenize{omega}\detokenize{_}\allowbreak{}\detokenize{winners}\detokenize{_}\allowbreak{}\detokenize{b}}, +6 more}
This is the phase diagram of \cref{fig:tte-phase-diagram} in exact form, with its
boundary lines; the Tie is the main text's Phase E. In Phase A \Muon wins at every
accuracy, in Phase B \SignSGD does, in Phase C \SignSGD wins at low accuracy and \Muon once
$\epsilon\ll N^{-\omega_c}$, and in Phase D \SignSGD wins above the accuracy scale and the two tie
below it. The chord $\texp=2\alpha-1$ between B and C is where $\omega_c$ reaches $p-1$, that is,
where the critical accuracy meets the accuracy scale, so as $\omega\to\infty$ the boundary
between B and C becomes the chord. The level $\omega_c=1$ is the curve $\texp_c(\alpha)$ of
\cref{prop:app-phase-winners}, the cut of the map at $\epsilon=N^{-1}$, where that proposition's
window ends. Not covered are the winner for $\omega\le p-1$ on
$\{\alpha=1,\texp\ge2\}\cup\{\alpha=\frac12,\texp\ge\frac32\}$, and the limit constants for $\omega\le p-1$ in
the balanced cases; at $\alpha=\texp=\omega=1$ only $\mathsf R\lesssim1$ is proved.

\begin{sketch}
Off $\mathsf E$ the winner is the sign of $S(\omega)$, and of $L$ where $S(\omega)=0$
(\cref{thm:app-phase-omega}(ii)). Write $\omega'=\min\{\omega,p-1\}$. For $0<\alpha<1$, $e_D=-\frac\alpha2$.
If $\texp<1$, then $S(\omega)=\frac\alpha2(\omega'/(p-1)-1)$ and $L=0$; this is negative below the
accuracy scale and zero at and beyond it, which is Phase D, with the limit $K_A$ of
\cref{thm:app-phase-omega}(i). If $\texp\ge1$, then $\chi(p-1)=\max\{\frac{\alpha+1-\texp}2,0\}\le\frac\alpha2$, so
$S(\omega)\le0$, with equality only at $\texp=1$, $\omega\ge p-1$, where $L=-1$; this is the first part of
Phase B. For $1<\alpha<2$, $e_D=\frac\alpha2-1<0$, and for $1\le\texp\le\alpha+1$
\[
S(\omega)=\frac\alpha2-1+\frac{\omega'}{p-1}\cdot\frac{\alpha+1-\texp}2 ,
\]
which vanishes at $\omega'=\omega_c$. At $\omega'=p-1$ it equals $(2\alpha-1-\texp)/2$, so
$\omega_c<p-1$ exactly when $\texp<2\alpha-1$. Below the chord \SignSGD wins for $\omega<\omega_c$
and \SignSVD for $\omega>\omega_c$, which is Phase C; at $\omega=\omega_c$, $L=-[\texp=1]$ gives the balance or
the win of order $1/\log N$. On the chord $S(\omega)<0$ for $\omega<p-1$ and $S=L=0$ beyond, with
the limit $K_c$; above it, and for $\texp>\alpha+1$, $S(\omega)<0$, which is the second part of
Phase B. The case $\texp<1$ is the same with $\alpha$ in place of $\alpha+1-\texp$.

For $\alpha\ge2$ the dimension factor does not decay, $e_D=0$, and only the accuracy
factor and the logarithms count. For $\alpha>2$, $S(\omega)>0$ when $\texp<\alpha+1$, which is
Phase A; for $\texp\ge\alpha+1$, $S(\omega)=0$ and $L=[\texp=\alpha+1]$, which gives \SignSVD on the line
and the Tie beyond it. At $\alpha=2$, $S(\omega)>0$ for $\texp<3$, and $S(\omega)=0$ with
$L=-1+[\texp=3]$ for $\texp\ge3$. At $\alpha=1$ and $1\le\texp<2$ only the upper bound of
\cref{thm:app-phase-omega}(iii) is available. There
$S(\omega)=-\frac12+\frac{\omega'}{p-1}\cdot\frac{2-\texp}2<0$ unless $\texp=1$ and $\omega\ge p-1$, which forces
$\mathsf R\to0$; at $\texp=1$, $\omega>1$ the frozen limit $\pi^{3/2}$ of $\rho_N$ applies.
\end{sketch}

\section{SignSGD}\label{sk:ch-signsgd}
\SignSGD keeps no momentum buffer, so this chapter has no momentum version, and every
statement concerns one step from a given state $\Delta$ with one fresh batch; none is a
statement about a trajectory. The step is $\Delta^+=\Delta-\lr S$ with $S=\sign(G^\circ)$,
the entrywise sign of the minibatch gradient
$G^\circ=\frac1B\sum_{a}(\Xout^{(a)})^\T\Delta\Xin^{(a)}\,\Xout^{(a)}\tensor\Xin^{(a)}$, which is,
when the risk is positive, also the sign of the normalized gradient $\Gt$ of \cref{eq:main_signSGD}. Since
$\norm S_\F^2\le\Nout\Nin$, the exact one-step identities of the risk and of the row energy
$\CMscr Q_i(\Delta)=\norm{\Delta^\T u_i}^2$ apply to it (\cref{eqn:app-sgd-risk,eqn:app-sgd-row}),
with $\langle V,W\rangle_\Sigma=\Tr(\Sigmaout V\Sigmain W^\T)$ and $\norm V_\Sigma^2=\langle V,V\rangle_\Sigma$:
\[
\begin{gathered}
\EE\,\cR(\Delta^+)=\cR(\Delta)-\lr\,\langle\Delta,\EE S\rangle_\Sigma+\frac{\lr^2}2\,\EE\norm S_\Sigma^2,\\
\EE\,\CMscr Q_i(\Delta^+)=\CMscr Q_i(\Delta)-2\lr\,\EE\langle\Delta^\T u_i,S^\T u_i\rangle+\lr^2\,\EE\norm{S^\T u_i}^2 .
\end{gathered}
\]
The drift constant of the main text is $\mathcal N_B=\EE\big[(\sum_{b=1}^Bx_b^2y_b^2)^{1/2}\big]$ for
independent standard Gaussians $x_b,y_b$ (\cref{eqn:signsgd-Nb}). The chapter bounds $\mathcal N_B$,
evaluates the two expectations in the risk identity for isotropic data, and then, for a
Haar eigenbasis, evaluates the row volatility and assembles the row recursion.

\begin{proposition}[The batch constant]\label{prop:app-sgd-batch-coefficient}
For every integer $B\ge1$, the constant \cref{eqn:signsgd-Nb} satisfies
$\mathcal N_1=2/\pi$ and
\begin{equation}\label{eqn:app-sgd-NB}
\max\Big\{\frac2\pi\sqrt B,\ \sqrt B-\frac4{\sqrt B}\Big\}\le\mathcal N_B\le\sqrt B,\qquad
0\le1-\frac{\mathcal N_B}{\sqrt B}\le\frac4B .
\end{equation}
In particular $\mathcal N_B/\sqrt B\to1$ as $B\to\infty$.
\end{proposition}

\skpointer{prop:app-sgd-batch-coefficient}{Lean: proved: \texttt{\detokenize{Muon}\detokenize{.}\allowbreak{}\detokenize{SignSGD}\detokenize{.}\allowbreak{}\detokenize{prop}\detokenize{_}\allowbreak{}\detokenize{app}\detokenize{_}\allowbreak{}\detokenize{sgd}\detokenize{_}\allowbreak{}\detokenize{batch}\detokenize{_}\allowbreak{}\detokenize{coefficient}\detokenize{_}\allowbreak{}\detokenize{i}}, \texttt{\detokenize{Muon}\detokenize{.}\allowbreak{}\detokenize{SignSGD}\detokenize{.}\allowbreak{}\detokenize{prop}\detokenize{_}\allowbreak{}\detokenize{app}\detokenize{_}\allowbreak{}\detokenize{sgd}\detokenize{_}\allowbreak{}\detokenize{batch}\detokenize{_}\allowbreak{}\detokenize{coefficient}\detokenize{_}\allowbreak{}\detokenize{ii}}, +2 more}
The batch constant is the drift coefficient of \SignSGD: the main text writes the
isotropic drift as $\mathfrak d=\mathcal N_B/\sqrt\pi$ and then uses $\mathcal N_B\approx\sqrt B$
throughout, which gives the tabulated drift $\mu_i\sqrt{B/(\pi\bar\mu)}$ of
\cref{tab:power-law-kernels}. The proposition justifies this replacement with an explicit
error: $\mathcal N_B$ differs from $\sqrt B$ by a factor in $[1-4/B,1]$, and it is never
smaller than $\frac2\pi\sqrt B$, a bound attained at $B=1$.

\begin{sketch}
The constant is the mean of the square root of a sum $T_B=\sum_bx_b^2y_b^2$ of $B$
independent terms with mean $1$ and second moment $(\EE x^4)(\EE y^4)=9$, so $T_B/B$ has
mean $1$ and variance $8/B$. At $B=1$ independence gives $\mathcal N_1=(\EE\abs g)^2=2/\pi$.
Jensen's inequality gives $\mathcal N_B\le(\EE T_B)^{1/2}=\sqrt B$, and the
Cauchy--Schwarz comparison $(\sum_bz_b^2)^{1/2}\ge B^{-1/2}\sum_b\abs{z_b}$ with
$z_b=x_by_b$ gives $\mathcal N_B\ge\frac2\pi\sqrt B$.

The Jensen deficit is at most half the variance. Rationalizing $\sqrt s-1$ gives the exact
identity, for $s\ge0$,
\[
\sqrt s=1+\frac{s-1}2-\frac{(s-1)^2}{2(\sqrt s+1)^2},
\]
whose last term lies between $0$ and $(s-1)^2/2$. At $s=T_B/B$ the linear term has mean
zero, so $0\le1-\EE\sqrt{T_B/B}\le\frac12\cdot\frac8B=\frac4B$, which is the second
inequality of \cref{eqn:app-sgd-NB} and the bound $\mathcal N_B\ge\sqrt B-4/\sqrt B$.
\end{sketch}

In the next proposition constants $K\ge1$, $0<f_-\le f_+$ and $C_\Delta>0$ are fixed, and a
state $\Delta$ with batch size $B$ is in the isotropic regime of \cref{def:app-sgd-iso-regime}
if, with $k\defeq\min\{\Nout,\Nin,B\}$,
\[
\max\{\Nout,\Nin,B\}\le Kk,\qquad f_-\le\norm\Delta_\F\le f_+,\qquad \max_{i,j}\abs{\Delta_{ij}}\le\frac{C_\Delta}k .
\]
The threshold is $k_0=\max\{4,8KC_\Delta^2/f_-^2\}$, $C_2$ is the universal constant of the
sign comparison \cref{lem:app-sgd-comparison}, and $C_r$ depends only on $K,f_-,C_\Delta$.

\begin{proposition}[Isotropic \SignSGD recursion]\label{prop:app-sgd-iso}\label{prop:signsgd-recursion-B}
Let $\Sigmain=\Idin$ and $\Sigmaout=\Idout$, and let $\Dt$ be the state at
step $t$. Suppose that, almost surely, $\Dt$ with the batch size $B$ is in
the isotropic regime of \cref{def:app-sgd-iso-regime} and $k\ge k_0$. Then the
\SignSGD step satisfies
\begin{equation}\label{eqn:signsgd-risk-recursion-B}
\EE_t\big[\cR(t+1)\big]=\cR(t)-\lr\,\mathcal N_B\sqrt{\frac{4\cR(t)}\pi}+\frac{\lr^2}2\,\Nin\Nout+E_t,\qquad
\abs{E_t}\le\lr\,K\sqrt{2\cR(t)}\,\Big(C_2+\frac{C_r}{\sqrt k}\Big),
\end{equation}
with $C_2$ universal and $C_r$ from \cref{thm:app-sgd-iso-drift}. The error
is at most $\pi K(C_2+C_r)/(2\sqrt{2/\pi}\sqrt k)$ times the drift term.
\end{proposition}

\skpointer{prop:app-sgd-iso}{Lean: proved: \texttt{\detokenize{Muon}\detokenize{.}\allowbreak{}\detokenize{SignSGD}\detokenize{.}\allowbreak{}\detokenize{prop}\detokenize{_}\allowbreak{}\detokenize{app}\detokenize{_}\allowbreak{}\detokenize{sgd}\detokenize{_}\allowbreak{}\detokenize{iso}\detokenize{_}\allowbreak{}\detokenize{filtered}}, \texttt{\detokenize{Muon}\detokenize{.}\allowbreak{}\detokenize{SignSGD}\detokenize{.}\allowbreak{}\detokenize{prop}\detokenize{_}\allowbreak{}\detokenize{app}\detokenize{_}\allowbreak{}\detokenize{sgd}\detokenize{_}\allowbreak{}\detokenize{iso}\detokenize{_}\allowbreak{}\detokenize{frozen}}}
This is the \SignSGD case of the scalar recursion \cref{eqn:comparison-general-form}, with
drift $\mathfrak d=\mathcal N_B/\sqrt\pi$ and volatility $\mathfrak v=\Nin\Nout$ ($N^2$ in the square
case). It holds one step at a time, from any state in the isotropic regime, whose entries are all of order $1/k$, with
no hypothesis on its rank or operator norm. The volatility is exact, and the drift carries a
relative error of order $k^{-1/2}$.

\begin{sketch}
The volatility is exact. With identity covariances $\norm S_\Sigma=\norm S_\F$, and every entry of
$G^\circ$ is a polynomial in the Gaussian batch that is not identically zero, hence nonzero almost
surely. So $\norm S_\F^2=\Nin\Nout$ (\cref{prop:app-sgd-iso-volatility}), and the risk identity reads
$\EE_t[\cR(t+1)]=\cR(t)-\lr\,\mathsf D+\frac{\lr^2}2\Nin\Nout$ with $\mathsf D=\langle\Dt,\EE_tS\rangle_\F$.

Each expected sign is compared with that of a shifted Gaussian. With
$F=\norm\Delta_\F$, the entry $BG^\circ_{ij}$ is a sum of $B$ independent copies of
$X=(z^\T\Delta w)z_iw_j$, of mean $\Delta_{ij}$ and variance $s^2=F^2+O(1/k)$
(\cref{lem:app-sgd-moments}). Writing the sign as a limit of
$\frac2\pi\int_0^\infty e^{-\varepsilon u}\sin(uv)\,\dif u/u$ compares expected signs through the
characteristic function $\phi_Y$ of the standardized coordinate $Y$, uniformly in the shift
$x=\sqrt B\Delta_{ij}/s$ and with universal $C_1,C_2$ (\cref{lem:app-sgd-comparison}):
\[
\Big|\EE\sign\Big(\frac1{\sqrt B}\sum_{a=1}^BY_a+x\Big)-\EE\sign(g+x)\Big|\le\frac{C_1\abs{\EE Y^3}}{\sqrt B}+\frac{C_2}B .
\]
This needs $\abs{\phi_Y(t)}$ to decay like $\abs t^{-1/2}$, and that is where small entries enter.
Given the other coordinates, $X$ is Gaussian in the coordinates of $z$ other than $z_i$, with a
variance controlled by the block $\bar\Delta$ of $\Delta$ without row $i$ and column $j$; every row
and column of $\Delta$ has squared norm at most $KC_\Delta^2/k$, so $\norm{\bar\Delta}_\F^2\ge\frac34F^2$
for $k\ge k_0$ (\cref{thm:app-sgd-iso-drift}), more than the $\norm{\bar\Delta}_\F\ge F/2$ that \cref{lem:app-sgd-decay} needs. The third cumulant $\EE Y^3$ is $O(1/k)$ as well. With
$\EE\sign(g+x)=\sqrt{2/\pi}\,x+O(\abs x^3)$ and $\abs{\sqrt B-\mathcal N_B}\le4/\sqrt B$
(\cref{prop:app-sgd-batch-coefficient}), every entry satisfies (\cref{thm:app-sgd-iso-drift})
\[
\Big|\EE\sign(G^\circ_{ij})-\sqrt{\frac2\pi}\,\mathcal N_B\frac{\Delta_{ij}}F\Big|\le\frac{C_2}B+\frac{C_r}{k^{3/2}} .
\]

The sum over entries costs the $\ell^1$ norm of the state, which is why the error is relative of
order $k^{-1/2}$. Since $\sum_{ij}\Delta_{ij}^2/F=F$, the leading terms add up to
$\mathsf D_0=\sqrt{2/\pi}\,\mathcal N_BF=\mathcal N_B\sqrt{4\cR/\pi}\ge(2/\pi)^{3/2}\sqrt kF$, while
$\sum_{ij}\abs{\Delta_{ij}}\le\sqrt{\Nout\Nin}\,F\le KkF$ bounds the summed errors by
$KF(C_2+C_r/\sqrt k)$, which is of order one.
\end{sketch}

The Haar model of \cref{def:app-sgd-state-law} fixes $K\ge1$, $0<c\le L$, $0<f_-\le f_+$ and
$C_{\rm col}>0$, dimensions with $k\defeq\min\{\Nout,\Nin,B\}\ge2$ and
$\max\{\Nout,\Nin,B\}\le Kk$, and deterministic $\Lambda=\diag(\mu_1,\dots,\mu_{\Nout})$ and
$D\in\R^{\Nout\times\Nin}$ with
\[
c\,\Idout\preceq\Lambda\preceq L\,\Idout,\qquad f_-\le\norm D_\F\le f_+,\qquad \max_j\norm{D_{\cdot j}}^2\le\frac{C_{\rm col}}k .
\]
With $U$ Haar distributed on $O(\Nout)$ and independent of the whitened batch $(z_b,w_b)_{b\le B}$,
it sets $\Sigmaout=U\Lambda U^\T$, $\Sigmain=\Idin$, $\Delta=UD$, $\Xout^{(b)}=U\Lambda^{1/2}z_b$ and
$\Xin^{(b)}=w_b$. Then $u_i=Ue_i$, $\CMscr Q_i(\Delta)=\norm{D_{i\cdot}}^2$, the risk
$2\cR=\Tr(\Lambda DD^\T)$ is deterministic, $\bar\mu=\Tr\Lambda/\Nout$ and $T_D=\Lambda DD^\T\Lambda$.
Expectations integrate the batch and $U$, a constant is admissible if it depends only on
$K,c,L,f_\pm,C_{\rm col}$, and the main text's row volatility is
$\mathfrak v_i^{\text{s-SGD}}=\Nin\big[1+\frac2\pi\big(\frac{\mu_i}{\bar\mu}-1\big)\big]$.

\begin{theorem}[Corrected row volatility]\label{thm:app-sgd-modes}
In the Haar model of \cref{def:app-sgd-state-law}, define the raw annealed
row volatility of the update and the corrected array
\begin{align}
\volraw_i&\defeq\EE\norm{S^\T u_i}^2,\label{eqn:app-sgd-actual-row}\\
\mathfrak v_i^{\rm corr}&\defeq\Nin\Big(1-\frac2\pi\Big)+\frac2\pi\,\Nout\Big[\frac{(\Nin+2)\,\mu_i}{\Tr\Lambda}+\frac{(B+2\Nin+3)\,(T_D)_{ii}}{2\cR\,\Tr\Lambda}\Big].\label{eqn:app-sgd-corr}
\end{align}
There is an admissible constant $C_{\rm row}$ such that, for every $i$,
\begin{equation}\label{eqn:app-sgd-modes}
\abs{\volraw_i-\mathfrak v_i^{\rm corr}}\le C_{\rm row}\,\frac{\Nin}{\sqrt k},\qquad
\Big|\frac{\volraw_i}{\mathfrak v_i^{\rm corr}}-1\Big|\le\frac{C_{\rm row}}{(1-2/\pi)\sqrt k}.
\end{equation}
The actual total is exact, $\sum_i\volraw_i=\Nout\Nin$, while
\begin{equation}\label{eqn:app-sgd-excess}
\sum_i\mathfrak v_i^{\rm corr}-\Nout\Nin=\frac2\pi\,\Nout\Big[2+\frac{(B+2\Nin+3)\Tr T_D}{2\cR\,\Tr\Lambda}\Big]\in\Big(0,\ \frac2\pi\Nout\big[2+(3K+3)L/c\big]\Big].
\end{equation}
The corrected array exceeds the main text's array row by row:
\begin{equation}\label{eqn:app-sgd-difference}
\mathfrak v_i^{\rm corr}-\mathfrak v_i^{\text{s-SGD}}=\frac2\pi\,\Nout\Big[\frac{2\mu_i}{\Tr\Sigmaout}+\frac{(B+2\Nin+3)\,(\Lambda DD^\T\Lambda)_{ii}}{2\cR\,\Tr\Sigmaout}\Big]\ge0 .
\end{equation}
\end{theorem}

\skpointer{thm:app-sgd-modes}{Lean: proved: \texttt{\detokenize{Muon}\detokenize{.}\allowbreak{}\detokenize{SignSGD}\detokenize{.}\allowbreak{}\detokenize{thm}\detokenize{_}\allowbreak{}\detokenize{app}\detokenize{_}\allowbreak{}\detokenize{sgd}\detokenize{_}\allowbreak{}\detokenize{modes}\detokenize{_}\allowbreak{}\detokenize{i}}, \texttt{\detokenize{Muon}\detokenize{.}\allowbreak{}\detokenize{SignSGD}\detokenize{.}\allowbreak{}\detokenize{thm}\detokenize{_}\allowbreak{}\detokenize{app}\detokenize{_}\allowbreak{}\detokenize{sgd}\detokenize{_}\allowbreak{}\detokenize{modes}\detokenize{_}\allowbreak{}\detokenize{ii}}, +2 more}
The theorem computes the \SignSGD row volatility $\mathfrak v_i$ of \cref{eqn:ri-recursion} on a Haar
eigenbasis and is the source of the footnote of \cref{sec:half-aniso}: the actual row volatility is the corrected array
$\mathfrak v_i^{\rm corr}$ up to a relative error of order $k^{-1/2}$, and this array exceeds the
main text's $\mathfrak v_i^{\text{s-SGD}}$ by a nonnegative term that depends on the state
through $(T_D)_{ii}$. In total the excess is of order $\Nout$ against $\Nout\Nin$, but for a
single row it can be of the same order as $\mathfrak v_i^{\text{s-SGD}}$: for $\Lambda=\Id$, $B=N$
and $D=e_1\tensor\mathbf 1/\sqrt N$ the first row has $\mathfrak v_1^{\rm corr}=(1+\frac6\pi)N+\frac{10}\pi$,
whereas $\mathfrak v_1^{\text{s-SGD}}=N$.

\begin{sketch}
Averaging over the eigenbasis reduces the row volatility exactly to the column directions of a
matrix that does not involve $U$. The residuals do not involve $U$, so $G^\circ=UH$ with $H=\frac1B\sum_b(z_b^\T\Lambda^{1/2}Dw_b)\,\Lambda^{1/2}z_b\tensor w_b$ independent of
$U$, and the $j$th column of $U^\T S$ is $U^\T\sign(Uh_j)$. For a fixed nonzero $a\in\R^n$ and
$Y=U^\T\sign(Ua)$, invariance under the rotations that fix $a$ leaves two unknown numbers in
$\EE_U[Y\tensor Y]$; the trace is $\norm Y^2=n$, and $a^\T Y/\norm a$ is the $\ell^1$ norm of a uniform
unit vector, whose second moment is $1+2(n-1)/\pi$. This gives (\cref{lem:app-sgd-haar-sign,lem:app-sgd-reduction})
\[
\EE_U\big[Y\tensor Y\big]=\Big(1-\frac2\pi\Big)\Id+\frac{2n}\pi\,\frac{a\tensor a}{\norm a^2},\qquad
\volraw_i=\Nin\Big(1-\frac2\pi\Big)+\frac2\pi\,\Nout\sum_{j=1}^{\Nin}\EE\frac{(h_j)_i^2}{\norm{h_j}^2}.
\]
Every entry of $UH$ is nonzero almost surely, so $\norm S_\F^2=\Nout\Nin$ and the total
$\sum_i\volraw_i=\norm S_\F^2$ is exact.

The random column norm can be replaced by $x_0=2\cR\Tr\Lambda/B$. The column $h_j$ is the mean of $B$ independent vectors of mean
$m_j=\Lambda D_{\cdot j}$, so $\norm{h_j}^2$ differs from $x_0$ by a bias of order $1/k$ and a
fluctuation whose fourth moment is of order $k^{-2}$ (a fourth-moment count over the samples), so
$\EE\abs{(h_j)_i^2/\norm{h_j}^2-(h_j)_i^2/x_0}=O(k^{-3/2})$ (\cref{lem:app-sgd-column-norm}), and summed
over the $\Nin$ columns with the factor $\frac2\pi\Nout\le\frac2\pi Kk$ this is the error
$C_{\rm row}\Nin/\sqrt k$.

The replaced sum is an exact Gaussian moment, and the mean of the columns produces the state
term. Gaussian pairings give (\cref{eqn:app-sgd-HH})
\[
\sum_j\EE\big[h_j\tensor h_j\big]=\EE\big[HH^\T\big]=\frac{2\cR\,(\Nin+2)}B\,\Lambda+\Big(1+\frac{2\Nin+3}B\Big)T_D,
\]
where $T_D=\sum_jm_j\tensor m_j$ collects the squared column means. The $(i,i)$ entry of $\EE[HH^\T]$ divided
by $x_0$ is the bracket of \cref{eqn:app-sgd-corr}. Summing over $i$ and subtracting the main text's
array give \cref{eqn:app-sgd-excess,eqn:app-sgd-difference}, with
$\Tr T_D\le2\cR L$ and $\Tr\Lambda\ge c\Nout$ for the upper bound.
\end{sketch}

In the next proposition the Haar model and its notation are those set before
\cref{thm:app-sgd-modes}. The threshold is $k_*=\max\{4,\lceil4LC_{\rm col}/(cf_-^2)\rceil\}$,
$C_*$ is the admissible constant of the half-anisotropic drift \cref{thm:app-sgd-half-drift}, and
$C_{\rm row}$ is the constant of \cref{thm:app-sgd-modes}.

\begin{proposition}[Half-anisotropic \SignSGD recursion]\label{prop:signsgd-recursion-anisotropic}
In the Haar model of \cref{def:app-sgd-state-law} with $k\ge k_*$, let
$\lr\ge0$ be deterministic and $\Delta^+=\Delta-\lr S$. Then, with
$\mathfrak d_i\defeq\mu_i\mathcal N_B/\sqrt{\pi\bar\mu}$, there are real numbers
$E_1,\dots,E_{\Nout}$ such that
\begin{equation}\label{eqn:signsgd-risk-recursion-anisotropic}
\EE\,\CMscr Q_i(\Delta^+)=\CMscr Q_i(\Delta)-\frac{2\lr\,\mathfrak d_i}{\sqrt\cR}\,\CMscr Q_i(\Delta)+\lr^2\,\mathfrak v_i^{\rm corr}+E_i,\qquad
\sum_i\abs{E_i}\le2\lr f_+\sqrt{C_*}+\lr^2C_{\rm row}\,\frac{\Nout\Nin}{\sqrt k}.
\end{equation}
The risk satisfies
\begin{equation}\label{eqn:app-sgd-aniso-risk}
\EE\,\cR(\Delta^+)=\cR-\lr\sum_i\frac{\mu_i\,\mathfrak d_i}{\sqrt\cR}\,\CMscr Q_i(\Delta)+\frac{\lr^2}2\sum_i\mu_i\,\mathfrak v_i^{\text{s-SGD}}+\mathcal E_\lr,
\end{equation}
with $\mathcal E_\lr$ bounded as in \cref{cor:app-sgd-half-risk}.
\end{proposition}

\skpointer{prop:signsgd-recursion-anisotropic}{Lean: proved: \texttt{\detokenize{Muon}\detokenize{.}\allowbreak{}\detokenize{SignSGD}\detokenize{.}\allowbreak{}\detokenize{prop}\detokenize{_}\allowbreak{}\detokenize{signsgd}\detokenize{_}\allowbreak{}\detokenize{recursion}\detokenize{_}\allowbreak{}\detokenize{anisotropic}}}
This is the one-step row recursion \cref{eqn:ri-recursion} for \SignSGD on a Haar eigenbasis, which
the main text cites for the claim that \SignSGD admits the same row recursion as \Muon. The
drift $\mathfrak d_i=\mu_i\mathcal N_B/\sqrt{\pi\bar\mu}$ is the tabulated $\mu_i\sqrt{B/(\pi\bar\mu)}$ of
\cref{tab:power-law-kernels} with $\sqrt B$ replaced by $\mathcal N_B$. The per-row volatility is the
state-dependent $\mathfrak v_i^{\rm corr}$ of \cref{thm:app-sgd-modes}, and the main text's array
$\mathfrak v_i^{\text{s-SGD}}$ is correct in the covariance-weighted total, the risk identity
\cref{eqn:app-sgd-aniso-risk}. The hypotheses of the Haar model, among them a covariance bounded below
and a state independent of the eigenbasis, are not shown to persist after the step.

\begin{sketch}
The whitened batch is independent of $U$, so the row identity of the chapter introduction holds given $U$;
averaged over $U$, its quadratic term is $\lr^2\volraw_i$, which \cref{thm:app-sgd-modes}
replaces by $\lr^2\mathfrak v_i^{\rm corr}$. It remains to compute the mean update $M(U)=\EE[S\mid U]$.

Each entry of $M(U)$ is an isotropic coordinate drift for a whitened state with small entries.
Rotating the output feature so that its $p$th coordinate
becomes a multiple of the first shows that $\sign(G^\circ_{pj})$ has the law of an entry of the
isotropic gradient of a state of Frobenius norm $\sqrt{2\cR}$ whose entry is
$\mu_{pj}=(\Sigmaout\Delta)_{pj}/\sqrt{v_p}$, with $v_p=(\Sigmaout)_{pp}$ (\cref{lem:app-sgd-whitening}). The
coordinate estimate of the isotropic case (\cref{prop:app-sgd-coordinate}) then gives $\sqrt{2/\pi}\,\mathcal N_B\mu_{pj}/\sqrt{2\cR}$.
Its hypothesis, a row and a column that carry a small part of the risk, is supplied by the
rotation: with $\omega_p$ the $p$th row of $U$, a uniform unit vector,
$\mu_{pj}=\omega_p^\T\Lambda D_{\cdot j}/\sqrt{v_p}$ has $\EE\abs{\mu_{pj}}^{2q}=O((\Nout k)^{-q})$, and the row
is heavy only with probability $O(\Nout^{-4})$ (\cref{lem:app-sgd-deletion}). The coordinate errors are
squared and summed before averaging over $U$, and their total stays bounded
(\cref{prop:app-sgd-coordinate-sum}).

The diagonal of the covariance concentrates, $\EE(v_p-\bar\mu)^4\le60L^4/\Nout^2$
(\cref{lem:app-sgd-sphere}), and replacing $v_p$ by $\bar\mu$ gives (\cref{thm:app-sgd-half-drift})
\[
\EE_U\big\|U^\T M(U)-L_B\,\Lambda D\big\|_\F^2\le C_*,\qquad L_B=\frac{\mathcal N_B}{\sqrt{\pi\cR\,\bar\mu}},\qquad
\norm{L_B\Lambda D}_\F^2\asymp B .
\]
The row pairing given $U$ is $\langle D_{i\cdot},(U^\T M)_{i\cdot}\rangle$, so the Cauchy--Schwarz inequality
bounds the summed row errors by $\norm D_\F\sqrt{C_*}\le f_+\sqrt{C_*}$, against a total of order $\sqrt B$
(\cref{cor:app-sgd-half-pairings}), and $L_B\mu_i=\mathfrak d_i/\sqrt\cR$. For the risk, the Haar sign matrix (\cref{lem:app-sgd-reduction}) reduces the weighted
volatility to $\Nin(1-\frac2\pi)\Tr\Lambda+\frac2\pi\Nout\sum_j\EE[h_j^\T\Lambda h_j/\norm{h_j}^2]$, with
$h_j=U^\T G^\circ_{\cdot j}$, and the mean of each
ratio is $\Tr\Lambda^2/\Tr\Lambda$ up to $O(k^{-1/2})$ by concentration of quadratic forms of a column, which
gives $\sum_i\mu_i\mathfrak v_i^{\text{s-SGD}}$ (\cref{thm:app-sgd-half-volatility,cor:app-sgd-half-risk}).
\end{sketch}

\newpage
\section*{NeurIPS Paper Checklist}

\begin{enumerate}

\item {\bf Claims}
    \item[] Question: Do the main claims made in the abstract and introduction accurately reflect the paper's contributions and scope?
    \item[] Answer: \answerYes 
    \item[] Justification: We are explicit in the abstract and introduction that we are studying a simplified model, matrix-valued least squares problem, with both \Muon and \SignSVD. 

\item {\bf Limitations}
    \item[] Question: Does the paper discuss the limitations of the work performed by the authors?
    \item[] Answer: \answerYes{} 
    \item[] Justification: In the conclusion paragraph, we discuss the limitations of the work including the simplified setting as well as only analyzing the dynamics of \SignSVD over \Muon and \SignSGD over \Adam. We also only use the terms theorem, lemma, proposition, etc when indeed these statements are proven. For instance, we do not claim a theorem for the deterministic equivalent for the risk as we do not \textbf{prove} the concentration result. 

\item {\bf Theory assumptions and proofs}
    \item[] Question: For each theoretical result, does the paper provide the full set of assumptions and a complete (and correct) proof?
    \item[] Answer: \answerYes{} 
    \item[] Justification: We have explicitly stated assumptions labeled within the text itself (e.g., the assumption on the Gaussianity of the data). Statements for results are referenced to the appropriate section in the appendix. 

    \item {\bf Experimental result reproducibility}
    \item[] Question: Does the paper fully disclose all the information needed to reproduce the main experimental results of the paper to the extent that it affects the main claims and/or conclusions of the paper (regardless of whether the code and data are provided or not)?
    \item[] Answer: \answerYes{} 
    \item[] Justification: We provide a description of the least-squares problem in Section~\ref{sec:model_set_up}, the algorithms implemented in Section~\ref{sec:spectral_algs}, and all hyperparameter choices in figure captions, sufficient to reproduce the synthetic experiments.

\item {\bf Open access to data and code}
    \item[] Question: Does the paper provide open access to the data and code, with sufficient instructions to faithfully reproduce the main experimental results, as described in supplemental material?
    \item[] Answer: \answerNo{} 
    \item[] Justification: Code is not yet publicly released, but we intend to release the code for numerically solving the self-consistent equations for the risk. The synthetic least-squares experiments can be reproduced directly from the descriptions in the paper and figure captions without additional code.

\item {\bf Experimental setting/details}
    \item[] Question: Does the paper specify all the training and test details (e.g., data splits, hyperparameters, how they were chosen, type of optimizer) necessary to understand the results?
    \item[] Answer: \answerYes{} 
    \item[] Justification: All hyperparameters for the synthetic experiments—including the power law exponents, problem dimensions $\Nin$, $\Nout$, $B$, and optimizer settings—are specified in figure captions. Details for the numerical experiments with \Muon and momentum on the least squares problem are given in Section~\ref{sec:momentum-ablation}.  

\item {\bf Experiment statistical significance}
    \item[] Question: Does the paper report error bars suitably and correctly defined or other appropriate information about the statistical significance of the experiments?
    \item[] Answer: \answerYes{} 
    \item[] Justification: We report error bars in the synthetic experiments (see e.g., Fig.~\ref{fig:isotropic-main}), computed over multiple random seeds; figure captions specify the number of trials and what variability is shown.

\item {\bf Experiments compute resources}
    \item[] Question: For each experiment, does the paper provide sufficient information on the computer resources (type of compute workers, memory, time of execution) needed to reproduce the experiments?
    \item[] Answer: \answerYes{} 
    \item[] Justification: The synthetic least-squares experiments are inexpensive: each run completes in seconds to a few minutes on a single CPU core; total compute for all figures is negligible. A small number of preliminary runs ($\sim$10) were used to verify the numerical solver before generating the final figures. Res-Net with CIFAR-5M \cite{nakkiran2021bootstrap} data were also used to generate plots. These are generally inexpensive and standard tests for algorithms. 
    
\item {\bf Code of ethics}
    \item[] Question: Does the research conducted in the paper conform, in every respect, with the NeurIPS Code of Ethics \url{https://neurips.cc/public/EthicsGuidelines}?
    \item[] Answer: \answerYes{} 
    \item[] Justification: We have reviewed the code of ethics. We have made our utmost attempt to adhere to the guidelines provided by NeurIPS. We do not use any human subjects nor any datasets. We did our best to cite all the relevent related work. Given that our work is in the foundational research, it is difficult to mitigate all the risks as the downstream effects of theory are long, but we have done our best. The model is completely synthetic using the standard stochastic spectral algorithms; thus we don't, to the best of our knowledge, anticipate any risks. We have included a "Broader Impact" statement at the beginning of the "Supplemental Materials."

\item {\bf Broader impacts}
    \item[] Question: Does the paper discuss both potential positive societal impacts and negative societal impacts of the work performed?
    \item[] Answer: \answerYes{} 
    \item[] Justification: The work presented is purely foundational research and is not directly tied to any particular application. We study a simple matrix-valued least squares problem on synthetic dataset and we solve the model using common algorithms such as \Muon and \SignSGD. Given the theoretical nature of this work, we do not anticipate any direct ethical and societal issues. We do state some potential negative and positive societal effects because the implications of our results are on scaling behavior of these commonly used algorithms. See our Broader Impact Statement in the appendix.
    
\item {\bf Safeguards}
    \item[] Question: Does the paper describe safeguards that have been put in place for responsible release of data or models that have a high risk for misuse (e.g., pre-trained language models, image generators, or scraped datasets)?
    \item[] Answer: \answerNA{} 
    \item[] Justification: The paper poses no such risks. The majority of the work is theoretical and focuses on the matrix-valued least squares problem using only synthetic data generated from a normal distribution. We do not release any pretrained models. 

\item {\bf Licenses for existing assets}
    \item[] Question: Are the creators or original owners of assets (e.g., code, data, models), used in the paper, properly credited and are the license and terms of use explicitly mentioned and properly respected?
    \item[] Answer: \answerYes{} 
    \item[] Justification: We cite the original papers for \Muon~\citep{jordan2024muon} and \Adam, and include a detailed comparison with prior work in Section~\ref{sec:additional_related_work}. We additionally cite datasets such as CIFAR-5M \cite{nakkiran2021bootstrap} and software \cite{jax2018github} when we use them. 

\item {\bf New assets}
    \item[] Question: Are new assets introduced in the paper well documented and is the documentation provided alongside the assets?
    \item[] Answer: \answerNA{} 
    \item[] Justification: We do not release new datasets or models. We will release code for solving the deterministic equivalent for the risk; this will be documented in the repository.

\item {\bf Crowdsourcing and research with human subjects}
    \item[] Question: For crowdsourcing experiments and research with human subjects, does the paper include the full text of instructions given to participants and screenshots, if applicable, as well as details about compensation (if any)? 
    \item[] Answer: \answerNA{} 
    \item[] Justification: The work does not involve crowdsourcing nor research with human subjects. The data used in this work is generated synthetically. 

\item {\bf Institutional review board (IRB) approvals or equivalent for research with human subjects}
    \item[] Question: Does the paper describe potential risks incurred by study participants, whether such risks were disclosed to the subjects, and whether Institutional Review Board (IRB) approvals (or an equivalent approval/review based on the requirements of your country or institution) were obtained?
    \item[] Answer: \answerNA{} 
    \item[] Justification: The work does not involve crowdsourcing nor research with human subjects. The data used in this work is generated synthetically.

\item {\bf Declaration of LLM usage}
    \item[] Question: Does the paper describe the usage of LLMs if it is an important, original, or non-standard component of the core methods in this research? Note that if the LLM is used only for writing, editing, or formatting purposes and does \emph{not} impact the core methodology, scientific rigor, or originality of the research, declaration is not required.
    \item[] Answer: \answerYes{} 
    \item[] Justification: We include a declaration of LLM usage at the beginning of the appendix. LLMs were used to assist with code implementation for validating theoretical results (e.g., generating code to solve the deterministic equivalent), as well as to support parts of the proof development and simplify complex expressions. All LLM-generated outputs, particularly proofs, were carefully verified and corrected by the authors to ensure accuracy.

\end{enumerate}

\end{document}